\documentclass[11 pt, leqno]{amsart}
\usepackage[all]{xy}
\usepackage {amsfonts}
\usepackage{amsthm}
\usepackage{amssymb}
\usepackage{latexsym}
\usepackage{amsmath}
\usepackage{hyperref}
\usepackage{comment}
\usepackage{a4wide}
\usepackage{color}
\usepackage[dvipsnames]{xcolor}
\usepackage{comment}

\usepackage[utf8]{inputenc}
\usepackage[T1]{fontenc}
\usepackage{lmodern}
\usepackage{mathtools}
\usepackage[inline,shortlabels]{enumitem}
\usepackage{microtype}
\usepackage[normalem]{ulem}

\usepackage{enumitem} 
\setlist[enumerate,1]{label=\textup{(\arabic*)}}

\numberwithin{equation}{section}
\theoremstyle{plain}
\newtheorem{thm}{Theorem}[section]

\newtheorem {lem} [thm]{Lemma}
\newtheorem {prop}[thm] {Proposition}

\newtheorem{cor}[thm]{Corollary}

\theoremstyle{definition}
\newtheorem {defn}[thm] {Definition}
\newtheorem {rem} [thm]{Remark}
\newtheorem{ex}[thm]{Example}
\newtheorem {quest} [thm]{Question}

\newcommand*{\cstar}{\texorpdfstring{\(\Cst\)\nobreakdash-\hspace{0pt}}{*-}}

\newcommand*{\nb}{\nobreakdash}
\newcommand*{\Cst}{\mathrm \cst}
\newcommand{\Mat}{\mathbb M} 

\newcommand*{\congto}{\xrightarrow\sim}

\newcommand*{\contc}{\mathrm{C_c}}
\newcommand*{\contz}{\mathrm{C_0}}
\newcommand*{\contb}{\mathrm{C_b}}

\newcommand*{\s}{s} 
\newcommand*{\rg}{r}

\newcommand{\eqn}{\begin{equation}}
\newcommand{\eqne}{\end{equation}}

\renewcommand{\star}{*}

\DeclareMathOperator{\id}{id}

\DeclareMathOperator{\red}{red}
\DeclareMathOperator{\rd}{r}
\DeclareMathOperator{\f}{f}
\DeclareMathOperator{\dashind}{-Ind}

\DeclareMathOperator{\supp}{supp}
\DeclareMathOperator{\sem}{sem}
\DeclareMathOperator{\Bis}{Bis}

\DeclareMathOperator{\Ad}{Ad}

\renewcommand{\H}{\mathcal H}

\newcommand{\B}{\mathcal B}
\newcommand{\BB}{\mathbb B}
\newcommand{\LL}{\mathcal L}
\newcommand{\M}{\mathcal M}
\newcommand{\m}{m}
\newcommand{\G}{\mathcal G} 

\newcommand{\CC}{\mathcal C}

\newcommand{\EE}{\mathcal E}

\newcommand{\x}{\widetilde x}

\newcommand{\K}{\mathcal{K}}

\renewcommand{\L}{\mathcal L}

\DeclarePairedDelimiterX{\braket}[2]{\langle}{\rangle}{#1\,\delimsize\vert\,\mathopen{}#2}

\newcommand{\FF}{\mathcal F}
\newcommand{\A}{\mathcal A}
\newcommand{\D}{\mathcal D}
\newcommand{\C}{\mathbb C}

\newcommand{\N}{\mathbb N}
\newcommand{\T}{\mathbb T}
\newcommand{\E}{\mathbb E}

\newcommand{\cst}{\ifmmode\mathrm{C}^*\else{$\mathrm{C}^*$}\fi}

\newcommand*{\into}{\rightarrowtail}

\begin{document}

\author{Alcides Buss}
    \address{Departamento de Matem\'atica, Universidade Federal de Santa Catarina, 88.040-900 Florian\'opolis SC, Brazil}
	\email{alcides.buss@ufsc.br }

\author{Bartosz Kwa\'sniewski}
    \address{ Faculty of Mathematics, University of Bia\l ystok, ul. K. Cio\l kowskiego 1M, 15-245, Bia\l ystok, Poland}
	\email{bartoszk@math.uwb.edu.pl}

\author{Andrew McKee}
    \address{ Faculty of Mathematics, University of Bia\l ystok, ul. K. Cio\l kowskiego 1M, 15-245, Bia\l ystok, Poland}
	\email{a.mckee@uwb.edu.pl}

\author{Adam Skalski}
    \address{Institute of Mathematics of the Polish Academy of Sciences, ul.~\'Sniadeckich 8, 00--656 Warszawa, Poland}
    \email{a.skalski@impan.pl}

\date{\today}

\title{\bf Fourier--Stieltjes category \\
for twisted groupoid actions} 

\begin{abstract} 
We extend the theory of Fourier--Stieltjes algebras to the category of twisted actions by \'etale groupoids on arbitrary $\cst$\nb-bundles, generalising  theories constructed previously by B\'{e}dos and Conti for twisted group actions on unital $\cst$-algebras, and by Renault and others for groupoid $\cst$\nb-algebras, in each case motivated by the classical theory of Fourier--Stieltjes algebras of discrete groups. 
To this end we develop a toolbox including, among other things, a theory of multiplier $\cst$-correspondences, multiplier $\cst$-correspondence bundles, Busby--Smith twisted groupoid actions, and the associated crossed products, equivariant representations and Fell's absorption theorems. 
For a fixed \'etale groupoid $\G$ a Fourier--Stieltjes multiplier is a family of maps acting on fibres, arising from an equivariant representation. 
It corresponds to a certain fibre-preserving strict completely bounded map between twisted full (or reduced) crossed products. 
We establish a KSGNS\nb-type dilation result which shows that the correspondence above restricts to a bijection between positive-definite multipliers and a particular class of completely positive maps.
Further, we introduce a subclass of Fourier multipliers, that  enjoys a natural absorption property with respect to Fourier--Stieltjes multipliers and   gives rise  to `reduced to full' multiplier maps on crossed products. 
Finally, we provide several applications of the theory developed, for example to the approximation properties, such as weak containment or nuclearity, of the crossed products and actions in question, and discuss outstanding open problems.
\end{abstract}

\subjclass[2010]{Primary 46L05; Secondary  20E26, 22A22, 43A22, 46L55}

\keywords{\'etale groupoids; twisted groupoid actions; crossed products; Fourier--Stieltjes algebras; multipliers; $\cst$-correspondences}

\maketitle

\setcounter{tocdepth}{1}
\tableofcontents

\section*{Introduction}

The notions of Fourier--Stieltjes and Fourier algebras of general  locally compact groups, motivated by the concepts appearing earlier in classical, abelian harmonic analysis, were first introduced sixty years ago in the thesis of Eymard \cite{Eymard}. 
Recall that the \emph{Fourier--Stieltjes algebra} $B(\Gamma)$ of a given discrete group $\Gamma$ can be viewed in at least three different guises: 
\begin{enumerate}[(B1)]
	\item\label{enu:picture1} as the algebra of functions given by coefficients of unitary representations of $\Gamma$, with the pointwise multiplication;
	\item\label{enu:picture2} as the dual of the universal group $\cst$-algebra  $\cst(\Gamma)$, equipped with the `convolution' type product induced by the comultiplication of $\cst(\Gamma)$;
	\item\label{enu:picture3} as the span of completely positive (in short cp) Herz--Schur multipliers of $\cst(\Gamma)$, with the multiplication given as the composition of maps. 
\end{enumerate}
Similarly, the \emph{Fourier algebra} $A(\Gamma)$, which is an ideal in $B(\Gamma)$, can be seen as the algebra of coefficients of the left regular representation, as the predual of the group von~Neumann algebra $\mathrm{vN}(\Gamma)$, or as a certain subclass of Herz--Schur multipliers of $\cst(\Gamma)$ (or of $\mathrm{vN}(\Gamma)$). 
Note that developing the equivalence between the first and last points of view displayed above requires proving a version of a dilation result for positive-definite functions, which in this context is nothing but the GNS construction.
On the other hand, the fact that $A(\Gamma)$ is an ideal inside $B(\Gamma)$ is a consequence of Fell's absorption principle: the tensor product of a regular representation and an arbitrary representation is unitarily equivalent to a multiple of the regular representation.

This flexibility is one of the reasons why the aforementioned Banach algebras have played a significant role in the study of operator algebras related to groups and in the noncommutative harmonic analysis in general (we refer to the recent book \cite{KaniuthLau} for an introduction to the topic and several applications). 
It is thus natural that together with the development of various generalisations of group operator algebras the literature has seen several attempts to investigate analogues of Fourier--Stieltjes algebras in these new contexts. 
We would like to mention here for example the study of Fourier--Stieltjes algebras of locally compact quantum groups \`a la Kustermans--Vaes \cite{KuV}, as discussed for example in \cite{Daws}. 
The study of Fourier algebras, Fourier--Stieltjes algebras and related multiplier maps for  measure groupoids was initiated in the articles \cite{RenaultFourier,RamsayWalter,Oty}, with the main focus on trying to provide the equivalence between the first and the third point of view offered above; in particular on proving suitable versions of dilation and absorption results. 
The two specific cases which are most relevant for our paper are however the  ones encoding various types of non-commutative topological dynamics, i.e.\ that of the \'etale groupoid $\cst$-algebras and that of the $\cst$-algebraic twisted crossed products, which we will discuss next.

Locally compact Hausdorff \'etale grupoids were introduced to the $\cst$-algebraic world in the early 1980s in \cite{Renault0}, largely motivated by the developments related to equivalence relations and von~Neumann algebra theory, that later culminated in the celebrated characterisation of Cartan inclusions \cite{Re}.
An important special case of the \'etale groupoid $\cst$-algebra construction is that of the crossed product $C(X)\rtimes \Gamma$, where $X$ is a compact space equipped with an action of a discrete group $\Gamma$. 
It has however taken relatively long until the theory of Fourier--Stieltjes algebras and related multiplier maps was developed in the context of general $\cst$-algebraic crossed products (although \cite[Section 7.6]{Pedersen} provides a description of the state space of the crossed product $A \rtimes \Gamma$, and an attempt to introduce the algebraic structure on this space was undertaken in \cite{Fujita}). 
During the last few years we have seen a flow of work, notably by B\'edos and Conti, see \cite{BedosConti2,BedosConti3} and references therein, which studied `multiplier type' maps on  crossed products by twisted actions of discrete groups on \emph{unital} $\cst$-algebras. 
In particular, in the articles \cite{BedosConti,BedosConti2} both the dilation, for `kernels of positive type', and the absorption results for the so-called `induced representations' were established. 
The new context necessitated the use of Hilbert modules and $\cst$-correspondences, similarly to the quantum group framework mentioned earlier, as seen for example in \cite{Daws2}.

The articles \cite{dong_ruan} and \cite{mstt} use multiplier maps similar to those considered by B\'edos and Conti to study approximation properties of the crossed products in question. 
We should mention here that the key technique which makes the last application possible, and goes back to the context of group $\cst$-algebras, is the so-called \emph{Haagerup trick}, which provides a way to `average' arbitrary bounded linear maps on  the crossed product algebra into multipliers. 
In the context of twisted groupoid actions a version of the Haagerup trick was applied very recently in \cite{BartoszKangAdam} by two of the authors of this work, to study the relative Haagerup property of $\cst$-algebras in connection with the UCT property. 
This was possibly the first time multiplier type maps appeared for the groupoid crossed products.

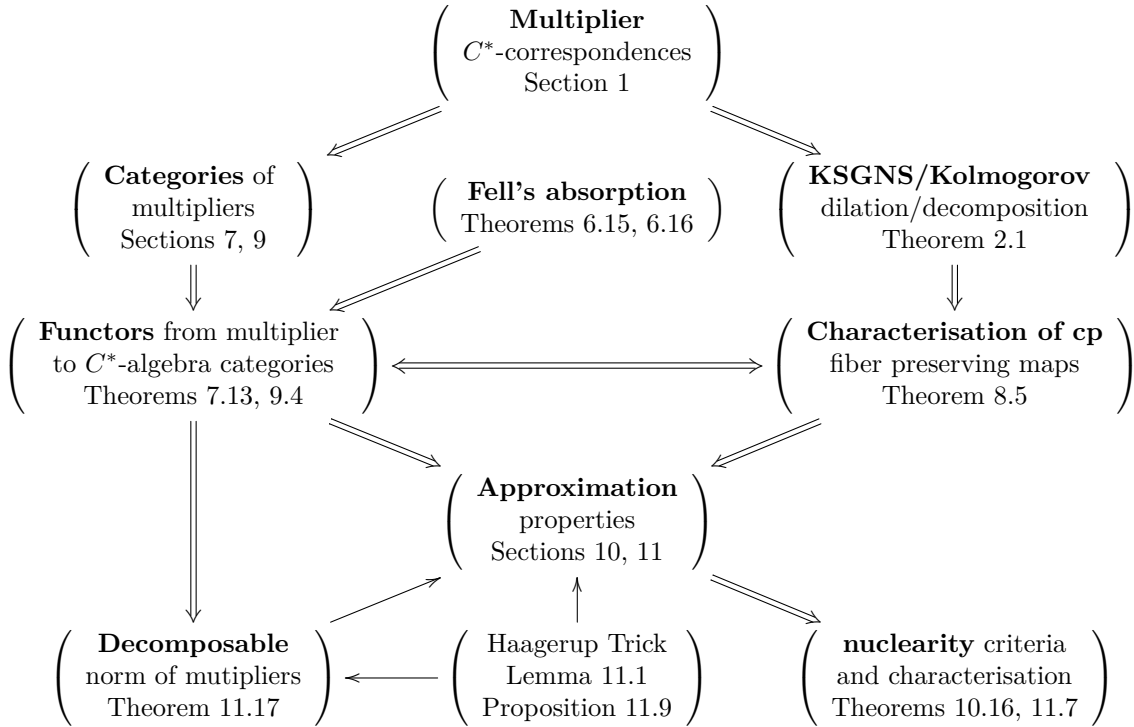
\begin{figure}[htb]
\centerline{\small
$   \xymatrixcolsep{0.8pc} \xymatrixrowsep{1.4pc}
\xymatrix{ 	
  &    
\text{$\left(\begin{array}{c} 
	\text{\textbf{Multiplier}} 
	\\
	\text{$C^*$-correspondences}
	\\
	\text{Section \ref{sec:multiplier}}
	\end{array}\right)$}  \ar@{=>}[ld] \ar@{=>}[rd] & 
	\\
\text{$\left(\begin{array}{c} 
	\text{\textbf{Categories} of  }
	\\
	\text{multipliers}
	\\
	\text{Sections \ref{Sec:FSmultipliers}, \ref{Sec:Fourierapproximation}}
	\end{array}\right)$} \ar@{=>}[d]  
	&   
			\text{$	\left(\begin{array}{c} 
	\text{\textbf{Fell's absorption}}
	\\ 	\text{Theorems \ref{thm:Fell_absorption}, \ref{thm:Fell_absorption II}}
\end{array}\right)  
	$}   \ar@{=>}[ld]
	& \text{$	\left(\begin{array}{c} 
\text{\textbf{KSGNS/Kolmogorov }}
	\\
	\text{dilation/decomposition}
\\
\text{Theorem \ref{thm:Murphy_Stinespring}}
\end{array}\right)  
	$}  \ar@{=>}[d] 
	\\
	\text{$\left(\begin{array}{c} 
	\text{\textbf{Functors} from multiplier }
	\\
	\text{to $C^*$-algebra categories}
	\\
	\text{Theorems \ref{thm:FSmultiplier_extends_to_reduced_and_full}, \ref{thm:Fmultipliers_are_reduced_to_full}}
	\end{array}\right)$} \ar@{<=>}[rr]   \ar@{=>}[rd] \ar@{=>}[dd]
		&  & \text{$	\left(\begin{array}{c} 
 \text{\textbf{Characterisation of cp}}
	\\
	\text{fiber preserving maps}
	\\
	\text{Theorem \ref{thm:FSmultiplier_positive_definite}}
	\end{array}\right)  
	$} 
		\ar@{=>}[ld]
		\\
  &  \text{$\left(\begin{array}{c} 
	\text{\textbf{Approximation}}
	\\
	\text{properties}
	\\
	\text{Sections \ref{Sec:Applicationapproximation}, \ref{Sec:Applications}}
	\end{array}\right)$}   \ar@{=>}[rd]
	& 
\\
	\text{$\left(\begin{array}{c} 
	\text{\textbf{Decomposable}}
	\\
	\text{norm of mutipliers}
	\\
	\text{Theorem \ref{thm:FSmultipliers_decomposable_discrete}}
	\end{array}\right)$} \ar@{->}[ru]  
		&
	\text{$	\left(\begin{array}{c} 
 \text{Haagerup Trick}
	\\ 	\text{Lemma \ref{lem:Haagerup trick discrete}}
	\\ \text{Proposition \ref{prop:Haagerup_trick}}
	\end{array}\right)  
	$}	\ar@{->}[l] \ar@{->}[u]		& 
		\text{$	\left(\begin{array}{c} 
 \text{\textbf{nuclearity} criteria}
	\\
	\text{and characterisation}
	\\
	\text{Theorems \ref{thm:approximation_prop_implies_weak_containment}, \ref{thm:characterisation of nuclearity}}\end{array}\right)  
	$}	
		}
	$}
	\caption{Structure and relationship between developed steps/results}\label{fig.structure}
\end{figure}

In this paper, inspired by the developments mentioned above, we undertake a systematic study of Fourier--Stieltjes and Fourier multipliers for crossed products by twisted actions of \'etale groupoids on arbitrary $\cst$-bundles. 
To include the work of B\'edos and Conti~\cite{BedosContiregular,BedosConti,BedosConti2} as a special case, we consider here twisted groupoid actions \'a la Busby--Smith~\cite{BusbySmith} with twists taking values in multiplier bundles. 
Such actions appear in \cite{BussMeyerZhu} but their crossed products, apart from the group case, have not yet been studied.
In particular, they are different from the Green--Renault twisted groupoid actions considered in \cite{Renault,Renault2}, where the twist is abelian (usually scalar valued). 
It has also occurred to us that, instead of looking at Banach algebras of multipliers for a single action, it is more natural and useful to study \emph{Banach categories of multipliers} between different actions of the same groupoid. 
The main steps in our development, as well as the relationships between them, are schematically presented in Figure~\ref{fig.structure}.

Apart from the novelties mentioned above we also for the first time consider crossed product multipliers for actions on \emph{non-unital} $C^*$-algebras.
To do so we develop systematically a theory of \emph{multiplier $C^*$-correspondences}.
It seems that so far this notion was considered only in the special case of Hilbert bimodules \cite{ER,Schweizer}, and we believe that the results of this investigation are of independent interest and will be applicable in a variety of contexts. 
With that tool in hand, establishing  a general framework of \emph{bundles of $C^*$-correspondences} and  of \emph{equivariant correspondences} that generalises notions appearing in \cite{LeGall,EKQR} and \cite{Deaconu}, we construct  Banach categories of \emph{Fourier--Stieltjes} and \emph{Fourier multipliers} between twisted actions of a fixed groupoid. 
This is a far reaching generalisation of the first picture \ref{enu:picture1} mentioned above for groups. 
We prove three different \emph{absorption results} to achieve our goals. 
The first of them implies that the Fourier category is an \emph{ideal} in the Fourier--Stieltjes category. 
The remaining two provide different generalisations of the classical \emph{Fell's absorption principle} for groups. 
The first implies that Fourier--Stieltjes multipliers act as completely bounded maps both on the reduced and universal crossed product algebras, which provides functors into the respective categories. 
The second one ensures that the Fourier multipliers even act `reduced-to-full'. 
In the context of group actions similar versions of Fell's absorption principle were proved in \cite{BedosContiregular, BedosConti}. 
For groupoid actions our second Fell's absorption principle is closely related, but not exactly comparable with, the corresponding recent results in \cite{Kranz} and \cite{BussMartinez}.

Having proved the above results we then focus on \ref{enu:picture3} --- the third of the pictures mentioned in the beginning of the introduction (as finding the right generalisation of the second, `functional' avatar is problematic even in the context of group crossed products, as can be seen by looking at \cite[Section 7.6]{Pedersen}). 
In order to obtain the desired equivalence we prove a dilation result, generalising simultaneously the KSGNS construction of Kasparov (\cite{Lance}) and Murphy's Kolmogorov decomposition of positive kernels for Hilbert modules from \cite{Mur}, and which might be of interest for its own sake. 
Using this we prove a core result of the paper.
It states that for any two twisted actions $(\alpha,u_{\alpha})$ and $(\beta, u_{\beta})$ of a given \'etale groupoid $\G$ we have natural bijections between the following objects: 
\begin{enumerate}
	\item\label{enu:intro_core_1} strict, bounded and positive-definite multipliers from $(\alpha,u_{\alpha})$ to $(\beta, u_{\beta})$;
    \item\label{enu:intro_core_2} strict completely positive fibre preserving maps between the reduced crossed products;
	\item\label{enu:intro_core_3} strict completely positive fibre preserving maps between the full crossed products;
	\item\label{enu:intro_core_4} Fourier--Stieltjes multipliers from $(\alpha,u_{\alpha})$ to $(\beta, u_{\beta})$ of the form $T_{\EE,L,\xi,\xi}$ 
	for some equivariant correspondence $(L,\EE)$ and a section $\xi\in \contb(\M(\EE))$.
\end{enumerate}	
This together with the standard polarization formula applied to \ref{enu:intro_core_4} gives  the appropriate generalisation of the equivalence between the guises \ref{enu:picture1} and \ref{enu:picture3}. 
The strictness of completely positive maps appearing in conditions \ref{enu:intro_core_1}--\ref{enu:intro_core_3} above is a very weak replacement for unitality of the algebras and maps in question. 
In fact, for the approximation properties that we study strictness plays no role, as we can always conjugate with approximate units to `strictify' a given map. 
Our categorical approach allows us to investigate \emph{approximation properties} of arbitrary multipliers. 
When applied to the identity this leads to a fruitful interaction between properties of the Fourier--Stieltjes algebra and the approximation properties of the corresponding crossed product algebra. 
For example, we provide natural approximation conditions guaranteeing that the relevant universal and reduced crossed product $\cst$-algebras coincide and are nuclear. 

We also return to the Haagerup trick. 
For general groupoids and general cp maps it is not completely clear what is the right replacement for this device, so we decided to consider two cases where we know that it works (this is reflected in Figure~\ref{fig.structure} by the use of thinner arrows). 
Namely, we assume that either the groupoid is discrete (which is still interesting, as for example it covers the case of discrete groups) or the maps in question have a certain bimodule property. 
Then we can use it to characterise the relevant approximation properties, and in the discrete case to identify the natural Fourier--Stieltjes norm with the decomposable norm of \cite{Haagerupdec} of the associated decomposable map, as well as give full characterisations of nuclearity of the associated crossed products.

The construction of the Fourier--Stieltjes category in the generality described above is not yet the end of the story. 
One could think of further extensions to multipliers of Fell bundles, with the latter already put to use for example in \cite{Exel:amenability,BussMartinez}. 
On the other hand, even the context of twisted crossed products given by groupoid actions offers many new perspectives. 
These are on one hand related to the analysis of more subtle approximation properties, and on the other to potential reconstruction results, related to the question: to what extent does the Fourier--Stieltjes or Fourier algebra remember the underlying action? 
At the end of the text we indicate several related avenues of future research and mention also specific problems left open in our work.

The specific plan of the paper is as follows: in Section~\ref{sec:multiplier}, we introduce multiplier modules, $\cst$-correspondences and Hilbert bimodules. 
In Section~\ref{sec:decomp} we prove a dilation result for a particular form of `positive kernels'. 
Next, in Section~\ref{Sec:bundles}, we return to correspondences in the context of bundles over a locally compact space and describe direct sums, tensor products and natural morphisms between the respective bundles. 
Section~\ref{Sec:Fell_bundles} passes to Fell bundles, and specifically to those which arise from twisted groupoid actions and whose cross sectional $C^*$-algebras give the relevant crossed products. 
Here we also connect the latter to twisted inverse semigroup actions in the sense of \cite{BussExel}. 
In Section~\ref{Sec:equivariant} we focus on equivariant correspondences for a pair of twisted groupoid actions, setting up the appropriate framework and proving that an \emph{equivariant correspondence induces} representations from one crossed product to another. 
The process of inducing representations from correspondences without equivariant action is introduced in Section~\ref{Sec:InducedAbsorption}. 
Such representations are called \emph{induced regular} because they always descend to reduced crossed products, while for equivariant induction this holds only if the equivariant representation is regular or if the representation from which we induce is  regular.
Our two counterparts of Fell's absorption principle describe the fundamental relationships between the two types of induction. 
Section~\ref{Sec:FSmultipliers} uses the stage set up in the first part of the paper to finally introduce the Fourier--Stieltjes Banach category and prove that Fourier--Stieltjes multipliers act both on reduced and full crossed product $\cst$-algebras. 
The following Section~\ref{Sec:positive-definite} specialises to Fourier--Stieltjes multipliers of positive type; here we also establish the advertised equivalence between different points of view on such multipliers. 
A short Section~\ref{Sec:Fourierapproximation} treats Fourier multipliers, using earlier absorption theorems to deduce that they form an ideal in the class of Fourier--Stieltjes multipliers and extend to reduced-to-full maps. 
Section~\ref{Sec:Applicationapproximation} introduces the concepts of a Fourier and Fourier--Stieltjes approximation of a given multiplier and uses it to provide applications to approximation properties of actions and the resulting crossed products. 
In Section~\ref{Sec:Applications} we discuss appropriate versions of the Haagerup trick, apply them to find other descriptions of approximation properties, as well as to provide an alternative view of the Fourier--Stieltjes norm, which in turn again has certain consequences for the approximation properties. 
Finally, in Section~\ref{Sec:QuestionsPerspectives}, we set up open problems and present further research perspectives.

\vspace*{0.3cm}

\textbf{Notation:} We will introduce the notation gradually throughout the text, but the most important conventions are as follows. 
We denote by $X$ a locally compact Hausdorff space and by $\G$ a Hausdorff \'etale groupoid with the unit space $X$. 
$C^*$-algebras are denoted by $A$, $B$, $\dots$, Fell bundles by $\A$, $\B$,  $\dots$, modules by $E$, $F$, and bundles of modules by $\EE$, $\FF$. 
If $p$ is a logical statement --- mostly an equality --- then we use the symbol $[p]$ to mean the boolean logical value of $p$, that is $[p]=1$ if $p$ is true and $[p]=0$ if $p$ is false.

\section{Multiplier C*-correspondences} 
\label{sec:multiplier}

In this section we discuss the notion of multiplier Banach bimodules, multiplier $\cst$-correspondences and multiplier Hilbert bimodules, extending the theory from \cite{Daws}.

\subsection{Multiplier Banach bimodules}

Throughout this paper $A$ and $B$ will be $\Cst$-algebras 
(although in this particular subsection it suffices to assume that $A$ and $B$ are Banach algebras with contractive approximate units).
For Banach spaces $E$, $F$, we denote by $\BB(E,F)$ the space of bounded operators $E\to F$.
Let $E$ be a Banach $A$-$B$-bimodule. 
In this paper Banach modules, by definition, will be \emph{nondegenerate}, 
i.e.\ we always assume that $AE = EB = E$. 
Note that, by the Cohen--Hewitt factorisation theorem, these products may be viewed either purely algebraically or as closed linear spans. 
Then $E$ is also naturally a Banach $\M(A)$-$\M(B)$-bimodule, 
where $\M(A)$ and $\M(B)$ denote the multiplier algebras of $A$ and $B$, respectively.
We recall the definition of the multiplier bimodule from \cite{Daws} 
(note that \cite{Daws} considers $A$-bimodules, but the constructions and results there remain valid for $A$-$B$-bimodules). 

Namely, the direct sum $\BB(A,E)\oplus\BB(B,E)$, equipped with the norm 
$$\|(R,L)\|_{\infty} = \max\{\|R\|,\|L\|\},
$$
becomes a Banach $\M(A)$-$\M(B)$-bimodule 
with module actions given by
\[
    (aR)(a_0) = R(a_0 a), \quad (Rb)(a_0) = R(a_0)b, \quad 
    (Lb)(b_0) = L(bb_0), \quad (aL)(b_0) = aL(b_0),
\]
for $(R,L)\in \BB(A,E)\oplus\BB(B,E)$, $a_0\in A$, $b_0\in B$, and $a\in\M(A)$, $b\in\M(B)$.
A \emph{multiplier} for $E$ is a pair of maps $(R,L)$, where $R:A\to E$ and $L:B\to E$, satisfying
\[
    a_0 L(b_0) = R(a_0)b_0, \qquad a_0\in A,\ b_0\in B.
\]
By \cite[Proposition~2.5]{Daws}, the set of all multipliers for $E$, denoted $\M(E)$, 
is a closed $\M(A)$-$\M(B)$-sub-bimodule of $\BB(A,E)\oplus\BB(B,E)$. 

We have an isometric $\M(A)$-$\M(B)$-bimodule map 
\[
    E\ni \xi \longmapsto (R_\xi,L_\xi), \qquad 
    R_\xi(a):=a\xi,\ L_\xi(b):=\xi b,
\]
which allows us to view $E$ as a sub-bimodule of $\M(E)$.
We call the $\M(A)$-$\M(B)$-bimodule $\M(E)$ the \emph{multiplier bimodule} of the $A$-$B$-bimodule $E$.
Below we will characterise it as the largest essential multiplier extension of~$E$.

\begin{defn}\label{defn:multiplier_extensions}
We say that an $\M(A)$-$\M(B)$-bimodule $F$ is a \emph{multiplier extension} of an $A$-$B$-bimodule $E$ if $E\subseteq F$ is a submodule such that $AF=FB=E$. 
We say that this extension is \emph{essential} if for every $\xi \in F$, we have $A\xi=0$  only if $\xi=0$ or equivalently, $\xi B=0$  only if $\xi=0$ (the equivalence follows from the equality $AF=FB=E$).
\end{defn}

\begin{rem}
If either of the algebras $A$ or $B$ is unital then $F=1_AF=F1_B=E$ for every multiplier extension $F$ of $E$.
\end{rem}

\begin{lem}\label{lem:universal_characterisation_of_multipliers}
For any $A$-$B$-bimodule $E$, the multiplier bimodule $\M(E)$ is the largest essential multiplier extension of $E$. 
More precisely, $\M(E)$ is an essential multiplier extension of $E$ and for any other multiplier extension $F$ of $E$ there is a unique (necessarily contractive) 
$A$-$B$-bimodule map $\mu:F\to \M(E)$ that extends the identity on $E$. 
Moreover, $\mu$ is injective if and only if $F$ is essential. 
\end{lem}
\begin{proof}
For any $(R,L)\in \M(E)$ and $a\in A$, $b\in B$ we have $a\cdot (R,L)=(R_{R(a)},L_{R(a)})$ and $(R,L)\cdot b=(R_{L(b)},L_{L(b)})$ (note that $L$ is a right module map and $R$ is a left module map, see \cite[Lemma 2.1]{Daws}). 
Hence $\M(E)$ is a multiplier extension of $E$ and clearly it is essential. 
For any multiplier extension $F$ we have the bimodule map $F\ni \xi \mapsto (R_\xi, L_{\xi})\in \M(E)$ where $R_{\xi}(a):=a \xi$ and $L_\xi(b):=\xi b$, for $\xi \in F$.
By nondegeneracy of bimodules this is the unique bimodule map $F\to \M(E)$ that restricts to the identity on $E$.  
By construction it is contractive, and it is injective if and only if $F$ is essential. 
\end{proof}

\begin{cor}\label{cor:universal_multipliers}
$\M(E)$ is the unique (up to an isometric isomorphism extending the identity on $E$) multiplier extension of $E$ with the property that for any multiplier extension $F$ of $E$ there is a unique bimodule map $F\to \M(E)$ that extends the identity $E\to E$. 
\end{cor}
\begin{proof} 
Follows from the usual universality considerations.
\end{proof}

We also need the following general morphism extension result. 
It is a common generalisation of \cite[Theorem 2.8]{Daws} and the known fact that any contractive homomorphism $\pi:A\to \M(B)$ which is \emph{nondegenerate} in the sense that $\pi(A)B=B$, extends uniquely to a homomorphism $\overline{\pi}:\M(A)\to \M(B)$ which is  necessarily contractive.

\begin{defn}\label{def:bimodule_morphism}
A \emph{morphism} from  a Banach $A$-$B$-bimodule $E$ to a Banach $C$-$D$-bimodule $F$ consists of a linear map $\pi:E\to \M(F)$ and nondegenerate homomorphisms $\pi_A:A\to \M(C)$, $\pi_B:B\to \M(D)$ such that $\pi(a\xi b)=\pi_A(a)\pi(\xi)\pi_B(b)$ for $\xi \in E$, $a\in A$, $b\in B$.
We say that this morphism is \emph{contractive} if all the maps $(\pi,\pi_A,\pi_B)$ are contractive.
\end{defn}

\begin{lem}\label{lem:extensions_lemma}
Any contractive morphism $(\pi,\pi_A,\pi_B):E\to F$ extends uniquely to a morphism $(\overline{\pi}, \overline{\pi}_A, \overline{\pi}_B):\M(E)\to \M(F)$ which is necessarily contractive.
\end{lem}
\begin{proof}
Take any $(R,L)\in \M(E)$. 
We claim that the formulas
\[
    \widetilde{R}(c\pi_A(a)):=c\pi(R(a)), \qquad \widetilde{L}(\pi_B(b)d):=\pi(L(b))d,\quad\mbox{ for $a\in A, b\in B, c\in C, d\in D$,}
\]
yield linear maps $\widetilde{R}:C\to F$, $\widetilde{L}:D\to F$. 
Indeed, letting $\{e_i^A\}_{i \in I}$ and $\{e_i^B\}_{i \in I}$ be approximate units respectively in $A$ and $B$, we get $c\pi(R(a))=\lim_{i} c\pi(R(ae_i^A))=\lim_{i} c\pi(aR(e_i^A))=\lim_{i} c\pi_A(a)\pi(R(e_i^A))$ and similarly  $\pi_B(b)d=\lim_{i}\pi(L(e_i^B))\pi_B(b)d$. 
This implies that $\widetilde{R}$ and $\widetilde{L}$ are well defined and contractive. 
We have 
\[
    \widetilde{R}(c)=\lim_{i} c \pi(R(e_i^A)), \qquad \widetilde{L}(d)=\lim_{i} \pi(L(e_i^B)) d, \qquad c\in C, d\in D.  
\]
Moreover, for $a\in A, b\in B, c\in C, d\in D$ we get
\[
    c\pi_A(a) \widetilde{L}(\pi_B(b)d) = c\pi(aL(b))d = c\pi(R(a)b)d= c\pi(R(a))\pi_B(b)d = \widetilde{R}(c\pi_A(a)) \pi_B(b)d.
\]
Hence $(\widetilde{R},\widetilde{L})\in \M(F)$. 
Thus putting $\overline{\pi}(R,L):=(\widetilde{R},\widetilde{L})$ we get a contractive  extension $\overline{\pi}:\M(E)\to \M(F)$ of $\pi$.
It can be readily checked that for $a\in A$, $b\in B$ and $\xi=(R,L)\in \M(E)$ we have $\pi_A(a)\overline{\pi}(\xi)\pi_B(b)=\pi(a\xi b)$. 
This property determines $\overline{\pi}$ uniquely, as for any other $\tilde{\pi}:\M(E)\to \M(F)$ with this property we have $(\overline{\pi}(\xi)-\tilde{\pi}(\xi))\pi_B(B)=0$ which implies $(\overline{\pi}(\xi)-\tilde{\pi}(\xi))D=0$ and therefore $\overline{\pi}(\xi)=\tilde{\pi}(\xi)$ because $\M(F)$ is an essential extension of $F$.

The above (in the special case where $E=A=B$ and $F=C=D$ are Banach algebras and $\pi=\pi_A=\pi_B$) also proves the aforementioned fact that contractive nondegenerate homomorphisms between Banach algebras extend uniquely to homomorphisms between their multiplier algebras, and they are necessarily contractive. 
Thus we have unique (contractive) homomorphic extensions $\overline{\pi}_A:\M(A)\to \M(C)$ and $\overline{\pi}_B:\M(B)\to \M(D)$ of  $\pi_A$ and $\pi_B$. We immediately get that $\overline{\pi}_A(a)\overline{\pi}(\xi)\overline{\pi}_B(b)=\overline{\pi}(a\xi b)$ for $a\in \M(A)$, $b\in \M(B)$ and $\xi\in \M(E)$. 
Thus $(\overline{\pi},\overline{\pi}_A,\overline{\pi}_B)$ is a unique morphism that extends $(\pi,\pi_A,\pi_B)$.
\end{proof}

\subsection{Multiplier C*-correspondences}

Let $E$ and $F$ be (right) $B$-Hilbert modules. 
We use the standard notation $\L(E,F)$ for the Banach space of adjointable operators from $E$ to $F$, and $\K(E,F)\subseteq \L(E,F)$ for the closed subspace of (generalised) compact operators, generated by the `rank-one' operators $\Theta_{x,y}(z):=x\langle y, z\rangle_B$, $x \in F, y, z\in E$. 
In particular, $\K(E):=\K(E,E)$ is an ideal in the $\cst$-algebra $\L(E):=\L(E,E)$. 
A $\cst$\emph{-correspondence} from $A$ to $B$ (or an $A$-$B$-$\cst$-correspondence) is a Hilbert module $E$ over $B$ together with a nondegenerate $*$-homomorphism from $A$ to $\L(E)$, which introduces a left $A$-module structure on $E$.
In particular, an $A$-$B$-$\cst$-correspondence is a Banach $A$-$B$-bimodule (where both actions are nondegenerate), so that the constructions of the previous subsection apply.
Any $\cst$-algebra $A$ can be treated as a trivial $\cst$-correspondence from $A$ to $A$ where the bimodule structure comes from $A$ and the $A$-valued inner product is given by the formula $\langle a, b\rangle:=a^*b$, $a,b\in A$. Then $\K(A)\cong A$ and $\L(A)=\M(A)$ is the multiplier algebra of $A$. 
For any $A$-$B$-$\cst$-correspondence $E$ the actions extend uniquely to multiplier algebras so that we may always treat $E$  as a $\M(A)$-$\M(B)$-$\cst$-correspondence. 
To relate it to the multiplier $\M(A)$-$\M(B)$-bimodule $\M(E)$, let us note that whenever $E_B$ is a $B$-Hilbert module, then $\L(E_B,E)$ is naturally a $\cst$-correspondence from $\M(A)$ to $\L(E_B)$ where
\[
    (a T)(x) := a T(x),\qquad (T b) (x):= Tb(x),\qquad \langle S, T\rangle_{\L(E_B)} :=S^* T,
\]
for $S, T\in  \L(E_B,E)$, $a\in M(A)$, $b\in \L(E_B)$, $x\in E_B$. 
In particular, $\L(E_B,E)$ contains $\K(E_B,E)$ as a sub-$\cst$-correspondence. 
If $E_B=B$ is the trivial $B$-Hilbert module, then $\L(B,E)$ is a $\cst$-correspondence from $\M(A)$ to $\M(B)=\L(B)$ and $\K(B,E)$ is naturally isomorphic to $E$; the isomorphism is determined by the map $\Theta_{\xi,b}\mapsto \xi b^*$, $\xi \in E, b \in B$. 
Under the identifications $B=\K(B)$ and $E=\K(B,E)$ we have $\L(B,E)\cdot B=E$, but in general $A\cdot \L(B,E)$ is not contained in $E$, unless for instance $E$ is \emph{proper}, i.e.\ $A$ acts on $E$ by compact operators. 
Hence in general $\L(B,E)$ is not a multiplier extension of $E$ in the sense of Definition~\ref{defn:multiplier_extensions}.
In fact, 
\[
    \M_{\cst}(E):=\{T\in \L(B,E):  aT\in \K(B,E) \text{ for all }a\in A\}
\]
is the largest sub-bimodule of $\L(B,E)$ which is a multiplier extension of $E=\K(B,E)$. 

\begin{lem}
For any $A$-$B$-$\cst$-correspondence $E$, the projection $\M(E)\ni (R,L)\mapsto L \in \M_{\cst}(E)$ is an isomorphism of Banach $\M(A)$-$\M(B)$-bimodules.
\end{lem}
\begin{proof}
If $(R,L) \in \M(E)$, then $L$ is adjointable and $L^*(a\xi)=\langle R(a^*),\xi\rangle_{B}$, for $a\in A$ and $\xi\in E$, because for any $b\in B$ we have 
\[
    \langle L(b) ,a\xi\rangle_{B}=\langle a^* L(b),\xi\rangle_{B}=\langle R(a^*)b,\xi\rangle_{B}= b^*\langle R(a^*),\xi\rangle_{B}= \langle b, \langle R(a^*),\xi\rangle_{B}\rangle_{B}.
\]
Thus $R$ is uniquely determined by $L$ and $\|R\|\leq \|L^*\|=\|L\|$. 
Also for any $a\in A$, writing $R(a)=\xi b^*$ for $\xi\in E$ and $b\in B$ we get that 
$a L=\Theta_{\xi,b}$. 
Thus $L\in \M_{\cst}(E)$ and we see that  $\M(E)\ni (R,L)\mapsto L \in \M_{\cst}(E)$ is an isometric embedding of bimodules.  
To show it is onto let $L\in \M_{\cst}(E)$. 
For any $a\in A$, $aL\in \K(B,E)$ and therefore for any approximate unit $\{e_i\}_{i \in I}$ in $B$ the limit $\lim_{i\in I} aL(e_i)\in E$ exists (because $aL$ can be  approximated in the norm by a finite sum $\sum_{k=1}^n\Theta_{\xi_k,b_k}$ and $\lim_{i\in I} \sum_{k=1}^n\Theta_{\xi_k,b_k}(e_i)=\sum_{k=1}^n \xi_kb_k^*\in E$). 
Putting $R(a):=\lim_{i} aL(e_i)$ one gets a multiplier $(R,L)$. 
We leave the remaining straightforward details to the reader.
\end{proof}

In the sequel we shall use the above lemma to identify $\M(E)=\M_{\cst}(E) \subseteq \L(B,E)$. 
The advantage is that then we have also the following equality
\[
    \M(E)=\{T\in \L(B,E): \langle S, aT\rangle_{\M(B)}\in B \text{ for }a\in A,\, S\in \L(B,E)\},
\]
see for instance \cite[Proposition 1.3]{katsura2}, which readily implies that $\M(E)$ is in fact a sub-$\cst$-correspondence of the  $\M(A)$-$\M(B)$-$\cst$-correspondence $\L(B,E)$.

\begin{rem} 
A $\cst$-correspondence multiplier extension of an $A$-$B$-$\cst$-correspondence $E$ is a $\M(A)$-$\M(B)$-$\cst$-correspondence $F$ such that $E=FB$ and 
\[
    \langle \xi, a \eta\rangle_{\M(B)}\in B, \qquad \text{ for all } a\in A,\  \xi, \eta \in F,
\]
because the displayed condition is equivalent to $AF=FB$, by \cite[Proposition 1.3]{katsura2}. 
Any $\cst$-correspondence multiplier extension is automatically essential.
\end{rem}

\begin{prop}\label{prop:universal_multiplier_description}
The multiplier bimodule $\M(E)$ of an $A$-$B$-$\cst$-correspondence $E$ has a uniquely determined structure of an $\M(A)$-$\M(B)$-$\cst$-correspondence extending the one from $E$. 
Moreover, $\M(E)$ is the largest among $\cst$-correspondences which are multiplier extensions of $E$. 
Namely, for any $\cst$-correspondence multiplier extension $F$ of $E$, the unique bimodule map $F\to \M(E)$ that extends the identity on $E$ is necessarily an isometry (preserves inner products). 
\end{prop}
\begin{proof} 
We already observed that $\M(E)$ has the $\cst$-correspondence structure inherited from $\L(B,E)$.
By Lemma~\ref{lem:universal_characterisation_of_multipliers}, the Banach $A$-$B$-bimodule structure of $\M(E)$ uniquely extends the one from $E$, and the  inner product $\langle \cdot, \cdot \rangle_{\M(B)}$ in $\M(E)$ is determined by the one $\langle \cdot, \cdot \rangle_{B}$ in $E$ because for any approximate unit $\{e_i\}_{i \in I}$ in $B$, and any $\xi,\eta \in F$ we have 
\[
    \langle \xi, \eta\rangle_{\M(B)}=\text{s-}\!\lim_{i,j\in I} e_j\langle \xi, \eta\rangle_{\M(B)} e_i=\text{s-}\!\lim_{i,j\in I} \langle \xi e_j , \eta e_i \rangle_{B},
\] 
where $\text{s-}\!\lim$ denotes the strict limit. 
Let $F$ satisfy the conditions in the second part of the assertion. 
Then $F$ is an essential multiplier extension of $E$, and so the identity on $E$ uniquely extends to a bimodule map $u:F\to \M(E)$ by Lemma~\ref{lem:universal_characterisation_of_multipliers}.
It necessarily preserves the inner products as for an approximate unit $\{e_i\}_{i \in I}$ in $B$, and  $\xi,\eta \in F$, as above we have $\langle u\xi, u\eta\rangle_{\M(B)}=\text{s-}\!\lim_{i,j\in I} e_j\langle u(\xi), u(\eta)\rangle_{\M(B)} e_i=\text{s-}\!\lim_{i,j\in I} \langle u(\xi e_j) , u(\eta e_i) \rangle_{B}=\text{s-}\!\lim_{i,j\in I} \langle \xi e_j , \eta e_i \rangle_{B}=\langle \xi, \eta\rangle_{\M(B)}$. 
\end{proof}


\begin{rem}  
An $A$-$B$-\emph{Hilbert bimodule} is an $A$-$B$-$\cst$-correspondence $E$ which is also a left $A$-Hilbert module and the inner products satisfy ${_{A}\langle} \xi, \zeta \rangle \cdot \eta =  \xi \cdot \langle \zeta,  \eta \rangle_{B}$, $\xi, \zeta, \eta\in E$. 
Equivalently, $E$ is an $A$-$B$-$\cst$-correspondence such that the left action  $\varphi:A \to \LL(E)$ restricts to a $*$-isomorphism $I\cong \K(E)$ where $I$ is an ideal in $A$ (we then necessarily have $I=(\ker\varphi)^\bot$). 
If $E$ is an  $A$-$B$-Hilbert bimodule, then  $M(E)$ is  $\M(A)$-$\M(B)$-Hilbert bimodule that agrees with the multiplier Hilbert bimodule constructed in \cite{Schweizer},  and earlier for equivalence bimodules in \cite{ER}. 
In particular, Proposition~\ref{prop:universal_multiplier_description} together with Corollary~\ref{cor:universal_multipliers}, generalise and improve \cite[Proposition 3.6]{Schweizer}, \cite[Proposition 1.2]{ER}.
\end{rem}

\begin{ex}[Matrix $\cst$-correspondences]\label{ex:matrixcorr}
Consider a $\cst$-$A$-$B$-correspondence $E$ and $n \in \N$. 
Then, as described in \cite[Section 3]{Ble}, the space
\[
    \Mat_n(E):=\{[\xi_{ij}]_{i,j=1}^n: \xi_{ij} \in E ,\ i,j=1,\ldots,n\}.
\]
becomes a Hilbert $\Mat_n(B)$-\cst-module with the natural right matrix multiplication and the inner product of $\xi=[\xi_{ij}]_{i,j=1}^n$,  $\zeta=[\zeta_{ij}]_{i,j=1}^n \in \Mat_n(E)$ given by 
\[ 
    \langle \xi, \zeta\rangle_{M_n(B)}= \left[\sum_{k=1}^n \langle \xi_{ki}, \zeta_{kj} \rangle_{B} \right]_{i,j=1,\ldots,n}.
\]
If the left action of $A$ on $E$ is given by $\phi:A\to \L(E)$, then the left action of $\Mat_n(A)$ on $\Mat_n(E)$ is given by $\phi^{(n)}:\Mat_n(A) \to \Mat_n(\L(E)) \subseteq \L(\Mat_n(E))$, where the last inclusion is given by $T=[T_{ij}]_{i,j=1}^n \in \Mat_n(\L(E))$ and $\xi \in \Mat_n (E)$ is as above, then 
\[ 
    T \xi = \left[\sum_{k=1}^n  T_{ik} (\xi_{kj}) \right]_{i,j=1,\ldots,n} . 
\]
Furthermore, we have the expected identification $\M(\Mat_n(E)) \cong \Mat_n(\M(E))$, compatible with the standard isomorphisms $\M(\Mat_n(A)) \cong \Mat_n(\M(A))$, $\M(\Mat_n(B)) \cong \Mat_n(\M(B))$. 
This is easiest to verify via Corollary~\ref{cor:universal_multipliers}. 
Indeed, if $F$ is a multiplier extension of the $\Mat_n(A)$-$\Mat_n(B)$-bimodule $\Mat_n(E)$, the desired map $F\to \Mat_n(\M(E))$ can be constructed from entrywise embeddings $\psi_{ij}:E \to \Mat_n(E) \subseteq F$, which are compatible with the `diagonal' actions of $A$ and $B$ --- so that via the universal property of $\M(E)$ one obtains unique maps $\gamma_{ij}:F \to E$. 
Naturally combining these yields a bimodule map $\gamma:F \to \Mat_n(E)$ which is a unique extension of the identity. 
\end{ex}

Let $E$ and $F$ be  $\cst$-correspondences from $A$ to $B$ and from $B$ to $C$, respectively. 
The (internal) \emph{tensor product} $E\otimes_B F$ is a $\cst$-correspondence from $A$ to $C$ constructed as the completion of the algebraic tensor product of $E$ and $F$ balanced over $B$, with inner product determined by 
\[
    \langle \xi_1\otimes \zeta_1,   \xi_2\otimes \zeta_2\rangle_{C}=\langle \zeta_1,   \langle \xi_1, \xi_2\rangle_{B}\cdot  \zeta_2\rangle_{C},
\] 
$\xi_i\in E$, $\zeta_i\in F$, $i=1,2$. 
This readily implies that we have a natural $\cst$-correspondence isomorphism $\M(E)\otimes_{\M(B)} F\cong E\otimes_{B} F$ that sends $\xi\otimes b\zeta$ to $\xi b\otimes \zeta$ for $\xi \in\M(E)$, $\zeta \in F$ and $b\in B$. 
Similarly, we also have the isomorphism $E\otimes_{\M(B)} \M(F)\cong E\otimes_{B} F$. Therefore we may and we will assume the identifications
\[
    E\otimes_{B} F= \M(E)\otimes_{\M(B)} F = E\otimes_{\M(B)} \M(F) \subseteq   \M(E)\otimes_{\M(B)}\M(F).
\]
Then $\M(E)\otimes_{\M(B)}\M(F)$ is clearly a multiplier extension of $E\otimes_{B} F$. 

\begin{lem}\label{lem:multiplier_tensor_product}
For $\cst$-correspondences $E$ from $A$ to $B$ and $F$ from $B$ to $C$ we have an embedding of $\cst$-correspondences
\[
    \M(E)\otimes_{\M(B)}\M(F)\into \M(E\otimes_B F)
\]
which is given by the unique bimodule map that extends the identity on $E\otimes_{B} F$. 
More specifically, the formula  $(\xi\otimes \zeta)(c):=\xi\otimes \zeta(c)\in \M(E)\otimes_{\M(B)} F=E\otimes_B F$ for $c\in C, \xi \in \M(E)$ and $\zeta \in \M(F)$ defines an operator $\xi\otimes \zeta \in\L(C, E\otimes_B F)$ and this gives a map
\[
    \M(E)\times \M(F)\ni (\xi,\zeta)\longmapsto \xi\otimes \zeta \in \M(E\otimes_B F) \subseteq \L(C, E\otimes_B F)
\]
that induces the aforementioned embedding.
\end{lem}
\begin{proof}
The first part follows from Proposition~\ref{prop:universal_multiplier_description}. 
For the second part, note that the operator $\xi\otimes \zeta$, for $\xi\in \M(E)$, $\zeta \in \M(F)$ is adjointable, where $(\xi\otimes \zeta)^*(\xi_0\otimes \zeta_0)=\zeta^*(\langle\xi,\xi_0\rangle_{\M(B)} \zeta_0)$ for $\xi_0\in \M(E)$, $\zeta_0 \in \M(F)$. 
Hence  $\xi\otimes \zeta\in \L(C, E\otimes_B F)$. 
Clearly, for $b\in B$ we have $\xi b\otimes \zeta= \xi\otimes b \zeta\in E\otimes_B F$. 
In particular, for any $a\in A$ we have $a(\xi\otimes \zeta)=(a\xi\otimes \zeta)\in E\otimes_{\M(B)} \M(F)=E\otimes_B F$. 
Thus the map $\M(E)\times \M(F)\ni (\xi,\zeta)\longmapsto \xi\otimes \zeta \in \M(E\otimes_B F)$ is well defined. Since it is bilinear, bimodule and $B$-balanced, it descends to a bimodule map on the algebraic $B$-balanced tensor product of $\M(E)$ and $\M(F)$, which in turn extends to the embedding $\M(E)\otimes_{\M(B)}\M(F)\into \M(E\otimes_B F)$ by the first part of the assertion. 
\end{proof}

\begin{rem}
The above embedding is, in general, not surjective. 
For instance, let $A$ be a non-unital $\cst$-algebra, and let $E=A$ and $F=A$ be viewed as $\cst$-correspondences from $A$ to $\M(A)$ and from $\M(A)$ to $A$, respectively, so that $M(E)=M(F) = A$. 
Then $\M(E) \otimes_{\M(A)} \M(F) = E \otimes_{\M(A)} F = A \subsetneq \M(A) = \M(E\otimes_{\M(A)}F)$ as $\cst$-correspondences from $\M(A)$ to $\M(A)$.
\end{rem}

\section{Dilation/Decomposition theorem} 
\label{sec:decomp}

Let $A, B$ be $\cst$-algebras. 
In the sequel we will study maps $T:A\to B$ that are of the form $T(a)=\langle \xi, a \zeta \rangle_{\M(B)}= \xi^* a \zeta$ for $\xi, \zeta \in \M(E)\subseteq \L(B,E)$, with $E$ a $\cst$-correspondence from $A$ to $B$. 
To this end, we develop in this section the following common generalisation of Kasparov's Stinespring dilation theorem (called the KSGNS construction in \cite[Theorem 5.6]{Lance}) and  Murphy's Kolmogorov decomposition of positive kernels for Hilbert modules \cite[Theorem 2.3]{Mur}.

\begin{thm}\label{thm:Murphy_Stinespring}
Let $k:S \times S \to \BB(A,\L(E_B))$ be a map where $A$ and $B$ are $\cst$-algebras, $E_B$ is a right $B$-Hilbert module, and $S$ is a set. 
The following are equivalent:
\begin{enumerate}
    \item\label{enu:Murphy_Stinespring2} there exists an $A$-$B$-$\cst$-correspondence $E$  and a map $\zeta:S \to \L(E_B,E)$ such that 
    \[ 
        k(s,t)(a) =  \zeta(s)^* \varphi_E(a) \zeta(t),\qquad   s, t \in S,\  a \in A,
    \]
    where $\varphi_E:A\to \L(E)$ is the $^*$-homomorphism defining the left action of $A$ on $E$.
    \item\label{enu:Murphy_Stinespring1} $k$ is a \emph{positive-definite kernel} in the sense that for all $n \in \N$, $s_1,\ldots, s_n \in S$, $a_1, \ldots ,a_n\in A$ and $\xi_1, \ldots, \xi_n \in E_B$, we have
    \[
        \sum_{i,j=1}^n \langle \xi_i, k(s_i,s_j) (a_i^*a_j)\xi_j\rangle_B \geq 0,
    \]
    and in addition $k$ is \emph{strict} in the sense that for each $s\in S$ the map $k(s,s):A\to \L(E_B)$ is strict, i.e.\ for an approximate unit $\{e_i\}_{i \in I}$ in $A$, the net $\{k(s,s)(e_i)\}_{i \in I}$ converges strictly in $\L(E_B)$.
\end{enumerate}
If the above conditions hold, the $\cst$-correspondence $E$ in \ref{enu:Murphy_Stinespring2} may be chosen so that  $\{a\zeta(s)\xi: a\in A, s \in S, \xi \in E_B\}$ is linearly dense in $E$; we call such pairs $(\zeta,E)$ \emph{minimal}. 
Given another minimal pair $(\zeta',E')$ as above we have a uniquely determined unitary isomorphism of $\cst$-correspondences $U:E \to E'$ such that $U\zeta(s) = \zeta'(s)$, $s \in S$.
\end{thm}
\begin{proof} 
Assume \ref{enu:Murphy_Stinespring2}. 
Using the notation in \ref{enu:Murphy_Stinespring1} we have
\begin{align*}
    \sum_{i,j=1}^n \langle \xi_i, k(s_i,s_j) (a_i^*a_j)\xi_j\rangle_B
        &=\langle  \sum_{i=1}^n a_i\zeta(s_i)\xi_i, \sum_{i=1}^n  a_i\zeta(s_i) \xi_i\rangle_B \geq 0,
\end{align*}
so $k$ is positive-definite. 
Also for any $\xi\in E_B$, $s \in S$ and an approximate unit $(e_i)_{i \in I}$ of $A$, the net $k(s,s)(e_i)\xi=\zeta(s)^* e_i\zeta(s)\xi$ converges to $\zeta(s)^* \zeta(s)\xi$. 
Hence $k$ is strict.

Now assume \ref{enu:Murphy_Stinespring1}. 
Begin by defining a new positive-definite kernel, now in the usual sense of \cite{Mur}, $\tilde{k}:(S \times A) \times (S \times A) \to \L(E_B)$, as follows:
\[ 
    \tilde{k} ((s,a) , (t,b)):= k(s,t)(a^*b), \qquad s, t \in S,\ a,b \in A.
\]	
By \cite[Theorem 2.3]{Mur} there exists a Hilbert $B$-module $E$ and a map $\tilde{\zeta}: S\times A\to \L(E_B,E)$ (with  $\|\tilde{\zeta}(s,a)\|\leq \|k(s,s)\|^{1/2}\|a\|$ for all $s \in S, a \in A$) such that for each $s,t\in S, a,b \in A$ we have  
\begin{equation}\label{eq:Murphy_given}
    k(s,t)(a^*b)=\tilde{k} ((s,a),(t,b)) =\tilde{\zeta}(s,a)^* \tilde{\zeta}(t,b). 
\end{equation}
Moreover, we may and we do assume that $\tilde{\zeta}$ is minimal, i.e.\ the set $\bigcup_{s\in S, a \in A}\tilde{\zeta}(s,a)E_B$ is linearly dense in $E$.
We claim that there is a natural $*$-homomorphism $\varphi_E:\M(A)\to \L(E)$ 
determined by the formula 
\[
    \varphi_E(u)\tilde{\zeta}(s,a)\xi :=\tilde{\zeta}(s,ua)\xi, \qquad   u \in \M(A),\  s\in S,\ a\in A,\ \xi \in E_B .
\]
To see that $\varphi_E(u)$ exists, by linearity and the Russo--Dye theorem (\cite{RussoDye}) it suffices to consider the case when $u\in \M(A)$ is unitary. 
Let then $u \in \mathcal{UM}(A)$. 
It is then easy to see that the map $\tilde\zeta_u: S \times A\to \L(E_B,E) $ given by 
\[
    \tilde\zeta_u (s,a) = \tilde\zeta (s, ua), \qquad s \in S,\ a \in A,
\]
yields another minimal Kolmogorov dilation of the kernel $\tilde k$. 
Hence, by \cite[Theorem 2.3]{Mur}, there exists a unitary $\varphi_E(u) \in \L(E)$ such that $\varphi_E(u)\tilde\zeta (s, a) = \tilde\zeta (s, ua)$ for all $s \in S$ and  $a \in A$.
It is then routine to check that the prescription $u \mapsto \varphi_E(u)$  extends to a unital $^*$-homomorphism from $\M(A)$ to $\L(E)$.
Restriction of $\varphi_E$ to $A$  gives $E$ the  structure of a $\cst$-correspondence from $A$ to $B$. 
Indeed, as $k$ is bounded-operator valued, we have for any   approximate unit $\{e_i\}_{i \in I}$ in $A$, $s, t \in S$ and $a,b \in A$ that $(k(s,t)(be_i a))_{i \in I}$ converges (pointwise) to $k(s,t) (ba)$. 
This it suffices to verify that $\varphi_E|_A:A \to \L(E)$ is nondegenerate.

Fix an approximate unit $\{e_i\}_{i \in I}$ in $A$ and $s\in S$. 
For any $\xi\in E_B$ and $i, j \in I, i\geq j$, using \eqref{eq:Murphy_given}, we get (noting that $k(s,s)$ is positive)
\[
    \|\tilde{\zeta}(s,e_i)\xi- \tilde{\zeta}(s,e_j)\xi\|^2=\|\langle \xi, k(s,s)\big((e_i-e_j)^2\big)\xi\rangle_B\|\leq \|\langle \xi, k(s,s)(e_i-e_j)\xi\rangle_B\|.
\]
Since  $\{k(s,s)(e_i)\}_{i \in I}$ is strictly Cauchy it follows that $\{\tilde{\zeta}(s,e_i)\}_{i \in I}$ is strictly Cauchy and hence it is convergent to an operator  $\zeta(s)\in  \L(E_B,E)$. 
In this way we get a map $\zeta:S \to \L(E_B,E)$, and for any $a\in A$ 
\[
    k(s,t)(a)=\lim_{i,j\in I}k(s,t)(e_iae_j)=\lim_{i,j\in I}\tilde{\zeta}(s,e_i)^*\varphi_E(a) \tilde{\zeta}(t,e_j)=\zeta(s)^* \varphi_E(a) \zeta(t).
\]
This proves \ref{enu:Murphy_Stinespring2}. 
Once again using the convergence of $\{k(s,t)(ae_i b)\}_{i \in I}$ to $k(s,t)(ab)$ we deduce that $a\zeta(s)\xi=\tilde{\zeta}(s,a)\xi$ and hence the pair $(\zeta,E)$ is minimal in the sense described in the last part of the assertion.

Suppose now that we are given another minimal pair $(\zeta',E')$ that dilates $k$ as in \ref{enu:Murphy_Stinespring2}. 
Putting $\tilde{\zeta'}(s,a):= a \cdot\zeta'(s)$ for all $s \in S$, $a \in A$,
we get a map $\tilde{\zeta'}: S\times A \to E'$ which is a minimal Kolmogorov decomposition for the kernel $\tilde{k}$. 
So once again we have a unique unitary $U\in \L(E;E')$ such that $U(\tilde{\zeta}(s,a)) = \tilde{\zeta'}(s,a)$ for all $s \in S$, $a \in A$. 
It is routine to check that such a $U$ necessarily intertwines the respective left actions of $A$ on $E$ and $E'$.
\end{proof}

\begin{rem}
When $S$ is a singleton, Theorem~\ref{thm:Murphy_Stinespring} reduces to the KSGNS dilation of strict completely positive maps $\varrho:A\to \L(E_B)$, \cite[Theorem 5.6]{Lance}. 
When $A=\C$ is trivial, then under the identification $\BB(\C,\L(E_B))\cong \L(E_B)$, via the map $\BB(\C,\L(E_B))\ni T\mapsto T(1)\in \L(E_B)$, Theorem~\ref{thm:Murphy_Stinespring} reduces to Murphy's Kolmogorov decomposition, \cite[Theorem 2.3]{Mur} (which was used in the proof above).
\end{rem}

\begin{rem}\label{rem:compact_valued_kernels}
If $k(s,s)\in \BB(A,\K(E_B))$ for all $s\in S$, then for any (minimal) decomposition $\zeta:S \to \L(E_B,E)$ as in Theorem~\ref{thm:Murphy_Stinespring} and any $a\in A$ we have $\varphi_E(a)\zeta:S \to \K(E_B,E)$. 
This follows for instance from \cite[Proposition 1.3]{katsura2} applied to  each $\varphi_E(a)\zeta(s)$.
In particular, if the kernel in Murphy's theorem takes values in compact operators $\K(E_B)$, then its Kolmogorov decomposition is also compact valued. 
\end{rem}

We will need yet another version of dilation that follows from  Theorem~\ref{thm:Murphy_Stinespring}. 

\begin{cor}\label{cor:Murphy_Stinespring}
Let $k:S \times S \to \BB(A,B)$ be a map where $A$ and $B$ are $\cst$-algebras and $S$ is a set.
The following are equivalent:
\begin{enumerate}
    \item\label{enu:cor_Murphy_Stinespring2} there exists a $\cst$-correspondence $E$ from $A$ to $B$ and a map $\zeta:S \to \M(E)$ such that 
    \[
        k(s,t)(a) = \langle   \zeta(s), a\cdot \zeta(t) \rangle_{\M(B)}, \qquad s, t \in S,\ a \in A;
    \]

    \item\label{enu:cor_Murphy_Stinespring1} $k$ is a \emph{positive-definite kernel} in the sense that for all $n \in \N$, $s_1,\ldots, s_n \in S$ and $a_1, \ldots, a_n \in A$ the matrix
    \[
        [k(s_i,s_j) ( a_i^*a_j)]_{i,j=1}^n \in \Mat_n(B)
    \]
    is positive, 
    and $k$ is \emph{strict} in the sense that for each $s\in S$ the map $k(s,s):A\to B$ is strict, i.e.\ for an approximate unit $\{e_i\}_{i \in I}$ in $A$, the net $\{k(s,s)(e_i)\}_{i \in I}$ converges strictly in $\M(B)$.
\end{enumerate}
Moreover, if the conditions above hold, then the $\cst$-correspondence $E$ in \ref{enu:Murphy_Stinespring2} may be chosen so that the set $\{a\zeta(s)b:s \in S, a\in A, b\in B\}$ is dense in $E$; we call such pairs $(\zeta,E)$ \emph{minimal}. 
Given another minimal pair $(\zeta',E')$ as above, we have a uniquely determined unitary isomorphism of $\cst$-correspondences $U:E \to E'$ such that $U\zeta(s) = \zeta'(s)$, $s \in S$.
\end{cor}
\begin{proof} 
Let $E_B:=B$ be the trivial Hilbert module over $B$, so that $B\cong \K(E_B)\subseteq \L(E_B)$. 
Then clearly \ref{enu:Murphy_Stinespring1} in Theorem~\ref{thm:Murphy_Stinespring} and \ref{enu:cor_Murphy_Stinespring1} in Corollary~\ref{cor:Murphy_Stinespring} are equivalent. 
Also by Remark~\ref{rem:compact_valued_kernels}, one readily sees that  \ref{enu:Murphy_Stinespring2} in Theorem~\ref{thm:Murphy_Stinespring} is equivalent to \ref{enu:cor_Murphy_Stinespring2} in Corollary~\ref{cor:Murphy_Stinespring}. 
\end{proof}


%
%
\section{Bundles of C*-correspondences} 
\label{Sec:bundles}

Throughout this paper $X$ will stand for a locally compact Hausdorff space.
In this section we introduce bundles of $\cst$-correspondences between two $\cst$-bundles over $X$ and present several relevant constructions.

A \emph{Banach bundle} $\EE$ over $X$ is a topological space with an open continuous surjection $p:\EE\to X$ such that each fibre $E_x=p^{-1}(x)$ is a Banach space whose topology agrees with that of $\EE$, the linear operations are continuous, and the norm $\EE\ni a\mapsto \|a\| \in [0,\infty)$ is \emph{upper semicontinuous} and nondegenerate, see, for instance, \cite[Definition 2.1]{BussExel} for a detailed definition. 
If the norm is continuous then we say $\EE$ is a \emph{continuous Banach bundle}. 
We will often write $\EE=\{E_x\}_{x\in X}$ for a Banach bundle. 
We denote by $\contc(\EE)$,  $\contz(\EE)$, $\contb(\EE)$  the spaces of continuous sections that are compactly supported,  vanishing at infinity, and bounded, respectively. 
By \emph{Fell's reconstruction theorem}, any of these spaces is enough to reconstruct the topology of $\EE$. 
More specifically (see \cite[Theorem II.13.18]{Fell_Doran} or \cite[Proposition 2.4]{BussExel}), for any family of Banach spaces $\EE=\{E_x\}_{x\in X}$ and any linear space $\Gamma$ of sections of $\EE$ such that $\overline{\{\xi (x): \xi \in \Gamma\}}=E_x$ for each $x\in X$ and $X\ni x\mapsto \|\xi(x)\|\in [0,\infty)$ is upper semicontinuous for each $\xi\in \Gamma$, there exists a unique topology on $\EE$ turning it into a Banach bundle with $\Gamma\subseteq C(\EE)$. 
A net $\{f_i\}_{i \in I}$ of elements of $ \EE$ converges in this topology to $f\in \EE$ if and only if $\lim_{i \in I}p(f_i)= p(f)$ and for every $\varepsilon>0$ there is $\xi\in \Gamma$  such that $\|f-\xi(p(f))\| < \varepsilon$ and $\|f_i-\xi(p(f_i))\|<\varepsilon$ for large enough $i \in I$.

It is well known that (upper semicontinuous) bundles of $\cst$-algebras $\A=\{A_x\}_{x\in X}$ are equivalent to $\contz(X)$-$\cst$-algebras, see, for instance, \cite[Appendix C]{WLbook}. 
More specifically, a $\contz(X)$-$\cst$-algebra is a $\cst$-algebra $A$ equipped with a nondegenerate homomorphism $\contz(X)\to Z(\M(A))$ turning $A$ into a $\contz(X)$-module. 
Given a bundle of $\cst$-algebras $\A$ we write $A=\contz(\A)$ for its $\cst$-algebra of $\contz$-sections, equipped with the $\contz(X)$-module structure given by
\[
    (fa) (x)= f(x)a(x), \qquad f\in \contz(X) , a\in \contz(\A).
\]
We now briefly explain how this equivalence extends to $\cst$-correspondences. 
We begin by introducing a bundle version of Le Gall's notion of a $\contz(X)$-representation \cite[Definition 4.1]{LeGall}, which was renamed a \emph{$\contz(X)$-$\cst$-correspondence} by Deaconu in \cite[Definition 5.8]{Deaconu}.

\begin{defn}\label{def:Correspondence_bundle}
A \emph{$\cst$-correspondence bundle} from a $\cst$-bundle $\A=\{A_x\}_{x\in X}$ to a  $\cst$-bundle $\B=\{B_x\}_{x\in X}$ is a Banach bundle $\EE=\{E_x\}_{x\in X}$, where each $E_x$ is a $\cst$-correspondence from $A_x$ to $B_x$, and the sections
\[
    X\ni x \mapsto  \langle \zeta(x),\xi(x)\rangle_{E_x}\in  \B, \qquad X\ni x \mapsto  a(x)\xi(x)\in  \EE, \qquad X\ni x \mapsto  \xi(x)b(x)\in  \EE
\]
are continuous for every $\zeta, \xi \in \contc(\EE)$, $a\in  \contc(\A)$, $b\in \contc(\B)$. 
\end{defn}

\begin{prop}\label{prop:correspondence_bundles_vs_C(X)_correspondence} 
For any $\cst$-correspondence bundle $\EE=\{E_x\}_{x\in X}$ from  $\A = \{ A_x \}_{x\in X}$ to $\B=\{B_x\}_{x\in X}$ the space $E := \contz(\EE)$ is a $\cst$-correspondence from $A:=\contz(\A)$ to $B:=\contz(\B)$ with the additional property that 
\begin{equation}\label{eq:Deaconu_homogenity_condition}
    (f a)\cdot   \xi\cdot  b=  a \cdot \xi\cdot  (f b)\quad  \text{for all $f\in \contz(X),\ a\in A,\ b\in B,\ \xi \in E$,}
\end{equation}
i.e.\ $E$ is a $\contz(X)$-$\cst$-correspondence from $A$ to $B$ in the sense of \cite{Deaconu}. 
Moreover, every $\contz(X)$-$\cst$-correspondence from $A$ to $B$ is of the above form for some $\cst$-correspondence bundle $\EE=\{E_x\}_{x\in X}$.
\end{prop}
\begin{proof} 
The first part of the assertion is straightforward. 
Let $E$ be any $\cst$-correspondence from $A$ to $B$ satisfying \eqref{eq:Deaconu_homogenity_condition}.
Recall that the fibres of the corresponding bundles $\A$ and $\B$ are given by $A_x:=A/I_x A$ and $B_x:=B/I_x B$ where $I_x:=\{f\in \contz(X): f(x)=0\}$, $x\in X$.
By \eqref{eq:Deaconu_homogenity_condition} for each $x\in X$, we have $(I_{x}A) E\subseteq E(I_{x}B)$. 
Thus the quotient $E_{x}:=E/(E I_x B)$ is naturally a $\cst$-correspondence from $A_x=A/I_x$ to $B_x=B/I_{x}$, cf.\ \cite[Lemma 2.3]{FMR} (note that $E_{x}$ is canonically isomorphic to the $A_x$-$B_x$-$\cst$-correspondence $E\otimes_B B_x$ considered in \cite{LeGall, Deaconu}). 
Denoting by $\xi(x)$ the image of $\xi$ in the quotient $E_{x}=E/(E I_x B)$, we get $\|\xi(x)\|^2=\|\langle \xi (x), \xi (x) \rangle_{B_x}\|=\|\langle \xi , \xi\rangle  (x) \|$, where $\langle \xi , \xi\rangle$ is in the $\contz(X)$-algebra $B$. 
Hence the map $X\ni x\mapsto \|\xi(x)\|$ is upper semicontinuous, vanishes at infinity and its maximum is $\|\xi\|$.
Thus by Fell's reconstruction theorem, there is a unique topology on  $\EE:=\{E_x\}_{x\in X}$ such that sections $X \ni x \mapsto \xi(x)\in E_x$ for $\xi\in E$ are continuous. 
The proof of \cite[Proposition C.24]{WLbook} can be readily adapted to show that the image of $E$ in $\contz(\EE)$ is dense. 
Since the embedding $E\hookrightarrow \contz(\EE)$ is isometric, and $E$ is complete, we conclude it must be an isomorphism.
\end{proof}

\begin{rem}
For any $\cst$-correspondence bundle $\EE=\{E_x\}_{x\in X}$ from  $\A=\{A_x\}_{x\in X}$ to $\B=\{B_x\}_{x\in X}$ the space $\contb(\EE)$ is naturally a $\cst$-correspondence from $\contb(\A)$ to $\contb(\B)$, and so $\contb(\EE)$ is also a $\contb(X)$-$\cst$-correspondence. 
Moreover, $\contz(\EE)$ can also be treated as a $\cst$-correspondence from $\contb(\A)$ to $\contb(\B)$, a $\cst$-subcorrespondence of $\contb(\EE)$.
\end{rem}

\begin{ex}
Every $\cst$-bundle $\A=\{A_x\}_{x\in X}$ can be viewed as a (trivial) $\cst$-correspondence bundle over $\A$, by viewing each fibre $A_x$ as the trivial $\cst$-correspondence. 
Proposition~\ref{prop:correspondence_bundles_vs_C(X)_correspondence} applied to such  $\cst$-correspondence bundles recovers the correspondence between $\contz(X)$-C$^*$-algebras and $\cst$-bundles, extending the arguments used to establish the latter fact.
\end{ex}

\begin{ex}
If $\A=\{\C\}_{x\in X}$ is the trivial one-dimensional $\cst$-bundle, so that $\contz(\A)=\contz(X)$, then a $\cst$-correspondence bundle from $\A$ to a $\cst$-bundle $\B=\{B_x\}_{x\in X}$ (with the associated $\cst$-algebra $B=\contz(\B)$), is simply a \emph{Hilbert module bundle} over $\B=\{B_x\}_{x\in X}$, 
see \cite[1.7]{Kumjian}, because we may ignore the left action of $\A$ (it carries no new information). 
In view of Proposition~\ref{prop:correspondence_bundles_vs_C(X)_correspondence} such Hilbert module bundles correspond to $\contz(X)$-Hilbert modules, 
where by a \emph{$\contz(X)$-Hilbert module}  over a $\contz(X)$-$\Cst$-algebra $B$ we mean  a right Hilbert module $E$ over $B$ which is also 
 a $\contz(X)$-module such that 
 $ f\cdot \xi \cdot b = \xi\cdot  (f b)$, $f\in \contz(X)$,  $b\in B$, $\xi \in E$. 
If, in addition, $\B=\{\C\}_{x\in X}$ is a trivial bundle, then all Hilbert modules over $B=\contz(X)$ are $\contz(X)$-Hilbert modules and they correspond to \emph{continuous bundles of Hilbert spaces} $\H=\{H_x\}_{x\in X}$, see \cite[Definition~II.13.5]{Fell_Doran} or \cite{RenaultFourier}. 
\end{ex}

\begin{ex} \label{ex:matrixbundles}
Let $\A=\{A_x\}_{x \in X}$ be a \cst-bundle and let $n \in \N$. 
The $\contz(X)$-$\cst$-algebra picture yields immediately that $\Mat_n(\A)= \{\Mat_n(A_x)\}_{x \in X}$ is again a \cst-bundle; and if $A$ is the $\contz(X)$-$\cst$-algebra associated to $\A$, then $\Mat_n(A)$ is the one associated to $\Mat_n(\A)$. 
Next, if $\EE=\{E_x\}_{x \in X}$ is a $\cst$-correspondence bundle from $\A$ to $\B$ as in Definition~\ref{def:Correspondence_bundle}, then $\Mat_n(\EE)=\{\Mat_n(E_x)\}_{x \in X}$, with $\Mat_n(E_x)$ constructed as in Example~\ref{ex:matrixcorr} and the natural topology on $\Mat_n(\EE)$, becomes a $\cst$-correspondence bundle from $\Mat_n(\A)$ to $\Mat_n(\B)$.
\end{ex}

\begin{defn}
Let $\EE=\{E_x\}_{x\in X}$ be a $\cst$-correspondence bundle from  $\A=\{A_x\}_{x\in X}$ to $\B=\{B_x\}_{x\in X}$. We say that a section $\xi$ of the bundle $\M(\EE):=\{\M(E_x)\}_{x\in X}$  of multiplier $\cst$\nb-correspondences is \emph{strictly continuous} if $a\cdot \xi, \xi \cdot b\in \contz(\EE)$ for any $a\in \contz(\A)$, $b\in \contz(\B)$. 
We denote by $\contb(\M(\EE))$ the set of all bounded strictly continuous sections of $\M(\EE)$.
\end{defn} 

For a $\cst$-bundle  $\A=\{A_x\}_{x\in X}$, treated as a trivial correspondence bundle from $\A$ to $\A$, the above definition of the multiplier bundle $\M(\A)=\{\M(A_x)\}_{x\in X}$ is consistent with the one given in \cite[Section 3]{Multipliers}. 
In particular, the following proposition generalises \cite[Theorem 3.3]{Multipliers} from $\cst$-bundles to $\cst$-correspondence bundles.

\begin{prop}\label{prop:Akemman_Pedersen_Tomiyama_generalization}
Let $\EE=\{E_x\}_{x\in X}$ be a $\cst$-correspondence bundle from $\A=\{A_x\}_{x\in X}$ to $\B=\{B_x\}_{x\in X}$. 
Equipped with pointwise operations, the section algebras $\contb(\M(\A))$, $\contb(\M(\B))$ are $\cst$-algebras and $\contb(\M(\EE))$ is a  $\contb(\M(\A))$-$\contb(\M(\B))$-$\cst$-correspondence isomorphic to the $\M(\contz(\A))$-$\M(\contz(\B))$-$\cst$-correspondence $\M(\contz(\EE))$, where the isomorphism
\[
    \contb(\M(\EE))\cong \M(\contz(\EE))
\]
is the unique morphism that extends the identities on $\contz(\EE)$, $\contz(\A)$ and $\contz(\B)$. 
In particular, $\contb(\M(\A))\cong \M(\contz(\A))$ and $\contb(\M(\B))\cong \M(\contz(\B))$ as $\cst$-algebras. 
\end{prop}
\begin{proof}
It is immediate that bounded strictly continuous sections of $\M(\EE)$ form a Banach $\contz(\A)$-$\contz(\B)$-bimodule with pointwise operations and norm  $\|\xi\|:=\sup\|\xi(x)\|$. 
When this construction is applied to the $\cst$-bundle $\A$, we also see that $\contb(\M(\A))$ is closed under pointwise multiplication, involution, and contains the unit section. So $\contb(\M(\A))$ is naturally a $\cst$-algebra that contains $\contz(\A)$ as a $\cst$-subalgebra, in fact, an ideal. We have a unital $*$-homomorphism $\pi_A:\contb(\M(\A))\to \M(\contz(\A))$ where $[\pi_A(b)a](x)=b(x)\cdot a(x)$, $b\in \contb(\M(\A))$, $a\in \contz(\A)$, $x\in X$.
It is isometric and extends the identity on $\contz(\A)$. 
To show it is surjective, let $x\in X$ and note that the evaluation  map $\pi_{A,x}:\contz(\A)\to A_x$, $\pi_{A,x}(a)=a(x)$, extends (uniquely) to a $*$-homomorphism $\overline{\pi}_{A,x}:\M(\contz(\A))\to \M(A_x)$. For $a\in \M(\contz(\A))$ the formula $\widehat{a}(x)=\overline{\pi}_{A,x}(a)$ defines $\widehat{a}\in \contb(\M(\A))$ with $\pi_A(\widehat{a})=a$. 
Hence we have $*$-isomorphisms 
\[
    \contb(\M(\A))\cong \M(\contz(\A)),\qquad \contb(\M(\B))\cong \M(\contz(\B))
\] 
which are in fact given by unique homomorphisms extending the identities on $\contz(\A)$ and $\contz(\B)$, respectively.
A similar argument works for the $\contb(\M(\A))$-$\contb(\M(\B))$-bimodule $\contb(\M(\EE))$. 
Namely, the map  $\contb(\M(\EE))\ni \xi\mapsto (R_\xi, L_{\xi}) \in \M(\contz(\EE))$ where $R_{\xi}(a):=a \xi$ and $L_\xi(b):=\xi b$, is a  linear isometry  $\pi:\contb(\M(\EE))\to \M(\contz(\EE))$ that extends the identity on $\contz(\EE)$. 
To see that $\pi$ is surjective, note that, by Lemma~\ref{lem:extensions_lemma}, for any $x\in X$, the  map $\pi_x:\contz(\EE)\to E_x$, $\pi_x(\xi)=\xi(x)$, extends (uniquely) to a linear map $\overline{\pi}_x:\M(\contz(\EE))\to \M(E_x)$ such that $\overline{\pi}_x(a\xi b)=\overline{\pi}_{A,x}(a)(x)\overline{\pi}_x(\xi)\overline{\pi}_{B,x}$ for $a\in \M(\contz(\A))$, $b\in \M(\contz(\B))$. 
For $\xi\in M(\contz(\EE))$ the formula $\widehat{\xi}(x)=\overline{\pi}_{x}(\xi)$ defines $\widehat{\xi}\in \contb(\M(\A))$ with $\pi(\widehat{\xi}):=\xi$.
Hence $(\pi,\pi_A,\pi_B)$ is an isometric isomorphism from  the $\contb(\M(\A))$-$\contb(\M(\B))$-bimodule $\contb(\M(\EE))$ to the $\M(\contz(\A))$-$\M(\contz(\B))$-bimodule $\M(\EE)$. 
It is the unique morphism that extends the identity morphism on $\contz(\EE)$.
Its inverse maps the inner product in $\M(\EE)$ to the inner product in $\contb(\M(\EE))$ defined pointwise. 
\end{proof}

For a $\cst$-bundle $\A=\{A_x\}_{x\in X}$, the canonical identification $\contb(\M(\A))\cong\M(\contz(\A))$ from Proposition~\ref{prop:Akemman_Pedersen_Tomiyama_generalization}
allows us to regard $\contb(\M(\A))$ as the multiplier algebra of $\contz(\A)$. In particular, this justifies writing informally $\M(\contz(\A))=\contb(\M(\A))$ in the sequel.

\begin{rem}
For a $\cst$-bundle $\A=\{A_x\}_{x\in X}$, the multiplier bundle $\M(\A)=\{\M(A_x)\}_{x\in X}$ is usually not a Banach bundle with the `strict topology'. 
It may also happen that $\contb(\M(\A))\neq \contb(\A)$ even when $\M(\A)=\A$ as sets. 
For instance, let $\A$ be the bundle corresponding to the $C([0,1])$-algebra:
\[
    A=\left\{a\in C([0,1],\Mat_2(\C)): a(0)=\left(\begin{array}{cc} \lambda & 0 \\ 0  &  0 \end{array}\right),\ \lambda \in \C\right\}.
\] 
Then $\A=\M(\A)$ because all the fibres $A_x$, $x\in [0,1]$ are unital ($A_0\simeq\C$, $A_x\simeq \Mat_2(\C)$ for $x>0$).
But $A=\contz(\A)=\contb(\A)$ is not unital and so it is smaller than $\contb(\M(\A))=\M(A)$. In fact we have,
\[
  \contb(\M(\A))=\M(A)=\left\{a\in C([0,1],\Mat_2(\C)): a(0)=\left(\begin{array}{cc}
        \lambda & 0 \\
        0  &  \mu
        \end{array}\right),\ \lambda, \mu \in \C\right\}.
\]
In general, the strictly continuous sections of $\M(\A)$ coincide with continuous sections of $\A$ if and only if the bundle $\A$ has a \emph{continuous unital section}, i.e.\ each fibre $A_x$, $x\in X$ is unital and the unit section $X \ni x\mapsto 1_x\in A_x\subseteq \A$ is continuous.
This latter condition was made a standing assumption in \cite{BartoszKangAdam}.
\end{rem}

The tensor product of $\contz(X)$-$\cst$-correspondences is a $\contz(X)$-$\cst$-correspondence, cf.\ \cite[Proposition 4.1]{LeGall}. 
In terms of bundles this corresponds to the following construction.

\begin{ex}[Tensor product of $\cst$-correspondence bundles]\label{ex:tensor_bundles}
If $\EE=\{E_x\}_{x\in X}$ and $\FF=\{F_x\}_{x\in X}$ are $\cst$-correspondence bundles from  $\A=\{A_x\}_{x\in X}$ to $\B=\{B_x\}_{x\in X}$ and from $\B=\{B_x\}_{x\in X}$ to $\CC=\{C_x\}_{x\in X}$ respectively, then  the collection of $\cst$-correspondence tensor products $\EE\otimes_\B \FF:=\{E_x\otimes_{B_x} F_x\}_{x\in X}$ is naturally a  $\cst$-correspondence bundle from  $\A$  to $\CC$. 
Indeed, for every $\xi_i\in \contc(\EE)$, $\zeta_i\in \contc(\FF)$, $i=1, \ldots ,n$, consider the section $x\mapsto  (\sum_{i=1}^n\xi_i\otimes \zeta_i)(x) := \sum_{i=1}^n \xi_i(x)\otimes \zeta_i(x)\in \EE\otimes_\B \FF$. 
It follows from Definition~\ref{def:Correspondence_bundle} that the map $x\mapsto \|\sum_{i=1}^n\xi_i\otimes \zeta_i(x)\|^2=\|\sum_{i,j=1}^n\langle \zeta_i(x), \langle \xi_i(x),\xi_j(x)\rangle_{B_x} \zeta_j(x)\rangle_{C_x}\|$ is upper semicontinuous. 
Hence $\EE\otimes \FF$ can be equipped with a unique topology such that the sections $\xi\otimes \zeta$, $\xi\in \contc(\EE)$, $\zeta\in \contc(\FF)$, are continuous.
Then $\EE\otimes_\B \FF$ is a $\cst$-correspondence bundle from  $\A$ to $\CC$ and  we have a canonical isomorphism 
\[
    \contz(\EE\otimes_\B \FF)\cong \contz(\EE)\otimes_{\contz(\B)} \contz(\FF). 
\]
of $\cst$-correspondences from $\contz(\A)$ to $\contz(\CC)$. 
Combining Lemma~\ref{lem:multiplier_tensor_product} and Proposition~\ref{prop:Akemman_Pedersen_Tomiyama_generalization} we also have the embedding $\contb(\M(\EE))\otimes_{\contb(\M(\B))} \contb(\M(\FF))\into  \contb(\M(\EE\otimes_{\B} \FF))$ of $\cst$-correspondences from $\contb(\M(\A))\cong \M(\contz(\A))$ to $\contb(\M(\CC)) \cong \M(\contz(\CC))$. 
\end{ex}

\begin{ex}[Direct sums of $\cst$-correspondence bundles] 
Let $\EE=\{E_x\}_{x\in X}$ and $\FF=\{F_x\}_{x\in X}$ be $\cst$-bundles from $\A=\{A_x\}_{x\in X}$ to $\B=\{B_x\}_{x\in X}$. 
We equip the family $\EE\oplus\FF:=\{E_x\oplus F_x\}_{x\in X}$ of direct sums of $\cst$-correspondence bundles with the unique topology such that for all $\xi\in \contc(\EE)$, $\zeta\in \contc(\FF)$ the section $ x\mapsto\xi\oplus \zeta(x):=\xi(x)\oplus \zeta(x)$ is in $\contc(\EE\oplus\FF )$. 
This is a $\cst$-correspondence bundle  from $\A=\{A_x\}_{x\in X}$ to $\B=\{B_x\}_{x\in X}$ and we have a canonical isomorphism of $\cst$-correspondences 
\[
    \contz(\EE\oplus\FF)\cong \contz(\EE)\oplus \contz(\FF), \qquad \contb(\EE\oplus\FF)\cong \contb(\EE)\oplus\contb(\FF).
\]
More generally, if $\{\EE_i\}_{i\in I}$ is a family of $\cst$-correspondence bundles from $\A$ to $\B$, with $\EE_i=\{E_x^i\}_{x \in X}$ for each $i \in I$, then we define the direct sum $\cst$-correspondence bundle by putting $\oplus_{i\in I} \EE_i:= \{\oplus_{i\in I} E_x^i\}_{x\in X}$ and equipping it with the unique topology such that for sections $\xi_i\in \contc(\EE_i)$, $i\in I$, such that $\sum_{i\in I} \|\xi_i(x)\|^2<\infty$ for every $x\in X$, the section $X\ni x \to \sum_{i\in I} \xi_i(x)\in \oplus_{i\in I} \EE_i$ is continuous.
\end{ex}

Finally we comment on morphisms between $\cst$-correspondence bundles.
	
\begin{lem}\label{lem:morphisms_disintegration} 
Let $\EE=\{E_x\}_{x\in X}$ and $\FF=\{F_x\}_{x\in X}$ be Banach bundles. 
The relation
\[
    L(\xi)(x)=L_{x}(\xi(x)) , \qquad x\in X,\ \xi \in \contz(\EE),
\]
establishes a bijective correspondence between bounded $\contz(X)$-linear maps $L:\contz(\EE)\to \contz(\FF)$ and bundles $\{L_{x}\}_{x\in X}$ of linear operators $L_x:E_x\to F_x$, such that $\sup_{x\in X}\|L_x\|<\infty$ and the map $\EE\supseteq E_x\ni \xi \longmapsto L_x(\xi)\in F_x\subseteq \FF$ is continuous. 
\end{lem}
\begin{proof}
This is standard, see for instance (the proofs of) \cite[Propositions 3.2, 3.5]{Bartosz}. 
\end{proof}

\section{Fell bundles and twisted actions} 
\label{Sec:Fell_bundles}

In this section we introduce twisted actions of groupoids on $\cst$-bundles and describe the associated crossed products. From now on, throughout this paper, 
\begin{quote}
	$\G$ is a \emph{locally compact Hausdorff \'{e}tale groupoid} with the \emph{unit space} $X:=\G^{(0)}$. 
\end{quote}
\noindent
We denote by $r,s:\G\to X\subseteq \G$ the \emph{range} and \emph{source} maps and by $\G^{(2)}=\{(g,h)\in \G\times \G: s(g)=r(h)\}$ the set of \emph{composable pairs}. 
Similarly we define \emph{composable triples} $\G^{(3)}$, and so on. 
We refer to \cite{Renault0,Exel:combinatorial,Sims} for more information on \'etale groupoids. 
In this section we will define their (twisted) actions on $\cst$-bundles. 
We will however begin the discussion in the context of general Fell bundles.

Let $\A$ be a \emph{Fell bundle over the groupoid~$\G$}, see \cite[Section 2]{BussExel} or \cite{Kumjian,Takeishi}. 
For the convenience of the reader we recall some basic facts here. 
By definition, $\A$ is an upper-semicontinuous Banach bundle $\A=\{A_g\}_{g\in \G}$ equipped with a  continuous multiplication ${\cdot}: \{(a,b)\in \A\times \A : a\in A_{g},\ b \in A_{h},\ (g,h)\in \G^{(2)}\} \to \A$ and a continuous involution
$ ^{\star}: \A\to \A$ such that for all $g,h \in \G$ with $(g,h) \in \G^{(2)}$
\[
    A_{g}\times A_h\ni (a,b)\mapsto a b\in A_{gh}, \qquad 	\A\ni A_g\ni a\mapsto a^{\star}\in A_{g^{-1}},  
\]
satisfying the standard set of axioms (note that we write simply $ab$ in lieu of $a \cdot b$). 
The linear space $\contc(\A)$ is naturally a $*$-algebra with algebraic operations coming from the Fell bundle: for $a,b\in \contc(\A)$ we put
\[
    (a*b)(g) := \sum_{h\in \G_{s(g)}} a(gh^{-1}) b(h)=\sum_{h\in \G^{\rg(g)}} a(h) b(h^{-1}g) \quad \text{and} \quad  a^*(g):= a(g^{-1})^{\star}.
\]
Here, for $x \in X$, we put $\G_x:=\{g \in \G : s(g) = x \}$ and $\G^x := \{ g\in \G : \rg(g) = x \}$.

\begin{defn} 
The \emph{full section $\cst$-algebra} $\cst(\A)$ of the Fell bundle $\A$ is the completion of $\contc(\A)$ in the maximal $\cst$-norm.
\end{defn} 

One shows that the largest $\cst$-(semi)norm indeed exists on $\contc(\A)$, see \cite{BussExel0}. 
To prove that it is a $\cst$-norm, one constructs the reduced norm as follows. 
The restriction $\A|_X = \{ A_x \}_{x\in X}$ of $\A$ to $X$ is a $\cst$-bundle and $\contc(\A|_X)\subseteq \contc(\A)$ is a $*$-subalgebra. 
The inclusion $\contc(\A|_X)\subseteq \contc(\A)$ extends to a $\cst$-inclusion $\contz(\A|_X)\subseteq \cst(\A)$. 
The restriction map $a \mapsto a|_{X}$, $a\in \contc(\A)$, extends to a conditional expectation $\mathbb{E}: \cst(\A)\to \contz(\A|_X)$. 
Consider the Hilbert $A_{x}$-module direct sum
\[
    \ell^2(\A)_x:=\Big\{\xi:\G_x\to \A: \xi(g)\in A_g,\, g \in \G_x,\; \sum_{g \in \G_x} \xi(g)^*\xi(g) \text{ converges in $A_x$} \Big\} = \bigoplus_{g \in \G_x} A_g.
\]
Then we have a representation $\Lambda_x:\cst(\A) \to \LL(\ell^2(\A)_x)$ determined by the formula
\[
    [\Lambda_x(a) \xi ](g) = \sum_{h \in \G_x} a(g h^{-1}) \xi (h), \qquad a \in \contc(\A),\ \xi \in \ell^2(A_x),\ g \in \G_x .
\]
There is a unique topology on $\ell^2(\A) = \{\ell^2(\A)_x\}_{x\in X}$ making it a Hilbert module bundle over $\A|_{X}$, and such that $\contz(\ell^2(\A))$ contains  $x\mapsto \bigoplus_{g \in \G_x} a(g)$, for all $a\in \contc(\A)$, cf.\ \cite[3.3]{Kumjian}. 
The field of representations $\{\Lambda_{x}\}_{x\in X}$ induces a representation $\Lambda:\cst(\A)\to \LL ( \contz(\ell^2(\A)) )$ 
\[
    \big( \Lambda(a)\xi \big) (x) = \Lambda_x(a)\xi(x) , \qquad  a\in \contc(\A) , \ \xi\in \contz(\ell^2(\A)), \ x\in X,
\]
that  we call the \emph{regular representation} of $\cst(\A)$. 
This representation is faithful on $\contc(\A)$, but not on $\cst(\A)$, in general. 
When $\Lambda$ is faithful we say that $\A$ has the \emph{weak containment property}. 
This is implied by \emph{amenability} of $\G$ \cite{Sims_Williams} or some sort of amenability of $\A$ \cite{Kranz, BussMartinez}, but there are examples of non-amenable groupoids with the weak containment property \cite{Willett}.

Notice that $\contz(\ell^2(\A))$ is a $\cst(\A)$-$\contz(\A|_X)$-$\cst$-correspondence that we call the \emph{regular $\cst$-correspondence} for $\A$. 
It can be defined equivalently as the completion of the pre-Hilbert module $\contc(\A)$ with respect to the inner product $\langle a, b\rangle_{\contz(\A|_X)} := \E(a^*b)$. 
Then $\Lambda:\cst(\A)\to \LL\left(\contz(\ell^2(\A))\right)$ is determined by the multiplication in $\contc(\A)$: $\Lambda(a)b=a*b$  for $a \in \contc(\A)\subseteq \cst(\A)$ and $b\in \contc(\A)\subseteq  \contz(\ell^2(\A))$. 

\begin{defn}
The \emph{reduced section $\cst$-algebra} $\cst_{\red}(\A)$ is the completion of $\contc(\A)$ in the norm  $\|a\|_{\red}:=\sup_{x\in X}\|\Lambda_x(a)\|$. 
Thus $\cst_{\red}(\A)\cong \Lambda(\cst(\A))$ and the restriction map $\E:\contc(\A)\to \contc(\A|_X)$ extends to a faithful conditional expectation $\E:\cst_r(\A)\to \contz(\A|_X)$.	
\end{defn}

We briefly comment on the relationship between Fell bundles over the groupoid $\G$ and Fell bundles over inverse semigroups. 
This will be useful when studying representations of the section $\cst$-algebras and also checking that twisted groupoid actions yield  Fell bundles.
To this end recall that the family of (open) bisections of the groupoid $\G$, i.e.\ 
\[
    \Bis(\G) = \{ U\subseteq \G : \text{$U$ is open and $r|_U , s|_U$ are injective} \} .
\]
with operations induced by composition and inverse from $\G$ forms naturally a unital inverse semigroup. 
The unit space $X$ is the unit for the inverse semigroup $\Bis(\G)$.

Let now $S$ be any unital inverse semigroup with unit $1$. 
A \emph{Fell bundle over the inverse semigroup} $S$ is a family $\{ A_{t} \}_{t\in S}$ of Hilbert $A$-$A$-bimodules, where $A := A_1$ is a $\cst$-algebra, and hence also a trivial Hilbert $A$-bimodule, equipped with a multiplication given by bimodule embeddings $A_{t}\otimes_A A_{u}\to A_{tu}$, $t,u\in S$, which are associative in an appropriate sense and induce isomorphisms $A_{t}\otimes_A A_{t^*}\otimes_A A_{t}\cong A_t$, $t\in S$. 
See \cite{BussMartinez} for details and relationship with previous equivalent definitions. 
In particular, this data induces inclusion maps $j_{t,u}:A_{t}\to A_{u}$ and involutions $*:A_{t}\to A_{t^*}$ compatible with  the above structures. 

\begin{ex}[Inverse semigroup Fell bundle from a groupoid Fell bundle]\label{ex:Inverse semigroup Fell bundle from a groupoid Fell bundle}
Every Fell bundle $\A=\{A_g\}_{g\in \G}$ over $\G$ \emph{induces} a Fell bundle $\A_{\sem}=\{A_U\}_{U\in \Bis(\G)}$ over $\Bis(\G)$, where the fibres are the spaces $A_{U}:=\contz(\A|_{U})$, $U\in \Bis(\G)$, and multiplication maps $A_{U}\times A_{V}\to A_{UV}$ and involution maps $A_{U}\to A_{U^*}$ are given by multiplication and involution on $\A$. 
We have natural inclusions $A_{U}\subseteq A_{V}$ for $U\subseteq V$. 
Moreover, $\A$ is saturated, that is $\overline{A_g\cdot A_h}=A_{gh}$
for all $(g,h)\in \G^{(2)}$, if and only if $\A_{\sem}$ is saturated, that is $\overline{A_{U}\cdot A_{V}}=A_{UV}$ for all $U,V\in \Bis(\G)$. 
\end{ex}

In the saturated case we have the following elegant characterisation of inverse semigroup Fell bundles coming from groupoid Fell bundles as described in Example~\ref{ex:Inverse semigroup Fell bundle from a groupoid Fell bundle}:

\begin{thm}[{\cite[Theorem 6.1]{BussMeyer}}] \label{thm:groupoid_Fell_bundle_from_inverse_semigroup_Fell_bundle}
Let $S:=\Bis(\G)$ be the inverse semigroup of bisections of $\G$.
A saturated Fell bundle $\A_{\sem}=\{A_U\}_{U\in S}$ over $S$ is induced from a Fell bundle $\A=\{A_{g}\}_{g\in \G}$ over $\G$, as described in Example~\ref{ex:Inverse semigroup Fell bundle from a groupoid Fell bundle}, if and only if the map $X\supseteq U\mapsto A_{U} \triangleleft A_X$ from the lattice of open sets in $X$ to the lattice  of ideals in $A_X$ preserves arbitrary suprema.

If this is the case, the Fell bundle $\A$ over $\G$ is unique up to an isomorphism.
\end{thm}

\begin{rem}\label{rem:Buss_Meyer_theorem_bundle_description} 
For further reference we describe how the bundle $\A$ over $\G$ in Theorem~\ref{thm:groupoid_Fell_bundle_from_inverse_semigroup_Fell_bundle} is constructed from $\A_{\sem}$ in Theorem~\ref{thm:groupoid_Fell_bundle_from_inverse_semigroup_Fell_bundle}. 
If the map $X\supseteq U\mapsto A_{U} \triangleleft A_X$ preserves suprema, then (since it automatically preserves infima) by general Stone duality $A_X$ is a $\contz(X)$-algebra where $A_U= A_X \contz(U)$ for every open $U\subseteq X$.
In particular, $A_X\cong \contz(\{A_x\}_{x\in X})$ where $A_x:=A_X/I_x$ and  $I_x:=A_X \contz(X\setminus \{x\})$ for $x\in X$.
To extend this construction to  $A_U$ for $U\in S$, we use that $A_U$ is naturally a Morita equivalence $A_{r(U)}$-$A_{s(U)}$-bimodule. 
Namely, for any $g\in U$ we have $A_{U}I_{ s(g)}=I_{r(g)}A_{U}=I_{r(g)}A_{U}I_{ s(g)}$ and this subspace of $A_{U}$ denoted by $I_{U,g}$ is naturally an $I_{r(g)}$-$I_{s(g)}$-Hilbert bimodule.
Hence  the quotient $A_{U,g}:=A_{U}/I_{U,g}$ is naturally an $A_{r(g)}$-$A_{s(g)}$-Hilbert bimodule. 
Writing $a_{U}(g)$ for the image of $a\in A_U$ in $A_{U,g}$ we have a unique topology on the bundle $\{A_{U,g}\}_{g\in U}$ such that we have an isomorphism
\[
    A_{U}\cong \contz(\{A_{U,g}\}_{g\in U})\quad \text{ where }\quad A_U\ni a\longmapsto a_{U}\in \contz(\{A_{U,g}\}_{g\in U}).
\]
Using the inclusion maps from $\A_{\sem}$ one may piece together the above structures to obtain a Banach bundle $\A=\{A_g\}_{g\in \G}$ where $A_g\cong A_{U,g}$ canonically for any $U\in \Bis(\G)$ containing $g\in \G$. 
Then the formulas 
\[
    a_{U}(g)^{\star}:=a^*_{U^*}(g^{-1}), \qquad 
    a_{U}(g)\cdot b_{V}(h):= (a\cdot b)_{UV}(gh),\qquad \text{ for }a\in A_U, b\in A_{V},
\]
induce the desired Fell bundle structure on $\A$.
\end{rem}

\subsection*{Twisted actions}

Most of the literature on groupoids only deals with the so-called Green--Renault twisted actions, with the abelian twist, see \cite{Renault,Renault2}. 
We need to consider here twisted actions of $\G$ \'a la Busby--Smith (\cite{BusbySmith}), where the twist is not necessarily abelian. 
This type of twisted action appeared in \cite{BussMeyerZhu}, where they are interpreted as certain functors between $2$-categories. 

If $A$ is a $\cst$-algebra we denote by $U\M(A)$ the unitary group in $\M(A)$, and if $\A$ is a $\cst$-bundle we write $U\M(\A):=\{U\M(A_{x})\}_{x\in X}\subseteq \M(\A)$ for the corresponding group bundle.
A twisted groupoid action on a $\cst$-bundle $\A$ should be viewed as a `continuous twisted functor' $\alpha$ from $\G$ to the category of $*$-isomorphisms between the fibres of $\A$ and the twist $u$ is a `strictly continuous 2-cocycle in  $U\M(\A)$'. 
A more precise interpretation is given in \cite{BussMeyerZhu} where one views twisted actions of $\G$ as `continuous morphisms' (or weak functors) from $\G$ to the $2$-category of \cstar{}algebras whose morphisms are nondegenerate homomorphisms $A\to \M(B)$, and $2$-morphisms are certain unitary multiplier intertwiners between these homomorphisms. 
We now make precise exactly what we need. 

\begin{defn}\label{def:twisted_groupoid_action}
A \emph{twisted action of  $\G$ on a $\cst$-bundle} $\A=\{A_x\}_{x\in X}$ 
is a pair $(\alpha, u)$ where 
\begin{enumerate}
    \item\label{enu:twisted_groupoid_action1} $\alpha=\{\alpha_{g}\}_{g\in \G}$ is a family of $*$-isomorphisms $\alpha_g:A_{s(g)}\to  A_{r(g)}$, $g\in \G$, such that the map $s^*\A \ni (a,g) \mapsto (\alpha_{g}(a),g) \in r^*\A$ is continuous, where $r^*\A :=\{ (a,g)\in  \A\times \G:  a\in A_{r(g)}\}$ and $s^*\A :=\{ (a,g)\in  \A\times \G:  a\in A_{s(g)}\}$ are the pullback bundles;

    \item\label{enu:twisted_groupoid_action2} $u=\{u(g,h)\}_{(g,h)\in \G^{(2)}}$ is a  section of the pullback group bundle $r^{(2)*}U\M(\A) = \{ (w,(g,h)) \in \M(\A) \times \G^{(2)}: w\in U\M(A_{r(g)}) \},$ 
    such that for every continuous section $a$ of the pullback bundle $r^{*}\A^{(2)}:=\{((a,g), (b,h))\in r^*\A \times r^*\A: s(g)=r(h)\}$ the formula $au(g,h):=a(g,h)u(g,h)$ for $g, h \in \G$ defines a continuous section of $r^{*}\A^{(2)}$.
\end{enumerate}
In addition, we require this pair to satisfy the following algebraic conditions
\begin{enumerate}[(a)]
    \item\label{enu:twisted_groupoid_action3} $\alpha_x = \id_{A_x}$, $u(x,x) = 1_{\M(A_x)}$ for all $x\in X$;
    \item\label{enu:twisted_groupoid_action4} $\Ad_{u(g,h)}\circ \alpha_{gh}=\alpha_{g}\circ\alpha_{h}$ for all $(g,h)\in \G^{(2)}$;
    \item\label{enu:twisted_groupoid_action5} $\overline{\alpha}_f(u(g,h)) u(f,gh) = u(f,g) u(fg,h)$, for all $(f,g,h)\in \G^{(3)}$;
\end{enumerate} 
where $\overline{\alpha}_{f} : \M(A_{s(f)})\to \M(A_{r(f)})$ is the unique homomorphic extension of $\alpha_f$ for $f\in \G$.
\end{defn}

\begin{rem}\label{rem:consequnces_twists_in_groupoid_actions}
Relations \ref{enu:twisted_groupoid_action3}--\ref{enu:twisted_groupoid_action5} have some standard consequences that we will sometimes use without warning. 
For instance, $\alpha_g^{-1}=\Ad_{u(g^{-1},g)^*}\circ \alpha_{g^{-1}}$, $u(g,s(g))=u(r(g),g)=1_{\M(A_{r(g)})}$ and $\alpha_g(u(g^{-1},g))=u(g,g^{-1})$, $g\in \G$.
\end{rem}

We first note that a twisted action of $\G$ induces a twisted inverse semigroup action of $\Bis(\G)$ in the sense of Sieben~\cite{Sieben98}. 
This is a special case of the inverse semigroup twisted actions of Buss--Exel~\cite[Section 5]{BussExel0}, which were also considered in \cite[Section 2]{BartoszKangAdam}. 
The latter actions are more general and cover  twists given by groupoid circle extensions,  see \cite[Theorem~7.2]{BussExel0} or \cite[Theorem~2.16]{BartoszKangAdam}.

\begin{lem}\label{lem:inverse_semigroup_twisted_action}
Let $(\alpha,u)$ be a twisted action of $\G$ on $\A = \{ A_x \}_{x\in X}$. 
For any open $U\subseteq X$ put $A_U:=\contz(\A|_{U})$. 
Then for any $U,V\in \Bis(\G)$, the  maps $\alpha_{U}:A_{s(U)}\to A_{r(U)}$ and  sections $\omega(U, V)\in U\M(A_{r(UV)})=U\contb(\A|_{r(UV)})$ given by the formulas
\[
    \alpha_{U}(a)(r(g)):= \alpha_{g}(a (s(g))),\qquad   \omega(U, V)(r(g)):=u(g,h), \qquad g\in U, h\in V, r(h)=s(g),
\]
form a twisted inverse semigroup action of $\Bis(\G)$ on $A=\contz(\A)$, in the sense of \cite{Sieben98} and hence also of \cite{BussExel0}. 
In particular, it gives rise to a saturated inverse semigroup Fell bundle $\A_{\sem}^{(\alpha,u)} = \{ (a, U) : a\in A_{r(U)} ,\ U\in \Bis(\G) \}$, where writing $a\delta_U$ for $(a,U)\in \A_{\sem}$ we put
\[
    (a\delta_U)\cdot (b\delta_V):= \alpha_U(\alpha_{U}^{-1}(a)b)\omega(U,V) \delta_{UV}, \quad  (a\delta_U)^*=\alpha_U^{-1}(a^*)\omega(U^*,U)^*\delta_{U^*},
\]
for all $a\in A_{r(U)}$, $b\in  A_{r(V)}$, $U,V\in \Bis(\G)$.
\end{lem}
\begin{proof} 
By \ref{enu:twisted_groupoid_action1} and \ref{enu:twisted_groupoid_action2} in Definition~\ref{def:twisted_groupoid_action}, the maps $\alpha_{U}$ are well-defined isometries, and $\omega(U, V)$ are well defined elements of $U\M(A_{UV})$. 
Thus we need to check that for any $U,V,W \in S$: 
\begin{enumerate}[(T1)]
    \item\label{enu:inverse_semigroup_action1} $A_{s(UV)} = A_{s(V)} \cap \alpha_V^{-1}(A_{s(U)})$ and  \(\alpha_U\circ \alpha_V= Ad_{\omega(U,V)}\alpha_{UV}\) on this ideal; 
    \item\label{enu:inverse_semigroup_action2}  \(\alpha_U\big(a \omega(V,W)\big) \omega(U,V W) = \alpha_U(a)\omega(U,V)\omega(UV,W)\) for \(a\in A_{s(V)}\cap
    A_{r(VW)}\); 
    \item\label{enu:inverse_semigroup_action3} \(\omega(U,V) = 1_{\M(A_{r(UV)})}\) if either $U\subseteq X$ or $V\subseteq Y$.
\end{enumerate}
Let $a\in A_{s(UV)}$, and $g\in U$, $h\in V$ with $s(g)=r(h)$. 
By Definition~\ref{def:twisted_groupoid_action}\ref{enu:twisted_groupoid_action4} we get 
\[
\begin{split}
	Ad_{\omega(U,V)}\alpha_{UV}(a)(r(g)) &=u(g,h) \alpha_{gh}(a(s(h))) u(g,h)^*=\alpha_{g}(\alpha_h(a(s(h))))
	\\
	&=\alpha_{g}(\alpha_V(a)(s(g)))=\alpha_U (\alpha_V(a))(r(g)) . 
\end{split}
\]
This implies \ref{enu:inverse_semigroup_action1}, as the above calculation implies that $A_{r(UV)}=\alpha_U(A_{r(V)}\cap A_{s(U)})=\alpha_U(A_{r(V)})\cap A_{r(U)}$, which by taking inverses is equivalent to $A_{s(UV)} = A_{s(V)} \cap \alpha_V^{-1}(A_{s(U)})$.

To show \ref{enu:inverse_semigroup_action2} let \(a\in A_{s(V)} \cap A_{r(VW)}\) and $f\in U$, $g\in V$, $h\in W$ with $(f,g,h)\in \G^{(3)}$.
Using Definition~\ref{def:twisted_groupoid_action}\ref{enu:twisted_groupoid_action5} we get 
\[
\begin{split}
    [\alpha_U\big(a \omega(V,W)\big) \omega(U,V W)] (r(f)) &=\alpha_{f}\Big(a(s(f))  u( g,h)\Big) u(f,gh) \\
        &= \alpha_{f}\Big(a(s(f))\Big) u(f,g)u(fg,h) \\
        &= \big( \alpha_U(a)\omega(U,V)\omega(UV,W) \big) (r(f)).
\end{split}
\]
Property \ref{enu:inverse_semigroup_action3} follows from Definition~\ref{def:twisted_groupoid_action}\ref{enu:twisted_groupoid_action3}, 
see also Remark~\ref{rem:consequnces_twists_in_groupoid_actions}.

Every twisted inverse semigroup action naturally yields a saturated Fell bundle over the inverse semigroup, see \cite[page 250]{BussExel0}, which in the present context is as described in the assertion.
\end{proof}

\begin{cor}\label{cor:Fell_bundle_from_twisted_action} 
Let $(\alpha, u)$ be a twisted action of $\G$ on $\A = \{A_x\}_{x\in X}$. 
Write the pullback bundle $\A^{(\alpha, u)} := r^{*}\A$, as $\A^{(\alpha, u)} = \bigsqcup_{ g\in \G } A_g$ where $A_g:= A_{r(g)}$ for $g\in \G$. 
Then $\A^{(\alpha, u)}$ is a Fell bundle over the groupoid $\G$ with multiplication and involution given by  
\begin{gather*}
    A_g\times A_h \ni (a,b) \longmapsto a\cdot \alpha_g(b)u(g,h)\in  A_{gh}, \\
    A_g \ni a\longmapsto u(g^{-1},g)^*\alpha_{g^{-1}}(a^*)=\alpha_{g}^{-1}(a^*)u(g^{-1},g)^*\in A_{g^{-1}}.
\end{gather*}
\end{cor}
\begin{proof} 
One could check the conditions in \cite[Definition 2.6]{BussExel}, which is a tedious verification. 	
To avoid that we  appeal to Lemma~\ref{lem:inverse_semigroup_twisted_action} and Theorem~\ref{thm:groupoid_Fell_bundle_from_inverse_semigroup_Fell_bundle}. 
Namely, by Theorem~\ref{thm:groupoid_Fell_bundle_from_inverse_semigroup_Fell_bundle} there is  a unique Fell bundle $\A^{(\alpha,u)}$ over $\G$, such that $\A_{\sem}^{(\alpha,u)}$ is induced by  $\A^{(\alpha,u)}$.
By the description in Remark~\ref{rem:Buss_Meyer_theorem_bundle_description}, the bundle $\A^{(\alpha,u)}$ as a Banach bundle can be naturally identified with the pullback bundle $\A^{(\alpha, u)}:=r^{*}\A$, and then the multiplication and involution have to be as in the statement.
\end{proof}

Note that we have in particular $\A^{(\alpha, u)}|_X=\A$.

\begin{defn}\label{def:Groupoid_action_and_related_objects}
If $(\alpha, u)$ is a twisted action of $\G$ on $\A = \{ A_x \}_{x\in X}$, we define the \emph{full} and the \emph{reduced crossed product} for $\alpha$ as the cross-sectional \cstar{}algebras $\cst(\A^{(\alpha,u)})$ and $C_{\red}^*(\A^{(\alpha,u)})$ for the Fell bundle $\A^{(\alpha,u)}$ described in Corollary~\ref{cor:Fell_bundle_from_twisted_action}.	
\end{defn}

More specifically, the full crossed product $\cst(\A^{(\alpha,u)})$ is the maximal $\cst$-completion of the $*$-algebra $\contc(\A^{(\alpha,u)})$ where
for $a, b \in \contc(\A^{(\alpha,u)})$, $g \in \G$, we have
\[
    (a*b)(g)=\sum_{h\in \G_{s(g)}} a(gh^{-1})\alpha_{g h^{-1}}(b(h))u(gh^{-1},h),\quad  \quad  a^*(g)= u(g,g^{-1})^*\alpha_{g}(a(g^{-1})^*).
\]
The \emph{regular representation} $\Lambda$  of $\cst(\A^{(\alpha,u)})$ on $\ell^2(\A^{(\alpha,u)})=\ell^2(r^*\A)=\{\bigoplus_{g \in \G_x} A_{r(g)}\}_{x\in X}$ is determined by the representations $\Lambda_x$ on the Hilbert $A_x$-modules 
$\ell^2(r^*\A)_x=\bigoplus_{g \in \G_x} A_{r(g)}$, $x\in X$, where $[\xi \cdot a_x](g)= \xi(g)\alpha_{g}(a_x)$, $\left\langle \xi, \zeta\right\rangle_{A_x}=\sum_{h \in \G_x} \alpha_{h}^{-1}(\xi(h)^*\zeta(h))$, for $a_x\in A_x$, $\xi, \zeta \in \bigoplus_{h \in \G_x} A_{r(h)}$, $g \in \G_x$ and
\begin{equation}\label{eq:regular-rep-twisted-actions}
    \big( \Lambda_x(a)\xi \big)(g) = \sum_{h \in \G_x} a(g h^{-1}) \alpha_{g h^{-1}}(\xi(h))u(gh^{-1},h) , \qquad a \in \contc(\A^{(\alpha,u)}) ,\ g \in \G_x .
\end{equation}


\begin{ex} \label{ex:groupoid crossed product}
When $\G=G$ is a discrete group then a twisted action $(\alpha, u)$ of $\G$ is the same as a twisted group action on a $\cst$-algebra $\A=A$ in the sense of Busby--Smith~\cite{BusbySmith}; this has been considered in many other papers, for example \cite{PackerRaeburn,BedosContiregular,BedosConti,BedosConti2}, and the algebras $\cst(\A^{(\alpha,u)})$ and $C_{\red}^*(\A^{(\alpha,u)})$ are the usual full and reduced crossed products, respectively, usually also denoted by $A\rtimes_{(\alpha,u)}G$ and $A\rtimes_{(\alpha,u),r}G$ in the literature.  
\end{ex}

\begin{ex}\label{ex:continuous-cocycles} 
Twisted actions of $\G$ on a trivial bundle $\A=\{\C\}_{x\in X}$ correspond to continuous groupoid cocycles, i.e.\ for any such action $(\alpha, u)$ we necessarily have $\alpha=\{ \mathrm{id}_{\C}\}_{x\in X}$, and $u$ is a \emph{continuous  normalised $2$-cocycle} on $\G$, that is, $u : \G^{(2)} \to \mathbb{T}$ is a continuous map satisfying $u\big( r(g),g \big) = 1 = u\big( g , s(g) \big)$  and $u(f,g)u(fg,h) = u(g,h)u(f,gh)$ for every composable triple $(f,g,h) \in \G^{(3)}$. 
The algebras $\cst(\A^{(\alpha,u)})$ and $C_{\red}^*(\A^{(\alpha,u)})$ are then the usual full and reduced twisted groupoid $\cst$-algebras, usually also written as $\cst(\G,u)$ and $\cst_{\red}(\G,u)$, see \cite{Renault0}.	 
In particular, the (left) regular representation ~\eqref{eq:regular-rep-twisted-actions} for $x\in X$ specialises in this case to a homomorphism $\lambda^u_x\colon \contc(\G,u)\to \LL(\ell^2(\G_x))$, which  on the canonical basis $\{\delta_g\}_{g\in \G_x}\subseteq \ell^2(\G_x)$ is given by
\begin{equation}\label{eq:regular-rep-twisted-groupoid}
    \lambda^u_x(f)\delta_h=\sum_{g\in \G_{r(h)}}f(g)u(g,h)\delta_{gh}.
\end{equation}
We use the superscript $u$ above in order to differentiate $\lambda^u$ from the ordinary (left) regular representation $\lambda$, which coincides with $\lambda^u$ if $u=1$ is the trivial $2$-cocycle. 
\end{ex}

\begin{ex}\label{ex:matrixtwistedactions}
Consider a twisted action $(\alpha, u)$ of a groupoid $\G$ on a \cst-bundle $\A$ and let $n \in \N$. 
The lifted action $(\alpha^{(n)}, u^{(n)})$ on $\Mat_n(\A)$ is given by the isomorphisms $\alpha_g^{(n)}$ from $\Mat_n(A_{s(g)})$ to $\Mat_n(A_{r(g)})$ and unitaries $u(g,h)^{(n)} = u(g,h) \otimes I_n \in  \Mat_n (\M(A_{r(g)})) \cong \M(\Mat_n(A_{r(g)}))$. 
The algebraic and topological conditions appearing in Definition~\ref{def:twisted_groupoid_action} are clearly satisfied by the lifted action. 
Further checks show that we have the isomorphisms $\Mat_n(\cst(\A^{(\alpha,u)})) \cong  \cst(\Mat_n(\A)^{(\alpha^{(n)},u^{(n)})})$ and $\Mat_n(C_{\red}^*(\A^{(\alpha,u)})) \cong C_{\red}^*(\Mat_n(\A)^{(\alpha^{(n)},u^{(n)})})$.
Defining a matrix lift of an abstract Fell bundle is more complicated. 
The question relates to the norm on the matrices over fibres, and it can addressed via the \cst-condition, but we will not need it in this paper.
\end{ex}

\subsection{Representations of crossed products} 
\label{sec:representations_of_crossed_products}

Let us fix a twisted action $(\alpha, u)$ of $\G$  on $\A = \{ A_x \}_{x\in X}$, and recall the inverse semigroup Fell bundle $\A_{\sem}^{(\alpha,u)}$ from Lemma~\ref{lem:inverse_semigroup_twisted_action}.
For every  $U\in \Bis(\G)$, and $a\in \contc(\A|_{r(U)})$ we may identify $a\delta_U\in \A_{\sem}^{(\alpha,u)}$ with an element of $\contc(\A^{(\alpha,u)})$ given by 
\[
    (a\delta_U)(g) := [g\in U] a(r(g)),\qquad  g\in \G.
\] 
The product and involutions in $\A_{\sem}^{(\alpha,u)}$ and $\contc(\A^{(\alpha,u)})$ are compatible, and every element in $\contc(\A^{(\alpha,u)})$ is of the form $\sum_{U\in F}a_U\delta_U$ where $F\subseteq \Bis(\G)$ is finite and $a_U\in \contc(\A|_{r(U)})$, $U\in F$. 
This extends to an isomorphism between the algebras $\cst(\A^{(\alpha,u)})$ and  $C_{\red}^*(\A^{(\alpha,u)})$ and the full and reduced algebras associated to the inverse semigroup Fell bundle $\A_{\sem}^{(\alpha,u)}$, respectively, see \cite[Theorem~5.5]{BussMeyer} and \cite[Theorem~2.15, Theorem~4.11]{BussExel}. 
In particular, the relation 
\[
    \psi(\sum_{U\in F}a_U\delta_U)=\sum_{U\in F} \psi_{U}(a_U\delta_U)
\]
gives a bijective correspondence between representations $\psi : \cst(\A^{(\alpha,u)}) \to B$ in a $\cst$-algebra $B$ and families $\{ \psi_{U} \}_{ U\in\Bis(\G) }$ of linear maps $\psi_{U} : A_{r(U)} \to B$ such that for all $a\in A_{r(U)}$, $b \in A_{r(V)}$, $U,V\in \Bis(\G)$, we have
\begin{enumerate}[labelindent=40pt,label={(SR\arabic*)},itemindent=1em] 
    \item\label{item:semigroup_representation1'} $\psi_U(a\delta_U) \psi_V(b\delta_{V}) = \psi_{UV}( \alpha_U(\alpha_{U}^{-1}(a_U)a_V ) \omega(U,V) \delta_{UV})$;
    \item \label{item:semigroup_representation2'} $\psi_U(a\delta_U)^* = \psi_{U^*}(\alpha_U^{-1}(a^*) \omega(U^*,U) \delta_{U^*})$. 
\end{enumerate}
If $H$ is a Hilbert space then the relation 
\[
    \psi(\sum_{U\in F}a_U\delta_U)=\sum_{U\in F} \pi(a_U) v_{U}
\] 
gives a bijective correspondence between representations $\psi:\cst(\A^{(\alpha,u)})\to \BB(H)$ and \emph{covariant representations of $(\alpha, u)$ on $H$}, i.e.\ pairs $(\pi,v)$ where $\pi: A \to \BB(H)$ is a representation (recall that $A$ denotes the \cst-algebra associated with the bundle $\A$) and $\{v_U\}_{U\in \Bis(\G)}\subseteq \BB(H)$ is a family of partial isometries satisfying: 
\begin{enumerate}[labelindent=40pt,label={(CR\arabic*)},itemindent=1em] 
    \item\label{item:covariant_representation1'} $v_U \pi(a) v_{U}^* = \pi(\alpha_U(a))$ for all $a\in A_{s(U)} $, $U\in \Bis(\G)$;
    \item \label{item:covariant_representation2'} the ranges of $v_U$ and $v_U^*$ are  $\overline{\pi(A_{r(U)})H}$ and $\overline{\pi(A_{s(U)})H}$, for all $U\in \Bis(\G)$;
    \item\label{item:covariant_representation3'} $\pi(a)v_U v_V = \pi(a\omega(U,V)) v_{UV}$ for all $a\in A_{r(UV)}$, $U,V\in \Bis(\G)$. 
\end{enumerate}


%
%
\section{Equivariant correspondences} 
\label{Sec:equivariant}

Throughout this section $\A$ and $\B$ will be $\cst$-bundles over $X:=\G^{(0)}$, carrying twisted actions $(\alpha,u_\alpha)$ and $(\beta,u_\beta)$ of $\G$ 
(where $\G$ is an \'{e}tale locally compact Hausdorff groupoid). 
In the untwisted case, the following notion was called an \emph{equivariant representation} of $\G$ by Le Gall in \cite[Definition 4.6]{LeGall}. 
B\'{e}dos and Conti~\cite{BedosConti, BedosConti2} follow this naming, while 
Echterhoff, Kaliszewski, Quigg and Raeburn~\cite[Definition 2.5]{EKQR} call such an object a \emph{compatible right-Hilbert module} (in the context of group actions).  
Deaconu~\cite[Definition 5.10]{Deaconu} calls it a \emph{groupoid action on a $\contz(X)$-$\cst$-correspondence}. 
Here we extend this notion to twisted actions, and call it an equivariant correspondence.
	
\begin{defn}\label{def:equivariant_representation} 
Let $(\alpha, u_{\alpha})$ and $(\beta,u_{\beta})$ be twisted actions of $\G$ on $\A$ and $\B$, respectively. 
An $(\alpha, u_{\alpha})$-$(\beta,u_{\beta})$ \emph{equivariant correspondence} is a $\cst$-correspondence bundle $\EE$ from $\A$ to $\B$ endowed with an \emph{equivariant representation} of $\G$, meaning a family $L=\{L_g\}_{g\in \G}$ satisfying  
\begin{enumerate}
    \item\label{enu:equivariant_representation1} for each $g\in \G$ the map $L_g:E_{s(g)}\to E_{r(g)}$ is an invertible isometry, $L_x = \mathrm{id}_{E_x}$ for $x\in X$, and 
    \[
        L_{g}(L_{h}(\xi)) = u_{\alpha}(g,h)L_{gh}(\xi)u_{\beta}(g,h)^*, \qquad \xi\in E_{s(h)},\, (g,h)\in \G^{(2)};
    \]
    \item\label{enu:equivariant_representation2} for $g \in \G$, $\xi,\zeta \in E_{s(g)}$, $a\in A_{s(g)}$, $b\in B_{s(g)}$ we have 
    \[
        \qquad L_{g} (a \xi)=\alpha_{g}(a)  (L_{g}\xi), \qquad L_{g}(\xi b)=L_{g}(\xi) \beta_{g}(b),\qquad \beta_{g}(  \langle  \xi,\zeta\rangle_{B_{s(g)}})= \langle  L_{g}\xi, L_{g}\zeta\rangle_{B_{r(g)}};
    \]
    \item\label{enu:equivariant_representation3} for each $\xi, \zeta \in \contc(\EE)$ the section $\G\ni g \longmapsto \langle \xi(r(g)),L_g \zeta(s(g))\rangle_{B_{r(g)}}\in r^*\B$ is continuous.
\end{enumerate}
\end{defn}

\begin{rem}\label{rem:relations_equivariant_reps} 
The relations in Definition~\ref{def:equivariant_representation} imply that $L_g^{-1}(\cdot) = u_{\alpha}(g^{-1},g)^* L_{g^{-1}}(\cdot)u_{\beta}(g^{-1},g)$, $L_{g}^{-1} (\alpha_{g}(a) \xi) = a L_{g}^{-1}(\xi)$, $L_{g}^{-1}(\xi \beta_{g}(b))=L_{g}^{-1}(\xi) b$ and  $\beta_{g}^{-1}( \langle \xi,\zeta\rangle_{B_{s(g)}}) = \langle  L_{g}^{-1}\xi, L_{g}^{-1}\zeta\rangle_{B_{r(g)}}$ for all $\xi, \zeta \in E_{s(g)}$, $a\in A_{s(g)}$, $b\in B_{s(g)}$ and $g\in \G$. 
\end{rem}

\begin{lem}\label{lem:L_is_continuous}
Condition \ref{enu:equivariant_representation3} in Definition~\ref{def:equivariant_representation}, assuming \ref{enu:equivariant_representation2}, is equivalent to continuity of 
\begin{equation}\label{eq:continuous_bundle_map}
    s^*\EE\supseteq E_{s(g)}\ni \xi \longmapsto L_g(\xi)\in E_{r(g)}\in r^*\EE.
\end{equation}
\end{lem}
\begin{proof}
Clearly, continuity of \eqref{eq:continuous_bundle_map} implies \ref{enu:equivariant_representation3} in Definition~\ref{def:equivariant_representation}.  
Conversely, assume \ref{enu:equivariant_representation3} and the last relation in \ref{enu:equivariant_representation2} in Definition~\ref{def:equivariant_representation}. 
The latter implies that all maps $L_g:E_{s(g)}\to E_{r(g)}$, $g\in \G$, are isometries and for all $\xi, \zeta \in \contc(\EE)$, the section  
\[
    \G\ni g \longmapsto   \beta_{g}(  \langle  \xi,\zeta\rangle_{B_{s(g)}})=\langle L_{g}\xi, L_{g}\zeta\rangle_{B_{r(g)}} \in B_{r(g)}\subseteq  r^*\B
\]
is continuous (because $\beta$ and the inner product are `continuous'). 
Let $(\xi_i)_{i \in I}$ be a net of elements  of $s^*\EE$ convergent to $\xi_0 \in s^*\EE$. 
This means that $g_i:=p(\xi_i)\to g:=p(\xi_0)$ in $\G$ and for every $\xi\in  \contz(s^*\EE)$ we have $\|\xi_i -\xi(g_i)\|\to \|\xi_0 -\xi(g)\|$. 
Choose $\xi\in \contz(s^*\EE)$ and  $\zeta \in \contz(r^*\EE)$ such that $\xi(g)=\xi_0$ and $\zeta(g)=L_g(\xi_0)$. 
Then 
\[
    \| L_{g_i}(\xi_i) -\zeta(g_i) \| \leq \|\xi_i -\xi(g_i)\| + \|L_{g_i}(\xi(g_i)) -\zeta(g_i)\|
\]
where the first summand tends to zero because $\xi_0 =\xi(g)$. 
To see the same for $\|L_{g_i}(\xi(g_i)) -\zeta(g_i)\|$ note that its square is the norm of the following element of $B_{g_i} \subseteq \B$:
\[
    \langle L_{g_i}(\xi(g_i)) , L_{g_i}(\xi(g_i)) \rangle_{B_{r(g)}} - \langle L_{g_i}(\xi(g_i)) , \zeta(g_i), \rangle_{B_{r(g)}} -\langle \zeta(g_i), L_{g_i}(\xi(g_i)) \rangle_{B_{r(g)}} +\langle \zeta(g_i), \zeta(g_i) \rangle_{B_{r(g)}} , 
\]
which tends to 
\[
    \langle \zeta(g), \zeta(g) \rangle_{B_{r(g)}}-\langle \zeta(g), \zeta(g) \rangle_{B_{r(g)}}-\langle \zeta(g), \zeta(g) \rangle_{B_{r(g)}}+\langle \zeta(g), \zeta(g) \rangle_{B_{r(g)}}=0
\] 
by continuity of all the maps involved. 
This proves that $L_{g_i}(\xi_i) \to  L_{g}(\xi_0)$ in $r^*\EE$.
\end{proof}

For a twisted groupoid action $(\alpha, u_{\alpha})$ on $\A$ we may think of a pair $(\overline{\alpha}, u_{\alpha})$, where $\overline{\alpha}:=\{\overline{\alpha}_{g}\}_{g\in \G}$ is a family of standard strict extensions of the $*$-isomorphisms $\alpha=\{\alpha_{g}\}_{g\in \G}$, as a `strictly continuous' twisted action on $\M(\A)$. 
Then any equivariant representation of $\G$ extends to a `strictly continuous' equivariant representation between the extended actions as follows. 

\begin{prop}\label{prop:strict_extension_of_groupoid_action}
Any $(\alpha, u_{\alpha})$-$(\beta,u_{\beta})$ equivariant representation $L=\{L_g\}_{g\in \G}$ of $\G$ on $\EE$ extends to a `strictly continuous' $(\overline{\alpha}, u_{\alpha})$-$(\overline{\beta}, u_{\beta})$ equivariant representation $\overline{L}=\{\overline{L}_g\}_{g\in \G}$ on  $\M(\EE)$, in the sense that each isometry $L_g$ extends uniquely to an isometry $ \overline{L}_{g} : \M(E_{s(g)}) \to \M(E_{r(g)})$, $g\in \G$, such that $\overline{L} = \{ \overline{L}_g \}_{g\in \G}$ satisfies the obvious analogues of conditions \ref{enu:equivariant_representation1} and \ref{enu:equivariant_representation2} in Definition~\ref{def:equivariant_representation}, and the following analogue of \ref{enu:equivariant_representation3}:
\begin{enumerate}[label={(\arabic*')}]  \setcounter{enumi}{2}
    \item\label{enu:equivariant_representation3'} the section $\G\ni g \longmapsto  \langle  \xi(r(g)),a(r(g)) \overline{L}_g \zeta(s(g))\rangle_{\M(B_{r(g)})}\in B_{r(g)}\subseteq r^*\B$ is continuous for all strictly continuous sections  $\xi, \zeta \in \contb(\M(\EE))$ and any $a\in \contc(\A)$.
\end{enumerate}
\end{prop}
\begin{proof} 
For each $g\in \G$, $(L_{g},\alpha_g,\beta_g)$ is an isomorphism from the $A_{s(g)}$-$B_{s(g)}$-bimodule $E_{s(g)}$ onto the $A_{r(g)}$-$B_{r(g)}$-bimodule $E_{r(g)}$. 
Hence by Lemma~\ref{lem:extensions_lemma}, $L_g$ extends uniquely to a linear map $\overline{L}_{g}:\M(E_{s(g)})\to \M(E_{r(g)})$ such that  
\[
    \overline{L}_{g} (a \xi)=\overline{\alpha}_{g}(a)  (\overline{L}_{g}\xi) , \qquad 
    \overline{L}_{g}(\xi b)=\overline{L}_{g}(\xi) \overline{\beta}_{g}(b),
\]
for all $a\in \M(A_{s(g)})$, $b\in \M(B_{s(g)})$, $\xi \in \M(E_{s(g)})$.
Moreover, $\overline{L}_{g}$ is necessarily an invertible isometry, and it is strictly continuous. 
Strict continuity readily implies that the remaining conditions in Definition~\ref{def:equivariant_representation}\ref{enu:equivariant_representation1}, \ref{enu:equivariant_representation2} remain valid when $L$ is replaced by $\overline{L}$. 
To show \ref{enu:equivariant_representation3'} note that any $a\in \contc(\A)$ may be written as the product $a=bc$ for $b,c\in \contc(\A)$. 
Then for $\xi, \zeta \in \contb(\M(\EE))$ we have
\[
    \langle \xi(r(g)),a \overline{L}_g \zeta(s(g))\rangle_{\M(B_{r(g)})}=\langle b^*(r(g))\xi(r(g)), L_g(\alpha_{g^{-1}}(c(r(g)))\zeta(s(g)))\rangle_{\M(B_{r(g)})}\in B_{r(g)}
\]
and  the section $\G\ni g \longmapsto \langle \xi(r(g)),a(r(g)) \overline{L}_g \zeta(s(g))\rangle_{\M(B_{r(g)})} \in r^*\B$ is continuous because $b^*\xi\in \contc(\EE)$ and for each open bisection $U\subseteq \G$ the map $s(U)\ni s(g) \mapsto \alpha_{g^{-1}}(a(r(g)))\zeta(s(g))$ is a continuous section of $\EE|_{s(U)}$.
\end{proof}

\begin{ex}
For a twisted group action $(A,G,\alpha,u)$, that is, when $X=\{e\}$ and $\G=G$ is a discrete group, the $\cst$-correspondence bundle is a single $\cst$-correspondence $E$ over $A$ and  $(\alpha,u)$-$(\alpha,u)$ equivariant representations of $G$ on $E$, as in Definition~\ref{def:equivariant_representation}, generalise equivariant representations considered in \cite{BedosConti} to the non-unital case. 
\end{ex}

\begin{ex}\label{ex:continuous_Hilbert_G-bundles}
When the bundle $\A=\{\C\}_{x\in X}$ is trivial, so that a twisted action $(\alpha,u)$ on $\A$ reduces to a continuous groupoid $2$-cocycle $u$ (i.e.\ $\alpha$ is trivial, see Example~\ref{ex:continuous-cocycles}), then $(\alpha,u)$-$(\alpha,u)$ equivariant correspondences are precisely \emph{continuous Hilbert $\G$-bundles}, that is, continuous bundles $\H=\{H_x\}_{x\in X}$ of Hilbert spaces together with a continuous groupoid homomorphism $L$ from $\G$ to the unitary groupoid of $\H$, meaning that $L_g\colon H_{s(g)}\!\to H_{r(g)}$ are unitary isomorphisms satisfying $L_g L_h=L_{gh}$ for all $(g,h)\in \G^2$, and such that for continuous sections $\xi,\eta\in \contc(\H)$, the map 
\[
g\longmapsto \braket{\xi(r(g))}{L_g(\eta(s(g)))}
\]
is continuous on~$\G$. 

In particular, we have the \emph{left} and \emph{right regular continuous Hilbert $\G$-bundles}  
$(\{\ell^2(\G^x)\}_{x\in X}, \lambda)$ and $(\{\ell^2(\G_x)\}_{x\in X}, \rho)$, 
where the topology is induced by the functions from $\contc(\G)$, and for each $g\in\G$ the unitaries 
\[
\lambda_g\colon \ell^2(\G^{s(g)})\!\to\! \ell^2(\G^{r(g)}),
\qquad 
\rho_{g}\colon \ell^2(\G_{s(g)})\!\to\! \ell^2(\G_{r(g)}),
\]
are given by
\[
    \lambda_g(\delta_h):=\delta_{gh},
    \qquad 
    \rho_g (\delta_{h}):=\delta_{hg^{-1}}.
\] 
The inverse operation in $\G$ induces a unitary equivalence between these bundles. 
The relationship between the equivariant regular representations of $\G$ introduced here and the usual regular representations of the associated $\cst$-algebra, as discussed in Example~\ref{ex:continuous-cocycles}, does not appear to be fully understood. This issue is related to the next remark.
\end{ex}

\begin{rem}\label{rem:do-not-depend-on-twist}
Equivariant correspondences in Example~\ref{ex:continuous_Hilbert_G-bundles}  do not depend on the cocycle $u$. 
This is compatible with \cite[Remark~3.2]{BedosConti2}.
We also note that our notion of equivariant representation only deals with the “continuous part” of the representation theory of $\G$. 
This idea was already considered by Paterson in \cite{Paterson}, while in \cite{RamsayWalter} general measurable Hilbert bundles are also taken into account by Ramsay and Walter in order to define Fourier–Stieltjes algebras of general locally compact groupoids.
\end{rem}

\begin{ex}
Consider a $2$-cocycle $u\in Z^2(G,\T)$ on a group $G$ and the trivial action $\id$ of $G$ on $A=\C$, so that $(\id,u)$ is a twisted action of $G$ on $\C$. 
To the cocycle $u$ we can assign its left $u$-representation $\lambda^u\colon G\to \LL(\ell^2(G))$ by $\lambda^u_g(\delta_h) := u(g,h)\delta_{gh}$ (which is a special case of~\eqref{eq:regular-rep-twisted-groupoid}). 
This is a unitary (projective) $u$-representation in the sense that $\lambda^u_g\lambda^u_h=u(g,h)\lambda^u_{gh}$ for all $g,h\in G$.
Viewing $\ell^2(G)$ as a correspondence from $\K(\ell^2(G))$ to $\C$, we may view $(\ell^2(G),\lambda^u)$ as an equivariant correspondence from the (untwisted) $G$-action $(\K(\ell^2(G)),\Ad\lambda^u)$ to the twisted action $(\C,\id,u)$. 
This is an \emph{equivariant Morita equivalence}. 
Indeed, at least in the separable case every twisted group action is Morita equivalent to an ordinary action by the Packer--Raeburn stabilisation trick \cite{PackerRaeburn}. 
A version of this also holds for twisted groupoid actions, see \cite[Theorem~5.3]{BussMeyerZhu}.
\end{ex}

The following canonical constructions apply to equivariant correspondences: 

\begin{ex}\label{ex:trivial_equivariant_action} 
Any twisted action $(\alpha,u)$ of $\G$ on $\A$ can be viewed as the  $(\alpha,u)$-$(\alpha,u)$-equivariant representation of $\G$ on the trivial bundle of $\cst$-correspondences $\EE=\A$  where $L=\alpha$. 
We call it \emph{the trivial equivariant representation} of $(\alpha,u)$.
\end{ex}

\begin{ex}[Tensoring with Hilbert bundles]\label{ex:TensorActionHilbertBundle}
We may tensor any $(\alpha,u_{\alpha})$-$(\beta,u_{\beta})$ equivariant representation $L=\{L_g\}_{g\in \G}$  of $\G$ on $\EE$ with any continuous Hilbert $\G$-bundle $(\H,\ell)=(\{H_x\}_{x\in X}, \{\ell_g\}_{g\in \G})$ to get another $(\alpha,u_{\alpha})$-$(\beta,u_{\beta})$ equivariant representation $L\otimes \ell$  on $\EE\otimes \H$. 
Namely, for each $x\in X$ the external tensor product $E_x\otimes H_x$ is naturally a $\cst$-correspondence from $A_x\otimes \C\cong A_x$ to $B_x\otimes \C\cong B_x$, see \cite{Lance}. 
We equip $\EE\otimes \H:=\{E_x\otimes H_x\}_{x\in X}$ with the unique topology such that the sections $X\ni x\to \xi(x)\otimes h(x)\in E_x\otimes H_x$ are continuous for all $\xi\in \contz(\EE)$ and $h\in \contz(\H)$. 
Then the maps 
\[
    L_{g}\otimes \ell_g (\xi\otimes h):= L_{g}(\xi)\otimes \ell_g (h), \qquad \xi \in E_{s(g)}, h\in H_{s(g)},\ g \in \G,
\]
satisfy the conditions in Definition~\ref{def:equivariant_representation}.  
Tensoring $(\EE,L)$  with left or right regular Hilbert bundles, viewed as a $\G$-bundle, see Example~\ref{ex:continuous_Hilbert_G-bundles}, we get equivariant correspondences
equivalent to a \emph{regularised} version of $L$, see Definition~\ref{def:regular_reps} and Example~\ref{ex:regularized_equivariant_correspondences_revisited} below.
%
\end{ex}

\begin{ex}[Amplifications of equivariant correspondences]\label{ex:matrixreps} 
Let $(\alpha, u_{\alpha})$ and $(\beta,u_{\beta})$ be twisted actions of $\G$ on \cst-bundles $\A$ and $\B$, respectively. 
Let $L=\{L_g\}_{g\in \G}$ be an $(\alpha, u_{\alpha})$-$(\beta,u_{\beta})$ \emph{equivariant representation} of $\G$ on a $\cst$-correspondence bundle $\EE$   from $\A$ to $\B$. 
Fix $n \in \N$. 
Then $L^{(n)}= \{L_g^{(n)}\}_{g\in \G}$, with $L_g^{(n)}:\Mat_n(E_{s(g)}) \to \Mat_n(E_{r(g)})$ being the entry-wise matrix lifting of $L_g$ becomes an  $(\alpha^{(n)}, u_{\alpha}^{(n)})$-$(\beta^{(n)},u_{\beta}^{(n)})$ \emph{equivariant representation} of $\G$ on a $\cst$-correspondence bundle $\Mat_n(\EE)$ from $\Mat_n(\A)$ to $\Mat_n(\B)$ (see Examples~\ref{ex:matrixbundles} and \ref{ex:matrixtwistedactions}). 
Once again checking the algebraic and topological conditions is straightforward; moreover the strict extensions $\overline{L_g^{{}_{(n)}}}$ constructed in Proposition~\ref{prop:strict_extension_of_groupoid_action} are the expected matrix liftings of $\overline{L}_g$. 
\end{ex}

\begin{ex}[Tensor product of equivariant correspondences] \label{ex:tensor_product_groupoid_actions}
Let $L$ be an $(\alpha, u_{\alpha})$-$(\beta,u_{\beta})$ equivariant representation of $\G$ on $\EE$  and $K$ a $(\beta,u_{\beta})$-$(\gamma, u_{\gamma})$ equivariant representation of $\G$ on $\FF$ where $(\alpha, u_{\alpha})$, $(\beta,u_{\beta})$ and $(\gamma, u_{\gamma})$ are twisted actions on  $\A$, $\B$, and  $\CC$ respectively. 
In such a situation we will say that $L$ and $K$ are `composable'. 
We have an $(\alpha, u_{\alpha})$-$(\gamma, u_{\gamma})$ equivariant representation $L \otimes K$ on the tensor product $\cst$-correspondence bundle $\EE\otimes_{\B} \FF$,
discussed in Example~\ref{ex:tensor_bundles}, determined by 
\[
    (L_g\otimes K_g )(\xi\otimes \zeta) := (L_g\xi)\otimes (K_g \zeta), \qquad \xi \in E_{s(g)},\ \zeta\in  F_{s(g)},\ g \in \G.
\]
Indeed, let us first note that $L_g\otimes K_g$ extends to an isometry from $E_{s(g)}\otimes_{B_{s(g)}} F_{s(g)}$ to $E_{r(g)}\otimes_{B_{r(g)}} F_{r(g)}$ because it satisfies the last relation from~\ref{enu:equivariant_representation2} in Definition~\ref{def:equivariant_representation} on elementary tensors:
\[
\begin{split}
    \gamma_{g}(  \langle\xi_1\otimes \zeta_1,  \xi_2\otimes \zeta_2\rangle_{C_{s(g)}}) &=\gamma_{g}(  \langle \zeta_1,  \langle \xi_1,\xi_2\rangle_{B_{s(g)}}  \zeta_2\rangle_{C_{s(g)}}) \\
    &=\langle K_{g}\zeta_1,  K_{g}(\langle \xi_1,\xi_2\rangle_{B_{s(g)}}  \zeta_2)\rangle_{C_{r(g)}}\\
    &=\langle K_{g}\zeta_1, \beta_{g}(\langle \xi_1,\xi_2\rangle_{B_{s(g)}} ) K_{g} \zeta_2\rangle_{C_{r(g)}} \\
    &=\langle K_{g}\zeta_1, \langle L_{g}\xi_1,L_{g}\xi_2\rangle_{B_{r(g)}}  K_{g} \zeta_2\rangle_{C_{r(g)}} \\
    &=\langle (L_g\xi_1)\otimes (K_g \zeta_1), (L_g\xi_2)\otimes (K_g \zeta_2) \rangle_{C_{r(g)}} \\
    &=\langle (L_g\otimes K_g )\xi_1\otimes \zeta_1, (L_g\otimes K_g )\xi_2\otimes \zeta_2 \rangle_{C_{r(g)}}.
\end{split}
\]
The remaining relations in \ref{enu:equivariant_representation2} in Definition~\ref{def:equivariant_representation} are immediate, and the relation in \ref{enu:equivariant_representation1} holds because we consider balanced tensor products.
Condition \ref{enu:equivariant_representation3} holds because the section of $ r^*\CC$ sending $g\in \G$ to $\langle (L_g\otimes K_g )\xi_1\otimes \zeta_1, \xi_2\otimes \zeta_2 \rangle_{A_{r(g)}} =\langle L_{g}\zeta_1, \langle K_{g}\xi_1,\xi_2\rangle_{B_{r(g)}}  \zeta_2\rangle_{C_{r(g)}}$ is continuous for every $\xi_1, \xi_2\in \contc(\EE)$, $\zeta_1,\zeta_2\in \contc(\FF)$.  

Moreover, it is straightforward to see that the extension  $\overline{L_g \otimes K_g}:\M(E_{s(g)}\otimes_{B_{s(g)}} F_{s(g)})\to \M(E_{r(g)}\otimes_{B_{r(g)}} F_{r(g)})$ of $L_g \otimes K_g$, as described in Proposition~\ref{prop:strict_extension_of_groupoid_action}, is given on elementary tensors defined in Lemma~\ref{lem:multiplier_tensor_product} by
\[
    \overline{L_g \otimes K_g}(\xi\otimes \zeta) = \overline{L}_g(\xi)\otimes  \overline{K}_g(\zeta), \qquad \xi \in \M(E_{r(g)}), \, \zeta \in \M(F_{r(g)}).
\]
\end{ex}

\begin{ex}[Direct sums of of equivariant correspondences] 
For any family of $(\alpha, u_{\alpha})$-$(\beta,u_{\beta})$ equivariant representations $\{L_i\}_{i\in I}$ of the groupoid $\G$ on $\cst$-bundles $\{\EE_i\}_{i\in I}$ from $\A$ to $\B$ the direct sum $\oplus_{i\in I} L_i := \{ \oplus_{i\in I} (L_i)_g \}_{g\in \G}$ is a well-defined $(\alpha, u_{\alpha})$-$(\beta,u_{\beta})$ equivariant representation of $\G$ on the direct sum  $\cst$-bundle $\oplus_{i\in I} \EE_i$. 
Also we have a natural embedding $\oplus_{i\in I} \M(\EE_i)\subseteq \M(\oplus_{i\in I} \EE_i)$, which is equality when $I$ is finite, and  the extended action  $\overline{\oplus_{i\in I}L_i}$ restricted to $\oplus_{i\in I} \M(\EE_i)$ coincides with $\{\oplus_{i\in I}(\overline{L_i})_g\}_{g\in \G}$. 
\end{ex}

\subsection{Representations induced by equivariant correspondences}

Equivariant representations allow us to transfer representations of the crossed product for $(\beta,u_{\beta})$ to the crossed product for $(\alpha, u_{\alpha})$. 
Firstly, let us notice that any such equivariant representation $(\EE, L)$ can be interpreted in terms of inverse semigroups as follows. 
For every open $V\subseteq X$ we put $E_{V}:=\contz(\EE|_{V})$. 
For every $U\in \Bis(\G)$ we have an isometry $L_U: E_{s(U)}\to E_{r(U)}$ where
\[
    L_U(\xi)(r(g)) = L_{g}( \xi(s(g)) ) , \qquad g\in U , \ \xi \in E_{s(U)},
\]
see Lemma~\ref{lem:L_is_continuous}. 
Conditions \ref{enu:equivariant_representation1} and \ref{enu:equivariant_representation2} in Definition~\ref{def:equivariant_representation} imply that 
\[
    L_{U}(L_{V}(\xi)) = \omega_{\alpha}(U,V)L_{UV} (\xi) \omega_{\beta}(U,V)^*
\]
for $\xi \in E_{s(UV)}$, and for $\xi,\zeta \in E_{s(U)}$, $a\in A_{s(U)}$ and  $b\in B_{s(U)}$ we have 
\[
    L_{U} (a \xi)=\alpha_{U}(a)  L_{U}(\xi), \qquad 
    L_{U}(\xi b)=L_{U}(\xi) \beta_{U}(b) , \qquad 
    \beta_{U}(  \langle \xi, \zeta\rangle_{B_{s(U)}}) = \langle  L_{U}(\xi), L_{U}(\zeta)\rangle_{B_{r(U)}}.
\]
In particular, for any $\xi\in \contz(\EE)$ and $a\in A_{r(U)}$ we have a naturally defined element $aL_U(\xi)\in E_{r(U)}$, which agrees with $L_U(\alpha_{U}^{-1}(a)\xi)$.

\begin{prop}\label{prop:representation_induced_by_equivariant_action}
Let $L=\{L_g\}_{g\in \G}$ be an $(\alpha, u_{\alpha})$-$(\beta,u_{\beta})$ equivariant representation of $\G$ on $\EE$ and let $\psi : \cst(\B^{(\beta,u_{\beta})})\to \LL(\H)$ be a representation of $\cst(\B^{(\beta,u_{\beta})})$ on a right Hilbert $C$-module $\H$.
Let $\contz(\EE)\otimes_{\psi} \H$ be the tensor product of $\cst$-correspondences where the left action of $\contz(\B)$ on $\H$ is given by $\psi$. 
Then we get a representation $L\dashind(\psi) : \cst(\A^{(\alpha,u_{\alpha})}) \to \LL(\contz(\EE)\otimes_{\psi} \H)$ determined by the formula
\[
    L\dashind(\psi)(a\delta_U)\, \xi\otimes \zeta :=  aL_U(\xi) \otimes \psi( \delta_U)\zeta , 
    \qquad \xi \in \contz(\EE), \zeta \in \H, a\in A_{r(U)}, U\in \Bis(\G) ,
\]
where $aL_U(\xi) \otimes \psi( \delta_U)\zeta := \lim_{i} L_U(\alpha_{U}^{-1}(a)\xi) \otimes \psi(e_{i}^{r(U)} \delta_U)\zeta$ for an approximate unit $\{e_{i}^{r(U)}\}_{i}$ in $B_{r(U)}$. 
This limit exists and does not depend on the choice of $\{e_i\}_i$.
\end{prop}
\begin{proof} 
If $H$ were a Hilbert space (equivalently $C=\C$), then we could disintegrate $\psi$ into a covariant representation $(\pi,v)$, as described in Subsection~\ref{sec:representations_of_crossed_products}, and then put $L\dashind(\psi)(a\delta_U)\, \xi\otimes \zeta =  aL_U(\xi) \otimes v_U\zeta$ where the corresponding partial isometry $v_U\in\BB(H)$ is given by the strong limit of $\psi(e_{i}^{r(U)} \delta_U)$.
Here we work with a general Hilbert module $\H$ (with a view towards one of our versions of the Fell absorption principle --- see Theorem~\ref{thm:Fell_absorption} below) where such a strong limit does not necessarily exist, so we need to struggle with approximate units. 

Let us first note that the limit $\lim_{i} L_U(\alpha_{U}^{-1}(a)\xi) \otimes \psi(e_{i}^{r(U)} \delta_U)\zeta$ exists and does not depend on the choice of the approximate unit. 
Indeed, we may write $aL_U(\xi)=\kappa b$ where $\kappa \in E_{r(U)}$ and $b\in B_{r(U)}$, and then $\lim_{i} aL_U(\xi) \otimes \psi(e_{i}^{r(U)} \delta_U)\zeta = \lim_{i} \kappa\otimes \psi(b e_{i}^{r(U)} \delta_U)\zeta = \kappa\otimes \psi(b \delta_U)\zeta$.
Let $\xi_j\in \contz(\EE)$ and $\zeta_j\in \H$ for $j=1, \ldots ,n$. 
A simple calculation shows that for $j,k \in \{1, \ldots,n\}$ we have
\[
    \langle L\dashind(\psi)(a\delta_U)\, \xi_j\otimes \zeta_j, L\dashind(\psi)(a\delta_U)\, \xi_k\otimes \zeta_k\rangle_{C} = \langle \zeta_j, \psi(\langle \xi_j, \alpha_{U}^{-1}(a^*a)\xi_k\rangle )\zeta_k\rangle_C .
\]
Since $\sum_{j,k=1}^n\langle \zeta_j, \psi(\langle \xi_j, \alpha_{U}^{-1}(a^*a)\xi_k\rangle )\zeta_k\rangle_C \leq \sum_{j,k=1}^n\langle \zeta_j, \psi(\langle \xi_j, \|a^*a\|\xi_k\rangle_C) \zeta_k \rangle$ it follows that
\[
    \Bigg\|L\dashind(\psi)(a\delta_U)\left(\sum_{i=1}^n \xi_i\otimes \zeta_i\right)\Bigg\|^2\leq \|a\|^2\Bigg\|\sum_{i=1}^n \xi_i\otimes \zeta_i\Bigg\|^2.
\]
Hence $L\dashind(\psi)(a\delta_U)$ is a well defined contractive map on $\contz(\EE)\otimes_{\psi} \H$. 
As in Remark~\ref{rem:relations_equivariant_reps} we get $L_{U}^{-1}(\xi) = \omega_{\alpha}(U^*,U)^* L_{U^*}(\xi )\omega_{\beta}(U^*,U)$ for $\xi\in E_{r(U)}$. 
Thus for $\xi, \xi'\in E=\contz(\EE)$ we have
\[
\begin{split}
    \langle aL_U(\xi) ,\xi' \rangle_{B} &= \lim_{i}\langle aL_U(\xi) ,e_{i}^{r(U)}\xi' \rangle_{B} =\lim_{i} \langle L_U(\alpha_{U}^{-1}(a)\xi) ,  L_U(L_{U}^{-1}(e_i^{r(U)}\xi')) \rangle_{B_{r(U)}} \\
        &=\lim_{i}\beta_U\Big(\langle \xi ,  \alpha_U^{-1}(a^*)L_{U}^{-1}(e_i^{r(U)}\xi') \rangle_{B_{s(U)}}\Big) \\
        &=\beta_U\Big(\langle \xi ,  \alpha_U^{-1}(a^*)\omega_{\alpha}(U^*,U)^* L_{U^*}(\xi') \rangle_{B_{s(U)}}\omega_{\beta}(U^*,U)\Big).
\end{split}
\]
Hence
\[
\begin{split}
    \lim_{i}(e_i^{r(U)}\delta_U)^* \langle aL_U(\xi), \xi'\rangle_B &= \lim_{i}\beta_U^{-1}(e_i^{r(U)})\omega_{\beta}(U^*,U)^*\delta_{U^*}\langle aL_U(\xi), \xi'\rangle_B \\
        &= \beta_{U}^{-1}(\langle aL_U(\xi), \xi'\rangle_B) \omega_{\beta}(U^*,U)^*\delta_{U^*} \\
        &=\langle \xi ,  \alpha_U^{-1}(a^*)\omega_{\alpha}(U^*,U)^* L_{U^*}(\xi') \rangle_{B_{s(U)}}\delta_{U^*} .
\end{split}
\]
Using this for $\zeta,\zeta'\in \H$ we get
\[
\begin{split}
    \langle aL_U(\xi)\otimes \psi(\delta_U)\zeta ,\xi'\otimes \zeta' \rangle_{C} &= \lim_{i}\langle \psi(e_i^{r(U)}\delta_U) \zeta, \psi( \langle aL_U(\xi), \xi'\rangle_B)\zeta' \rangle_{C} \\
        &= \lim_{i}\langle  \zeta, \psi((e_i^{r(U)}\delta_U)^* \langle aL_U(\xi), \xi'\rangle_B)\zeta' \rangle_{C} \\
        &=\langle  \zeta, \psi(\langle \xi , \alpha_U^{-1}(a^*)\omega_{\alpha}(U^*,U)^* L_{U^*}(\xi') \rangle_{B_{s(U)}}\delta_{U^*}) \zeta' \rangle_{C} \\
        &= \langle\xi\otimes \zeta,  \alpha_U^{-1}(a^*)\omega_{\alpha}(U^*,U)^* L_{U^*}(\xi') \otimes \psi(\delta_{U^*})\zeta' \rangle_{C}.
\end{split}
\]
This shows that $L\dashind(\psi)(a\delta_U)\in \LL(\contz(\EE)\otimes_{\psi} \H)$ and
\[
    L\dashind(\psi)(a\delta_U)^* = L\dashind(\psi)(\alpha_U^{-1}(a^*)\omega_{\alpha}(U^*,U)^* \delta_{U^*}) = L\dashind(\psi)((a\delta_U)^*) .
\]
Now if $b\in A_{r(V)}$ for $V\in \Bis(\G)$, then
\[
\begin{split}
    L\dashind(\psi)&(a\delta_U) [L\dashind(\psi)(b\delta_V) \xi\otimes \zeta] = \lim_{i}\lim_{j} aL_U(bL_V(\xi)) \otimes \psi(e_{i}^{r(U)} \delta_U) \psi(e_{j}^{r(V)} \delta_V)\zeta \\
        &= \lim_{i} \lim_{j} \alpha_U(\alpha_{U}^{-1}(a)b)L_U(L_V(\xi)) \otimes  \psi \Big( \beta_U\big(\beta_{U}^{-1}(e_{i}^{r(U)})e_{j}^{r(V)}\big)\omega_{\beta}(U,V)\delta_{UV} \Big) \zeta \\
        &= \lim_{i} \alpha_U(\alpha_{U}^{-1}(a)b)\omega_{\alpha}(U,V) L_{UV}(\xi) \otimes \psi\Big(\omega_{\beta}(U,V)^* e_{i}^{r(UV)}\omega_{\beta}(U,V)\delta_{UV}\Big) \zeta \\
        &= L\dashind(\psi)[ (a\delta_U)\cdot (b\delta_V)] (\xi\otimes \zeta). 
\end{split}
\]
Thus $L\dashind(\psi)(a\delta_U) \cdot L\dashind(\psi)(b\delta_V) 
=L\dashind(\psi)(   (a\delta_U)\cdot (b\delta_V))$, 
and the assertion now follows, cf.\ Subsection~\ref{sec:representations_of_crossed_products}.
\end{proof}

\begin{ex}
For group actions on unital \cstar{}algebras as in \cite{BedosContiregular,BedosConti,BedosConti2,BedosConti3}, the above proposition takes the following form, see \cite[Lemma 4.9]{BedosContiregular} for its progenitor. 
Let $(\alpha,u_\alpha)$ and $(\beta,u_\beta)$ be two twisted actions  of a group $G$ on unital \cstar{}algebras $A$ and $B$, let $L$ be an equivariant representation on an $A$-$B$-\cstar{}correspondence $E$. 
Given a representation $\psi\colon B\rtimes_{(\beta,u_\beta)}G\to \LL(\H)$ on a Hilbert $C$-module $\H$, there is a representation $L\dashind(\psi)\colon A\rtimes_{(\alpha,u_\alpha)}G\to \LL(E\otimes_\psi \H)$ which on the generators $a\delta_g\in A\rtimes_{(\alpha,u_\alpha)}G$, $a\in A$, $g\in G$, is given by the formula
\[
    L\dashind(\psi)(a\delta_g)(\xi\otimes\zeta)=aL_g(\xi)\otimes\psi(\delta_g)\zeta,\quad \xi\in E,\eta\in \H.
\]
\end{ex}

\section{Induced regular representations and absorption principles} 
\label{Sec:InducedAbsorption}

In this section we will discuss induced/regular representations of bundle $\cst$-algebras related to twisted groupoid actions. 
Once again we first discuss the case of general Fell bundles.

Let $\A=\{A_g\}_{g\in \G}$ be a Fell bundle over $\G$. 
We begin this section by introducing various forms of ``regular induced representations'' for $\A$ using the $\cst(\A)$-$\contz(\A|_X)$-$\cst$-correspondence $\contz(\ell^2(\A))$, and a $\cst$-correspondence bundle from $\A|_X=\{A_x\}_{x\in X}$ to another $\cst$-bundle $\B=\{B_x\}_{x\in X}$. 
These do not seem to have appeared explicitly in the literature. 

\begin{defn}\label{def:regular induced rep}
Let $\EE=\{E_x\}_{x\in X}$ be a $\cst$-correspondence bundle from $\A|_X=\{A_x\}_{x\in X}$ to $\B=\{B_x\}_{x\in X}$. 
The \emph{regular representation of $\cst(\A)$ induced by $\EE$} is given by the left action $\Lambda^{\EE} := \Lambda\otimes \mathrm{id}_{\contz(\EE)}$ of $\cst(\A)$ on the $\cst(\A)$-$\contz(\B)$-$\cst$-correspondence $\contz(\ell^2(\A)) \otimes_{\contz(\A|_X)} \contz(\EE)$ (see Example~\ref{ex:tensor_bundles}).
\end{defn}

\begin{rem} 
By construction $\ker \Lambda \subseteq \ker \Lambda^{\EE}$ and we have  $\ker\Lambda = \ker\Lambda^{\EE}$ whenever $\EE$ is \emph{faithful}, i.e.\ the left action of $\contz(\A|_X)$ on $\contz(\EE)$ is faithful (this holds when there is a dense set of points $X_0\subseteq X$ such that the $A_x$-$B_x$-$\cst$-correspondence $E_x$ is faithful for all $x \in X_0$). 
Hence $\Lambda^{\EE}$ factors through a representation $\Lambda^{\EE}_{\red}$ of  $\cst_{\red}(\A)$ and $\Lambda^{\EE}_{\red}$ is faithful if $\EE$ is so.
\end{rem}

\begin{rem}\label{rem:induced_regular_alternative}
We give an alternative description of $\Lambda^{\EE}$ as follows. 
For each $x\in X$, we have a $\cst$-correspondence $\ell^2(\A|_{\G_x})\otimes_{A_x} E_x$ from $\cst(\A)$ to $B_x$. 
As a right Hilbert module it is equal to the direct sum $\ell^2(\A\otimes_{A_x} E_x)$ of Hilbert $B_x$-modules $\A\otimes_{A_x} E_x := \{ A_g\otimes_{A_x}E_x \}_ {g\in \G_x}$. 
The multiplication of $\A$ yields bounded linear maps $A_{gh^{-1}} \otimes_{A_x} A_h \otimes_{A_x} E_x \to A_g \otimes_{A_x} E_x$, $g,h\in \G_x$, that we write again as `multiplication' maps whenever $s(g)=s(h)$. 
Then the formula
\begin{equation}\label{eq:Alcides_formula}
    [\Lambda_{x}^{\EE}(a) \xi](g) := \sum_{h \in \G_x} a(g h^{-1})\cdot \xi(h),\qquad  a\in \contc(\A), \ \xi \in \ell^2(\A\otimes_{A_x} E_x),\ g \in \G_x ,
\end{equation}
makes sense and gives the $*$-homomorphism $\Lambda_{x}^{\EE} : \cst(\A)\to \LL(\ell^2(\A\otimes_{A_x} E_x))$ induced from $\Lambda_x$ by $E_x$. 
We equip the bundle $\ell^2(\A\otimes_{\A|_X} \EE) := \{ \ell^2(\A\otimes_{A_x} E_x)\}_{x\in X}$ with the topology generated by declaring sections of the form $ x\mapsto \oplus_{g\in \G_x}\xi(g)$ to be continuous for all $\xi\in \contz(s^*\EE)$, i.e.\ requesting that $\G\ni g\mapsto \xi(g)\in E_{s(g)}\subseteq \EE$ is continuous. 
Then we get an isomorphism of Hilbert $\contz(\B)$-modules
\[
    \contz(\ell^2(\A)) \otimes_{\contz(\A|_X)} \contz(\EE) \cong \contz(\ell^2(\A\otimes_{\A|_X} \EE))
\]
that maps $a\otimes \xi \in \contz(\ell^2(\A))\otimes_{\contz(\A|_X)}\contz(\EE)$ to $\widehat{a\otimes \xi}\in \contz(\ell^2(\A\otimes_{\A|_X} \EE))$ where $\widehat{a\otimes \xi}(x)=a(x)\otimes \xi(x)$. 
Under this isomorphism the \emph{induced regular representation} is $\Lambda^\EE : \cst(\A)\to \LL(\contz(\ell^2(\A\otimes_{\A|_X} \EE)))$, where
\[
    (\Lambda^\EE(a)\xi)(x) = \Lambda_{x}^{\EE}(a) \xi(x) = \bigoplus_{g\in \G_x}\sum_{h\in \G_x}a(gh^{-1})\cdot\xi(h),
\]
$a\in \contc(\A), \xi\in \contz(\ell^2(\A\otimes_{\A|_X} \EE))$, $x\in X$. 
\end{rem}

We will now specify the context to Fell bundles coming from twisted groupoid actions. 
Assume then now that $\A$ is a $\cst$-bundle over $X$ which carries a twisted action $(\alpha,u)$ of $\G$.
Induced regular representations of $\cst(\A^{(\alpha,u)})$ can be described as follows. 
To any $\A$-$\B$-$\cst$-correspondence bundle $\EE$ we associate the \emph{regularised $\A$-$\B$-$\cst$-correspondence bundle} $\ell^2(s^*\EE) = \{ \ell^2(s^*\EE)_x \}_{x\in X}$ whose fibres are direct sums
\[
    \ell^2(s^*\EE)_x = \bigoplus_{g \in \G_x} E_x
\]
of $A_x$-$B_x$-$\cst$-correspondences $E_x$. 
We equip $\ell^2(s^*\EE)$ with the unique topology making it a Banach bundle and such that $\contc(\ell^2(s^*\EE))$ contains $x\mapsto \bigoplus_{g \in \G_x} \xi(g)$ for $\xi\in \contc(s^*\EE)$. 
Then $\ell^2(s^*\EE)$ is an $\A$-$\B$-$\cst$-correspondence bundle. 
Alternatively, the $\cst$-correspondence bundle $\ell^2(s^*\EE)$ is determined by the $\contz(X)$-$\cst$-correspondence from $\contz(\A)$ to $\contz(\B)$ given by the completion of $\contc(s^*\EE)$ with structure determined by the operations 
\[
    \langle\xi, \zeta\rangle_{\contz(\B)}(x) := \sum_{h\in \G_x}\langle \xi(h),\zeta(h)\rangle_{B_x}, \qquad 
    a\cdot \xi\cdot b(g) := a(s(g))\xi(g) b(s(g))
\]
for $\xi, \zeta \in \contc(s^*\EE)$, $a\in\contz(\A)$, $b\in\contz(\B)$, $x\in X$, $g \in \G$.

\begin{prop}\label{prop:induced_regular_representations_for_actions}
Let $(\alpha, u)$ be a twisted action of $\G$ on $\A$ and let $\EE$ be an $\A$-$\B$-$\cst$-correspondence bundle for another $\cst$-bundle $\B$ over $X$. 
The induced regular representation $\Lambda^{\EE}$ of $\cst(\A^{(\alpha,u)})$ is naturally unitarily equivalent to a representation $\pi^{\EE}:\cst(\A^{(\alpha,u)})\to \LL(\contz(\ell^2(s^*\EE)))$, where
\[
    (\pi^\EE(a)\xi)(x)= \pi_{x}^{\EE}(a)\xi(x), \qquad a\in \contc(\A^{(\alpha,u)}) ,\ \xi\in \contz(\ell^2(s^*\EE)) ,\ x \in X ,
\]
and $\pi_{x}^{\EE} : \cst(\A^{(\alpha,u)})\to \LL(\bigoplus_{g \in \G_x} E_x)$ is a representation of $\cst(\A^{(\alpha,u)})$ on the  Hilbert $B_{x}$-module $\bigoplus_{g \in \G_x} E_x$, given by
\[
    [\pi_{x}^{\EE}(a) \xi](g) := \sum_{h \in \G_x} \alpha_{g}^{-1} \left( a(g h^{-1})u(gh^{-1},h) \right) \xi(h) , \qquad a\in \contc(\A^{(\alpha,u)}),\ \xi \in \bigoplus_{h \in \G_x} E_x ,\ g \in \G_x.
\]
In particular, $\pi^{\EE}$ descends to a representation of $\cst_{\red}(\A^{(\alpha,u)})$, which is faithful if $\pi^{\EE}$ is faithful on $\contz(\A)$.
\end{prop}
\begin{proof}
We use Remark~\ref{rem:induced_regular_alternative} to treat $\Lambda^{\EE}$ as a   representation on the bundle $\ell^2(\A\otimes_{\A|_X} \EE) := \{ \ell^2( \A\otimes_{A_x} E_x)\}_{x\in X}$ where $\A\otimes_{A_x} E_x:=\{A_g\otimes_{A_x}E_x\}_ {g\in \G_x}$ for $x\in X$. For each $g\in \G_x$, we have an isomorphism $A_g\otimes_{A_x}E_x\cong E_x$ of $B_x$-Hilbert modules given by $a\otimes \xi \mapsto \alpha_g^{-1}(a)\xi$.
These isomorphisms give a unitary isomorphism $\ell^2(\A\otimes_{A_x} E_x)\cong \bigoplus_{g \in \G_x} E_x$ under which \eqref{eq:Alcides_formula} translates to the formula for $\pi_{x}^{\EE}(a)$ as in the assertion.  
These unitaries yield an isomorphism $\ell^2(\A\otimes_{\A|_X} \EE)\cong \ell^2(s^*\EE)$ of Hilbert module bundles that intertwines $\pi^\EE$ and $\Lambda^{\EE}$.
\end{proof}

\begin{ex}[Regular representation revisited] \label{ex:regular_rep_revisited}
Suppose that $(\alpha, u)$ is a twisted $\G$-action on $\A = \{A_x\}_{x\in X}$. 
Then the bundle $\A=\{A_x\}_{x\in X}$ can be viewed as a faithful $\cst$-correspondence bundle from $\A$ to $\A$.
Hence, by Proposition~\ref{prop:induced_regular_representations_for_actions}, it induces a representation $\pi^{\A}:\cst(\A^{(\alpha,u)})\to \LL(\contz(\ell^2(s^*\A)))$, where $\ell^2(s^*\A):=\{\bigoplus_{g \in \G_x} A_x\}_{x\in X}$, which factors through a faithful representation of $\cst_{\red}(\A^{(\alpha,u)})$. 
In fact, the formula
\[
	W(\xi)(x) = \bigoplus_{g\in \G_x } \alpha_g(\xi(g)), \qquad  \xi\in \contz(\ell^2(s^*\A)),\ x \in X,
\]
defines a Hilbert $\contz(\A)$-module unitary $W:\contz(\ell^2(s^*\A))\to \contz(\ell^2(\A^{(\alpha, u)}))$ that intertwines $\pi^{\A}$ and the regular representation $\Lambda : \cst(\A^{(\alpha,u)}) \to \LL(\contz(\ell^2(\A^{(\alpha, u)})))$ from~\eqref{eq:regular-rep-twisted-actions}. 
%
\end{ex}

\begin{ex} \label{ex:groupoid regular}
For the trivial bundle $\A=\{\C\}_{x\in X}$ we are treating a twisted groupoid $(\G,u)$ as in Example~\ref{ex:continuous-cocycles}. 
In this case, the induced regular representation $\pi^\EE\colon C^*(\G,u)\to \LL(\contz(\ell^2(s^*\EE)))$ associated with the $\A$-$\B$-correspondence $\EE$ is given by: 
\[
    (\pi^\EE_x(f)\xi)(g) = \sum_{h\in \G_x}f(gh^{-1})\xi(h) u(gh^{-1},h) = \sum_{h\in \G^{r(g)}}f(h)\xi(h^{-1}g)u(h,h^{-1}g) 
\]
for $f\in \contc(\G)$, $\xi\in \oplus_{g\in \G_x} E_x=\ell^2(\G_x,E_x)$ and $g\in \G_x$. 
Here each $E_x$ is just a Hilbert module over $B_x$ because $A_x=\C$. 
Of course, this is a special case of Proposition~\ref{prop:induced_regular_representations_for_actions} where the action $\alpha$ is trivial. 
The above representation generalises the ordinary regular representation $\lambda^u$ from~\eqref{eq:regular-rep-twisted-groupoid} by adding the coefficient bundle $\EE$. 
For an equivariant representation $L$ on $\EE$, its associated regularised equivariant representation $L^\G$ acts on $\ell^2(s^*\EE)$ by the formula~\eqref{eq:induced-regular-rep}. 
There is no further simplification here, except that each $E_x$ is just a Hilbert $B_x$-module.
Further assuming that also $\B=\{\C\}_{x\in X}$ is the one-dimensional trivial bundle, then equivariant representations are described by continuous Hilbert $\G$-bundles, that is, continuous Hilbert bundles $\H=\{H_x\}_{x\in X}$ together with a continuous family of unitaries $L_g\colon H_{s(g)}\congto H_{r(g)}$ satisfying $L_gL_h=L_{gh}$ for all $(g,h)\in \G^{(2)}$, as in Example~\ref{ex:continuous_Hilbert_G-bundles}.
\end{ex}

Note that the (regularised) $\contz(\A)$-$\contz(\B)$-$\cst$-correspondence $\contz(\ell^2(s^*\EE))$ contains $\contz(\EE)$ as a sub-$\cst$-correspondence (the embedding is given by the inclusions $\contc(\EE)\subseteq \contc(s^*\EE)\subseteq \ell^2(s^*\EE)$). 
Any equivariant representation of $\G$ on $\EE$ extends naturally to $\ell^2(s^*\EE)$ as follows.

\begin{lem}\label{lem:induced_equivariant_representation}
Let $\A$, $\B$ be $\cst$-bundles over $X := \G^{(0)}$, carrying twisted actions $(\alpha,u_\alpha)$ and $(\beta,u_\beta)$ of $\G$. 
An $(\alpha,u_{\alpha})$-$(\beta,u_{\beta})$ equivariant representation $L = \{ L_g \}_{g \in \G}$ of $\G$ on a $\contz(\A)$-$\contz(\B)$-$\cst$-correspondence $\EE$ induces an $(\alpha,u_{\alpha})$-$(\beta,u_{\beta})$ equivariant representation $L^\G := \{ L^{\G}_g \}_{g \in \G}$ of $\G$ on $\ell^2(s^*\EE)$, where for $g \in \G$  the map $L^{\G}_{g}:\ell^2(s^*\EE)_{s(g)}\to \ell^2(s^*\EE)_{r(g)}$ is given by  
\begin{equation}\label{eq:induced-regular-rep}
    L^{\G}_{g}(\xi)(t) = L_g (\xi (t g)), \qquad 
    \xi\in \ell^2(s^*\EE)_{s(g)},\,\, t \in \G_{r(g)}.    
\end{equation}
\end{lem}
\begin{proof}
Given $g \in \G$ and $\xi,\zeta \in  \ell^2(s^*\EE)_{s(g)} = \bigoplus_{h \in \G_{s(g)}} E_{s(g)}$, we have 
\[
\begin{split} 
    \langle L^{\G}_g(\xi), L^{\G}_g(\zeta) \rangle_{B_{r(g)}} &= \sum_{t \in \G_{r(g)}} \langle L^{\G}_g(\xi)(t), L^{\G}_g(\zeta)(t)\rangle_{B_{r(g)}} = \sum_{t \in \G_{r(g)}} \langle L_g (\xi(t g)) , L_g (\zeta (t g))  \rangle_{B_{r(g)}} \\
        &= \beta_g\left(\sum_{t \in \G_{r(g)}} \langle \xi (t g), \zeta(t g) \rangle_{B_{s(g)}} \right) = \beta_g\left(\langle \xi, \zeta\rangle_{B_{s(g)}} \right). 
\end{split}
\]
This implies that $L^{\G}_{g} : \ell^2(s^*\EE)_{s(g)}\to \ell^2(s^*\EE)_{r(g)}$ is a well-defined isometry. 
Furthermore, for $b \in B_{s(g)}$ and $t \in \G_{r(g)}$, we have
\[ 
    L^{\G}_g ( \xi b)(t) = L_g ((\xi b)(t g)) = L_g (\xi (t g) b) = L_g (\xi (t g))  \beta_g (b) = L^{\G}_g ( \xi )(t)\beta_g (b) = (L^{\G}_g ( \xi ) \beta_g (b) )(t),  
\]
so that $L^{\G}_g (\xi b ) = L^{\G}_g (\xi) \beta_g (b)$. 
Similarly, we get $L^{\G}_g ( a \xi ) = \alpha_g (a) L^{\G}_g (\xi) $ for $a\in A_{s(g)}$.
If $g, h, t \in \G$ with $s(g)=r(h), s(t) = r(g)$ and $\xi \in \ell^2(s^*\B)_{s(h)}$ then we have
\[
\begin{split} 
    L^{\G}_g (L^{\G}_h \xi)(t) &= L_g ((L^{\G}_h \xi) (t g)) = L_g \big(L_h (\xi(t g h))\big) = u_{\alpha}(g,h) L_{g h} (\xi (t g h)) u_{\beta}(g,h)^* \\
        &= u_{\alpha}(g,h) L^\G_{g h}(\xi)(t) u_{\beta}(g,h)^*. 
\end{split}
\]
This shows conditions \ref{enu:equivariant_representation1}, \ref{enu:equivariant_representation2} in Definition~\ref{def:equivariant_representation}.
To check condition \ref{enu:equivariant_representation3}, we take sections $\xi_0,\zeta_0\in \contc(\s^*\EE)$ and consider the corresponding continuous sections $\xi, \zeta \in \contc(\ell^2(\s^*\EE))$ given by $x\mapsto \xi_0|_{\G_x}, \zeta_0|_{\G_x}$ as in the definition of the topology of $\ell^2(\s^*\EE)$. 
Then for $g \in \G$ we get
\[
    \langle \xi(r(g)), L^{\G}_g \zeta(s(g)) \rangle_{B_{r(g)}}  
	=\sum_{t \in \G_{r(g)}} \langle  \xi_0(t), L_g(\zeta_0(tg))\rangle_{B_{r(g)}} . 
\]
This is a continuous section of the pullback $r^*\B$ because $L$ is continuous and  $\xi_0,\zeta_0\in \contc(s^*\EE)$, so in particular the sum above is finite. 
\end{proof}

\begin{defn}\label{def:regular_reps} 
We call equivariant representations of $\G$ of the form $L^{\G}$ \emph{regular} (or, if we want to stress the original equivariant representation $L$, \emph{regularised}).
\end{defn}

\begin{ex}\label{ex:trivial_equivariant_action2} 
Given a twisted action $(\alpha,u)$ of $\G$ on $\A$ we can view the collection of maps $(\alpha_g)_{g \in \G}$ as an equivariant representation of $\G$ on $\EE=\A$, see Example~\ref{ex:trivial_equivariant_action}. 
Its regularisation $\alpha^{\G}$ acts on $\ell^2(s^*\A)=\{\ell^2(s^*\A)_x\}_{x\in X}$ by the formula $(\alpha^{\G}_{g}(a))_t = \alpha_g (a_{t g})$, $a \xi\in \ell^2(s^*\A)_{s(g)}$, $t \in \G_{r(g)}$. 
\end{ex}

\begin{ex}
For a twisted  group action $(A,G,\alpha, u)$, i.e.\ when $\G=G$ is a discrete group, the regular $(\alpha,u)$-$(\alpha,u)$ equivariant representations of $\G$ coincide with induced equivariant representations introduced in \cite[Example 4.8]{BedosContiregular}, and $\alpha^G$ is called the regular equivariant representation of $\alpha$ in \cite[Example 4.7]{BedosContiregular}.
\end{ex}

\begin{defn}
Let $\A$, $\B$ be $\cst$-bundles over $X$, carrying twisted actions $(\alpha,u_\alpha)$ and $(\beta,u_\beta)$ of $\G$.
Two $(\alpha,u_{\alpha})$-$(\beta,u_{\beta})$-equivariant representations of $\G$, $L$ and $L'$, defined respectively on $\cst$-correspondence bundles $\EE$ and $\FF$ from $\A$ to $\B$, are \emph{unitarily equivalent} if there is a $\contz(X)$-unitary  $\Gamma:\contz(\EE) \to \contz(\FF)$ such that $\Gamma_{r(g)} L_g = L'_g \Gamma_{s(g)}$ for all $g \in \Gamma$, where $\{\Gamma_x\}_{x \in X}$ is the associated continuous family of unitaries $\Gamma_x:E_x \to F_x$, $x\in X$, see Lemma~\ref{lem:morphisms_disintegration}.
\end{defn}

\begin{ex}[Regularised equivariant representations revisited]\label{ex:regularized_equivariant_correspondences_revisited}
Instead of using the source map $s$ in the construction of a regular equivariant representation, we could also use the range map $r$. 
More specifically, let $(\EE,L)$ be an $(\alpha,u_{\alpha})$-$(\beta,u_{\beta})$-equivariant correspondence. 
We have an $(\alpha,u_{\alpha})$-$(\beta,u_{\beta})$-equivariant correspondence $(\ell^2(r^*\EE), {^{\G}L})$ where the fibres of $\ell^2(r^*\EE) = \{ \ell^2(r^*\EE)_x \}_{x\in X}$ are direct sums $\ell^2(r^*\EE)_x = \bigoplus_{g \in \G^x} E_x$ and for each $g \in \G$  the map ${^{\G}L}_{g}:\ell^2(r^*\EE)_{s(g)}\to \ell^2(r^*\EE)_{r(g)}$ is given by  
\[
    ^{\G}L_{g}(\xi)(t) = L_g (\xi (g^{-1}t)), \qquad 
    \xi\in \ell^2(r^*\EE)_{s(g)},\,\, t \in \G^{s(g)}.    
\]
Then the regularised equivariant representations $(\ell^2(s^*\EE), L^{\G})$ and $(\ell^2(r^*\EE), {^{\G}L})$  are unitarily equivalent, with equivalence given by the bundle of unitaries $\ell^2(s^*\EE)_x\congto \ell^2(r^*\EE)_x$, where $\delta_h\mapsto \delta_{h^{-1}}$ for $h\in \G_{x}$.
Also unitary equivalence of Hilbert bundles mentioned in Example~\ref{ex:continuous_Hilbert_G-bundles} is a special case of the above notion of unitary equivalence, and tensoring $(\EE,L)$  with the left and right regular Hilbert bundles $(\{\ell^2(\G^x)\}_{x\in X}, \lambda)$ and  $(\{\ell^2(\G_x)\}_{x\in X}, \rho)$, see Examples~\ref{ex:continuous_Hilbert_G-bundles}, \ref{ex:TensorActionHilbertBundle}, we get equivariant correspondences $(\EE\otimes \{\ell^2(\G^x)\}_{x\in X},L\otimes \lambda)$ and $(\EE\otimes \{\ell^2(\G_x)\}_{x\in X}, L\otimes\rho)$ where, for each $g\in \G$ and $\xi \in E_{s(g)}$, we have
\begin{align*}
    L_g\otimes \lambda_g(\xi \otimes \delta_h)&=L_g(\xi)\otimes \delta_{gh}, \qquad\quad h\in \G^{s(g)},
    \\
    L_g\otimes \rho_g(\xi \otimes \delta_h)&=L_g(\xi)\otimes \delta_{hg^{-1}}, \qquad h\in \G_{s(g)}.
\end{align*}
Indeed, for instance the unitaries $E_x\otimes \ell^2(\G_x)\ni \sum_{h\in \G_x} \xi_h\otimes \delta_h \longmapsto \oplus_{h\in \G_x}\xi_h \in \ell^2(s^*\EE)_x=\bigoplus_{h\in \G_x} E_x$, for $x\in X$, establish the equivalence $(\ell^2(s^*\EE), L^{\G})\cong (\EE\otimes \{\ell^2(\G_x)\}_{x\in X}, L\otimes\rho)$. 
Concluding, we have the following unitary equivalences
\[
    (\ell^2(s^*\EE), L^{\G})\cong (\EE\otimes \{\ell^2(\G_x)\}_{x\in X}, L\otimes\rho)\cong (\EE\otimes \{\ell^2(\G^x)\}_{x\in X},L\otimes \lambda)\cong (\ell^2(r^*\EE), {^{\G}L}) , 
\]
 yielding equivalent pictures of the regularisation of the equivariant correspondence $(\EE,L)$.
\end{ex}

 
\begin{prop}[Absorption for equivariant representations] \label{pr:RegularActionsTensor} 
Let $\A$, $\B$  be $\cst$-bundles over $X$, carrying twisted actions $(\alpha,u_\alpha)$ and $(\beta,u_\beta)$ of $\G$.
Let $L$ be an $(\alpha,u_{\alpha})$-$(\beta,u_{\beta})$ equivariant representation of $\G$ on a $\cst$-correspon\-dence bundle $\EE$ from $\A$ to $\B$. 
The tensor product equivariant representations $L \otimes \beta^\G$ and $\alpha^\G \otimes L$, acting on $\EE \otimes_{\B} \ell^2(s^*\B)$ and $\ell^2(s^*\A)\otimes_\A \EE$, respectively, are both unitarily equivalent to the induced regular equivariant representation $L^\G$ on $ \ell^2(s^*\EE)$. 
\end{prop}
\begin{proof} 
We construct the desired $\contz(X)$-unitary $\Gamma^{\beta} : \contz(\EE \otimes_\B \ell^2(s^*\B)) \to \contz(\ell^2(s^*\EE))$. 
Given $x \in X$, there is a unique unitary $\Gamma_x^{\beta} : E_x \otimes_{B_x} \ell^2(s^*\B)_x\to \ell^2(s^*\EE)_x$ satisfying
\begin{equation}\label{eq:unitary_map_tensor} 
    \Gamma_x^{\beta}(\xi \otimes b)(t) =\xi \cdot b (t), \qquad \xi \in E_x, \ b \in \ell^2(s^*\B)_x ,\ t \in \G_x.
\end{equation}
Indeed, for $\xi, \xi' \in E_x$, $b, b' \in \ell^2(s^*\B)_x$ we have  $\xi \otimes b, \xi' \otimes b'\in E_x \otimes_{B_x} \ell^2(s^*\B)_x$ and
\[
\begin{split}
    \langle \Gamma_x^{\beta}(\xi \otimes b), \Gamma_x^{\beta}(\xi' \otimes b')\rangle_{B_{x}} &=\sum_{t \in \G_x} \langle\xi b(t),  \xi'b'(t)\rangle_{B_x} = \sum_{t \in \G_x} b(t)^*\langle\xi , \xi'\rangle_{B_x} b'(t) \\
        &= \sum_{t \in \G_x}  \langle \xi \otimes b(t), \xi' \otimes b'(t) \rangle_{B_x} = \langle \xi \otimes b, \xi' \otimes b'\rangle_{B_x} . 
\end{split}
\]
This implies that \eqref{eq:unitary_map_tensor} extends uniquely to an adjointable isometry, which is invertible as its range clearly contains the dense subspace $\contc(s^*\EE|_{\G_x})\subseteq  \ell^2(s^*\EE)_x$. 
To see that the bundle $\{\Gamma_x^{\beta}\}_{x\in X}$ of unitaries is continuous take $\xi\in \contc(\EE)$ and $b\in \contc( \ell^{2}(s^*\B))$ of the form $b(x)=\oplus_{g\in \G_x} b_0(g)$ for some $b_0\in \contc( s^*\B)$. 
Then the formula $\zeta(g):=\xi(s(g))\cdot b_0(g)$ defines an element $\zeta \in\contc( s^*\EE)$ and so
\[
    \Gamma^{\beta}(\xi\otimes b)(x) := \Gamma_x^{\beta}(\xi(x)\otimes b(x))= \bigoplus_{g\in \G_x} \xi(x)\cdot b_0(g), \qquad x \in X, 
\]
defines an element of $\contz(\ell^{2}(s^*\EE))$. 
Hence $\{\Gamma_x^{\beta}\}_{x\in X}$ defines the $\contz(X)$-unitary $\Gamma^{\beta}$. 
It intertwines the equivariant representations as for $\xi \in E_x, b \in \ell^2(s^*\B)_x$, $t,g \in \G_x$, we have
\[
\begin{split} 
    (\Gamma_{r(g)}^{\beta}  \circ (L \otimes \beta^\G)_g) (\xi \otimes b)(t) &= \Gamma_{r(g)} (L_g \xi \otimes \beta^\G_g(b))(t) = L_{g} (\xi)\cdot  \beta^\G_g(b)(t ) = L_g (\xi) \beta_g (b (t g)) \\
        &= L_g (\xi b (tg)) = L_g (\Gamma_{s(g)}^{\beta} (\xi \otimes b (tg)))=(L^{\G}_{g} \circ \Gamma_{s(g)}^{\beta}) (\xi \otimes b)(t).
\end{split}
\]
This proves the unitary equivalence between $L \otimes \beta^\G$ and $L^\G$. 
The equivalence between $\alpha^\G \otimes L$ and $L^\G$ can be proved analogously, and it is implemented by the $\contz(X)$-unitary $\Gamma^{\alpha}:\contz(\ell^2(s^*\A \otimes_{\A} \EE ) \to \contz(\ell^2(s^*\EE))$, where for $x \in X$ we set
\[ 
    \Gamma_x^{\alpha}(a \otimes\xi)(t) = a(t) \xi, \qquad \xi \in E_x ,\ a \in \ell^2(s^*\A)_x ,\ t \in \G_x.
\]
We leave the details to the reader. 
\end{proof}

\begin{cor}\label{cor:regular_equivariant_absorbtion}
Let $\A$, $\B$, $\CC$ be $\cst$-bundles over $X$, carrying twisted actions $(\alpha,u_\alpha)$, $(\beta,u_\beta)$ and $(\gamma, u_\gamma)$ of $\G$. 
Let $L$ and $K$ be two composable equivariant representations of $\G$, so $L$ is   $(\alpha,u_{\alpha})$-$(\beta,u_{\beta})$ equivariant and $K$ is $(\beta,u_{\beta})$-$(\gamma,u_{\gamma})$ equivariant, see Example~\ref{ex:tensor_product_groupoid_actions}. 
Then 
\[
    L^\G\otimes K\cong  L\otimes K^\G \cong (L\otimes K)^\G .
\]
\end{cor}
\begin{proof}
Using the `associativity' of the tensor product and applying Proposition~\ref{pr:RegularActionsTensor} four times we get $L^\G\otimes K\cong  L\otimes \beta^\G \otimes K \cong L\otimes K^\G \cong L\otimes K \otimes \gamma^\G \cong (L\otimes K)^\G$. 
\end{proof}

We will use the above absorption statement to infer that our Fourier multipliers form an ideal in the category of Fourier--Stieltjes multipliers (to be defined in the next section). 
Also, Proposition~\ref{pr:RegularActionsTensor} implies that 
\[
    \contz(\EE) \otimes_{\contz(\B)} \contz(\ell^2(s^*\B)) \cong\contz(\ell^2(s^*\EE)).
\]
To show that Fourier--Stieltjes multipliers yield completely bounded maps on reduced crossed products we will need yet another form of absorption, where we change the left multiplication in the regular $\cst$-correspondence $\ell^2(s^*\B)$ from the pointwise to the one given by the regular representation. 
The following versions of the Fell absorption principle are what we need; they are related to similar results appearing, for instance, in \cite[Lemma 2]{Sims_Williams} and \cite[Theorem 3.20]{BussMartinez}. 
The first one is a generalisation of \cite[Theorem 4.10]{BedosConti}.

\begin{thm}[Fell's absorption I]\label{thm:Fell_absorption}
Let $\A$, $\B$ be $\cst$-bundles over $X$, carrying twisted actions $(\alpha,u_\alpha)$ and $(\beta,u_\beta)$ of $\G$. 
Let $L$ be an $(\alpha,u_{\alpha})$-$(\beta,u_{\beta})$ equivariant representation of $\G$ on a $\cst$-correspondence bundle $\EE$ from $\A$ to $\B$. 
The representation $L\dashind(\pi^\B)$ induced by $L$ (as in Proposition~\ref{prop:representation_induced_by_equivariant_action}) from the regular representation $\pi^\B:\cst(\B^{(\beta,u_{\beta})})\to \LL(\contz(\ell^2(s^*\B)))$ (see Example~\ref{ex:regular_rep_revisited}) is unitarily equivalent to the regular representation $\pi^{\EE}\cong\Lambda^\EE$ 
induced by $\EE$. 
This equivalence is implemented by the isomorphism of right Hilbert $\contz(\B)$-modules 
\[
    W : \contz(\EE)\otimes_{\pi^{\B}} \contz(\ell^2(s^*\B)) \stackrel{\cong}{\longrightarrow} \contz(\ell^2(s^*\EE))
\]
where $W(\xi\otimes b )(x):= W_x(\xi \otimes b(x) )$, $\xi \in \contz(\EE)$, $b\in \contz(\ell^2(s^*\B))$, $x \in X$, and 
\[
    W_x : \contz(\EE) \otimes_{\pi^{\B}_x} \ell^2(s^*\B)_x \stackrel{\cong}{\longrightarrow} \ell^2(s^*\EE)_x, \quad
    W_x \left(\xi\otimes b_x\right)(g) := L_{g}^{-1}\left(\xi(r(g))\right) b_x(g),
\]
for $\xi\in \contz(\EE)$, $b_x\in \ell^2(s^*\B)_x$. 
\end{thm}
\begin{proof}
Let $x \in X$.
For any $\xi, \ \zeta \in E_{r(g)}$, $g\in \G_x$, and $a, b \in B_{x}$, using Remark~\ref{rem:relations_equivariant_reps}, we have 
\[
\begin{split}
    \langle L_{g}^{-1}(\xi) b ,L_{g}^{-1}(\zeta)a\rangle_{B_{x}} &= \beta_{g}^{-1}(\langle \xi \beta_g(b), \zeta \beta_g(a)\rangle_{B_{r(g)}}) = b^* \beta_{g}^{-1}( \langle \xi , \zeta \rangle_{B_{r(g)}} ) a .
\end{split}
\]
Using this, for any $\xi,\zeta \in \contz(\EE)$ and $a$, $ b\in \ell^2(s^*\B)_x$ we get
\[
\begin{split}
    \left\langle W_x \left(\xi\otimes b\right) , W_x \left(\zeta\otimes  a\right) \right \rangle_{B_x} &= \sum_{g\in \G} b(g)^* \beta_{g}^{-1}( \langle \xi(r(g))  , \zeta(r(g))\rangle_{B_{r(g)}} ) a(g) 
    =\left\langle \xi\otimes b , \zeta\otimes  a\right\rangle_{B_x}.
\end{split}
\]
This readily implies that $W_x$ is a well-defined isometry. 
It clearly has dense range, hence in fact $W_x$ is an adjointable unitary.
Now it suffices to see that $W$ is well defined, i.e.\ that for $\xi \in \contc(\EE)$, $b\in \contz(\ell^2(s^*\B))$ the formula $W(\xi\otimes b )(x)= W_x(\xi \otimes b(x) )$ defines an element of $\contz(\ell^2(s^*\EE))$.
We may assume here that $b$ is of the form $x\mapsto \oplus_{g\in \G_x} b_0(g)$ for some $b_0\in \contz(s^*\B)$, as such elements form a dense subset of $\contz(\ell^2(s^*\B))$. 
The section $\G\ni g\mapsto L_{g^{-1}}\left(\xi(r(g))\right)$ is in $\contb(s^*\EE)$ by continuity and boundedness of $\xi$ and $L$, see Lemma~\ref{lem:L_is_continuous}. 
Therefore the section $\G\ni g\mapsto L_{g^{-1}}\left(\xi(r(g)\right))u_{\beta}(g^{-1},g)b_0(g)$ is in $\contc(s^*\EE)$. 
Since $W(\xi\otimes b )(x) = \oplus_{g\in \G_x} L_{g^{-1}} \left( \xi(r(g)) \right) u_{\beta}(g^{-1},g)b_0(g)$ we conclude that $W(\xi\otimes b )\in \contz(\ell^2(s^*\EE))$.

To prove that $W$ intertwines $L\dashind(\pi^\B)$ and $\pi^{\EE}$ we need to show that for every  $U\in \Bis(\G)$, $a\in \contc(\A|_{r(U)})$, $\xi\in \contz(\EE),\ b\in \ell^2(s^*\B)$, we have
\begin{equation}\label{eq:Fell_absorption_conjugates}
    W^*\pi^{\EE}(a)W (\xi \otimes b)= a L_U(\xi) \otimes \pi^{\B}(\delta_U)(b), 
\end{equation}
where $L_U(\xi)(r(g)):=L_{g}(\xi(s(g))$ for $g\in U$ and 
\[
    \pi^{\B}(\delta_U)(b)(x) := \oplus_{g\in \G_x} \sum_{h\in \G^{r(g)}\cap U} \overline{\alpha}_{g}^{-1}\big(u(h,h^{-1}g)\big)b(h^{-1}g), \qquad x\in X.
\]
Moreover, by linearity and continuity, it suffices to consider $b$ such that for every $x\in X$,  $b(x)\in \ell^2(s^*\B)_x$, is supported on single element of $\G_x$. 
Let us then fix $x\in X$ and assume that $b_x := b(x)\in \ell^2(s^*\B)_x$, is supported on $t\in \G_x$. 
There is at most one element in $U\cap \G_{r(t)}$, that we denote by $h$ if it exists. 
We have   
\[
\begin{split}
    \big( \pi^{\EE}_x (a) W_x (\xi\otimes & b_x) \big)(g) = [g=ht] \alpha_{g}^{-1} \left( a(h)u_{\alpha}(h,t) \right) L_{t}^{-1} \left( \xi(r(t)) \right) b_x(t) \\
        &= [g=ht] \alpha_{g}^{-1} \left( a(h) u_{\alpha}(h,t) \right) L_{g}^{-1} L_{h^{-1}}^{-1} \left( u_{\alpha}(h^{-1},g)\xi(r(t)) u_{\beta}(h^{-1},g)^* \right) b_x(t) \\
        &= [g=ht] L_{g}^{-1} \left( a(h)u_{\alpha}(h,t) L_{h^{-1}}^{-1} (u_{\alpha}(h^{-1},g)\xi(r(t)) u_{\beta}(h^{-1},g)^* \right) b_x(t).
\end{split}
\]
Consider $c:=u_{\alpha}(h,t) L_{h^{-1}}^{-1} \left( u_{\alpha}(h^{-1},g)\xi(r(t)) u_{\beta}(h^{-1},g)^* \right)$ where $g=ht$. 
Since $r(t)=s(h)$ and $L_{h^{-1}}^{-1}(\cdot) = u_{\alpha}(h,h^{-1})^* L_{h}(\cdot) u_{\beta}(h,h^{-1})$ we get
\[
    c = u_{\alpha}(h,t) u_{\alpha}(h,h^{-1})^* \alpha_h(u_{\alpha}(h^{-1},g)) L_{h}\left(\xi(s(h)) \right)\beta_{h}(u_{\beta}(h^{-1},g))^* u_{\beta}(h,h^{-1}).
\]
By the cocycle identity we have $\alpha_h(u_{\alpha}(h^{-1},g)) = u_{\alpha}(h,h^{-1})u_{\alpha}(h,t)^*$  (and similarly for $\beta$). 
Thus $c = L_h(\xi(s(h))u_{\beta}(h,t)$. 
Therefore continuing the above calculations we get
\[
\begin{split}
    \big( \pi^{\EE}_x(a) W_x (\xi\otimes  b_x) \big)(g) &= [g=ht] L_{g}^{-1} \left( a(h) L_h ( \xi(s(h) ) u_{\beta}(h,t) \right) b_x(t) \\
        &= W_x \left( a L_U(\xi) \otimes \pi^{\B}(\delta_U)(b) \right)(g).
\end{split}
\]
This implies the equality \eqref{eq:Fell_absorption_conjugates} and finishes the proof.
\end{proof}

The above theorem states that for any equivariant correspondence $(\EE, L)$ the following diagram commutes
\[
\xymatrixcolsep{2.6pc} \xymatrixrowsep{2.6pc}
\xymatrix{ 
        \Cst( \A^{(\alpha,u_{\alpha})} )	 \ar@{->}[d]_{\text{\normalsize $L\dashind(\pi^\B)$}} \ar@{->}[rrd]^{\pi^\EE} \ar@{->}[rr]^{\Lambda\otimes \mathrm{id}_{\contz(\EE)\qquad\,\,} } 
        &  & \LL\Big(\contz(\ell^2(r^*\A)) \otimes_{\contz(\A)} \contz(\EE) \Big)  \ar@{<->}[d]^{\cong}
            \\
        \LL\Big( \contz(\EE)\otimes_{\pi^{\B}} \contz(\ell^2(s^*\B))\Big) \ar@{<->}[rr]_{\qquad \cong}
        & & \LL\Big( \contz(\ell^2(s^*\EE))\Big)
}
\]
where $\cong$ denotes an isomorphism implemented by a unitary between Hilbert modules. 
Thus regular representations absorb the inducing functor in the sense that composing $L\dashind$ with the regular representation one gets a  regular representation. 
For groups the representation $L\dashind(\pi^{\B})$ is nothing but the tensor product of a representation $L$ and the regular representation $\pi^{\B}$ and the above theorem reduces to the classical Fell's absorption.


Next we derive the second version of Fell's absorption principle that will be needed later, and which in particular generalises \cite[Theorem 4.11]{BedosContiregular}. Recall that regular representations induced by $\cst$-correspondence bundles were introduced in Definition~\ref{def:regular induced rep}, regularisations of equivariant representations in Lemma~\ref{lem:induced_equivariant_representation} and Definition~\ref{def:regular_reps}, and representations induced by equivariant actions in Proposition~\ref{prop:representation_induced_by_equivariant_action}.

\begin{thm}[Fell's absorption II]\label{thm:Fell_absorption II}
Let $\A$, $\B$ be $\cst$-bundles over $X$, carrying twisted actions $(\alpha,u_\alpha)$ and $(\beta,u_\beta)$ of $\G$. 
Let $L$ be an $(\alpha,u_{\alpha})$-$(\beta,u_{\beta})$ equivariant representation of $\G$ on a $\cst$-correspondence bundle $\EE$ from $\A$ to $\B$, and let $\psi : \cst( \B^{(\beta,u_{\beta})} ) \to \LL(H)$ be a representation of $\cst(\B^{(\beta,u_{\beta})})$ on a Hilbert space $H$.
The representation $L^{\G}\dashind(\psi) : \cst( \A^{(\alpha,u_{\alpha})} ) \to \LL(\contz(\ell^2(s^*\EE))\otimes_{\psi} H)$ induced from $\psi$ by the regularisation of $L$ is unitarily equivalent to the amplification $\pi^{\EE}\otimes 1_{H} : \cst(\A^{(\alpha,u_{\alpha})}) \to \LL(\contz(\ell^2(s^*\EE)) \otimes_{\psi} H)$ of the regular representation $\pi^{\EE}$ induced by $\EE$. 
\end{thm}
\begin{proof} 
Disintegrate $\psi$ into a covariant representation $(\pi,v)$ as in Subsection~\ref{sec:representations_of_crossed_products}. 
So $\pi = \psi|_{\contz(X)}$ and for any $U\in\Bis(\G)$ the partial isometry $v_{U}$ is the strong limit of $\psi(e^{r(U)}_i\delta_{U})$, where for an open set $V\subseteq X$, we choose an approximate unit $\{e^V_i\}_i$ in $ B_{V}=\contz(\B|_{V})$.
We will describe the action of the desired unitary on the dense subspace $\contc(s^*\EE)\otimes_{\pi} H$ of $\contz(\ell^2(s^*\EE))\otimes_{\pi}H$.
For $U\in\Bis(\G)$ and $\xi \in \contc(\EE|_{s(U)})$, define  $\delta_U \xi \in \contc(s^*\EE|_U)$ by 
\[
    (\delta_U \xi)(g) := \xi(s(g)) , \qquad g\in U .
\]
For any $U\in \Bis(\G)$ and any $\eta\in H$ put
\begin{equation}\label{eq:unitary_for_Fell_absorption_II}
    W (\delta_U \xi \otimes \eta):=\delta_{U^*}L_{U}(\xi) \otimes v_U\eta,
\end{equation}
where $v_U$ is the partial isometry from the disintegration of $\psi$, see Subsection~\ref{sec:representations_of_crossed_products}.
We first show that for any $\xi'\in \contc(\EE|_{s(U')})$, $U'\in\Bis(\G)$, $\eta'\in H$ we have
\begin{equation}\label{eq:inner_product_preserve_Fell}
    \langle W (\xi \delta_U\otimes \eta),  W (\xi' \delta_{U'}\otimes \eta')\rangle = \langle  \xi \delta_U\otimes \eta,  \xi' \delta_{U'}\otimes \eta'\rangle. 
\end{equation}
Let $U_0=U\cap U'$. 
The section $\langle \delta_{U^*} L_{U}(\xi) , \delta_{U'^*}L_{U'}(\xi') \rangle_{\contz(\B)}$ vanishes outside $r(U_0)$ and for $g\in U_0$ we have $\langle  \delta_{U^*}L_{U}(\xi), \delta_{U'^*}L_{U'}(\xi') \rangle_{\contz(\B)}(r(g))=\langle L_g(\xi(s(g))), L_g(\xi'(s(g)))\rangle_{B_{r(g)}}$. 
Thus 
\[
\begin{split}
    \langle \delta_{U^*}L_{U}(\xi) , \delta_{U'^*}L_{U'}(\xi') \rangle_{\contz(\B)} &= \lim_{i,j} e_{i}^{r(U_0)}\cdot\langle \delta_{U^*}L_{U}(\xi) ,  \delta_{U'^*}L_{U'}(\xi')\rangle_{\contz(\B)}\cdot e_{j}^{r(U_0)} \\
        &= \lim_{i,j} \langle \delta_{U^*}L_{U}(\xi) e_{i}^{r(U_0)}, \delta_{U'^*}L_{U'}(\xi')  e_{j}^{r(U_0)}\rangle_{\contz(\B)} \\
        &= \lim_{i,j} \langle \delta_{U^*}L_{U}(\xi e_{i}^{s(U_0)}) , \delta_{U'^*}L_{U'}(\xi' e_{i}^{s(U_0)}) \rangle_{\contz(\B)} \\
        &= \lim_{i,j} \langle \delta_{U_0^*} L_{U_0}(\xi e_{i}^{s(U_0)}) ,  \delta_{U_0^*}L_{U_0}(\xi' e_{i}^{s(U_0)})\rangle_{\contz(\B)}.
\end{split}
\]
Similarly, $\langle \delta_{U}\xi , \delta_{U'}\xi' \rangle_{\contz(\B)} = \lim_{i,j} \langle \delta_{U_0^*}\xi e_{i}^{s(U_0)}, \delta_{U_0^*}\xi' e_{j}^{s(U_0)} \rangle_{\contz(\B)}$. 
We also have $\psi(e_i^{r(U_0)})v_{U} = \psi(e_i^{r(U_0)})v_{U_0}$ and $\psi(e_i^{r(U_0)})v_{U'}=\psi(e_i^{r(U_0)})v_{U_0}$.
Using this and the covariance conditions in Subsection~\ref{sec:representations_of_crossed_products}, we get 
\[
\begin{split}
    \langle W (\delta_U\xi \otimes \eta),  W (\delta_{U'}\xi' \otimes & \eta')\rangle = \langle  v_U\eta,\pi(\delta_{U^*}\langle L_{U}(\xi) ,\delta_{U'^*} L_{U'}(\xi') \rangle) v_{U'}\eta' \rangle \\
        &= \lim_{i,j} \left\langle \eta, v_{U_0}^*\pi\left(\langle \delta_{U_0^*} L_{U_0}(\xi e_{i}^{s(U_0)}) ,  \delta_{U_0^*} L_{U_0}(\xi' e_{j}^{s(U_0)}) \rangle_{\contz(\B)} \right)v_{U_0} \eta' \right\rangle \\
        &= \lim_{i,j} \left\langle \eta , \pi \left( \beta_{U_0}^{-1} \Big( \langle\delta_{U_0^*} L_{U_0}(\xi e_{i}^{s(U_0)}) , \delta_{U_0^*} L_{U_0}(\xi' e_{j}^{s(U_0)}) \rangle_{\contz(\B)}\Big)\right)\eta' \right\rangle \\
        &= \lim_{i,j} \left\langle \eta, \pi\left(\langle \delta_{U_0} \xi e_{i}^{s(U_0)}, \delta_{U_0} \xi' e_{j}^{s(U_0)} \rangle_{\contz(\B)}\right)\eta'\right\rangle \\
        &= \left\langle \eta, \pi\left(\langle \xi \delta_{U}, \xi' \delta_{U'} \rangle_{\contz(\B)}\right)\eta'\right\rangle = \langle \xi \delta_U\otimes \eta,  \xi' \delta_{U'}\otimes \eta'\rangle.
\end{split}
\]
Now by linearity, continuity and \eqref{eq:inner_product_preserve_Fell}, we see that \eqref{eq:unitary_for_Fell_absorption_II} determines an isometry $W$ on $\contz(\ell^2(s^*\EE))\otimes_{\pi} H$. 
Using that $\omega_{\alpha}(U,U^*)a=\omega_{\alpha}(U_0,U_0^*)a$ for any $U_0\subseteq U\in \Bis(\G)$ and $a\in A_{U^0}=\contz(\A|_{U_0})$, one can adopt the above arguments to see that the formula
\[
    W^* (\delta_U\xi \otimes \eta) := \delta_{U^*}\omega_{\alpha}(U,U^*)^*L_{U}(\xi) \otimes v_{U}\eta,
\]
determines an isometry on $\contz(\ell^2(s^*\EE))\otimes_{\pi} H$. 
Since $\omega_{\alpha}(U,U^*)^*L_{U}(\cdot)\omega_{\beta}(U,U^*) = L_{U^*}^{-1}$, $\overline{\pi}(\omega_{\beta}(U,U^*)^*)v_{U}=v_{U^*}^*$ and the tensor product is balanced over $\overline{\pi}(\M(A))$, we also have
\[
    W^* (\delta_U \xi \otimes \eta) := \delta_{U^*}L_{U^*}^{-1}(\xi) \otimes v_{U^*}^*\eta.
\]
This readily implies that $WW^* = W^*W = 1$, as for instance we have 
\[
\begin{split}
    WW^* (\delta_U \xi \otimes \eta) &= \delta_{U}\xi\otimes v_{U^*}v_{U^*}^*\eta = \text{s-}\lim_{i}\delta_{U}\xi\otimes \pi(e_i^{s(U)})\eta \\
        &= \text{s-}\lim_{i}\delta_{U}\xi \otimes \pi( e_i^{s(U)}) \eta = \text{s-}\lim_{i}\delta_{U}\xi e_i^{s(U)}\otimes \eta = \delta_U \xi \otimes \eta.
\end{split}
\]
Hence $W$ is a unitary and $W^*$ is its adjoint. 

For $V\in\Bis(G)$ and $a\in A_{r(V)}$ we define $a\delta_V\in \contc(\A^{(\alpha, u_{\alpha})}|_{V})$ by $a\delta_V(g)=a(r(g))$ for $g\in V$.
One computes that $\pi^\EE(a\delta_V)\delta_{U^*}L_{U^*}^{-1}(\xi) = \delta_{VU^*}\zeta$ where $\zeta\in \contc(\EE|_{s(VU^*)})$ for $gh\in VU^*$ where $g\in V$, $h\in U^*$ is given by 
\[
\begin{split}
    \zeta(s(g)) &= \alpha_{gh}^{-1}\Big(a(r(g))u_{\alpha}(g,h)\Big) L_{h}^{-1}\big(\xi(s(g))\big) \\
        &= \alpha_{h}^{-1} \Big( \alpha_g^{-1}\big(u_{\alpha}(g,h) a(r(g))\big) \Big) L_{h}^{-1}(\xi(s(g))) \\
        &= L_{h}^{-1} \Big( \alpha_g^{-1}\big(u_{\alpha}(g,h) a(r(g))\big)\xi(s(g) \Big).
\end{split}
\]
Thus on one hand we have 
\[
    (\pi^{\EE}\otimes 1_{H}) (a\delta_V)  W^*(\delta_U\xi \otimes \eta) = \delta_{VU^*}\zeta \otimes v_{U^*}^*\eta.
\]
On the other hand, a direct computation shows that $aL_{V}^{\G}(\delta_U\xi) = \delta_{UV^*}\zeta_0$ where $\zeta_0\in \contc(\EE|_{s(UV^*)}) = \contc(\EE|_{r(VU^*)})$ is given by $\zeta_0(r(g)) := L_g\Big(\alpha_g^{-1}(a(r(g))\xi (s(g)))\Big)$ for $g\in V$. 
Moreover, for $gh\in VU^*$ where $g\in V$, $h\in U^*$, we have
\[
\begin{split}
    L_{VU^*}^{-1}(\zeta_0)(s(g)) &= L_{gh}^{-1} \Big( L_g\Big(\alpha_g^{-1}(a(r(g)))\xi (s(g)))\Big) \Big) \\
        &= L_{h}^{-1}L_g^{-1} L_g \Big( \alpha_g^{-1} (u_{\alpha}(g,h)a(r(g)))\xi (s(g))\beta_g^{-1}(u_{\beta}(g,h)^*) \Big) \\
        &= L_{h}^{-1} \Big( \alpha_g^{-1}(u_{\alpha}(g,h) a (r(g)))\xi (s(g))) \Big) \beta_{h}^{-1}( \beta_g^{-1}(u_{\beta}(g,h)^*) \\
        &= \zeta(s(gh)) \beta_{gh}^{-1}(u_{\beta}(g,h)^*).
\end{split}
\]
That is $L_{VU^*}^{-1}(\zeta_0) = \zeta \beta_{UV^*}^{-1}(\omega_{\beta}(V,U^*)^*)$. Using this we get 
\[
\begin{split}
    W^* L^{\G}\dashind(\psi) (a\delta_V) (\delta_U\xi \otimes \eta) &= W^* \big( aL_{V}^{\G}(\delta_U\xi) \otimes v_{V}\eta \big) = W^* \big( \delta_{UV^*} \zeta_0 \otimes v_{V}\eta \big) \\
        &= \delta_{VU^*}\Big(\zeta \beta_{UV^*}^{-1}(\omega_{\beta}(V,U^*)^*) \Big)\otimes v_{VU^*}^* v_{V}\eta \\
        &= \delta_{VU^*}\zeta \otimes \pi(\beta_{UV^*}^{-1}(\omega_{\beta}(V,U^*)^*))v_{VU^*}^*v_V\eta.
\end{split}
\]
Then using the covariance relations and the fact that $v_{V}^*v_V$ is a strong limit of $\{\pi(e_i^{s(V)})\}$ we obtain that
\[
\begin{split}
    \pi( \beta_{UV^*}^{-1}(\omega_{\beta}(V,U^*)^*) )v_{VU^*}^*v_V\eta &= \lim_{i}\pi(e_i^{s(VU^*)}) v_{U^*}^*v_{V}^*v_V\eta = \lim_{i,j}v_{U^*}^*\pi(\beta_{U^*}^{-1}(e_i^{s(VU^*)})) \pi(e_j^{s(V)})\eta \\
        &= \lim_{i}v_{U^*}^* \pi( \beta_{U^*}^{-1}(e_i^{s(VU^*)}))\eta = \lim_{i}\pi(e_i^{s(VU^*)})v_{U^*}^*\eta .
\end{split}
\]
Further exploiting the fact that the tensor product is balanced we get
\[
    W^* L^{\G}\dashind(\psi) (a\delta_V)  (\delta_U\xi \otimes \eta) = \delta_{VU^*}\zeta  \otimes v_{U^*}^*\eta = (\pi^{\EE}\otimes 1_{H}) (a\delta_V)  W^*(\delta_U\xi \otimes \eta).
\]
This finishes the proof.
\end{proof}

The above theorem states that for any equivariant correspondence $(\EE, L)$ and for any representation $\psi$ of $\Cst( \B^{(\alpha,u_{\alpha})} )$ on a Hilbert space  $H$ the following diagram commutes
\[
\xymatrixcolsep{2.6pc} \xymatrixrowsep{2.6pc}
\xymatrix{ 
    \Cst( \A^{(\alpha,u_{\alpha})} )	 \ar@{->}[d]_{\text{\normalsize $L^{\G}\dashind(\psi)$}} \ar@{->}[rrd]^{\pi^\EE\otimes\mathrm{id}_{H}} 
    \ar@{->}[rr]^{\Lambda\otimes \mathrm{id}_{\contz(\EE)}\otimes\mathrm{id}_{H} \qquad\qquad \quad} 
    &  & \LL\Big(\contz(\ell^2(r^*\A)) \otimes_{\contz(\A)} \contz(\EE) \otimes_{\psi} H\Big)  \ar@{<->}[d]^{\cong}
    \\
    \LL\Big( \contz(\ell^2(s^*\EE))\otimes_{\psi} H\Big) \ar@{<->}[rr]_{ \cong}
    & & \LL\Big( \contz(\ell^2(s^*\EE))\otimes_{\psi} H\Big)
    }
\]
where the isomorphisms $\cong$ are implemented by unitaries between Hilbert modules. 
Thus the regular equivariant representation $L^{\G}$ absorbs all representations in the sense that the induced representation is an amplification of a regular representation, and hence descends to the reduced crossed product. 
For groups the representation $L^{\G}\dashind(\psi)$ is nothing but the tensor product of a regular representation $L^{\G}$ and the representation $\psi$, and the above theorem recovers the classical Fell absorption principle.

\subsection{Absorption principles for groupoids} 

Let us specialise the above results to the case of trivial one-dimensional \cstar{}bundle $\A=\B=\C\times X$ over $X$. 
In this case, we are simply considering a  twisted groupoid $(\G,u)$, and equivariant representations correspond to continuous Hilbert $\G$-bundles $(\H,L)$ as in Example~\ref{ex:continuous_Hilbert_G-bundles}. 
Recall the identifications and notation introduced in Proposition \ref{prop:induced_regular_representations_for_actions}. 
The regular induced representation $\pi^\H\colon C^*(\G,u)\to \LL(\contz(\ell^2(s^*\H)))$  acts on $\contz(\ell^2(s^*\H))$, which is the Hilbert module over $\contz(X)$ obtained as the completion of $\contc(s^*\H)$ with respect to the obvious right action and inner product
\[
    \braket{\xi}{\eta}(x):=\sum_{g\in \G_x}\braket{\xi(g)}{\eta(g)},\qquad x\in X,\, \xi,\, \eta \in \contc(s^*\H).
\]
The representation $\pi^\H\colon\Cst(\G,u)\to \LL(\contz(\ell^2(s^*\H)))$ is given by the formula
\[
    \pi^\H(f)\xi(g)=\sum_{h\in \G^{r(g)}}f(h)\xi(h^{-1}g)u(h,h^{-1}g),
\]
where $ f \in \contc(\G), \xi \in \contc(s^*\H)$ and $g \in \G$.

The left regular representation $\lambda^u$ of $C^*(\G,u)$ is given by a similar formula, as it can be seen as the regular representation induced from the trivial Hilbert bundle $\C\times X$, which acts on the $\contz(X)$-Hilbert module which we will simply denote $\ell^2(\G)$: the completion of $\contc(\G)$ in the $\contz(X)$-valued inner product $\braket{\xi}{\eta}(x):=\sum_{g\in \G_x} \xi(g)\overline{\eta(g)}$, $x\in X,\, \xi,\, \eta \in \contc(\G)$. 
The induced representation $L\dashind(\lambda^u)$ acts thus on $\contz(\H)\otimes_{\contz(X)}\ell^2(\G)$ by the formula
\[
    L\dashind(\lambda^u)(f\delta_U)(\xi\otimes\eta)= f\cdot L_U(\xi)\otimes \lambda^u(\delta_U)\eta
\]
for an open bisection $U\subseteq \G$, $f\in \contc(r(U))\subseteq \contc(X)$, $\xi\in \contz(\H|_{r(U)})\subseteq \contz(\H)$, $\eta \in \contc(\G)\subseteq \ell^2(\G)$,
where $f\delta_U(g)=[g\in U]f(r(g))$, $(f\cdot L_U)(\xi)(r(g))=f(r(g))L_g(\xi(s(g)))$, and $(\lambda^u(\delta_U)\eta)(g)=[g \in \G^{r(h)}, h \in U]\eta(h^{-1}g)$ for $g \in \G$.
With the above notation, Theorem~\ref{thm:Fell_absorption} gives the following result. 

\begin{cor}
The unitary $W\colon \contz(\H)\otimes_{\contz(X)} \ell^2(\G)
\congto \contz(\ell^2(s^*\H))$ given by $W(\xi\otimes\eta)(g):=L_g^{-1}(\xi(r(g)))\eta(g)$, for $\xi \in \contz(\H), \eta \in 
\ell^2(\G)$, $g \in \G$, yields a unitary equivalence between the $L$-induced representation $L\dashind(\lambda^u)\colon C^*(\G,u)\to \LL\big(\contz(\H)\otimes_{\contz(X)} \ell^2(\G)\big)$
and the regular induced representation $\pi^\H\colon C^*(\G,u)\to \LL(\contz(\ell^2(s^*\H)))$.
\end{cor}

\begin{rem}
By Proposition~\ref{prop:induced_regular_representations_for_actions}, $\pi^{\H}$ is naturally unitarily equivalent to the representation $\Lambda^{\H} := \lambda^u\otimes \mathrm{id}_{\contz(\H)}$ of $\cst(\G,u)$ on  $\contz(\ell^2(\G)) \otimes_{\contz(X)} \contz(\H)$. Hence the above corollary amounts to saying that $\lambda^u\otimes \mathrm{id}_{\contz(\H)}\cong L\dashind(\lambda^u)$, and so the regular representation induced by $\H$ from $\lambda^u$ is unitarily equivalent to the representation induced by $L$ from $\lambda^u$. 
Pictorially,  we can tensor $\H$ by $\ell^2(\G)$ on both sides and equivariance allows us to move regular representations from  one side to another. 
\end{rem}

We recall, see Example~\ref{ex:regularized_equivariant_correspondences_revisited}, that the regularisation $(\ell^2(s^*\H), L^{\G})$ of the Hilbert $\G$-bundle $(\H,L)$ is unitarily isomorphic to the continuous Hilbert $\G$-bundle $(\H\otimes \{\ell^2(\G_x)\}_{x\in X}, L\otimes\rho)$, where  $(\{\ell^2(\G_x)\}_{x\in X},\rho)$ is the right regular continuous Hilbert $\G$-bundle, cf.\  Examples~\ref{ex:continuous-cocycles},~\ref{ex:continuous_Hilbert_G-bundles}, \ref{ex:groupoid regular}.  
Thus for any representation $\psi\colon C^*(\G,u) \to B(H)$ on a Hilbert space $H$, the induced representation $L^\G\dashind(\psi)$ acts  on $\contz(\H)\otimes_{\contz(X)}\ell^2(\G)\otimes_{\psi} H$ by the formula
\[
    L^\G\dashind(\psi)(f\delta_U)(\xi\otimes\eta\otimes\zeta)=\lim_i f\cdot L_U(\xi)\otimes \rho_U(\eta)\otimes \psi(\delta_U)\zeta
\]
for an open bisection $U\subseteq \G$, $f\in \contc(r(U))\subseteq \contc(X)$, $\xi\in \contz(\H|_{r(U)})\subseteq \contz(\H)$, $\eta \in \contc(\G)\subseteq \ell^2(\G)$,
where  $\rho_U(\eta)(g)=[g\in \G_{r(h)}, h\in U] \eta(gh)$ and $\psi(\delta_U)\zeta$ is the strong limit of $\psi(e_i\delta_U)\zeta$ where $\{e_i\}_{i\in I}$ is an approximate unit of the ideal $\contz(r(U))\subseteq \contz(X)$.
Theorem~\ref{thm:Fell_absorption II} specialises in this context to the following corollary.

\begin{cor}
Consider a continuous Hilbert $\G$-bundle $(\H, L)$ over a twisted groupoid $(\G,u)$ and a representation $\psi\colon C^*(\G,u) \to B(H)$ of $C^*(\G,u)$ on a Hilbert space $H$. 
Then the representation $L^\G\dashind(\psi): C^*(\G,u) \to \LL(\contz(\ell^2(s^*\H))\otimes_{\psi} H)$ induced from $\psi$ by the regularisation of $L$ is unitarily equivalent to the amplification $\Lambda^{\H}\otimes 1_{H} : C^*(\G,u) \to \LL(\contz(\ell^2(s^*\H)) \otimes_{\psi} H)$ of the regular representation $\Lambda^{\H}$ induced by $\H$. 
\end{cor}


%
%
\section{Fourier--Stieltjes multipliers for twisted groupoid actions}
\label{Sec:FSmultipliers}

This section is central for the work, as we are ready to exploit the results established earlier in the paper to construct the Banach category of Fourier--Stieltjes multipliers for twisted groupoid crossed products. 
Recall that $\G$ is a fixed locally compact Hausdorff \'etale groupoid with unit space $X=\G^{(0)}$. 
As in the last section we begin the discussion in the general context of Fell bundles.

\begin{defn}\label{de:MultiplierFellBundle}
Let $\A = \{ A_g \}_{g \in \G}$ and $\B = \{ B_g \}_{g \in \G}$ be Fell bundles over $\G$. 
A \emph{multiplier} from $\A$ to $\B$ is a continuous bundle of fibrewise bounded operators $T=\{T_g\}_{g\in \G}$, i.e.\ a family of maps $T_{g}\in \BB(A_g, B_g)$ such that for each $a\in \contc(\A)$ the map
\begin{equation}\label{eq:multiplier_map_defn}
    \G\ni  g \longmapsto (\m_Ta)(g) := T_{g}(a(g))\in B_{g},
\end{equation}
is in $\contc(\B)$. 
If $A$ and $B$ are completions of $\contc(\A)$ and $\contc(\B)$ respectively, we  say that the multiplier $T$ is  $A$-$B$ \emph{bounded} (resp.\ \emph{completely bounded} or \emph{completely positive}) if the map $\m_{T}:\contc(\A)\to \contc(\B)$ extends to a bounded (resp.\ completely bounded, completely positive) map $\m_{T}:A\to B$ (note that we still denote it in the same way).
Usually, we will be interested in the case when $A$ and $B$ are reduced or full section $\cst$-algebras, in which case we will call $T$ a \emph{reduced} or \emph{full multiplier}, respectively. 
\end{defn}

\begin{rem}\label{rmk:confusion}
There is a somewhat unfortunate long tradition of calling both the family $T=\{T_g\}_{g\in \G}$ and the resulting map $m_T$ a multiplier. 
As it is difficult to fight the tradition, we will rather follow it and hope that it does not cause confusion to the readers.
\end{rem}

\begin{ex}[Herz--Schur multipliers] 
If $\A = \B$, then any bounded strictly continuous section $\varphi\in \contb(r^*\M(\A|_{X}))$ defines a multiplier $T^{\varphi}$ from $\A$ to $\A$ where $T^\varphi_g a := \varphi(g) a$ for all $a\in A_{g}$, $g\in \G$. 
One may then call $T^{\varphi}$ a \emph{Herz--Schur multiplier} for $\A$ and write $\m_{\varphi} := \m_{T^{\varphi}}$, so
\begin{equation}\label{eq:Herz-Schur_multiplier_map_defn}
	(\m_{\varphi}a)(g) := \varphi(g)a(g)\in A_{g}, \qquad a\in \contc(\A),\ g\in \G.
\end{equation}
Specialising to the case of the Fell bundle $\A^{(\alpha,u)}$ of a twisted action $(\alpha,u)$ of $\G$, we have $\varphi\in \contb(r^*\M(\A))$ and the corresponding multipliers were  studied in \cite{BartoszKangAdam} (in the context of actions of twisted groupoids). 
Note that in the case of trivial bundles, every multiplier is a Herz--Schur multiplier (as fibres are one-dimensional).
The Herz--Schur multipliers studied in \cite{mtt}, \cite{mstt} for group actions, are what we would call completely bounded reduced multipliers. 
\end{ex}

\begin{ex}[Exel multipliers]\label{ex:Alcides_approximation_property_general}
Let $\A$ be a Fell bundle over $\G$ and let $\A^0:=\A|_{X}$ be the restriction of $\A$ to $X$. 
Consider two sections $\xi,\zeta\in \contc(r^*\A^0)$
and define
\[
    T_g(a) = \sum_{h\in \G^{r(g)}}\xi(h)^*\cdot a\cdot\zeta(g^{-1}h),\quad a\in A_g ,\ g \in \G.
\]
Here we use the natural $A_{r(g)}$-$A_{s(g)}$-bimodule structure of $A_g$. 
The sum above is finite because $\xi$ has compact support and $\G^{r(g)}$ is discrete. 
It is clear that $T_g$ is a bounded linear map $A_g\to A_g$.
Moreover, if $a\in \contc(\A)$, then $g\mapsto T_g(a(g))$ is again a continuous section of $\A$ because $\contc(\A)$ is a right module over $\contc(s^*(\A^0))$ in the canonical way.

For groups, multipliers of the above form first appeared in Exel's work \cite{Exel:amenability} and are used to define a notion of \emph{approximation property} for Fell bundles over groups. 
More recently, the last notion has been generalised to \'{e}tale groupoids or inverse semigroups in \cite{Kranz,BussMartinez}. 
\end{ex}


As mentioned in Remark~\ref{rmk:confusion}, in the sequel we will often identify a multiplier $T=\{T_g\}_{g\in \G}$ with the associated linear map $\m_T:\contc(\A)\to \contc(\B)$ given by \eqref{eq:multiplier_map_defn}. 
\emph{The set of all multipliers} $\m(\A, \B)$ from $\A$ to $\B$, when viewed as maps on continuous sections, is naturally an involutive $\contb(X)$-bimodule. 
Namely, for $T=\{T_g\}_{g\in \G}, S=\{S_g\}_{g\in \G}\in \m(\A, \B)$ and $f \in \contb(X)$ these structures are given by $(\lambda T)_g := \lambda T_g$, $(T+S)_g := T_g +S_g$, $ (f \cdot T)_g := f(r (g)) T_g$,  $(T \cdot f)_g := f(s(g))T_g $, and
\begin{equation}\label{eq:def-involuition-multipliers}
    T^{\dagger}_g(a) := (T_{g^{-1}} (a^{\star}))^{\star}, \qquad g \in \G,\ a \in A_g.
\end{equation}
The involution $\dagger$ is conjugate-linear and for $f,T$ as above we have $(f \cdot T)^{\dagger}=T^{\dagger}\cdot \overline{f}$, $(T\cdot f)^{\dagger}=\overline{f} \cdot T^{\dagger}$. 
In terms of the map $\m_T:\contc(\A)\to \contc(\B)$ it is given by $\m_{T^\dagger}(a):= \m_T(a^*)^*$, $a\in \contc(\A)$. 
We can also compose multipliers if the domain and codomain agree: 
for $T=\{T_g\}_{g\in \G}\in \m(\A, \B)$, $S = \{S_g\}_{g\in \G}\in \m(\B, \CC)$ we put
\[
    S\circ T := \{ S_g \circ T_g \}_{g\in \G}.
\]
In particular, all Fell bundles over $\G$ with multipliers as morphisms form a category. 
Also $\m(\A):=\m(\A, \A)$ is an algebra and an involutive $\contb(X)$-bimodule.
The involution~\eqref{eq:def-involuition-multipliers} is multiplicative on the algebra $\m(\A)$, so it is not an `involution' in the ordinary sense unless $\m(\A)$ is commutative. 

When restricting to Fell bundles coming from twisted actions of $\G$ we have a distinguished subcategory with morphisms defined as follows. 
Let $(\alpha,u_{\alpha})$ and $(\beta,u_{\beta})$ be twisted actions of $\G$ on $\cst$-bundles $\A$ and $\B$ over $X$, respectively, and consider the associated Fell bundles $\A^{(\alpha,u_{\alpha})}=r^*\A$ and $\B^{(\beta,u_{\beta})}=r^*\B$ over $\G$.
For any  $(\alpha,u_{\alpha})$-$(\beta,u_{\beta})$-equivariant representation $L$ of $\G$ on a $\cst$-$\A$-$\B$-correspondence bundle $\EE$ and any bounded strictly continuous sections $\xi,\zeta\in \contb(\M(\EE))$ the formula  
\[
    T_{g}(a) := \langle \xi(r(g)), a \overline{L}_{g} \zeta(s(g)) \rangle_{\M(B_{r(g)})} \in B_{r(g)}, \qquad a\in A_{r(g)},\ g\in \G,
\]
defines a multiplier $T_{\EE,L,\xi,\zeta}:=\{T_g\}_{g\in \G}$ from $\A^{(\alpha,u_{\alpha})}$ to $\B^{(\beta,u_{\beta})}$, cf.\ Proposition~\ref{prop:strict_extension_of_groupoid_action}\ref{enu:equivariant_representation3'}. 
We can alternatively express $T_{g}$ without using the extended action $\overline{L}$, as the following formula holds
\[
    T_{g}(a)=\lim_{i} \left\langle e_i \xi(r(g)),  L_{g}\Big(\alpha_{g^{-1}}(a)\zeta(s(g))\Big)\right\rangle_{B_{r(g)}} , 
\]
where $\{e_i\}_{i}$ is an approximate unit in $A_{r(g)}$, $a\in A_{r(g)}$, $g\in \G$.

\begin{defn} 
Let $\A$, $\B$  be $\cst$-bundles over $X$, carrying twisted actions $(\alpha,u_\alpha)$ and $(\beta,u_\beta)$ of $\G$.
We call $T_{\EE,L,\xi,\zeta}$ defined above the \emph{Fourier--Stieltjes multiplier} associated to the $(\alpha,u_{\alpha})$-$(\beta,u_{\beta})$-equivariant representation $L$ of $\G$ and sections $\xi,\zeta\in \contb(\M(\EE))$. 
We denote by $FS[(\alpha,u_{\alpha}),(\beta,u_{\beta})]$ the set of all Fourier--Stieltjes multipliers from $(\alpha,u_{\alpha})$ to $(\beta,u_{\beta})$. 
We also write $FS(\alpha,u_{\alpha}):=FS[(\alpha,u_{\alpha}),(\alpha,u_{\alpha})]$.
\end{defn}


\begin{ex} 
If $\G = G$ is a discrete group then $r^*\A \equiv  G\times A$ and $A_{r(g)}=A$ for every $g\in G$. 
Thus if $(\alpha,u_{\alpha})$ and $(\beta, u_{\beta})$ are twisted group actions of $\G$ respectively on $\cst$-algebras $A$ and $B$, we may identify the Fourier--Stieltjes multipliers in $FS[(\alpha,u_{\alpha}),  (\beta, u_{\beta})]$ with maps $T_{E,L,\xi,\zeta} : G\times A \to B$ given by $T_{E,L,\xi,\zeta}(g,a)=\langle \xi, a \overline{L}_{g}\zeta\rangle_{\M(B)}$, for $a\in A$, $\xi,\zeta \in \M(E)$, where $L$ is an equivariant representation of $\G$ on an $A$-$B$-$\cst$-correspondence $E$. 
When $A=B$ is unital, these maps were called  Fourier--Stieltjes coefficients in \cite{BedosConti}. 
In contrast to \cite{BedosConti} we do not require the unitality of the $\cst$-algebras in question and also allow different `source' and `target' algebras.   
\end{ex}

\begin{ex} \label{ex:FS groupoid}
If $\A$ is the trivial one-dimensional bundle over $X$, so that $A=\contz(X)$ and $r^*\A$ is the trivial one-dimensional $\G$-bundle, we may identify the Fourier--Stieltjes multipliers with Fourier--Stieltjes coefficients $T_{\H, \ell,\xi,\zeta}:\G \to \C$ given by $T_{\H,\ell,\xi,\zeta}(g)=\langle \xi(r(g)), \ell_g\zeta(s(g))\rangle$, where $(\H, \ell)$ is a $\G$-Hilbert bundle and $\xi,\zeta \in \contb(\H)$, see for instance \cite{Paterson} (in the untwisted case) and compare to Examples~\ref{ex:continuous-cocycles} and~\ref{ex:continuous_Hilbert_G-bundles}. 
\end{ex}

\begin{ex}\label{ex:identity_multiplier} 
If $\A$ is a $\cst$-bundle over $X$, carrying a twisted action $(\alpha,u_\alpha)$ of $\G$, then the identity map on $\contc(\A^{(\alpha,u_{\alpha})})$ arises as
a Fourier--Stieltjes multiplier. 
Namely, the unit section $1$ belongs to $\contb(\M(\A))$, for the multiplier $\cst$-correspondence bundle $\M(\A)=\{\M(A_x)\}_{x\in X}$, and $T_{\A,\alpha,1,1}\in FS(\alpha,u_{\alpha})$ is a  multiplier consisting of identities on each fibre. 
\end{ex}

\begin{lem}\label{lem:FS_multiplier_bounded} 
Let $\A$, $\B$ be $\cst$-bundles over $X$, carrying twisted actions $(\alpha,u_\alpha)$ and $(\beta,u_\beta)$ of $\G$.
Every Fourier--Stieltjes multiplier $T:=T_{\EE,L,\xi,\zeta} \in FS[(\alpha,u_{\alpha}),(\beta,u_{\beta})]$ is  
$\contz$\nb-completely bounded, meaning that the operator $\m_{T} : \contz(\A^{(\alpha,u_{\alpha})}) \to \contz(\B^{(\beta,u_{\beta})})$ is completely bounded, where the spaces are equipped with the supremum norm. 
In fact, $\| \m_{T} \|_{cb} \leq \|\xi\| \cdot \|\zeta\|$. 
The same statement holds if we view $\m_{T}$ as an operator mapping $\contb(\A^{(\alpha,u_{\alpha})})$ to $\contb(\B^{(\beta,u_{\beta})})$.
\end{lem}
\begin{proof} 
Fix $T:=T_{\EE,L,\xi,\zeta} \in FS[(\alpha,u_{\alpha}),(\beta,u_{\beta})]$.
The Schwarz inequality in the $\contz(\B)$-Hilbert module $\contz(\EE)$ implies that $\|\m_T(a)\|\leq \|a\| \|\xi\| \cdot \|\zeta\|$ for $a\in \contc(\A^{(\alpha,u_{\alpha})})$. 
Hence we have a bounded operator $\m_{T}: \contz(\A^{(\alpha,u_{\alpha})})\to \contz(\B^{(\beta,u_{\beta})})$ with $\|\m_{T}\|\leq \|\xi\| \cdot \|\zeta\|$. 
This generalises to amplifications.
Indeed, for any $n \in \N$, we have a natural identification of matrix amplifications $\Mat_n(\contz(r^*\A))\cong \contz(r^*(\Mat_n(\A))$, where $\Mat_n(\A)$ is the Fell bundle over $X$ with fibres given by $n$ by $n$ matrices over fibres of $\A$ (and the twisted action of $\G$ extended to $\Mat_n(\A)$ in the obvious way), see Example~\ref{ex:matrixtwistedactions}.
Thus to estimate the norm of $m_T^{(n)}: \Mat_n(\contz(\A^{(\alpha,u_{\alpha})})) \to \Mat_n(\contz(\B^{(\beta,u_{\beta})}))$ it suffices to fix $g \in \G$ and estimate the norm of $T_g^{(n)}: \Mat_n(A_{r(g)}) \to \Mat_n (B_{r(g)})$. 
A straightforward computation shows that given a matrix $a=[a_{ij}]_{i,j=1}^n \in \Mat_n (A_{r(g)})$ we have 
\[ 
    T_g^{(n)}(a) = \langle \tilde{\xi} , a \cdot \tilde{\zeta} \rangle,
\]
where $\tilde{\xi}= [\xi(r(g)) \delta_{ij}]_{i,j=1}^n \in \Mat_n (\EE_{r(g)})$, $\tilde{\zeta}= [\overline{L}_g \zeta(s(g)) \delta_{ij}]_{i,j=1}^n \in \Mat_n (\EE_{r(g)})$.
As before, the Schwarz inequality shows that 
\[
    \|T_g^{(n)}(a)\| \leq \|a\| \|\tilde{\xi}\| \|\tilde{\zeta}\|,
\]
and it remains to note that by \cite[Section 3, Formula (1')]{Ble} we have 
\begin{align*}
    \|\tilde{\xi}\| &= \Big\| \Big[\sum_{k=1}^n \langle \tilde{\xi}_{ki}, \tilde{\xi}_{kj}\rangle\Big]_{i,j=1}^n \Big\|^{\frac{1}{2}} = \Big\| \big[ \langle \xi(r(g)), \xi(r(g)) \rangle \delta_{i,j} \big]_{i,j=1}^n \Big\|^{\frac{1}{2}} \\
		&= \|\langle \xi(r(g)), \xi(r(g)) \rangle \|^{\frac{1}{2}} 
		= \|\xi(r(g))\|\leq \|\xi\|
\end{align*}
and similarly $\|\tilde{\zeta}\| =  \|\zeta(s(g))\|\leq \|\zeta\|$. 
Hence $\| \m_{T}^{(n)} \|=\sup_{g\in \G} \|T_g^{(n)}\|\leq \|\xi\| \cdot \|\zeta\|$.
\end{proof}


For fixed $\G$, the class of twisted actions with (positive-definite) Fourier--Stieltjes multipliers as morphisms forms a Banach $\contb(X)$-bimodule category. 

\begin{prop}\label{prop:BanachNormFS} 
Let $\A$, $\B$ be $\cst$-bundles over $X$, carrying twisted actions $(\alpha,u_\alpha)$ and $(\beta,u_\beta)$ of $\G$.
The Fourier--Stieltjes multipliers $FS[(\alpha,u_{\alpha}),(\beta,u_{\beta})]$ form an involutive Banach $\contb(X)$-bimodule with the algebraic structure inherited from $\m(\A, \B)$ and the norm 
\begin{equation}\label{eq:norm_Fourier-Stieltjes}
    \|T\|_{FS} := \inf\{\|\xi \|\|\zeta \|: T=T_{\EE,L,\xi,\zeta} \text{ for some  }(\EE, L)\text{ and }\xi,\zeta\in \contb(\M(\EE)) \} .
\end{equation}
Moreover, the composition of Fourier--Stieltjes multipliers is a Fourier--Stieltjes multiplier, the norm $\|\cdot\|_{FS}$ is submultiplicative and the identities are Fourier--Stieltjes multipliers.
In particular, $FS(\alpha,u_{\alpha})$ is a unital Banach algebra and involutive Banach $\contb(X)$-bimodule with a multiplicative, conjugate-linear involution. 
\end{prop}
\begin{proof} 
The fact that the collection of Fourier--Stieltjes multipliers $FS[(\alpha,u_{\alpha}),(\beta,u_{\beta})]$ is closed under linear combinations, and in fact forms a $\contb(\G^{(0)})$-bimodule follows from the following straightforward relations
\[
    f \cdot T_{\EE,L,\xi,\zeta}= T_{\EE,L,\overline{f}\xi,\zeta}, \quad 
    T_{\EE,L,\xi,\zeta}\cdot f= T_{\EE,L,\xi,f\zeta} , \quad 
    T_{\EE,L,\xi_1,\zeta_1}+T_{\FF,K,\xi_2,\zeta_2} = T_{\EE\oplus \FF,L\oplus K,\xi_1\oplus \xi_2,\zeta_1\oplus \zeta_2}.
\]
These also imply that $\|\cdot\|_{FS}$ is a seminorm. 
Since $\|\cdot\|_{FS}$ majorises the operator norm in $\BB(\contz(r^*\A),\contz(r^*\B))$, see Lemma~\ref{lem:FS_multiplier_bounded}, and clearly $m_T=0$ if and only if $T=0$, the expression $\|\cdot\|_{FS}$ is in fact a norm.  
That $FS[(\alpha,u_{\alpha}),(\beta,u_{\beta})]$ is closed
under the involution, and that the involution is isometric, follows from the relation
\[ 
    (T_{\EE,L,\xi,\zeta})^\dagger = T_{\EE,L,\zeta,\xi}.
\]
Indeed, set $ T_{\EE,L,\xi,\zeta}=\{T_{g}\}_{g\in \G}$ and let  $g \in \G$, $a\in A_{r(g)}$. 
Using relations in Remarks~\ref{rem:consequnces_twists_in_groupoid_actions}, \ref{rem:relations_equivariant_reps} and the fact that $\alpha_g(u_\alpha(g^{-1},g)^*) u_\alpha(g,g^{-1}) =1_{\M(A_{r(g)})}$ we get: 
\[
\begin{split}
    T_{g^{-1}}(a^\star)^\star &= \left(\langle \xi(r(g^{-1})), u_\alpha(g^{-1},g)^* \alpha_{g^{-1}}(a^*) \overline{L}_{g^{-1}}\zeta(s(g^{-1}))\rangle \right)^{\star} \\
        &= u_\beta(g,g^{-1})^* \beta_g \left( \langle u_\alpha(g^{-1},g)^* \alpha_{g^{-1}}(a^*)  \overline{L}_{g^{-1}}\zeta(r(g)),   \xi(s(g)) \rangle \right) \\ 
        &= u_\beta(g,g^{-1})^* \langle  L_g u_\alpha(g^{-1},g)^* L_{g^{-1}}a^*\zeta(r(g)), L_g  \xi(s(g))\rangle \\
        &= u_\beta(g,g^{-1})^* \langle \alpha_g(u_\alpha(g^{-1},g)^*) L_g L_{g^{-1}}a^*\zeta(r(g)), L_g  \xi(s(g))\rangle \\
        &= u_\beta(g,g^{-1})^* \langle \alpha_g(u_\alpha(g^{-1},g)^*) u_\alpha(g,g^{-1}) a^*\zeta(r(g)) u_\beta(g,g^{-1})^*, L_g  \xi(s(g))\rangle \\
        &= \langle a^*\zeta(r(g)), L_g  \xi(s(g))\rangle  =  \langle \zeta(r(g)), aL_g  \xi(s(g))\rangle = T_{\EE,L,\zeta,\xi}(a)(g).
\end{split}
\]
That Fourier--Stieltjes multipliers are closed under composition follows from the relation
\begin{equation}\label{eq:compostion_of_Fourier_Stieltjes}
    T_{\FF,K,\xi_2,\zeta_2}\circ  T_{\EE,L,\xi_1,\zeta_1}= T_{\EE\otimes \FF,L\otimes  K,\xi_1\otimes  \xi_2,\zeta_1\otimes  \zeta_2} , 
\end{equation}
where we use the notation from Examples~\ref{ex:tensor_bundles} and~\ref{ex:tensor_product_groupoid_actions}. 
The displayed relation holds because 
\[
\begin{split}
    T_{\FF,K,\xi_2,\zeta_2}  (T_{\EE,L,\xi_1,\zeta_1}(a))(g) &= \langle \xi_2(r(g)), T_{\EE,L,\xi_1,\zeta_1}(a) \overline{K}_{g}\zeta_2(s(g))\rangle 
        \\
    &= \langle \xi_2(r(g)), \langle \xi_1(r(g)), a(g) \overline{L}_{g} \zeta_1(s(g)) \rangle_{\M(A_{r(g)})} \overline{K}_{g}\zeta_2(s(g))\rangle 
        \\
    &=\langle (\xi_1\otimes \xi_2)(r(g)), a(g) \overline{L}_{g}\zeta_1(s(g)) \otimes \overline{K}_{g}\zeta_2(s(g)) \rangle
        \\
    &= T_{\EE\otimes \FF,L\otimes  K,\xi_1\otimes  \xi_2,\zeta_1\otimes  \zeta_2}(a)(g).
\end{split}
\]
Using this relation one sees that the norm $\|\cdot\|_{FS}$ is submultiplicative.  
To prove that $FS[(\alpha,u_{\alpha}),(\beta,u_{\beta})]$ is complete in this norm we take $\{T_i\}_{i=1}^\infty\subseteq FS[(\alpha,u_{\alpha}),(\beta,u_{\beta})]$ such that $\sum_{i=1}^\infty \|T_i\|_{FS}<\infty$ and we need to show that $\sum_{i=1}^\infty T_i$ is convergent in $\|\cdot\|_{FS}$. 
We may pick $(\alpha,u_{\alpha})$-$(\beta,u_{\beta})$ equivariant representations of $\G$ on $\cst$-correspondences $\{(\EE_i, L_i)\}_{i=1}^\infty$  and non-zero sections $\{\xi_i,\zeta_i\}_{i=1}^\infty\in \contb(\M(\EE_i))$ such that $T_i=T_{\EE_i,L_i,\xi_i,\zeta_i}$ and $\|\xi_i \|\|\zeta_i\|\leq \|T_i\|+1/2^i$ for $i\in \N$.
In addition (by passing to $\sqrt{\frac{|\zeta_i\|}{\|\xi_i\|}}\xi_i$ and $\sqrt{\frac{|\xi_i\|}{\|\zeta_i\|}}\zeta_i$) we may assume that $\| \xi_i \| = \| \zeta_i \|$ for each $i \in \N$. 
Then it follows that $\oplus_{i=1}^\infty \xi_i$ and $\oplus_{i=1}^\infty \zeta_i$ belong to $\contb(\M(\oplus_{i=1}^\infty \EE_i))$ because $\sum_{i=1}^\infty \|\xi_i\|^2 = \sum_{i=1}^\infty \|\zeta_i\|^2 < \sum_{i=1}^\infty \|T_i\|_{FS} +1$. 
Let $T := T_{\oplus_{i=1}^\infty \EE_i,\oplus_{i=1}^\infty L_i,\oplus_{i=1}^\infty \xi_i,\oplus_{i=1}^\infty \zeta_i}$. 
Then 
\[
\begin{split}
    \| T-\sum_{i=1}^n T_i\|_{FS} 
    &\leq \|\oplus_{i=n+1}^\infty \xi_i\| \cdot \|\oplus_{i=n+1}^\infty \zeta_i\| \leq \sum_{i=n+1}^\infty \|\xi_i\|^2 \sum_{i=n+1} \|\zeta_i\|^2 \\
        &< \left( \sum_{i=n+1}^\infty \|T_i\|_{FS} +1/2^i\right)^2 \stackrel{n\to \infty}{\longrightarrow} 0. 
\end{split}
\]
By Example~\ref{ex:identity_multiplier}, $T_{\A,\alpha,1,1}$ is the unit in $FS(\alpha,u_{\alpha})$ and clearly its FS-norm equals $1$. 
\end{proof}


\begin{ex}\label{ex:Fourier-Stietjes-algebra-groupoids}
Applying the above construction to actions on a trivial line bundle $\{\C\}_{x\in X}$, cf.\ Example~\ref{ex:continuous_Hilbert_G-bundles}, we get the \emph{Fourier--Stieltjes algebra $FS(\G)$ of the groupoid $\G$}, similar to what was considered in \cite{RamsayWalter} or \cite{Oty}. 
In our setting, $FS(\G)$ is the space of all (continuous, bounded) functions $\varphi\colon \G\to \C$ of the form $\varphi(g) = \langle \xi(r(g)) , \ell_g(\zeta(s(g))) \rangle$, for $g\in \G$ and  
sections $\xi ,\zeta \in \contb (\H)$ of a (continuous) $\G$-Hilbert bundle $(\H,\ell)$. 
By the above proposition, $FS(\G)$ is an involutive Banach algebra with respect to the pointwise product of functions $\G\to \C$, the conjugation as involution, and the norm given as in~\eqref{eq:norm_Fourier-Stieltjes}.
By Remark~\ref{rem:do-not-depend-on-twist}, we get the same algebra $FS(\G,u)=FS(\G)$ for any (continuous) $2$-cocycle $u$ on $\G$.
\end{ex}

\begin{prop}\label{prop:InclusionGroupoidFSAlg}
For a twisted action $(\alpha, u)$ of $\G$ on $\A$, we have an injective contractive involutive algebra homomorphism $FS(\G)\ni \varphi \mapsto T^\varphi\in FS(\alpha,u)$ where  $T^\varphi_g (a) := \varphi(g) a$, for $a \in A_{r(g)}$, $g\in \G$. 
\end{prop}
\begin{proof}
Take any continuous $\G$-Hilbert bundle $(\H,\ell)$ and sections $\xi ,\zeta \in \contb (\H)$ that define $\varphi$, i.e.\ $\varphi(g) = \langle \xi(r(g)) , \ell_g(\zeta(s(g))) \rangle$, for $g\in \G$. 
Tensor $(\H,\ell)$ with the trivial equivariant representation $(\A,\alpha)$ of $(\alpha,u)$, see Examples~\ref{ex:trivial_equivariant_action}, \ref{ex:TensorActionHilbertBundle}. 
For $a \in A_{r(g)}$, $g\in \G$ we have 
\[
\begin{split}
    T_{\A\otimes \H, \alpha \otimes \ell , 1 \otimes \xi , 1 \otimes \zeta}(g)(a) &= \langle (1 \otimes \xi) \big( r(g) \big) , a\cdot \overline{\alpha \otimes \ell}_g (1 \otimes \zeta) \big( s(g) \big) \rangle = a \langle \xi \big( r(g) ) ,  \ell_g \big( \zeta(s(g)) \big) \rangle \\
        &= \varphi(g) a = T^\varphi_g(a).
\end{split}
\]
Hence $T^\varphi \in FS(\alpha,u)$ and $\| T^\varphi \|_{FS} \leq \| 1 \otimes \xi \| \| 1 \otimes \zeta \| \leq \| \xi \| \| \zeta \|$, which implies $\| T^\varphi \|_{FS} \leq \| \varphi \|_{FS}$. 
The map $\varphi \mapsto T^\varphi$ is clearly injective, and it is easily checked that it is an involutive algebra homomorphism. 
\end{proof}

The next theorem is one of the central results in this paper. 
Its proof uses the absorption principles established in Sections~\ref{Sec:equivariant} and \ref{Sec:InducedAbsorption} and the inverse semigroup picture introduced in Section~\ref{Sec:Fell_bundles}.

\begin{thm}\label{thm:FSmultiplier_extends_to_reduced_and_full} 
Let $\A$, $\B$  be $\cst$-bundles over $X$, carrying twisted actions $(\alpha,u_\alpha)$ and $(\beta,u_\beta)$ of $\G$.
Every Fourier--Stieltjes multiplier $T\in FS[(\alpha,u_{\alpha}),(\beta,u_{\beta})]$ yields strict completely bounded maps $\m_T^{\rd} : \cst_{\red}(\A^{(\alpha,u_{\alpha})})\to \cst_{\red}(\B^{(\beta,u_{\beta})})$ and $\m_T^{\f} : \cst(\A^{(\alpha,u_{\alpha})})\to \cst(\B^{(\beta,u_{\beta})})$ between reduced and full crossed products, and $\|\m_T^{\rd}\|_{cb} , \|\m_T^{\f}\|_{cb} \leq \|T\|_{FS}$. 
If $T=T_{\EE,L,\xi,\xi}$ for some equivariant representation $L$ of $\G$ on a $\cst$-$\A$-$\B$-correspondence bundle $\EE$ and section $\xi \in \contb(\M(\EE))$, then $\m_T^{r}$ and $\m_T^{f}$ are completely positive and $\|\m_T^{\rd}\|_{cb} = \| \m_T^{\f} \|_{cb} = \| T \|_{FS} = \|\xi\|^2$.
\end{thm}
\begin{proof} 
Let us fix an equivariant representation $L$ of $\G$ on a $\cst$-$\A$-$\B$-correspondence bundle $\EE$, sections $\xi,\zeta\in \contb(\M(\EE))$ and the corresponding Fourier--Stieltjes multiplier $T := T_{\EE,L,\xi,\zeta} = \{ T_g \}_{g\in \G}$. 
To deal with reduced crossed products we will use our first version of the Fell absorption principle. 
Using the notation from Example~\ref{ex:tensor_bundles} and Theorem~\ref{thm:Fell_absorption}, we define the (creation) operator $\theta_\xi : \contz(\ell^2(s^*\B)) \to \contz(\EE)\otimes_{\pi^{\B}} \contz(\ell^2(s^*\B))$ by the formula $\theta_\xi (b):=\xi\otimes b$, $b\in \contz(\ell^2(s^*\B))$; for an approximate unit $\{e_i\}_{i\in I}$ in $\contz (\B)$ we have $\xi\otimes b= \lim_{i} \xi \otimes \pi^{\B}(e_i) b =\lim_{i} \xi e_i\otimes b \in \contz(\EE)\otimes_{\pi^{\B}} \contz(\ell^2(s^*\B))$. 
The operator $\theta_\xi$ is adjointable with norm dominated by $\|\xi\|$ and adjoint given by the formula 
\begin{equation}\label{eq:auxillary_adjoint_theta}
    \theta_\xi^*\big( a\zeta \otimes b) = \pi_{\B} (\langle \xi,  a \zeta\rangle_{\contz(\B)}) b,
\end{equation}
for $a\in \contz(\A)$ and $b\in \contz(\ell^2(s^*\B))$ (then $a\zeta\in \contz(\EE)$, so the inner product takes values in $\contz(\B)$ and the formula makes sense).
Using the unitary $W: \contz(\EE)\otimes_{\pi^{\B}} \contz(\ell^2(s^*\B))\to  \contz(\ell^2(s^*\EE))$ from Theorem~\ref{thm:Fell_absorption} we define a  map $\Psi : \LL(\contz(\ell^2(s^*\EE))) \to \LL( \contz(\ell^2(s^*\B)) )$ by the formula
\[
    \Psi(S) := \theta_\xi^* W^* S W\theta_\zeta, \qquad S\in \LL(\contz(\ell^2(s^*\EE))).
\]
By its form, $\Psi$ is completely bounded and strict, with $\|\Psi\|_{cb} \leq \|\xi\| \|\zeta\|$.
We claim that $\pi^{\B}(\m_{T}(a))=\Psi(\pi^{\EE}(a))$ 
for $a\in \contc(\A^{(\alpha,u_{\alpha})})$. 
Indeed, by \eqref{eq:Fell_absorption_conjugates} it suffices to check that for every  $U\in \Bis(\G)$, $a\in \contc(\A|_{r(U)})$, $b\in \ell^2(s^*\B)$, we have $\pi^{\B}(\m_{T}(a)) b= \theta_\xi^* ( aL_U(\zeta) \otimes \pi^{\B}(\delta_U)(b))$. 
But this readily follows from \eqref{eq:auxillary_adjoint_theta}.
Since $\pi^\B$ extends to a faithful representation of $\cst_{\red}(\B^{(\beta,u_{\beta})})$, and similarly for $\pi^\EE$ and $\cst_{\red}(\A^{(\alpha,u_{\alpha})})$ (see the last part of Proposition~\ref{prop:induced_regular_representations_for_actions}),
we conclude that $\m_{T} : \contc(\A^{(\alpha,u_{\alpha})})\to \contc(\B^{(\beta,u_{\beta})})$ extends to a map $\m_{T}^{\rd}:\cst_{\red}(\A^{(\alpha,u_{\alpha})})\to \cst_{\red}(\B^{(\beta,u_{\beta})})$ with
\[
    \pi^{\B}(\m_{T}^{\rd}(a)) = \Psi(\pi^{\EE}(a)) = \theta_\xi^* W^*\pi^{\EE}(a) W\theta_\zeta, \qquad a\in \cst_{\red}(\A^{(\alpha,u_{\alpha})}).
\]
This implies that $\m_{T}^{\rd}$ is strict, completely bounded and
\[
    \| \m_{T}^{\rd} \|_{cb} = \| \Psi\circ \pi^{\EE} \|_{cb} \leq \|\Psi\|_{cb}\leq \|\xi\|\|\zeta\|.
\]
In particular, $\|\m_{T}^{\rd}\|_{cb}\leq \|T\|_{FS}$ (since the presentation of $T$ was arbitrary). 
If $\xi=\zeta$, then $\Psi$ and hence also $\m_{T}^{\rd}$ is completely positive and for any approximate unit $\{e_i\}_{i\in I}$ in $\contz(\A)$ we get
\[
    \| \m_{T}^{\rd} \|_{cb} = \| \Psi\circ \pi^{\EE} \|_{cb} = \lim_{i} \|\Psi\circ \pi^{\EE} (e_{i}) \| = \| \theta_\xi^*W^*W\theta_\xi\| = \|\xi\|^2.
\]
Thus we have $\|\m_{T}^{\rd}\|_{cb}= \|T\|_{FS}=\|\xi\|^2$.

To deal with full crossed products let $\psi : \cst(\B^{(\beta,u_{\beta})})\to \BB(\H)$ be a faithful nondegenerate representation of $\cst(\B^{(\beta,u_{\beta})})$ on a Hilbert space $H$. 
We view $\psi=\pi\rtimes v$ as the integrated form of a covariant representation $(\pi,v)$ of the associated inverse semigroup action, see Subsection~\ref{sec:representations_of_crossed_products}.
Consider the induced representation $L\dashind(\psi):\cst(\A^{(\alpha,u_{\alpha})})\to \BB(\contz(\EE)\otimes_{\pi}H)$ as in Proposition~\ref{prop:representation_induced_by_equivariant_action}. 
Similarly to above we define the (creation) operator $\theta_\xi : H \to \contz(\EE)\otimes_{\pi} H$ by $\theta_\xi (\pi(b)h):=(\xi\cdot b)\otimes h$ for $b\in B=\contz(\B)$ and $h\in H$.
In other words, $\theta_\xi (h)=\lim_{i} \xi e_i^B \otimes h$ where $\{e_i^B\}$ is an approximate unit in $B$. 
Its adjoint is determined by $\theta_\xi^*\big( a\zeta \otimes h)= \pi (\langle a^*\xi,\zeta\rangle_{B}) h$, where $\zeta\in E$, $a\in A$, and $h\in H$. 
So $\theta_\xi^* (\zeta \otimes h)=\lim_{i} \pi (\langle e_i^A \xi,\zeta\rangle_{B}) h$ where $\{e_i^A\}$ is an approximate unit in $A$. 
We define a completely bounded map $\Psi:\cst(\A^{(\alpha,u_{\alpha})})\to B(H)$ by the formula
\[
    \Psi(b) := \theta_\xi^* (L\dashind(\psi))(b) \theta_\zeta.
\]
We claim that $\Psi(b) = \psi (T_{\EE,L,\xi,\zeta}(b))$ for all $b\in \contc(\A^{(\alpha,u_{\alpha})})$. 
It suffices to check that for $b=a\delta_{U}$ where $a\in A_{r(U)}$, $U\in \Bis(\G)$. 
Then for any $g\in U$ we have
\[
    T_{\EE,L,\xi,\zeta}(b)(g) = \langle \xi(r(g)),  L_{g}(\alpha_g^{-1}(a(r(g))\zeta(s(g))\rangle_{\M(B_{r(g)})} = \lim_{i} \langle e_i^A \xi,L_U(\alpha_{U}^{-1}(a)\zeta ) \rangle_{B}(r(g)).
\]
That is $T_{\EE,L,\xi,\zeta}(b)\circ r|_U^{-1}=\lim_{i} \langle e_i^A \xi,L_U(\alpha_{U}^{-1}(a)\zeta ) \rangle_{B}$.
Using this we get
\[
\begin{split}
    \Psi(b)h &= \lim_{i} \theta_\xi^* (L\dashind(\psi))(b) (\zeta\cdot e_i^B)\otimes h = \lim_{i} \theta_\xi^* \big( L_U(\alpha_{U}^{-1}(a)\zeta)\otimes v_{U} h \big) \\
        &= \lim_{i} \pi (\langle e_i^A \xi,L_U(\alpha_{U}^{-1}(a)\zeta ) \rangle_{B}) v_Uh = \psi ( T_{\EE,L,\xi,\zeta}(b)) h.
\end{split}
\]
This proves our claim and implies that the completely bounded map $\Psi$ maps $\cst(\A^{(\alpha,u_{\alpha})})$ to $\psi(\cst(\A^{(\alpha,u_{\alpha})}))$. 
Thus, putting $\m_T^{\f} := \psi^{-1}\circ \Psi$, we get the desired multiplier map $\m_T^{\f} : \cst(\A^{(\alpha,u_{\alpha})})\to \cst(\B^{(\beta,u_{\beta})})$.
The assertions concerning norms follow as in the first part. 
\end{proof}

Note that the above theorem will be later enriched by the discussion of the \emph{decomposable} norms of the maps $m_T^r$ and $m_T^s$, see Proposition~\ref{prop:FSmultipliers_decomposable} and Theorem~\ref{thm:FSmultipliers_decomposable_discrete}.

The set $CB[(\alpha,u_{\alpha}),(\beta,u_{\beta})]$ of completely bounded $\contz(X)$-bimodule maps $\cst( \A^{(\alpha,u_{\alpha})} ) \to \cst( \B^{(\beta,u_{\beta})} )$ has a similar structure to that of Fourier--Stieltjes multipliers described in Proposition~\ref{prop:BanachNormFS}. 
Namely, it is a Banach space with pointwise operations and the norm $\|\cdot\|_{cb}$. 
It is also an involutive Banach $\contb(X)$-bimodule, where the bimodule structure is determined by
\[
    f \cdot \Psi (a) := \Psi((f\circ r)\cdot a), \qquad 
     \Psi \cdot f (a) := \Psi( a \cdot (f\circ s)), \qquad 
    f\in \contb(X), \,a\in \contc(\B^{(\beta,u_{\beta})}), 
\]
with $\Psi\in CB[(\alpha,u_{\alpha}),(\beta,u_{\beta})]$, and the involution is given by 
\[
    \Psi^{\dagger}(a) := \Psi (a^*)^*, \qquad a\in \contc(\B^{(\beta,u_{\beta})}),
\]
cf.\ \cite[Proposition 1]{Walter}. 
This involution is called \emph{conjugation} in \cite{Walter}, \cite{BedosConti}. 
In the same way, the set $CB^{\red}[(\alpha,u_{\alpha}),(\beta,u_{\beta})]$ of completely bounded $\contz(X)$-bimodule maps $\cst_{\red}(\A^{(\alpha,u_{\alpha})})\to \cst_{\red}(\B^{(\beta,u_{\beta})})$ forms an involutive Banach $\contb(X)$-bimodule.

\begin{rem}\label{rem:functors_from_FS}
Viewing the class of all twisted actions of $\G$ as objects, denote by $FS_{TA}(\G)$ the category where morphisms are Fourier--Stieltjes multipliers, and by $CB_{TA}(\G)$  (resp.\ $CB_{TA}^{\rd}(\G)$) the category where morphisms are completely bounded, strict $\contz(X)$-bimodule maps between full (resp.\ reduced) crossed products, equipped with the norm $\|\cdot\|_{cb}$. 
The extensions in Theorem~\ref{thm:FSmultiplier_extends_to_reduced_and_full} give contractive involutive $\contb(X)$-bimodule functors $\m^{\f}:FS_{TA}(\G)\to CB_{TA}(\G)$ and $\m^{\rd}:FS_{TA}(\G)\to CB_{TA}^{\rd}(\G)$.
\end{rem}

\section{Positive-definiteness and multipliers on section C*-algebras} 
\label{Sec:positive-definite}

In this section we will specify our considerations to the completely positive multipliers.

\begin{defn}
Let $\A = \{ A_g \}_{g \in \G}$ and $\B = \{ B_g \}_{g \in \G}$ be Fell bundles over $\G$. 
We say that a multiplier $T=\{T_g\}_{g\in \G}$ from $\A$ to $\B$ is \emph{positive-definite} if for any $x \in X$, any $n \in \N$, any $g_1 , \ldots , g_n \in \G_x$ and any collections $\{a_{g_i} \in A_{g_i}: 1 \leq i \leq n \}$ and $\{b_{g_i} \in B_{g_i}: 1 \leq i \leq n \}$ we have
\[
    \sum_{i,j = 1}^n b_{g_i}^{\star} \left( T_{g_i g_j^{-1}} (a_{g_i} a_{g_j}^{\star}) \right) b_{g_j} \geq 0
\]
(note that the last sum takes values in $B_x$).
Then each of the maps $T_x \in \BB(A_{x},B_{x})$, $x \in X$, is completely positive, and we say that $T$ is \emph{strict} if each map $T_x : A_x \to B_x$, $x\in X$, is strict, and we call $T$ \emph{bounded} if $\sup_{x\in X} \| T_x \| < \infty$.
\end{defn}


\begin{rem}\label{rem:equivalent_for_positive_definiteness}
By passing to adjoints and inverses (replacing $g_i$ by $g_i^{-1}$ and putting $\widetilde{a}_{g_i} := a_{g_i^{-1}}^{\star}$, $\widetilde{b}_{g_i} := b_{g_i^{-1}}^{\star}$) one sees that a multiplier $T = \{ T_g \}_{g\in \G}$ from $\A$ to $\B$ is positive-definite if and only if for any $n \in \N$, any $g_1 , \ldots , g_n \in \G^x$ and any collections $\{a_{g_i} \in A_{g_i}: 1 \leq i \leq n \}$ and $\{b_{g_i} \in B_{g_i}: 1 \leq i \leq n \}$ we have  
\[
    \sum_{i,j = 1}^n b_{g_i} \left( T_{g_i^{-1} g_j} (a_{g_i}^{\star} a_{g_j} ) \right) b_{g_j}^{\star} \geq 0 . 
\]
\end{rem}

\begin{lem}\label{lem:completely_positive_implies_positive_definite}
Let $\A = \{ A_g \}_{g \in \G}$ and $\B = \{ B_g \}_{g \in \G}$ be Fell bundles over $\G$. 
Let $T = \{ T_g \}_{g\in \G}$ be a reduced or full multiplier from $\A$ to $\B$. 
If the corresponding map $\m_T$ between cross-sectional $\cst$-algebras is completely positive, then $T$ is positive-definite and bounded (in fact we have $\|\m_T\|=\sup_{x\in X}\|T_x\|$), and $T$ is strict if $\m_T$ is strict.
\end{lem}
\begin{proof}
Fix $x \in X$, $n \in \N$, a family $g_1 , \ldots , g_n \in \G_x$ and collections $\{ a_{g_i} \in A_{g_i} : 1 \leq i \leq n \}$, $\{ b_{g_i} \in B_{g_i} : 1 \leq i \leq n \}$. 
We may find sections $\{ f_i \}_{i = 1}^n \in \contc(\A)$ supported on open bisections and such that $f_{i}(g_i) = a_{g_i}$, $1 \leq i \leq n$. 
Then $(f_i * f_j^*)(g_i {g_j}^{-1}) = a_{g_i} a_{g_j}^{\star}$ for $1 \leq i,j \leq n$. 
Since the matrix $(f_i * f_j^*)_{i,j = 1}^n$ is positive both as an element of $\Mat_{n}(\cst(\A))$ and of $\Mat_{n}(\cst_{\red}(\A))$, and $\m_T^{(n)}$ is positive by assumption, we conclude that the operator 
\[ 
    b := \Lambda_x^{(n)} \circ \m_T^{(n)}  \big( (f_i * f_j^*)_{i , j = 1}^n \big) \in \L(\ell^2(\B)_{x}^{\oplus n})  
\]
is positive. 
Choose $\xi_i \in \ell^2(\B)_{x} = \bigoplus_{h \in \G_x} B_{h}$ such that $\xi_i(g_j) = [i = j] b_{g_i}$ for $1 \leq i,j \leq n$, and let $\xi = \bigoplus_{i = 1}^n \xi_{i} \in \ell^2(\B)_{x}^{\oplus n} =\bigoplus_{i = 1}^n \bigoplus_{h \in \G_x} B_{h}$. 
Then 
\[
    0 \leq \langle \xi, b \xi \rangle_{B_x} = \sum_{i , j = 1}^n \langle \xi_i , \Lambda_x(\m_T (f_i * f_j^*)) \xi_j \rangle_{B_x} = \sum_{i,j = 1}^n b_{g_i}^{\star} \left( T_{g_i g_j^{-1}} (a_{g_i} a_{g_j}^{\star}) \right) b_{g_j} .
\]
Hence $T$ is a positive-definite multiplier. 
Since any approximate unit $\{e_i\}_{i\in I}$ in $\contz(\A|_X)$ is also an approximate unit in $\cst_{\red}(\A)$ and in $\cst(\A)$, and $\m_T$ is a completely positive map, we get  
\[
    \| \m_T \| = \lim_{i\in I} \| \m_T(e_i) \| = \lim_{i\in I} \sup_{x\in X} \| T_x(e_i(x)) \| .
\]
Similarly, since, for any $x\in X$, the net $\{e_i(x)\}_{i \in I}$ is an approximate unit in $A_x$, and $T_x : A_x \to B_x$ is a completely positive  map, we also get $\| T_{x} \| = \lim_{i\in I} \| T_x (e_i(x)) \|$. 
Combining these equalities one gets $\|\m_T\|=\sup_{x\in X}\|T_x\|$.
Finally, if  $\m_T$ is strict, then the net $\{T(e_i)\}_{i \in I}$ converges strictly to a positive element $T(1)\in \M(\contz(\B|_X))$. 
For any $x\in X$, the net $\{e_i(x)\}_{i\in I}$ is an approximate unit in $A_x$.
For any $a\in B_x$ we may find $b \in \contc(\B)\subseteq \cst_r(\B,\G,\alpha)$ with $b(x) = a$, and then
\[
    \lim_{i\in I} T_x(e_i(x))a = \lim_{i\in I} T_x(e_i(x))b(x) = \lim_{i\in I} \big( T(e_i) b \big)(x) = \big( T(1)b \big)(x).
\]
Thus the map $T_x:A_x\to B_{x}$ is strict. 
\end{proof}

We will now pass to the context of twisted groupoid actions.

\begin{lem}\label{lem:positive_action}
Let $\A$, $\B$ be $\cst$-bundles over $X$, carrying twisted actions $(\alpha,u_\alpha)$ and $(\beta,u_\beta)$ of $\G$.
For a  multiplier $T=\{T_g\}_{g\in \G}$ from $\A^{(\alpha,u_{\alpha})}$ to $\B^{(\beta,u_{\beta})}$ the following are equivalent:
\begin{enumerate}
    \item\label{enu:positive_action1} $T$ is positive-definite;
    \item\label{enu:positive_action2} for any $x \in X$, $g_1 , \ldots , g_n \in \G_x$, $n \in \N$, and any collection $\{ a_{g_i} \in A_x : 1 \leq i \leq n \}$ the following matrix is positive: 
	\[
	   \left( \beta_{g_i}^{-1} \left(T_{g_i g_j^{-1}} \big( \alpha_{g_i}(a_{g_i} a_{g_j}^*) u_{\alpha}(g_i g_j^{-1} , g_j)^* \big) u_{\beta}(g_i g_j^{-1} , g_j) \right) \right)_{i,j = 1}^n \in \Mat_{n}(B_x);
	\]
	\item\label{enu:positive_action3} for any $x \in X$, $g_1 , \ldots , g_n \in \G^x$, $n \in \N$, and any collection $\{ a_{g_i} \in A_x : 1 \leq i \leq n \}$ the following matrix is positive: 
	\[
	   \left( \beta_{g_i} \left(T_{g_i^{-1} g_j} \big( \alpha_{g_i}^{-1}(a_{g_i}^* a_{g_j} u_{\alpha}(g_i , g_i^{-1} g_j)^*) \big) \right) u_{\beta}(g_i , g_i^{-1} g_j) \right)_{i,j = 1}^n \in \Mat_{n}(B_x). 
	\]
	\end{enumerate}
\end{lem}
\begin{proof} 
\ref{enu:positive_action1}$\iff$\ref{enu:positive_action2} Fix $x \in X$, $n \in \N$ and a $g_1 , \ldots , g_n \in \G_x$.  
Choosing $\tilde{a}_{g_i} \in A_{g_i}^{(\alpha,u_{\alpha})} = A_{r(g_i)}$ and $\tilde{b}_{g_i} \in B_{g_i}^{(\beta,u_{\beta})} = B_{r(g_i)}$, for $1 \leq i \leq n$, is equivalent to choosing $a_{g_i} \in A_{x}$ and $b_{g_i} \in B_{x}$, for $1 \leq i \leq n$, by putting $\tilde{a}_{g_i} = \alpha_{g_i}(a_{g_i}) \delta_{g_i}$ and $\tilde{b}_{g_i} = \beta_{g_i}(b_{g_i})\delta_{g_i}$.
Using the fact that for all $1 \leq i,j \leq n$
\[
    \alpha_{g_i}(u_{\alpha}(g_j^{-1} , g_j)) = \alpha_{g_i}(u_{\alpha}(g_j^{-1},g_j))u_{\alpha}(g_i,x) = u_{\alpha}(g_i,g_j^{-1})u_{\alpha}(g_i g_j^{-1},g_j) , 
\]
one gets $\tilde{a}_{g_i} \cdot (\tilde{a}_{g_j})^{\star} = \alpha_{g_i}(a_{g_i} a_{g_j}^*)u_{\alpha}(g_i g_j^{-1},g_j)^*\delta_{g_i g_j^{-1}}$. 
Similarly, for $1 \leq i,j \leq n$, 
\[
    \tilde{b}_{g_i}^{\star} \left( T_{g_i g_j^{-1}} (\tilde{a}_{g_i} \tilde{a}_{g_j}^{\star} ) \right) \tilde{b}_{g_j} = b_{g_i}^* \cdot \beta_{g_i}^{-1} \left( T_{g_i g_j^{-1}} \big(\alpha_{g_i}(a_{g_i} a_{g_j}^*) u_{\alpha}(g_i g_j^{-1} , g_j)^*\big) u_{\beta}(g_i g_j^{-1},g_j)\right)\cdot  b_{g_j}.
\]
Hence positivity of the elements $\sum_{i,j = 1}^n \tilde{b}_{g_i}^{\star} \left( T_{g_i g_j^{-1}} (\tilde{a}_{g_i} \tilde{a}_{g_j}^{\star}) \right) \tilde{b}_{g_j}$ for all choices of elements $\tilde{b}_{g_i} \in B_{g_i}^{(\beta,u_{\beta})}$ is equivalent to positivity of the matrix in the assertion.

\ref{enu:positive_action1}$\iff$\ref{enu:positive_action3} We will use the version of \ref{enu:positive_action1} described in Remark~\ref{rem:equivalent_for_positive_definiteness}.
Fix $x \in X$, $n \in \N$, and $g_1 , \ldots , g_n \in \G^x$. 
Choosing $a_{g_i} \in A_{x}$ and  $b_{g_i} \in B_{x}$, for $1 \leq i \leq n$, is equivalent to choosing $\tilde{a}_{g_i} \in A_{g_i}^{(\alpha,u_{\alpha})}$ and $\tilde{b}_{g_i} \in B_{g_i}^{(\beta,u_{\beta})}$, by putting $\tilde{a}_{g_i} = a_{g_i} \delta_{g_i}$ and $\tilde{b}_{g_i} = \beta_{g_i}(b_{g_i}) \delta_{g_i}$. 
Fix then $i,j$.
The equality $\alpha_{g_i}(u_{\alpha}(g_i^{-1},g_j)) u(g_i , g_i^{-1} g_j) = u_{\alpha}(g_i , g_i^{-1}) = \alpha_{g}(u_{\alpha}(g_i^{-1},g_i))$ is equivalent to 
\[
    u_{\alpha}(g_i^{-1},g_i)^* u_{\alpha}(g_i^{-1},g_j) = \alpha_{g_i}^{-1}(u_{\alpha}(g_i , g_i^{-1} g_j)^*) .
\] 
Using the last equality one gets $\tilde{a}_{g_i}^{\star} \cdot (\tilde{a}_{g_j}) = \alpha_{g_i}^{-1} \big( a_{g_i}^* a_{g_j} u_{\alpha}(g_i , g_i^{-1} g_j)^* \big) \delta_{g_i^{-1} g_j}$. 
Similarly,
\[
    \tilde{b}_{g_i} \left( T_{g_i^{-1} g_j} (\tilde{a}_{g_i}^{\star} \tilde{a}_{g_j}) \right) \tilde{b}_{g_j}^{\star} = b_{g_i}\beta_{g_i} \left(T_{g_i^{-1} g_j} \big( \alpha_{g_i}^{-1} (a_{g_i}^* a_{g_j} u_{\alpha}(g_i , g_i^{-1} g_j)^*) \big) \right) u_{\beta}(g_i , g_i^{-1} g_j) \cdot b_{g_j}^*.
\]
Hence positivity of the elements $\sum_{i,j = 1}^n \tilde{b}_{g_i} \left( T_{g_i^{-1} g_j} (\tilde{a}_{g_i}^{\star} \tilde{a}_{g_j}) \right) \tilde{b}_{g_j}^{\star}$ for all choices of elements $\tilde{b}_{g_i} \in B_{g_i}^{(\beta,u_{\beta})}$ is equivalent to positivity of the matrix in the assertion. 
\end{proof}

We are ready to formulate and prove the main result of this section.

\begin{thm}\label{thm:FSmultiplier_positive_definite}
Let $\A$, $\B$ be $\cst$-bundles over $X$, carrying twisted actions $(\alpha,u_\alpha)$ and $(\beta,u_\beta)$ of $\G$, and set $A:=\contz(\A),$ $B:=\contz(\B)$. 
For any multiplier $T$ from $\A^{(\alpha,u_{\alpha})}$ to $\B^{(\beta,u_{\beta})}$, the following are equivalent: 
\begin{enumerate}
	\item\label{enu:FSmultiplier_positive_definite1} $T$ is  strict, bounded and positive-definite;
    \item\label{enu:FSmultiplier_positive_definite2} $\m_T$ 
	extends to a strict completely positive map $\m_T^{\rd} : \cst_{\red}(\A^{(\alpha,u_{\alpha})}) \to \cst_{\red}(\B^{(\beta,u_{\beta})})$;
	\item\label{enu:FSmultiplier_positive_definite3} $\m_T$ 
	extends to a strict completely positive map $\m_T^{\f}:\cst(\A^{(\alpha,u_{\alpha})}) \to \cst(\B^{(\beta,u_{\beta})})$;
	\item\label{enu:FSmultiplier_positive_definite4} $T = T_{\EE,L,\xi,\xi}$ 
	for some equivariant representation $L$ of $\G$ on a $\cst$-$\A$-$\B$-correspondence bundle $\EE$ and section $\xi\in \contb(\M(\EE))$.
\end{enumerate}	
If the above equivalent conditions hold, then  $\| \m_T^{\f} \|_{cb} = \| \m_T^{\rd} \|_{cb} = \| \m_T^{\f} \| = \| \m_T^{\rd} \| = \| T \|_{FS} = \| \xi \|^2$ and the triple $(\EE,L,\xi)$ in \ref{enu:FSmultiplier_positive_definite4} can be chosen to be cyclic, in the sense that the set
\[
    \{ L_U (a \xi b) : a \in \contz(\A|_{s(U)}) ,\ b \in \contz(\B|_{s(U)}) ,\ U\in \Bis(\G) \}
\]
is linearly dense in $\contz(\EE)$. 
Moreover, a cyclic triple $(\EE,L,\xi)$ in \ref{enu:FSmultiplier_positive_definite4} is unique up to a unitary equivalence in the sense that if  $(\EE',L',\xi')$ is another cyclic triple with $T=T_{\EE',L',\xi',\xi'}$, then there is an $A$-$B$-bimodule $\contz(X)$-unitary $W:\contz(\EE)\to \contz (\EE')$ establishing equivalence between $L$ and $L'$, and sending $\xi$ to $\xi'$.
\end{thm}
\begin{proof} 
The implications \ref{enu:FSmultiplier_positive_definite4}$\implies$\ref{enu:FSmultiplier_positive_definite2},\ref{enu:FSmultiplier_positive_definite3} follow from Theorem~\ref{thm:FSmultiplier_extends_to_reduced_and_full} and the implications \ref{enu:FSmultiplier_positive_definite2},\ref{enu:FSmultiplier_positive_definite3}$\implies$\ref{enu:FSmultiplier_positive_definite1} are a consequence of Lemma~\ref{lem:completely_positive_implies_positive_definite}.

Thus we only need to prove that \ref{enu:FSmultiplier_positive_definite1} implies \ref{enu:FSmultiplier_positive_definite4}. 
Fix a strict, bounded  and positive-definite multiplier $T$. 
For each $x \in X$ define $k_x : \G^x \times \G^x \to \B(A_x,B_x)$ via
\[
    k_x(g, h)(a) := \beta_{g} \left(T_{g^{-1} h} \big( \alpha_g^{-1} (au_{\alpha}(g,g^{-1}h)^*) \big) \right) u_{\beta} (g,g^{-1}h) 
    \qquad g, h \in \G_x. 
\]
Then by Lemma~\ref{lem:positive_action}, $k_x$ is positive-definite and strict in the sense of Corollary~\ref{cor:Murphy_Stinespring}. 
Thus, by this latter result, there is a $\cst$-correspondence $E_x$ from $A_x$ to $B_x$ and a map $\xi_x : \G^x \to \M(E_x)$ such that the set $E_x^0 := \{ a \xi_x(g)b : g \in \G^x ,\ a \in A_x, b \in B_x \}$ is linearly dense in $E_x$, and for $g, h \in \G^x$, $a \in A_x$, we have 
\begin{equation} \label{kernelformula}
    \langle \xi_x(g), a \cdot \xi_x(h) \rangle_{\M(B_x)} =  \beta_{g} \left(T_{g^{-1} h} \big(\alpha_g^{-1}(au_{\alpha}(g,g^{-1}h)^*) \big) \right)u_{\beta}(g,g^{-1}h) \in B_x. 
\end{equation}
In this way, we get a family of $\cst$-correspondences $\EE:=\{E_x\}_{x\in X}$. 
Define the section $\xi : \G \to \M(\EE) = \{ \M(E_x) \}_{x\in X}$ by setting $\xi(g):=\xi_{r(g)}(g)$, $g \in \G$.
To topologise $\EE$ we let $\Gamma$ be the linear span of sections of the form $a\xi_U b$, where $U\in \Bis(\G)$, $a\in \contc(\A|_{r(U)})$, $b\in \contc(\B|_{r(U)})$ and
\[
    (a\xi_U b)(x) := a(x) \xi(r|_{U}^{-1}(x))b(x) , \qquad x\in X.
\]
If $c\xi_V d\in \Gamma$ is another element of this form, with $V\in \Bis(\G)$, $c\in \contc(\A|_{r(V})$, $d\in \contc (\B|_{r(V)})$, then $\langle a\xi_U b, c \xi_V d\rangle_{B} (x):=  \langle a\xi_U b(x), c \xi_V d (x)\rangle_{B_x}$ is zero outside $r(U)\cap r(V)$ and writing $g := r|_{U}^{-1}(x)$ and $h := r|_{V}^{-1}(x)$ for $x\in r(U)\cap r(V)$, by \eqref{kernelformula} we get 
\[
    \langle a\xi_U b, c \xi_V d\rangle_{B} (x) = b(x)^* \Big( \beta_{g} \Big(T_{g^{-1} h} \big( \alpha_g^{-1}( a(x)^* c(x)u_{\alpha}(g,g^{-1}h)^*) \big) \Big)u_{\beta}(g,g^{-1}h)\Big) d(x). 
\]
By continuity of all the operations involved, we see that the map $X\ni x\mapsto \langle a\xi_U b, c \xi_V d\rangle_{B} (x) \in B_x$ belongs to $\contc(\B)$. 
It extends uniquely to the (positive-definite) conjugate-linear map $\langle \cdot, \cdot \rangle_{B}:\Gamma\times \Gamma \to \contc(\B)\subseteq B$. 
In particular, for any $a\in \Gamma$ the map $X\mapsto \|a(x)\|^2=\|\langle a, a \rangle_{B}(x)\|$ is upper semicontinuous.
The fact that for each $x \in X$ the values of the prescribed sections are dense in the fibre $E_x$ follows immediately from the minimality assumption on each $E_x$ and the above construction of $\Gamma$.
Hence $\EE$ admits a unique topology making it a Banach bundle  such that $\Gamma\subseteq \contc(\EE)$.
The left $\A$ and right $\B$-module actions on $\EE$ are continuous by the very definition of sections, and the $B$-valued inner product $\langle \cdot, \cdot \rangle_{B}$ extends uniquely to $\contz(\EE)$ so that $\contz(\EE)$ becomes a $\contz(X)$-$\cst$-correspondence from $A$ to $B$. 

We claim that there is an $(\alpha,u_{\alpha})$-$(\beta,u_{\beta})$-equivariant representation of $\G$ on $\EE$ determined by
\begin{equation}\label{eq:Ldef} 
    L_g (a \xi(h)b) := \alpha_g(a)u_{\alpha}{(g,h)} \xi(gh)u_{\beta}(g,h)^*\beta_g(b),
\end{equation}
where $g\in \G$, $h \in \G^{s(g)}, a \in A_{s(g)}, b \in B_{s(g)}$. 
The most difficult (technically involved) part of the proof is showing that for $h, f \in \G^{s(g)}$, $a, c \in A_{s(g)}$, $b, d \in B_{s(g)}$, we have 
\begin{equation}\label{eq:twisted_thing_to_prove} 
    \langle L_g (a \xi(h)b) , L_g (c \xi(f)d) \rangle_{B_{r(g)}} = \beta_g(\langle a \xi(h)b , c \xi(f)d\rangle_{B_{s(g)}}).
\end{equation}
We denote by $L$ and $R$ the left and right hand side of \eqref{eq:twisted_thing_to_prove}. 
To make the calculations easier to read and write we abbreviate $u_{g,h} := u_{\alpha}{(g,h)}$ and $u_{g,h}:=u_{\beta}{(g,h)}$ as it is clear from the context whether we twist $\alpha$ or $\beta$. 
Then 
\[
\begin{split}
    R &= \beta_g \Big( b^*\beta_{h} \Big(T_{h^{-1} f} \big( \alpha_h^{-1}( a^* c\cdot  u^{*}_{h,h^{-1}f} \big) \Big) u_{h,h^{-1}f}\, d\Big) \\
    &= \beta_g(b^*)u_{g,h}\,\beta_{gh} \Big( T_{h^{-1} f} \big( \alpha_h^{-1}( a^* c\cdot  u^{*}_{h,h^{-1}f}) \big) \Big) u^{*}_{g,h} \beta_g(u^{\beta}_{h,h^{-1}f}) \beta_g(d) \\
    &= \beta_g(b^*)u_{g,h} \beta_{gh} \Big( T_{h^{-1} f} \big( \alpha_h^{-1}( a^* c\cdot  u^{*}_{h,h^{-1}f}) \big) \Big) u_{gh,h^{-1}f} u^{*}_{g,h} \beta_g(d) ,
\end{split}
\]
where in the last equality we used the cocycle identity. 
The left hand side is equal to 
\[
\begin{split} 
    L &= \Big \langle \alpha_g(a)u_{g,h} \xi(gh)u^{*}_{g,h}\beta_g(b) , \alpha_g(c)u^{}_{g,f} \xi(gf)u^{*}_{g,f}\beta_g(d)  \Big\rangle_{B_{r(g)}} \\
        &= \beta_g(b^*)u_{g,h}\beta_{gh} \Big(T_{h^{-1} f} \big(\alpha_{gh}^{-1}(u^{*}_{g,h}  \alpha_g(a^* c)u_{g,f} u^{*}_{gh,h^{-1}f}) \big)\Big) u_{gh,h^{-1}f}u^{*}_{g,h} \beta_g(d).
\end{split}
\]
Thus we need to check that the terms that $T_{h^{-1} f}$ acts on are equal. 
We compute 
\[
\begin{split} 
    \alpha_{gh}^{-1} \big(u^{*}_{g,h} \alpha_g(a^* c)u_{g,f} u^{*}_{gh,h^{-1}f}\big) &= \alpha_{h}^{-1} \alpha_{g}^{-1} \Ad_{u_{g,h}} \big( u^{*}_{g,h}  \alpha_g(a^* c)u_{g,f} u^{*}_{gh,h^{-1}f} \big) \\
        &= \alpha_{h}^{-1} \Big( \alpha_{g}^{-1} \big(  \alpha_g(a^* c)u_{g,f} u^{*}_{gh,h^{-1}f} u^{*}_{g,h} \big) \big) \\
        &= \alpha_{h}^{-1} \Big( a^* c\cdot \alpha_{g}^{-1} \big( u_{g,f} u^{*}_{gh,h^{-1}f} u^{*}_{g,h} \big) \Big) \\
        &= \alpha_{h}^{-1}( a^* c\cdot  u^{*}_{h,h^{-1}f} ) , 
\end{split}
\]
where the equality $\alpha_{g}^{-1}(u_{g,f} u^{*}_{gh,h^{-1}f} u^{*}_{g,h}) = u^{*}_{h,h^{-1}f}$ at the last step follows from the cocycle identity, as we have $ \alpha_g (u_{h,h^{-1}f})u_{g,f} = u_{g,h}u_{gh,h^{-1}f}$. 
This proves that $L=R$, that is \eqref{eq:twisted_thing_to_prove} holds.
This last equality implies that the linear extension of the prescription in \eqref{eq:Ldef} defines consistently a linear isometric map $L_g: E_{s(g)}^0 \to E_{r(g)}^0$ such that, for all $\xi_1, \xi_2 \in E_{s(g)}^0$, 
\begin{equation}\label{Lscalar}
    \langle L_h(\xi_1), L_h(\xi_2) \rangle_{B_{r(h)}} = \beta_h( \langle \xi_1, \xi_2 \rangle_{B_{s(h)}}) . 
\end{equation}
By continuity we get an isometry $L_g \in \B(E_{s(g)}, E_{r(g)})$ where \eqref{Lscalar} holds for all $\xi_1, \xi_2 \in E_{s(g)}$. 
Next note that \eqref{eq:Ldef} readily implies that 
\[ 
    L_g (a \cdot a_0 \xi(h)b_0\cdot b) = \alpha_g(a)\cdot  L_g(a_0 \xi(g)b_0) \cdot \beta_g(b), 
\]
for all $a,a_0 \in A_{s(g)}$, $b,b_0 \in B_{s(g)}$, $h\in \G^{s(g)}$. 
Thus by linearity and continuity 
\[ 
    L_g (a\cdot \xi \cdot b) = \alpha_g(a)\cdot L_g(\xi) \cdot \beta_g(b), \qquad  \xi \in E_{s(g)} ,\ a \in A_{s(g)} ,\ b \in  B_{s(g)}. 
\]
Similarly, to see that for $(f,g)\in \G^{(2)}$ we have $L_{f}\circ L_{g}(\cdot) = u_{\alpha}(f,g)L_{fg}(\cdot)u_{\beta}(f,g)^*$ on $E_{s(g)}$, we only need to check it on $a\xi(h)b$ where $f\in \G^{s(g)}$, $a \in A_{s(g)}$, $b \in B_{s(g)}$. 
We compute using our shorthand notation for twists: 
\[
\begin{split} 
    L_{f}( L_{g}(a\xi(h)b)) &= L_f\Big(\alpha_g(a)\,u_{g,h} \cdot\xi(gh)\cdot u_{g,h}^*\,\beta_g(b)\Big) \\
        &= \alpha_f\big(\alpha_g(a) u_{g,h}\big) u_{f,gh}\cdot \xi(fgh) \cdot u_{f,gh}^* \,\beta_f\big(u_{g,h}^* \beta_g(b)\big) \\
        &= \alpha_f\big(\alpha_g(a)\big) \big( \alpha_f (u_{g,h}) u_{f,gh}\big) \cdot \xi(fgh) \cdot \big( u_{f,gh}^* \,\beta_f(u_{g,h}^*)\big) \beta_f\big(\beta_g(b)\big) \\
        &= \alpha_f\big(\alpha_g(a)\big) (u_{f,g}  u_{fg,h}) \cdot \xi(fgh) \cdot (u_{fg,h}^* u_{f,g}^*) \beta_f\big(\beta_g(b)\big) \\
        &= u_{f,g}\alpha_{fg}(a)u_{fg,h}\cdot \xi(fgh) \cdot u_{fg,h}^*  \, \beta_{fg}(b) u_{f,g}^* \\
        &= u_{f,g} \, L_{fg}(a\xi(h)b) \, u_{f,g}^* . 
\end{split}
\]
Hence $L = \{ L_{g} \}_{g \in \G}$ is an $(\alpha,u_{\alpha})$-$(\beta,u_{\beta})$-equivariant representation of $\G$  on $\EE$.
Moreover, for any $g\in \G$ and $a \in A_{r(g)}$, \eqref{kernelformula} implies that $T_g(a) = \langle \xi(r(g)), a \xi(g) \rangle_{\M(B_{r(g)})}\in B_{r(g)}$ and \eqref{eq:Ldef} implies that $\overline{L}_{g}(\xi(s(g)) = \xi(g)$.
Thus
\[
    T_g(a) = \langle \xi(r(g)), a \overline{L}_g\big(\xi(s(g))\big)\rangle_{\M(B_{r(g)})}, 
\]
so that $T=T_{\EE,L,\xi|_{X},\xi|_{X}}$ where $\xi|_{X} \in \contb(\M(\EE))$ is bounded because $T$ is bounded, and in view of the above we have $\|\xi(x)\|^2=\|T_x\|$ for every $x\in X$. 
This finishes the proof of \ref{enu:FSmultiplier_positive_definite1}$\implies$\ref{enu:FSmultiplier_positive_definite4}. 
Moreover, in view of \eqref{eq:Ldef} we have $L_{U}(a\xi|_{X} b)=\alpha_{U}(a)\xi_{U}\beta_{U}(b)$ for every $a\in A_{s(U)}$, $b\in A_{s(U)}$, $U\in \Bis(\G)$. 
The linear span of such elements is equal to $\Gamma$, which by construction is dense in $\contz(\EE)$. 
Hence the triple $(\EE,L,\xi|_{X})$ is cyclic as in the second part of the assertion.

Finally, let $(\EE,L,\xi)$, $(\EE',L',\xi')$, be two cyclic triples such that $T_{\EE',L',\xi',\xi'}=T_{\EE,L,\xi,\xi}$.
We claim that for $a\in \contz(\A|_{s(U)})$,  $U\in \Bis(\G)$, and $c\in \contz(\A|_{s(V)})$,  $V\in \Bis(\G)$ 
\begin{equation}\label{eq:inner_product_preserve_cyclic}
    \langle  L_U(a \xi),  L_{V}(b\xi )\rangle = \langle  L_U'(a \xi' ),   L_{V}'(c\xi' )\rangle.
\end{equation}
To prove this we express the left hand side only in terms of $a,c$, $U,V$ and $T$.  
If $x\notin r(U)\cap r(V)$, then $\langle L_U (a\xi),L_V(b\xi) \rangle (x)=0$.
Assume then that $x\in r(U)\cap r(V)$. 
There are unique $g\in U$ and $h\in V$ such that $r(h)=r(g)=x$. 
Put $a_0=a(s(g))^* \alpha_g^{-1}(u_{\beta}(g, g^{-1}h)^*)$ and $c_0=\alpha_{g^{-1}{h}}(c(s(h)))$ and compute 
\[
\begin{split}
    \langle L_U (a\xi),L_V(c\xi) \rangle (x) &= \langle  L_{g}(a\xi(s(g))),L_{h}(c\xi(s(h)) \rangle \\
        &= \langle  L_{g}(a\xi(s(g))),L_{h}(c\xi(s(h))) u_{\beta}(g, g^{-1}h)^*\rangle u_{\beta}(g, g^{-1}h) \\
        &= \langle  L_{g}(a\xi(s(g))), u_{\beta}(g, g^{-1}h)^*L_{g} L_{g^{-1}h}(c\xi(s(h))) \rangle u_{\beta}(g, g^{-1}h) \\
        &= \beta_g\left(\langle  \xi(s(g)),  a(s(g))^* \alpha_g^{-1}(u_{\beta}(g, g^{-1}h)^*)L_{g^{-1}h}(c\xi(s(h))) \rangle\right) u_{\beta}(g, g^{-1}h) \\
        &= \beta_g\left( \langle \xi(s(g)), a_0 c_0L_{g^{-1}h}(\xi(s(h))) \rangle\right) u_{\beta}(g, g^{-1}h) \\
        &= \beta_g\left(T_{gh^{-1}}(a_0 c_0)\right) u_{\beta}(g, g^{-1}h).
\end{split}
\]
This proves \eqref{eq:inner_product_preserve_cyclic}, which implies that for $a\in \contz(\A|_{s(U)})$, $b\in \contz(\B|_{s(U)})$, $U\in \Bis(\G)$, and $c\in \contz(\A|_{s(V)})$, $d\in \contz(\B|_{s(V)})$, $V\in \Bis(\G)$ we have
\[
    \langle L_U(a \xi b), L_{V}(c\xi d)\rangle = \langle L_U'(a \xi' b) , L_{V}'(c\xi' d)\rangle.
\]
Thus if the triples $(\EE,L,\xi)$, $(\EE',L',\xi')$ are cyclic, then by linearity and continuity 
the formula $W (L_U(a \xi b)):=L_U'(a \xi' b)$ determines a unitary  $W:\contz(\EE)\to \contz (\EE')$ with the properties described in the last part of the assertion.
\end{proof} 

\begin{rem}\label{rem:functors_from_FS+} 
Extending Remark~\ref{rem:functors_from_FS}, let $FS_{TA}^+(\G)$, $CP_{TA}(\G)$, $CP_{TA}^{\rd}(\G)$ be the subcategories of $FS_{TA}(\G)$, $CB_{TA}(\G)$, $CB_{TA}^{\rd}(\G)$ consisting of positive-definite multipliers, and completely positive maps, respectively. 
By Theorem~\ref{thm:FSmultiplier_positive_definite}, the functors from Remark~\ref{rem:functors_from_FS} restrict to isometric functors $\m^{\f} : FS_{TA}^+(\G)\to CP_{TA}(\G)$ and $\m^{\rd} : FS_{TA}^+(\G)\to CP_{TA}^{\rd}(\G)$, and the ranges of these functors consist of fibre-preserving maps in $CP_{TA}(\G)$
and $CP_{TA}^{\rd}(\G)$, respectively.
\end{rem}

The next corollary extends the characterisation obtained in the above theorem further in the case of Herz--Schur multipliers. 

\begin{cor}\label{cor:completely_positive_Herz_Schur_multipliers}
Let $\A$ be a $\cst$-bundle over $X$, carrying a twisted action $(\alpha,u_\alpha)$ of $\G$ and let $\varphi \in \contb(r^*\M(\A))$. 
Denote by $T^{\varphi}$ the Herz--Schur multiplier associated with $\varphi$ and by $\m_{\varphi}$ the associated multiplier map defined as in~\eqref{eq:Herz-Schur_multiplier_map_defn}. 
The following statements are equivalent:
\begin{enumerate}
    \item\label{enu:positive_Herz-Shur1} $T^\varphi$ is a positive-definite multiplier for $\A^{(\alpha,u_{\alpha})}$;
    \item\label{enu:positive_Herz-Shur2} $\varphi$ takes values in central elements, i.e.\ $\varphi(g)\in Z\M(A_{r(g)})$, $g\in \G$, and for every $x \in X$ and every $g_1 , \ldots , g_n \in \G_x$, $n \in \N$, the matrix $\left( \alpha_{g_i}^{-1}( \varphi(g_i g_j^{-1}) ) \right)_{i,j = 1}^n \in \Mat_{n}(\M(A_x))$ is positive; 
    \item\label{enu:positive_Herz-Shur3} $\varphi$ takes values in central elements, i.e.\ $\varphi(g)\in Z\M(A_{r(g)})$, $g\in \G$, and for every $x \in X$ and every $g_1 , \ldots , g_n \in \G^x$, $n \in \N$, the matrix $\left( \alpha_{g_i} ( \varphi(g_i^{-1} g_j) ) \right)_{i,j = 1}^n \in \Mat_{n}(\M(A_x))$ is positive;
	\item\label{enu:positive_Herz-Shur4} $\m_\varphi$ extends to a completely positive map $\m_{\varphi}^{\rd} : \cst_{\red}(\A^{(\alpha,u_{\alpha})}) \to \cst_{\red}(\A^{(\alpha,u_{\alpha})})$; 
	\item\label{enu:positive_Herz-Shur5} $\m_\varphi$ extends to a completely positive map $\m_{\varphi}^{\f} : \cst(\A^{(\alpha,u_{\alpha})}) \to \cst(\A^{(\alpha,u_{\alpha})})$;
	\item\label{enu:positive_Herz-Shur6} $T^{\varphi} = T_{\EE,L,\xi,\xi}$ for some $\EE$, $L$ and $\xi\in \contb(\M(\EE))$, which is central in the sense that $a \xi =\xi a$ for all $a\in A=\contz(\A)$.
\end{enumerate}
If the above equivalent conditions hold, then $\|\varphi\|_{\infty} = \| \varphi|_{X} \|_{\infty} = \| \m_{\varphi}^{\rd} \| = \| \m_{\varphi}^{\f} \| = \| \xi \|$.
\end{cor}
\begin{proof} 
By its form, $T^{\varphi}$ is strict and bounded (we have $\sup_{x\in X}\|T_x^\varphi\|=\|\varphi|_{X}\|_{\infty}\leq \|\varphi\|_{\infty}<\infty$). 
By Lemma~\ref{lem:positive_action}\ref{enu:positive_action2}, $T^{\varphi}$ is positive-definite if and only if for every $x \in X$ and every $g_1 , \ldots , g_n \in \G_x$ and $\{a_{g_i} : 1 \leq i \leq n \} \subseteq A_x$, $n \in \N$, the matrix $\left( \alpha_{g_i}^{-1} ( \varphi(g_i g_j^{-1}) ) a_{g_i} a_{g_j}^* \right)_{i,j = 1}^n \in \Mat_{n}(A_x)$ is positive. 
This condition implies that for any $g \in \G$, and any $a_g\in A_{s(g)}$, the matrix 
\[
	\left( \begin{array}{cc}
	\varphi(s(g)) & \varphi(g^{-1}) a_g^*
	\\
	\alpha_{g}^{-1}(\varphi(g))a_g  & \alpha_{g}^{-1}(\varphi(s(g)))a_g a_g^*
	\end{array} \right)
\]
is positive and, in particular, self-adjoint. 
This implies $\alpha_{g}^{-1}(\varphi(g)) = \varphi(g^{-1})^*\in Z\M(\A_{s(g)})$, and so $\varphi$ is central. 
When $\varphi$ is central, positivity of $\left( \alpha_{g_i}^{-1}(\varphi(g_i g_j^{-1})) a_{g_i} a_{g_j}^*\right)_{i,j = 1}^n \in \Mat_{n}(A_x)$ for any set $\{a_{g_1},\ldots,a_{g_n}\}\subseteq A_x$ is clearly equivalent to positivity of $\left( \alpha_{g_i}^{-1}(\varphi(g_i g_j^{-1}))\right)_{i,j = 1}^n \in \Mat_{n}(\M(A_x))$. 
This proves the equivalence \ref{enu:positive_Herz-Shur1}$\iff$\ref{enu:positive_Herz-Shur2}.
The equivalence \ref{enu:positive_Herz-Shur1}$\iff$\ref{enu:positive_Herz-Shur3} can be checked analogously, cf.\ also the considerations before \cite[Proposition 4.3]{BedosConti2}.

Finally, Theorem~\ref{thm:FSmultiplier_positive_definite} gives the implications \ref{enu:positive_Herz-Shur6}$\implies$\ref{enu:positive_Herz-Shur1}$\iff$\ref{enu:positive_Herz-Shur4}$\iff$\ref{enu:positive_Herz-Shur5} and further shows that \ref{enu:positive_Herz-Shur1} implies that $T^{\varphi} = T_{\EE,L,\xi,\xi}$ for some equivariant representation $L$ of $\G$ on a $\cst$-$\A$-$\A$-correspondence bundle $\EE$ and section $\xi\in \contb(\M(\EE))$, which is \ref{enu:positive_Herz-Shur6} except for the centrality of $\xi$.  
To show the last fact assume \ref{enu:positive_Herz-Shur1} and invoke the construction from the proof of Theorem~\ref{thm:FSmultiplier_positive_definite}.
By \eqref{kernelformula} for any $a\in \contz(\A)$ we have
\[
    \langle \xi, a\xi\rangle_{A}(x) = \langle \xi, \xi \rangle_{\M(A)}(x)a(x),
\]
where $\langle \xi, \xi\rangle_{\M(A)}(x) = \sum_{h,g\in \G^x}\alpha_{g}(\varphi(g^{-1}h))$ is central in $\M(A_x)$ because $\varphi$ is central. 
This implies that 
\[
    \langle \xi, a^*a\xi \rangle_{A} = a^*\langle \xi, a\xi\rangle_{A} = \langle \xi, a^*\xi\rangle_{A} a = a^* \langle \xi, \xi\rangle_{A} a = a^*a\langle \xi, \xi\rangle_{A},
\]
and consequently $\langle a\xi - \xi a , a\xi - \xi a\rangle_A = 0$. 
Hence $a\xi=\xi a$. 
\end{proof}

\begin{rem}\label{rem:positive_type_conditions} 
The positivity condition in \ref{enu:positive_Herz-Shur2} implies the one in \cite[Definition 3.4]{BartoszKangAdam}. 
In the context of group actions, the positivity condition in \ref{enu:positive_Herz-Shur3} appears in \cite{Clare1}, \cite{Clare2} and \cite{BedosConti2}, where the corresponding $\varphi$ is called of \emph{positive-type} or \emph{AD-positive-definite}, respectively.	
\end{rem}

\section{Fourier multipliers} 
\label{Sec:Fourierapproximation}

In this section we introduce Fourier multipliers as those Fourier--Stieltjes multipliers which arise from regular equivariant representations.

\begin{defn}\label{de:FourierAlgebra}
Let $\A$, $\B$ be $\cst$-bundles over $X$, carrying twisted actions $(\alpha,u_\alpha)$ and $(\beta,u_\beta)$ of $\G$. 
We say that a multiplier $T\in FS[(\alpha,u_{\alpha}),(\beta,u_{\beta})]$ is a \emph{Fourier multiplier} if it can be represented by a regular equivariant representation of $\G$, i.e.\ if there exists an equivariant representation $ L$ of $\G$ on a $\cst$-$\A$-$\B$-correspondence bundle $\EE$ and sections $\xi,\zeta\in \contb(\M(\ell^2(s^*\EE)))$ such that for any $a\in A_{r(g)}$ and $g\in \G$, 
\begin{equation}\label{eq:Fourier_multiplier}
    T_g(a) = \sum_{t\in \G_{r(g)}} \langle \xi(t), a \overline{L}_{g}(\zeta(tg))\rangle_{\M(B_{r(g)})}.
\end{equation}
We denote by $F[(\alpha,u_{\alpha}),(\beta,u_{\beta})]$ the set of all such Fourier multipliers and by $F[(\alpha,u_{\alpha}),(\beta,u_{\beta})]^+$ the set of Fourier multipliers that can represented in the form \eqref{eq:Fourier_multiplier} with $\xi=\zeta$. 
\end{defn}

Note that the last part of the definition is related to a certain subtlety regarding various possible notions of positivity --- see Question~\ref{quest:positive}.

\begin{ex}\label{ex:approximation_property_maps_as_Fouriers}
Let $\A$ be a $\cst$-bundle over $X$, carrying a twisted action $(\alpha,u_\alpha)$ of $\G$. 
The multipliers described in Example~\ref{ex:Alcides_approximation_property_general}, given by sections in $\contc(r^*\A)$ applied to a Fell bundle $\A^{(\alpha,u_{\alpha})}$, give rise to Fourier multipliers in $F[(\alpha,u_{\alpha}),(\alpha,u_{\alpha})]$. 
More concretely, given $\xi,\zeta\in \contc(r^*(\A))$, we get the Fourier multiplier 
\[
    T_g(a)=\sum_{h\in \G^{r(g)}}\xi(h)^*a\alpha_g(\zeta(g^{-1}h)),\quad a\in A_{r(g)} .
\]
Specialising further to the case of the trivial bundle $\A=\C\times X$, that is, the case where we have just a twisted groupoid $(\G,u)$ (see Example~\ref{ex:continuous-cocycles}), given functions $\xi,\zeta\in \contc(\G)$, we get from the above formula a multiplier $T$ of $(\G,u)$, that we view again as the function in $\contc(\G)$ given by 
\[
    T_g = \sum_{h\in \G^{r(g)}}\overline{\xi(h)}\zeta(g^{-1}h),\quad g\in \G.
\]
Notice that the above function equals the (untwisted) convolution product $\tilde\xi^**\tilde\zeta\in \contc(\G)$, where $\tilde\xi(g):=\xi(g^{-1})$.
In particular, if $\xi=\zeta$, this is a prototype example of a positive-type function, see \cite[Definition~5.6.15]{BO}.
\end{ex}

The next proposition shows that the collection of Fourier multipliers has a natural `ideal' property.

\begin{prop}\label{prop:FourierAlgebraProperties}
Let $\A$, $\B$, $\D$  be $\cst$-bundles over $X$, carrying twisted actions $(\alpha,u_\alpha)$, $(\beta,u_\beta)$ and $(\gamma, u_\gamma)$ of $\G$. 
The Fourier multipliers $F[(\alpha,u_{\alpha}),(\beta,u_{\beta})]$ form an involutive $\contb(X)$-submodule of $FS[(\alpha,u_{\alpha}),(\beta,u_{\beta})]$. 
Moreover, $F[(\alpha,u_{\alpha}),(\beta,u_{\beta})]$ equipped with the norm 
\[
    \| T \|_{F} := \inf \big\{ \|\xi \| \|\zeta \| : T = T_{\ell^2(s^*\EE),L^\G,\xi,\zeta} \ \text{for some $(\EE, L)$ and $\xi,\zeta\in \contb(\M(\ell^2(s^*\EE)))$} \big\},
\]
is a Banach space. 
Further
\begin{gather*}
    FS[(\alpha,u_{\alpha}),(\beta,u_{\beta})]\cdot F[(\beta,u_{\beta}),(\gamma,u_{\gamma})] \subseteq F[(\alpha,u_{\alpha}),(\gamma,u_{\gamma})], \\
    F[(\alpha,u_{\alpha}),(\beta,u_{\beta})] \cdot FS[(\beta,u_{\beta}),(\gamma,u_{\gamma})]\subseteq F[(\alpha,u_{\alpha}),(\gamma,u_{\gamma})] .
\end{gather*}
In particular, $F(\alpha,u_{\alpha})$ is a two-sided ideal of the Banach algebra $FS(\alpha,u_{\alpha})$. 
\end{prop}
\begin{proof} 
The same arguments as in the proof of Proposition~\ref{prop:BanachNormFS} show that $F[(\alpha,u_{\alpha}),(\beta,u_{\beta})]$ is a Banach $\contb(X)$-bimodule with norm $\|\cdot \|_{F}$. 
For instance, the direct sum of regular equivariant representations of $\G$ is a regular equivariant representation, i.e.\ $(L\oplus K)^\G= L^\G\oplus K^\G$, or more generally $(\oplus_{i\in I} L_i)^\G=\oplus_{i\in I} L_i^\G$. 
Also, by Corollary~\ref{cor:regular_equivariant_absorbtion}, we have $L^\G\otimes K\cong L\otimes K^\G \cong (L\otimes K)^\G$, which implies the inclusions in the assertion. 
Indeed, if  $T_{\ell^2(\G , \EE) , L^\G , \xi_1, \zeta_1}\in FS[(\alpha,u_{\alpha}),(\beta,u_{\beta})]$, $T_{\FF , K , \xi_2, \zeta_2}\in F[(\beta,u_{\beta}),(\gamma,u_{\gamma})]$, and 
$\Gamma : L^\G\otimes K \to (L\otimes K)^\G$ is the unitary from  Corollary~\ref{cor:regular_equivariant_absorbtion}, then using \eqref{eq:compostion_of_Fourier_Stieltjes} we get 
\[
\begin{split}
    T_{\FF,K,\xi_2,\zeta_2} \circ  T_{\ell^2(\G , \EE) , L^\G , \xi_1, \zeta_1} &= T_{\ell^2(\G , \EE)\otimes \FF, L^\G\otimes  K,\xi_1\otimes  \xi_2,\zeta_1 \otimes \zeta_2} \\
        &= T_{\ell^2(\G , \EE\otimes \FF), (L\otimes K)^\G , \Gamma(\xi_1\otimes  \xi_2) , \Gamma(\zeta_1\otimes  \zeta_2)} \in F[(\alpha,u_{\alpha}),(\gamma,u_{\gamma})].
\end{split}
\]
Hence $FS[(\alpha,u_{\alpha}),(\beta,u_{\beta})]\cdot F[(\beta,u_{\beta}),(\gamma,u_{\gamma})] \subseteq F[(\alpha,u_{\alpha}),(\gamma,u_{\gamma})]$, and the other inclusion is proved analogously.
\end{proof}

The equality of the norms $\|\cdot\|_F$ and $\|\cdot\|_{FS}$ on the set of Fourier multipliers remains unclear. 
See Question~\ref{quest:norms} for further discussion.


\begin{thm}\label{thm:Fmultipliers_are_reduced_to_full} 
Let $\A$, $\B$ be $\cst$-bundles over $X$, carrying twisted actions $(\alpha,u_\alpha)$ and $(\beta,u_\beta)$ of $\G$. 
Every Fourier multiplier $T\in F[(\alpha,u_{\alpha}),(\beta,u_{\beta})]$ extends to a strict completely bounded map 	$\m_T^{\rd,\f} : \cst_{\red}(\A^{(\alpha,u_{\alpha})}) \to \cst(\B^{(\beta,u_{\beta})})$ from reduced to full crossed products, and $\| \m_T^{\rd,\f} \|_{cb} \leq \|T\|_{F}$. 
If $T = T_{s^*\EE,L^\G,\xi,\xi} \in F[(\alpha,u_{\alpha}),(\beta,u_{\beta})]^+$, then $\m_T^{\rd,\f}$ is completely positive and $\| \m_T^{\rd,\f} \|_{cb} = \| T \|_{FS} = \| T \|_{F} = \|\xi\|^2$.
\end{thm}
\begin{proof} 
Write $T := T_{s^*\EE,L^{\G},\xi,\zeta} = \{ T_g \}_{g\in \G}$ where $\xi,\zeta\in \contb(\M(s^*\EE))$. 
Let $\psi : \cst(\B^{(\beta,u_{\beta})}) \to \BB(\H)$ be a faithful nondegenerate representation on a Hilbert space $H$. 
As in the second part of the proof of Theorem~\ref{thm:FSmultiplier_extends_to_reduced_and_full}, we have 
\begin{equation}\label{eq:m_T_for_Fourier_multplier}
    \m_T(a) := \theta_\xi^* (L^\G\dashind(\psi))(a) \theta_\zeta, \qquad a\in\contc(\A^{(\alpha,u_{\alpha})}) ,
\end{equation}
where $L^\G\dashind(\psi) : \cst(\A^{(\alpha,u_{\alpha})})\to \BB(\contz(s^*\EE)\otimes_{\psi}H)$ is the induced representation and $\theta_\xi, \theta_{\zeta} : H \to   \contz(s^*\EE)\otimes_{\psi} H$ are creation operators. 
By Fell absorption (Theorem~\ref{thm:Fell_absorption II}) $L^\G\dashind(\psi)$ descends to a representation of the reduced crossed product, and by abuse of notation we may write $L^\G\dashind(\psi) : \cst_{\red}(\A^{(\alpha,u_{\alpha})}) \to \BB(\contz(s^*\EE)\otimes_{\psi}H)$.
Then \eqref{eq:m_T_for_Fourier_multplier} implies that $\m_T$ extends to a strict completely bounded map $\m_T^{\rd,\f} : \cst_{\red} (\A^{(\alpha,u_{\alpha})}) \to \cst(\B^{(\beta,u_{\beta})})$ and $\| \m_T^{\rd,\f} \|_{cb} \leq \| \xi \| \cdot \|\zeta\|$. 
In particular, $\|\m_T^{\rd,\f}\|_{cb} \leq \|T\|_{F}$. 
Also if $\xi=\zeta$, then \eqref{eq:m_T_for_Fourier_multplier} implies that $\m_T^{\rd,\f}$ is completely positive and $\| \m_T^{\rd,\f} \|_{cb} =\|\xi\|^2$, which also implies $\| \m_T^{\rd,\f} \|_{cb} =\|T\|_{F}$.
\end{proof}

\begin{rem}\label{rem:functors_from_F}
To use the categorical language, as in Remarks~\ref{rem:functors_from_FS} and~\ref{rem:functors_from_FS+}, denote by $F_{TA}(\G)$ and $CB_{TA}^{\rd, \f}(\G)$ the non-unital categories of all twisted actions of $\G$ where morphisms in $F_{TA}(\G)$ are Fourier multipliers, and morphisms in $CB_{TA}^{\rd, \f}(\G)$ are completely bounded strict $\contz(X)$-bimodule maps from reduced to full crossed products whose composition is defined by $T S:= T \circ \Lambda \circ S$ where $\Lambda$ is the relevant regular representation.  
The extension in Theorem~\ref{thm:FSmultiplier_extends_to_reduced_and_full} gives a contractive involutive $\contb(X)$-bimodule functor $\m^{\rd,\f}:F_{TA}(\G)\to CB_{TA}^{\rd,\f}(\G)$. 
It restricts to an isometric functor $\m^{\rd,\f}:F_{TA}^+(\G)\to CP_{TA}^{\rd,\f}(\G)$, where morphisms in $F_{TA}^+(\G)\subseteq F_{TA}(\G)$ are elements of $F[(\alpha,u_{\alpha}),(\beta,u_{\beta})]^+$ and morphisms in $CP_{TA}^{\rd,\f}(\G)\subseteq CB_{TA}^{\rd, \f}(\G)$ are completely positive maps.
\end{rem}

\begin{defn}\label{def:support_of_a_multiplier}
Given two Fell bundles $\A$, $\B$ over $\G$ we define the \emph{support} of a multiplier $T=\{T_g\}_{g\in \G}$ as $\supp (T) := \overline{ \{ g\in \G: T_g\neq 0 \} }$.
\end{defn}
 
\begin{prop}\label{prop:compactly_supported}
Compactly supported Fourier--Stieltjes multiplier are Fourier multipliers, i.e.\ if $\A$, $\B$ are $\cst$-bundles over $X$, carrying twisted actions $(\alpha,u_\alpha)$ and $(\beta,u_\beta)$ of $\G$, and $T\in FS[(\alpha,u_{\alpha}),(\beta,u_{\beta})]$ has compact support, then $T\in F[(\alpha,u_{\alpha}),(\beta,u_{\beta})]$.
\end{prop}
\begin{proof}
Let $K$ be a compact subset of $\G$. 
We claim that there exist $\xi , \zeta\in  \contb(\M(\ell^2(s^*\A)))$ such that $T_{\ell^2(s^*\A), \alpha^\G, \xi , \zeta}(g)(a) = a$ for all $a \in A_{r(g)}$ and $g \in K$. 
Indeed, pick $\xi_0 \in \contc(X)\subseteq\contc(\G )$ such that $\xi_0(x) = 1$ for all $x\in r(K)$, and $\zeta\in \contc(\G)$ such that $\zeta(g) = 1$ for all $g \in K$. 
Consider the corresponding strictly continuous sections $\xi, \zeta \in \contc(\M(\ell^2(\s^*\A)))$ given by $x\mapsto \xi_0|_{\G_x}, \zeta_0|_{\G_x}$, cf.\ Proposition~\ref{prop:InclusionGroupoidFSAlg}.
Then for $a \in A_{r(g)}$ and $g \in K$ we have
\[
\begin{split}
    T_{\ell^2(s^*\A) , \alpha^\G , \xi , \zeta}(g)(a) &= \langle \xi (r(g)) , a \overline{\alpha^\G}_g (\zeta(s(g))) \rangle_{\M(A_{r(g)})} \\  
        &= \sum_{t \in \G_{r(g)}} \xi_0(t)^* a \alpha_g \big( \zeta_{0}(t g) \big) = \xi_0(r(g))^* a  \zeta_0(g)=a.
\end{split}
\]
This proves our claim. 

Suppose now that $T\in FS[(\alpha,u_{\alpha}),(\beta,u_{\beta})]$ has compact support. 
The claim above yields $T_{\ell^2(s^*\A), \alpha^\G, \xi , \zeta} \in F[(\alpha,u_{\alpha}),(\beta,u_{\beta})]$ such that $T_{\ell^2(s^*\A), \alpha^\G, \xi , \zeta}\circ T = T$.
Thus $T \in F[(\alpha,u_{\alpha}),(\beta,u_{\beta})]$, by   Proposition~\ref{prop:FourierAlgebraProperties}.
\end{proof}

In terms of maps on crossed products ``compactness of the support'' means that the range of the map is supported on a precompact open set. 
We formalise this as follows. 

\begin{lem}\label{lem:precompacts_embed_into_crossed_products}
For any Fell bundle $\A$ over $\G$ and every precompact open set $V\subseteq \G$ the subspace $\contz(\A|_{V})$ of $\contc(\A)$ is closed both in the reduced and (hence) also in the universal (maximal) norm.
\end{lem}
\begin{proof}
Since the reduced norm is dominated by the universal norm it suffices to show that  $\contz(\A|_{V})$ is closed in the reduced norm.
We use the known fact that the inclusion $\contc(\A^{(\alpha,u_{\alpha})}) \subseteq \contz(\A^{(\alpha,u_{\alpha})})$ extends to a contractive linear map $j : \cst_{\red}(\A^{(\alpha,u_{\alpha})}) \to \contz(\A^{(\alpha,u_{\alpha})})$, where the codomain is equipped with the supremum norm, see \cite[Proposition 7.10]{BartoszRalf2}. 
This implies the assertion because $j$ is the identity on $\contz(\A^{(\alpha,u_{\alpha})}|_{V})$ and $\contz(\A^{(\alpha,u_{\alpha})}|_{V})$ is closed in $\contz(\A^{(\alpha,u_{\alpha})})$.
\end{proof} 

\begin{cor}\label{cor:compactly_supported_vs_range}
Let $T = \{ T_g \}_{g\in \G}$ be a reduced (or full) multiplier from a Fell bundle $\A$ to a Fell bundle $\B$ over $\G$. 
Then $T$ has compact support if and only if the range of $\m_T^{\rd}$ (resp.\ $\m_T^{\f}$) is contained in  $\contz(\B|_{V})$ for some precompact open $V\subseteq \G$.
\end{cor}
\begin{proof}
If the range of $\m_T^{\rd}$ (resp.\ $\m_T^{\f}$) is contained in  $\contz(\B|_{V})$, then clearly $\supp (T)\subseteq V$. 
Hence precompactness of $V$ implies compactness of $\supp (T)$. 
Conversely, if $\supp (T)$ is compact then there is a precompact open set $V\subseteq \G$ containing $\supp (T)$. 
Since $\m_t(\contc(\A))\subseteq \contz(\B|_V)$ and $\contz(\B|_V)$ is closed in both $\cst_{\red}(\B)$ and $\cst(\B)$ one concludes that the range of the continuous extension of $\m_T$ is also in $\contz(\B|_V)$.
\end{proof}

\section{Approximation properties for multipliers}
\label{Sec:Applicationapproximation}

In this section we turn to the discussion of approximation properties. 
We begin in the general context of abstract Fell bundle multipliers, proving a lemma that shows one can always approximately `strictify' a given multiplier. 
\begin{lem}\label{lem:multiplier_approx_with strict}
Let $\A$, $\B$  be $\cst$-bundles over $\G$. 
Let $T$ be a multiplier from $\A$ to $\B$ such that $\sup_{g\in \G} \| T_g \|<\infty$. 
Choose an approximate unit $\{e_i\}_{i\in I}$ in $\contz(\A)$.
Putting 
\[
    T_{i,g}(a) := T_g \big( e_i(r(g)) a \alpha_g(e_i(s(g))) \big)
\] 
for all $a\in A_{r(g)}$, $g\in \G$ and $i\in I$,  
we obtain a net $\{T_{i}\}_{i\in I}$ of strict multipliers from $\A$ to $\B$ such that $\m_{T_i}(f)$ converges uniformly on compact sets to $\m_{T}(f)$  for every $f\in\contc(\A)$. 
\end{lem}
\begin{proof} 
By definition, for $f\in\contc(\A)$ and $i \in I$ we have $\m_{T_i}(f)=\m_T(e_if e_i)$ and therefore
\[
    \| \m_{T_i}(f) - \m_{T}(f) \|_{\infty} = \| \m_{T}(e_ife _i-f) \|_{\infty} \leq \sup_{g\in \G} \| T_g \|\cdot \|e_ife _i-f\|_{\infty}\to 0.
\]
This proves the assertion, as each $T_{i}$ is a strict multiplier by construction. 
\end{proof}

We now specify the context to twisted groupoid actions.

\begin{lem}\label{lem:multiplier_approx_with strict2}
Let $\A$, $\B$  be $\cst$-bundles over $X$, carrying twisted actions $(\alpha,u_\alpha)$ and $(\beta,u_\beta)$ of $\G$. 
Let $T$ be a multiplier from $\A^{(\alpha,u_{\alpha})}$ to $\B^{(\beta,u_{\beta})}$ such that $\sup_{g\in \G} \| T_g \|<\infty$. 
Choose an approximate unit $\{e_i\}_{i\in I}$ in $\contz(\A)$. As before put
\[
    T_{i,g}(a) := T_g \big( e_i(r(g)) a \alpha_g(e_i(s(g))) \big)
\] 
for all $a\in A_{r(g)}$, $g\in \G$ and $i\in I$.  
Then
\begin{enumerate}
    \item\label{ite:multiplier_approx_with strict1} if $T\in FS[(\alpha,u_{\alpha}), (\beta,u_{\beta})]$, then for every $i \in I$ we have $T_i \in FS[(\alpha,u_{\alpha}), (\beta,u_{\beta})]$, $\|T_i\|_{FS}\leq \|T\|_{FS}$ and $\|T_i\|_{FS}\to \|T\|_{FS}$ (similar statement holds with $FS$ replaced by $F$);
    \item\label{ite:multiplier_approx_with strict2} if $T$ is positive-definite, then each $T_{i}$ is a  positive-definite Fourier-Stieltjes multiplier. Moreover we have the estimate $\sup_{i\in I} \| T_i \|_{FS} \leq \sup_{x\in X} \|T_{x}\|$.
\end{enumerate}
\end{lem}
\begin{proof}

If $T = T_{\EE,L,\xi,\zeta}$, for an equivariant representation $L$ of $\G$ on a $\cst$-$\A$-$\B$-correspondence bundle $\EE$ and sections $\xi,\zeta\in \contb(\M(\EE))$ then for each $i \in I$ we have $T_i=T_{\EE,L,e_i\xi_i,e_i\zeta}$.
Clearly $\|e_i\xi\|\leq \|\xi\|$, and $\|e_i\xi\|\to \|\xi\|$ because $\contz(\EE)=\contz(\A)\contz(\EE)$, as we assume $\cst$-correspondences are non-degenerate.  
We have similar relations for $\zeta$ and this readily implies the assertion in \ref{ite:multiplier_approx_with strict1}.

Further one can easily check that if $T$ is positive-definite, so is $T_i$ for each $i \in I$.
Thus the assertion in \ref{ite:multiplier_approx_with strict2} follows by Theorem~\ref{thm:FSmultiplier_positive_definite}.
\end{proof}

\begin{lem}\label{lem:Fourier_explicite_approximation}
Let $\A$, $\B$  be $\cst$-bundles over $X$, carrying twisted actions $(\alpha,u_\alpha)$ and $(\beta,u_\beta)$ of $\G$ and let $L$ be an equivariant representation of $\G$ on a $\cst$-$\A$-$\B$-correspondence bundle $\EE$. 
If $T=T_{\ell^2(s^*\EE) , L^\G , \xi , \zeta}$ for sections $\xi,\zeta\in \contb(\M(\ell^2(s^*\EE)))$, then there are nets $\{\xi_i\}_{i\in I}, \{\zeta_i\}_{i \in I} \subseteq  \contc(s^*\EE) \subseteq \contc(\ell^2(s^*\EE))$ such that  $\| \xi_i \| \leq \|\xi\|$, $\|\zeta_i\|\leq \|\zeta\|$, for every $i \in I$, and putting $T_i:=T_{\ell^2(s^*\EE) , L^\G , \xi_i , \zeta_i}$ it follows that $T_i$ is compactly supported and
$\m_{T_i}(a)$ converges uniformly on compact sets to $\m_{T}(a)$ for every $a\in\contc(\A^{(\alpha,u_{\alpha})})$.  
\end{lem}
\begin{proof}
Using Lemmas~\ref{lem:multiplier_approx_with strict} and \ref{lem:multiplier_approx_with strict2}, and by passing to $\xi_i:= e_i \xi \in \contz(\ell^2(s^*\EE))$ and $\zeta_i := e_i \zeta\in \contz(\ell^2(s^*\EE))$ for an approximate unit $\{e_i\}_{i\in I}$ in $\contz(\A)$, we may reduce the proof to the case when $\xi , \zeta \in \contz(\ell^2(s^*\EE))$. 
Since $\contc(s^*\EE)$ is dense in $\contz(\ell^2(s^*\EE))$ we may find nets $\{\xi_i\}_{i \in I}, \{\zeta_i\}_{i \in I}\subseteq \contc(s^*\EE)\subseteq \contc(\ell^2(s^*\EE))$ with $\| \xi_i \| \leq \| \xi \|$, $\|\zeta_i\|\leq \|\zeta\|$ for each $i \in I$ and $\|\xi-\xi_i\|\to 0$, $\|\zeta-\zeta_i\|\to 0$. 
Then putting $T_i:=T_{\ell^2(s^*\EE) , L^\G , \xi_i , \zeta_i}$, we have 
\[
    \| \m_{T_i}(a) - \m_{T}(a) \|_{\infty} \leq \| \xi - \xi_i \| \cdot \| a \|_{\infty} \cdot \| \zeta \| + \|\zeta-\zeta_i\| \cdot \| a \|_{\infty} \cdot \| \zeta_i \| \to 0 
\]
for every $a\in\contc(\A^{(\alpha,u_{\alpha})})$. 

Finally, if $\xi$ and $\zeta$ are supported on precompact bisections $U$ and $V\in\Bis(\G)$, respectively, then \eqref{eq:Fourier_multiplier} implies that $T$ is supported on a precompact bisection $U^*V$. 
Since elements in $\contc(s^*\EE)$ are finite sums of elements with supports in precompact bisections, this proves the compact support assertion, as whenever $\tilde\xi,\tilde\zeta \in  \contc(s^*\EE)$, the multiplier $T=T_{\ell^2(s^*\EE) , L^\G , \tilde\xi , \tilde\zeta}$ is compactly supported.
\end{proof}

The last lemma can be used to approximate Fourier multipliers by compactly supported ones, but a priori only in the sense of uniform convergence on compact sets.

\begin{prop}\label{pr:CompactSupportFourier}
Let $\A$, $\B$  be $\cst$-bundles over $X$, carrying twisted actions $(\alpha,u_\alpha)$ and $(\beta,u_\beta)$ of $\G$. 
If $T\in F[(\alpha,u_{\alpha}),(\beta,u_{\beta})]$, then  there is a net $\{T_{i}\}_{i\in I}$ of compactly supported multipliers in $ F[(\alpha,u_{\alpha}),(\beta,u_{\beta})]$ such that  $\| T_{i} \|_{F} \leq \| T \|_{F}$ for every $i \in I$ and $\m_{T_i}(a)$ converges uniformly on compact sets to $\m_{T}(a)$ for every $a\in\contc(\A^{(\alpha,u_{\alpha})})$.
\end{prop}
\begin{proof}
For each $k \in \N$ we can find an equivariant representation $L_k$ of $\G$ on a $\cst$-$\A$-$\B$-correspondence bundle $\EE_k$ and non-zero sections $\xi_k , \zeta_k \in \contb( \M(\ell^2(s^*\EE_k)) )$ such that $T=T_{\ell^2(s^*\EE_k) , L_k^\G , \xi_k , \zeta_k}$ and $\|\xi_k\|\cdot \|\zeta_k\|\leq \|T\|_F +{1/k}$, so that $\lambda_k := \frac{\|T\|_{F}}{\|\xi_k\|\cdot \|\zeta_k\|}\leq 1$ converges to $1$ as $k$ tends to infinity. 
Thus setting $T_k := \lambda_k T=T_{\ell^2(s^*\EE_k) , L_k^\G , \lambda_k\xi_k , \zeta_k} \in F[(\alpha,u_{\alpha}),(\beta,u_{\beta})]$, for each $k \in \N$ we get $\|T_k\|_F\leq  \|\lambda_k\xi_k\|\cdot\|\zeta_k\|=\|T\|_F$ and $m_{T_k}\to m_{T}$.
Applying Lemma~\ref{lem:Fourier_explicite_approximation} to each $T_k$ one gets the desired net.
\end{proof}

We are ready to define an analogue of the property above for arbitrary multipliers. 

\begin{defn}\label{def:Fourier_approximation_property}
Let $\A$, $\B$ be $\cst$-bundles over $X$, carrying twisted actions $(\alpha,u_\alpha)$ and $(\beta,u_\beta)$ of $\G$. 
A multiplier $T$ from $\A^{(\alpha,u_{\alpha})}$ to $\B^{(\beta,u_{\beta})}$ has the \emph{(positive) Fourier approximation property} if there is a $\|\cdot\|_{F}$-bounded net $\{T_{i}\}_{i\in I}$ in $F[(\alpha,u_{\alpha}),(\beta,u_{\beta})]$ (resp.\ in $ F[(\alpha,u_{\alpha})],(\beta,u_{\beta})]^+$) such that $\m_{T_i}(a)$ converges uniformly on compact sets to $\m_{T}(a)$ for all $a\in\contc(\A^{(\alpha,u_{\alpha})})$.
\end{defn}

Later we will often use the acronym \emph{AP} for the approximation property. 
The next proposition gathers several equivalent conditions for a given multiplier to have the Fourier AP. 

\begin{prop}\label{prop:characterisation_Fourier_approx}
Let $\A$, $\B$ be $\cst$-bundles over $X$, carrying twisted actions $(\alpha,u_\alpha)$ and $(\beta,u_\beta)$ of $\G$. 
For a multiplier $T$ from $\A^{(\alpha,u_{\alpha})}$ to $\B^{(\beta,u_{\beta})}$ the following are equivalent:
\begin{enumerate}
    \item\label{enu:characterisation_Fourier_approx1} $T$ has the (positive) Fourier approximation property;  
    \item\label{enu:characterisation_Fourier_approx2} there is a  $\|\cdot\|_{F}$-bounded net $\{T_{i}\}_{i\in I}$ in $ F[(\alpha,u_{\alpha}),(\beta,u_{\beta})]$ (resp.\ in $ F[(\alpha,u_{\alpha})],(\beta,u_{\beta})]^+$) of compactly supported  multipliers such that $\m_{T_i}(a)$ converges  uniformly on compact sets to $\m_{T}(a)$ for every $a\in\contc(\A^{(\alpha,u_{\alpha})})$; 
    \item\label{enu:characterisation_Fourier_approx3} there is a net $\{L_i\}_{i \in I}$ of equivariant representations of $\G$ acting respectively on $\cst$-$\A$-$\B$-correspondence bundles $\EE_i$ and sections $\xi_i$, $\zeta_i\in \contc(s^*\EE_i)$, $i\in I$, (resp. $\xi_i=\zeta_i$) such that 
    \begin{enumerate}
        \item $\sup_{i\in I}\sup_{x\in X}\|\sum_{t\in \G_{x}} \langle \xi_i(t), \xi_i(t)\rangle_{B_x}\| \cdot \sup_{x\in X}\|\sum_{t\in \G_{x}} \langle \zeta_i(t), \zeta_i(t)\rangle_{B_x}\|<\infty$,
        \item for every $a\in\contc(\A^{(\alpha,u_{\alpha})})$ the functions 
        \[
            \G\ni g\longmapsto \sum_{t\in \G_{r(g)}}\langle \xi_i(t), a(g)L_{g} \zeta_i(tg)\rangle_{B_{r(g)}}
        \]
        converge uniformly on compact sets to the function $\G\ni g\longmapsto T_{g}(a(g))\in B_{r(g)}$.
    \end{enumerate}
\end{enumerate}
\end{prop}
\begin{proof}
Combine Proposition~\ref{pr:CompactSupportFourier} and Lemma~\ref{lem:Fourier_explicite_approximation}.
\end{proof}

\begin{defn}\label{def:weak_approximation_property}
Let $\A$, $\B$  be $\cst$-bundles over $X$, carrying twisted actions $(\alpha,u_\alpha)$ and $(\beta,u_\beta)$ of $\G$. 
A multiplier $T$ from $\A^{(\alpha,u_{\alpha})}$ to $\B^{(\beta,u_{\beta})}$ has the \emph{Fourier--Stieltjes approximation property} if there is a $\|\cdot\|_{FS}$-bounded net $\{T_{i}\}_{i\in I}\subseteq FS[(\alpha,u_{\alpha}),(\beta,u_{\beta})]$ of compactly supported multipliers such that $\m_{T_i}(a)$ converges uniformly on compact sets to $\m_{T}(a)$ for every $a\in\contc(\A^{(\alpha,u_{\alpha})})$.  
If one can find $T_i$'s as above which are positive-definite, we speak of the \emph{positive Fourier--Stieltjes approximation property}. 
\end{defn}

\begin{rem}
By Propositions~\ref{prop:compactly_supported} and \ref{prop:characterisation_Fourier_approx}\ref{enu:characterisation_Fourier_approx2}, the Fourier AP and Fourier--Stieltjes AP differ only in the boundedness condition, formulated respectively in terms of the Fourier and the Fourier--Stieltjes norm. 
Since $\|\cdot\|_{FS}\leq \|\cdot\|_{F}$, the $FS$-boundedness is  formally weaker than $F$-boundedness. 
Hence the Fourier AP implies the Fourier--Stieltjes AP. 
It seems reasonable to conjecture that the two notions coincide at least for identity multipliers and positive approximation that we discuss below (see Question \ref{quest:norms}). 
Note also that for the (twisted) actions of discrete groups on unital $\cst$-algebras the Fourier AP was introduced in \cite{BedosContiregular} under the name \emph{weak approximation property}.
\end{rem}

Theorem~\ref{thm:FSmultiplier_positive_definite} implies the following characterisation of the positive Fourier--Stieltjes AP. 

\begin{prop}\label{prop:characterisation_positive_FS_approx}
Let $\A$, $\B$  be $\cst$-bundles over $X$, carrying twisted actions $(\alpha,u_\alpha)$ and $(\beta,u_\beta)$ of $\G$. 
For a multiplier $T$ from $\A^{(\alpha,u_{\alpha})}$ to $\B^{(\beta,u_{\beta})}$ the following are equivalent:
\begin{enumerate}
    \item\label{enu:characterisation_positive_FS_approx1} $T$ has the positive Fourier--Stieltjes approximation property;  
    \item\label{enu:characterisation_positive_FS_approx2} there is a net $\{T_{i}\}_{i\in I}$  of compactly supported  and positive-definite (not necessarily strict) multipliers from $\A^{(\alpha,u_{\alpha})}$ to $\B^{(\beta,u_{\beta})}$ such that $\sup_{i\in I}\sup_{x\in X}\|(T_i)_{x}\|<\infty$ and $\m_{T_i}(a)$ converges uniformly on compact sets to $\m_{T}(a)$ for every $a\in\contc(\A^{(\alpha,u_{\alpha})})$;
    \item\label{enu:characterisation_positive_FS_approx3} there is a net $\{L_i\}_{i \in I}$ of equivariant representations of $\G$, acting respectively on $\cst$-$\A$-$\B$-correspondence bundles $\EE_i$, and sections $\xi_i\in \contc(\EE_i)$, $i\in I$, such that $\sup_{i\in I}\|\xi_i\| <\infty$, and for every $a\in\contc(\A^{(\alpha,u_{\alpha})})$ the functions 
    \[
        \G\ni g\longmapsto \sum_{t\in \G_{r(g)}}\langle \xi_i(t), a(g)L_{i,g}\big(\xi_i(tg)\big)\rangle_{B_{r(g)}}
    \]
    converge uniformly on compact sets to the function $\G\ni g\longmapsto T_{g}(a(g))\in B_{r(g)}$. 
\end{enumerate}
\end{prop}
\begin{proof}  
Theorem~\ref{thm:FSmultiplier_positive_definite} implies the equivalence \ref{enu:characterisation_positive_FS_approx1}$\iff$
\ref{enu:characterisation_positive_FS_approx2} modulo strictness of the $T_i$s, which can be ignored thanks to Lemma~\ref{lem:multiplier_approx_with strict}. 
The same theorem also shows that both conditions are equivalent to a version of \ref{enu:characterisation_positive_FS_approx3} where sections $\xi_i$ are in $\contb(\M(\EE_i))$ rather than in $\contc(\EE_i)$, but then as in the proof of Lemma~\ref{lem:Fourier_explicite_approximation} we may replace $\xi_i$ by $e_j\xi_i\in \contc(\EE_i)$ for an approximate unit $\{e_j\}_{j \in J}\subseteq \contc(\A)$ in $\contz(\A)$. 
\end{proof}

One of the main consequences of the approximation property for a given multiplier is the following result which shows that the multipliers enjoying this property yield (completely) bounded maps between the appropriate $\cst$-algebras.

\begin{lem}\label{lem:weak_approx_implies_reduced_to_full}
Let $\A$, $\B$ be $\cst$-bundles over $X$, carrying twisted actions $(\alpha,u_\alpha)$ and $(\beta,u_\beta)$ of $\G$. 
Let $T$ and $\{T_{i}\}_{i\in I}$ be multipliers from $\A^{(\alpha,u_{\alpha})}$ to $\B^{(\beta,u_{\beta})}$ such that $\m_{T_i}(a)$ converges uniformly on compact sets to $\m_{T}(a)$  for every $a\in\contc(\A^{(\alpha,u_{\alpha})})$. 
\begin{enumerate}
    \item\label{ite:weak_approx_implies_reduced_to_full1} If $\{ T_{i} \}_{i\in I} \subseteq F[(\alpha,u_{\alpha}), (\beta,u_{\beta})]$ and $M := \sup_{i\in I} \|T_i\|_{F} < \infty$, then $\m_{T}$ extends to a completely bounded map $\m_T^{\rd,\f} : \cst_{\red}(\A^{(\alpha,u_{\alpha})})\to \cst(\B^{(\beta,u_{\beta})})$ such that $\| \m_T^{\rd,\f} \|_{cb} \leq M$ and $\m_{T_i}^{\rd,\f} \to \m_T^{\rd,\f}$ pointwise. 
    
    \item\label{ite:weak_approx_implies_reduced_to_full2} If $\{ T_{i} \}_{i\in I} \subseteq FS[(\alpha,u_{\alpha}), (\beta,u_{\beta})]$ and $M := \sup_{i\in I} \|T_i\|_{FS}< \infty$, then $\m_{T}$ extends to completely bounded maps $\m_T^{\rd} : \cst_{\red}(\A^{(\alpha,u_{\alpha})}) \to \cst_{\red}(\B^{(\beta,u_{\beta})})$ and $\m_T^{f} : \cst(\A^{(\alpha,u_{\alpha})}) \to \cst(\B^{(\beta,u_{\beta})})$, where $\| \m_T^{\rd} \|_{cb}$, $\| \m_T^{\f} \|_{cb} \leq M$ and $\m_{T_i}^{\rd}\to \m_T^{\rd}$, $\m_{T_i}^{\f}\to \m_T^{\f}$ pointwise.
\end{enumerate}
\end{lem}
\begin{proof} 
Note first that the uniform convergence on compact sets in $\contc(\A^{(\alpha,u_{\alpha})})$ implies convergence in the maximal $\cst$-norm. 
Indeed, for  $a\in\contc(\A^{(\alpha,u_{\alpha})})$ we have $a = \sum_{U\in F} a_{U}$ for some $a_U \in \contc(\A^{(\alpha,u_{\alpha})}|_{U})$, $U\in F$, where $F\subseteq \Bis(\G)$ is finite. 
The norm of $\cst(\B^{(\beta,u_{\beta})})$ restricted to $\contc(\B^{(\beta,u_{\beta})}|_{U})$, $U\in \Bis(\G)$, coincides with the supremum norm. 
Therefore $\| \m_{T}(a) - \m_{T_i}(a) \|_{\cst(\B^{(\beta,u_{\beta})})} \leq \sum_{U\in F} \| \m_{T}(a_U) - \m_{T_i}(a_U) \|_{\infty} \to 0$. 
Assume now that we are given a net $\{T_{i}\}_{i\in I}\subseteq  F[(\alpha,u_{\alpha}), (\beta,u_{\beta})]$ and $M:=\sup_{i}\|T_{i\in I}\|_{F}<\infty$. 
Then $\|\m_{T_i}(a)\|_{\cst(\B^{(\beta,u_{\beta})})} \leq M \cdot \|a\|_{\cst_{\red}(\A^{(\alpha,u_{\alpha})})}$, by Theorem~\ref{thm:Fmultipliers_are_reduced_to_full}, and thus $\| \m_{T}(a) \|_{\cst(\B^{(\beta,u_{\beta})})} \leq M \cdot \|a\|_{\cst_{\red}(\A^{(\alpha,u_{\alpha})})}$. 
Hence $\m_T$ yields a bounded operator $\m_T^{\rd,\f}:\cst_{\red}(\A^{(\alpha,u_{\alpha})})\to \cst(\B^{(\beta,u_{\beta})})$ with norm not greater than $M$. 
Applying this reasoning to matrices, we deduce that  $\m_T^{\rd,\f}$ is in fact completely bounded and $\| \m_T^{\rd,\f} \|_{cb}\leq M$. 
This proves \ref{ite:weak_approx_implies_reduced_to_full1}. 
The proof of \ref{ite:weak_approx_implies_reduced_to_full2} is the same (instead of Theorem~\ref{thm:Fmultipliers_are_reduced_to_full} use Theorem~\ref{thm:FSmultiplier_extends_to_reduced_and_full}).
\end{proof}

The next two corollaries are immediate consequences of the above result and appropriate definitions (in the second we also use Lemmas~\ref{lem:multiplier_approx_with strict} and \ref{lem:multiplier_approx_with strict2}).

\begin{cor}\label{cor:weak_approx_implies_reduced_to_full}
Every multiplier with the Fourier approximation property is reduced to full. 
Every multiplier with the Fourier--Stieltjes approximation property is both reduced and full.
\end{cor}

\begin{cor}\label{cor:positive_definite_is_full_and_reduced}
Every bounded positive-definite (not necessarily strict) multiplier $T$ from $\A^{(\alpha,u_{\alpha})}$ to $\B^{(\beta,u_{\beta})}$ is both reduced and full (and we have $\|\m_T^{\rd}\|_{cb}=\|\m_T^{\f}\|_{cb}=\sup_{x\in X}\|T_x\|$).
\end{cor}

%
\subsection*{Fourier and Fourier--Stieltjes approximation properties for actions}

We will now study the definitions considered above in the case of the identity multiplier and view them as approximation properties of a given twisted groupoid action.

\begin{defn}\label{def:weak_approximation_property_for_actions}
We say that an action $(\alpha,u_{\alpha})$ of a groupoid $\G$ on a $\cst$-bundle $\A$ has the \emph{(positive) Fourier} or \emph{Fourier--Stieltjes approximation property} if the identity multiplier of $\A^{(\alpha,u_{\alpha})}$ has this property. 
\end{defn}

\begin{cor}\label{cor:positive_FS_AP_for_actions}
A twisted groupoid action $(\alpha,u_{\alpha})$ has the positive Fourier--Stieltjes approximation property if and only if there is a net $\{T_{i}\}_{i\in I}$ of compactly supported    and positive-definite multipliers of $\A^{(\alpha,u_{\alpha})}$ such that $\sup_{i\in I}\sup_{x\in X}\|(T_i)_{x}\|<\infty$ and $\m_{T_i}(a)\to a$  uniformly on compact sets for every $a\in\contc(\A^{(\alpha,u_{\alpha})})$.
\end{cor}
\begin{proof} 
Apply Proposition~\ref{prop:characterisation_positive_FS_approx}.
\end{proof}


\begin{rem} \label{rem:diagram}  
We will focus on the Fourier--Stieltjes approximation property as it is (at least formally) weaker than the Fourier approximation property. 
Exel's approximation property \cite[Definition~4.4]{Exel:amenability}, recently generalised to Fell bundles over \'etale groupoids in \cite[Definition~3.1]{Kranz} and \cite[Definition~4.10]{BussMartinez}, when formulated for the bundle $\A^{(\alpha,u_{\alpha})}$ associated with a twisted groupoid action, means that the identity multiplier has the positive Fourier AP with the approximating net given by  coefficients of the regular equivariant representation of $\G$, see Examples~\ref{ex:trivial_equivariant_action2}, \ref{ex:regular_rep_revisited}. 
More precisely, we say that the action $(\alpha,u_{\alpha})$ of a groupoid $\G$ on a $\cst$-bundle $\A$ has \emph{Exel's approximation property} if there is a net $\{\xi_{i}\}_{i\in I}\subseteq \contc(r^*\A)$ such that $\sup_{i\in I} \sup_{x\in X} \| \sum_{g\in \G_x}\xi_i(g)^*\xi_i(g) \| <\infty$ and for every $a\in\contc(\A^{(\alpha,u_{\alpha})})=\contc(r^*\A)$, the net of sections (compare to Example~\ref{ex:approximation_property_maps_as_Fouriers})
\[
    \G\ni g\longmapsto \sum_{t\in \G^{r(g)}} \xi_i(h)^* a(g)\alpha_{g}(\xi_i(g^{-1}h)) \in A_{r(g)} = \A^{(\alpha,u_{\alpha})}_g \subseteq \A^{(\alpha,u_{\alpha})}
\]
converges uniformly on compact sets to $a$. 
Following \cite[Definition 3.5]{BussEchterhoffWillett}, \cite[Definition 2.2]{Takeishi}, \cite[Theorem 5.6.18]{BO}, one could say that the twisted action $(\alpha,u_{\alpha})$ on $\A$ is \emph{strongly amenable} if there is a bounded net $\{\varphi_i\}_{i}\subseteq \contc(r^*Z\M(\A))$ of central sections of positive-type (see Remark~\ref{rem:positive_type_conditions}) such that $\sup_{g\in G} \|\varphi_i(g)a(g) - a(g)\|_{\infty}\longrightarrow 0$ for every $a \in \contc(\A^{(\alpha,u_{\alpha})})$. 
By Corollary~\ref{cor:completely_positive_Herz_Schur_multipliers}, strong amenability implies the  positive Fourier--Stieltjes AP.
The diagram 
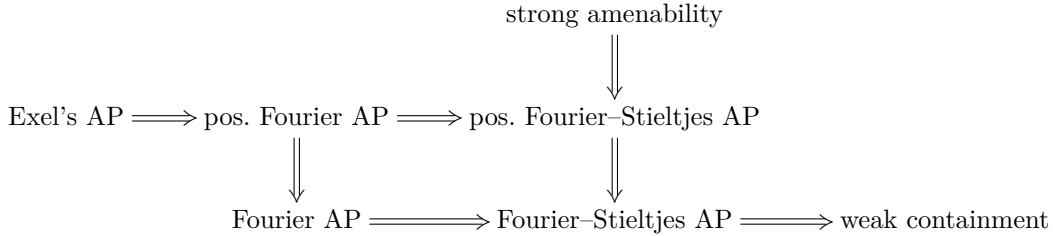
\small
\[
    \xymatrix{
   &  & \text{strong amenability} \ar@{=>}[d]  & \\ 
    \text{Exel's AP} \ar@{=>}[r]  & \text{pos.\ Fourier AP} \ar@{=>}[r] \ar@{=>}[d]  & \text{pos.\ Fourier--Stieltjes AP}\ar@{=>}[d] & \\ 
    &  \text{Fourier AP} \ar@{=>}[r]  & \text{Fourier--Stieltjes AP} \ar@{=>}[r] & \text{weak containment} 
}
\] 
\normalsize
summarises the general relationships between the discussed approximation properties. 
The last implication from the Fourier--Stieltjes AP to the weak containment property will be proved below (Theorem~\ref{thm:approximation_prop_implies_weak_containment}). 
For a brief discussion of the converse we refer to Question~\ref{quest:diagram}.
\end{rem}
 
The next result should be compared to \cite[Theorem 5.8]{BedosContiregular}.

\begin{thm}\label{thm:approximation_prop_implies_weak_containment}
Let $\A$ be a $\cst$-bundle over $X$, carrying a twisted action $(\alpha,u_\alpha)$ of $\G$. 
Assume $(\alpha,u_{\alpha})$ has the Fourier--Stieltjes approximation property. 
Then it has the weak containment property, i.e.\ $\cst(\A^{(\alpha,u_{\alpha})}) = \cst_{\red}(\A^{(\alpha,u_{\alpha})})$, and $\cst(\A^{(\alpha,u_{\alpha})})$ is nuclear if and only if $A = \contz(\A)$ is nuclear.
\end{thm}
\begin{proof}
For the first part we cannot use Corollary~\ref{cor:weak_approx_implies_reduced_to_full}, as the Fourier--Stieltjes AP is (at least formally) weaker than the Fourier AP. 
We adapt the proof of B\'{e}dos--Conti~\cite[Theorem 4.6]{BedosConti2} for crossed products, which was in turn inspired by \cite[Theorem 2.6]{BO} for groups.

Recall from Lemma~\ref{lem:precompacts_embed_into_crossed_products} that if $V\subseteq \G$ is a precompact open set then the subspace $\contz(\A^{(\alpha,u_{\alpha})}|_{V})$ of $\contc(\A^{(\alpha,u_{\alpha})})$ is closed in the reduced and hence also in the universal (maximal) norm. 
Now let $\{T_{i}\}_{i\in I} \subseteq FS[(\alpha,u_{\alpha})]$  be a net witnessing the approximation property. 
For every $i \in I$ the support of each $T_i$ is contained in a precompact open set $V_i$, and so $\m_{T} : \contc(\A^{(\alpha,u_{\alpha})})\to \contz(\A^{(\alpha,u_{\alpha})}|_{V_i})$. 
By the above claim the extended maps $\m_{T_i}^{\rd} : \cst_{\red}(\A^{(\alpha,u_{\alpha})}) \to \contz( \A^{(\alpha,u_{\alpha})}|_{V_i} )$ and $\m_{T_i}^{\f} : \cst(\A^{(\alpha,u_{\alpha})}) \to \contz( \A^{(\alpha,u_{\alpha})}|_{V_i})$ (that exist by Theorem~\ref{thm:FSmultiplier_extends_to_reduced_and_full}) take values in $\contc(\A^{(\alpha,u_{\alpha})})$. 
The regular representation $\Lambda : \cst(\A^{(\alpha,u_{\alpha})}) \to \cst_{\red}(\A^{(\alpha,u_{\alpha})})$ is the identity on $\contc(\A^{(\alpha,u_{\alpha})})$, and therefore for every $a \in \cst(\A^{(\alpha,u_{\alpha})})$ 
\[
    \m_{T_i}^{\rd} (\Lambda(a)) = \Lambda ( \m_{T_i}^{\f}(a)) = \m_{T_i}^{\f}(a) .
\]
Hence if $\Lambda(a) = 0$, then $\m_{T_i}^{\f}(a) = 0$ for all $i\in I$, which implies $a=0$ because $\m_{T_i}^{\f}(a)\to a$, see Lemma~\ref{lem:weak_approx_implies_reduced_to_full}\ref{ite:weak_approx_implies_reduced_to_full2}. 
Thus $\cst(\A^{(\alpha,u_{\alpha})}) = \cst_{\red}(\A^{(\alpha,u_{\alpha})})$.

If $\cst(\A^{(\alpha,u_{\alpha})})$ is nuclear, then so is $A = \contz(\A)$ because we have a conditional expectation $E:\cst(\A^{(\alpha,u_{\alpha})})\to A$. 
To prove the converse implication we adapt Takeishi's proof of \cite[Theorem 4.1]{Takeishi} for amenable groupoids, which in turn was based on \cite[Theorem 5.6.18]{BO}. 
Assume that $A$ is nuclear. 
Let $D$ be any $\cst$-algebra.  
We need to show that the canonical quotient map $Q : \cst(\A^{(\alpha,u_{\alpha})}) \otimes_{\max} D \rightarrow \cst(\A^{(\alpha,u_{\alpha})})\otimes_{\min} D$ is injective. 
The first part of the proof of \cite[Theorem 4.1]{Takeishi} shows (using nuclearity of $A$) that for any precompact open set $V\subseteq \G$, the quotient map $Q$ is injective on   $\contz(\A^{(\alpha,u_{\alpha})}|_{V})\otimes_{\max} D\subseteq \cst(\A^{(\alpha,u_{\alpha})}) \otimes_{\max} D$.
For the approximating net $\{ T_{i} \}_{i\in I} \subseteq  FS[(\alpha,u_{\alpha})]$, as above, the range of each $\m_{T_i}^{\f}$ is contained in the closed subspace $\contz(\A^{(\alpha,u_{\alpha})}|_{V_i})$, and therefore ${\m}_{T_i}^{\f} \otimes_{\max} \mathrm{id} \left( {\cst}(\A^{(\alpha,u_{\alpha})}) \otimes_{\max}D \right)$ is contained in $\contz(\A^{(\alpha,u_{\alpha})}|_{V_i})\otimes_{\max} D$.  
Let $a \in {\cst}(\A^{(\alpha,u_{\alpha})})\otimes_{\max} D$ with $Q(a)=0$. 
Using that the following  diagram commutes: 
\[
\xymatrix{
    {\cst}(\A^{(\alpha,u_{\alpha})}) \otimes_{\max}  D\ar[d]_{Q} \ar[rr]^{{\m}_{T_i}^{\f} \otimes_{\max} \mathrm{id}} & & {\cst}(\A^{(\alpha,u_{\alpha})}) \otimes_{\max} D \ar[d]^{Q} & 
        \\
    {\cst}(\A^{(\alpha,u_{\alpha})}) \otimes_{\min} D \ar[rr]^{{\m}_{T_i}^{\f} \otimes_{\min} \mathrm{id}} & & {\cst}(\A^{(\alpha,u_{\alpha})})\otimes_{\min} D &,}
\]
we get $Q\circ ({\m}_{T_i}^{\f} \otimes_{\max} \mathrm{id})(a) = ({\m}_{T_i}^{\f} \otimes_{\min} \mathrm{id})\circ Q(a) = 0$. 
Since $Q$ is injective on the range of ${\m}_{T_i}^{\f}\otimes_{\max} \mathrm{id}$ this means that ${\m}_{T_i}^{\f} \otimes_{\max} \mathrm{id}(a)=0$ for all $i\in I$. 
Hence $a = \lim_i{\m}_{T_i}^{\f} \otimes_{\max} \mathrm{id}(a)=0$. 
This proves that $Q$ is injective. 
\end{proof}

We will now specify the above results to the context of (twisted) groupoid $\cst$-algebras.


\begin{cor}\label{cor:amenable-groupoids-vs-FS-AP}
Let $u$ be a continuous $2$-cocycle on a locally compact Hausdorff \'etale groupoid $\G$. 
Then the following assertions are equivalent:
\begin{enumerate}
    \item\label{enu:amenable-groupoids-vs-FS-AP1} the trivial twisted action $(\id,u)$ on the trivial one-dimensional \cstar{}bundle $\C\times X$ has the Fourier--Stieltjes approximation property;
    \item\label{enu:amenable-groupoids-vs-FS-AP2} $(\id,u)$ or, equivalently, the trivial (untwisted) action $\id$ on $\C\times X$ has the positive Fourier approximation property.
    \item\label{enu:amenable-groupoids-vs-FS-AP3} $C^*(\G,u)$ or, equivalently, $C^*_r(\G,u)$ is nuclear;
    \item\label{enu:amenable-groupoids-vs-FS-AP4} $C^*(\G)$ or, equivalently, $C^*_r(\G)$ is nuclear;
    \item\label{enu:amenable-groupoids-vs-FS-AP5}  $\G$ is amenable.
\end{enumerate}
\end{cor}
\begin{proof}
If $(\id,u)$ has the Fourier--Stieltjes approximation property, then by the previous theorem $C^*(\G,u)=C^*_r(\G,u)$ is nuclear, and this is known to be equivalent to amenability of $\G$, see \cite{Takeishi}. 
But amenability of $\G$ is, in turn, equivalent to the existence of a net $\{\varphi_i\}_{i\in I}$ of positive-type functions in $\contc(\G)$ that converge pointwise to $1$ uniformly on compacts, see \cite[Proposition~2.2.13]{AD-Renault-amenable} or \cite[Theorem~5.6.18]{BO}. 
This is another way to say that $(\id,u)$ or just $\id$ has the positive Fourier--Stieltjes AP as $FS(\G,u)=FS(\G)$, see Example~\ref{ex:Fourier-Stietjes-algebra-groupoids}.
\end{proof}

\begin{rem}
Most of the statements in the above corollary were already known to be equivalent before, as already indicated in the references throughout the proof, except for the implication from \ref{enu:amenable-groupoids-vs-FS-AP1} to the other statements \ref{enu:amenable-groupoids-vs-FS-AP2}--\ref{enu:amenable-groupoids-vs-FS-AP5}. 
\end{rem}

\section{Some further applications: Haagerup trick and decomposable norms} 
\label{Sec:Applications}

The Haagerup trick is a way of producing multipliers of a dynamical system from maps on the associated reduced algebra, applied first for groups by Uffe Haagerup \cite[Lemma 2.5]{Haagerup}, later used in many different contexts and in particular generalised to group actions in \cite[Proposition 4.12]{BedosConti2} and \cite[Proposition 3.4]{mstt}. 
A version of this trick for twisted groupoid actions and bimodule maps was a main tool in \cite{BartoszKangAdam}. 
In general, for \'etale groupoids and their actions it is not clear what the Haagerup trick should be. 
However when either the groupoid $\G$ is discrete or when we consider only bimodule maps the situation becomes much more transparent. 
We discuss here these two cases and some of the applications, notably to identifying the Fourier--Stieltjes norm of the associated multiplier. 
For the last purpose we will also use the fact that maps coming from a Fourier--Stieltjes multiplier $T = T_{\EE,L,\xi,\zeta}$  are \emph{decomposable}, that is they are finite linear combinations of completely positive maps, which follows from the usual polarisation formula 
\[
    T =\frac{1}{4}\sum_{k=0}^3 i^k T_{\EE,L,\xi+i^k\zeta ,\xi+i^k\zeta}
\]
and the  multipliers of the form $T_{\EE,L,\xi,\xi}$ induce completely positive maps, cf.\ Theorem~\ref{thm:FSmultiplier_positive_definite}. 
Finally we will use the well-known fact that for any  Fell bundle $\A$ over $\G$, 
the inclusion $\contc(\A)\subseteq \contz(\A)$ extends to a contractive injective linear embedding \(\cst_{\red}(\A) \to \contz(\A)\), cf.\ \cite[Proposition 7.10]{BartoszRalf2}.

\subsection{Twisted actions of discrete groupoids} 

In this subsection we will assume that the groupoid $\G$ is discrete. 

\begin{lem}[Haagerup Trick I]\label{lem:Haagerup trick discrete}
Let $\A$ and $\B$ be Fell bundles over a discrete groupoid $\G$. 
For any bounded linear map $\Phi:\cst_{\red}(\A)\to \cst_{\red}(\B)$  
the formula  
\[
    T_{g}^{\Phi}(a)=\Phi(a\delta_g)(g), \qquad a\in  A_{g},\,\, g \in \G,
\]
where $a\delta_g$ is  a section in $\contc(\A|_{\{g\}})$ with $a\delta_g(g)=a$, defines a multiplier $T^{\Phi}=\{T_g^{\Phi}\}_{g\in \G}$ from $\A$ to $\B$ with $\sup_{g\in \G}\|T_{g}^{\Phi} \|\leq \|\Phi\|$. 
If $\Phi$ is completely positive, then $T^{\Phi}=\{T_g^{\Phi}\}_{g\in \G}$ is  positive-definite.  
\end{lem}
\begin{proof}
It is clear that the formula above defines a multiplier. 
Moreover, $\| T_{g}^{\Phi}(a) \| = \| \Phi(a\delta_g)(g) \| \leq \| \Phi(a\delta_g) \| \leq \| \Phi \| \cdot \| a \|$ for any $a \in A_{g}$, $g \in \G$. 
Hence $\sup_{g\in \G} \| T_{g}^{\Phi} \| \leq \| \Phi \|$. 
Assume that $\Phi$ is completely positive. 
Fix $x \in X$, $n \in \N$, elements $g_1 , \ldots , g_n \in \G_x$, and collections $\{ a_{g_i} \in A_{g_i} : 1 \leq i \leq n \}$, $\{b_{g_i} \in B_{g_i}: 1 \leq i \leq n \}$. 
The matrix $( a_{g_i} a_{g_j}^\star \delta_{g_i g_j^{-1}})_{i,j = 1}^n = (a_{g_i} \delta_{g_i} * (a_{g_j} \delta_{g_j})^*)_{i,j = 1}^n$ is positive in $\Mat_n(\cst_{\red}(\A))$ and so $b := \Phi^{(n)} ((a_{g_i} a_{g_j}^\star \delta_{g_i g_j^{-1}})_{i,j = 1}^n)$ is positive in $\L(\ell^2(\B)_{x}^{\oplus n})$. 
Let $\xi = \bigoplus_{i = 1}^n \xi_{g_i} \in \ell^2(\B)_{x}^{\oplus n}$ be given by $\xi_g(h)=[h=g] b_g$, for all $h,g \in \G$. 
Then 
\[ 
    0 \leq \langle \xi, b \xi \rangle_{B_x} = \sum_{i,j = 1}^n \langle \xi_{g_i} , \Phi ( a_{g_i} a_{g_j}^\star \delta_{g_i g_j^{-1}} ) \xi_{g_j} \rangle_{B_x} = \sum_{i,j = 1}^n b_{g_i}^{\star} \left( T_{g_i g_j^{-1}}^{\Phi} (a_{g_i}  a_{g_j}^{\star} ) \right) b_{g_j} . 
\]
Hence $T^{\Phi}$ is positive definite. 
\end{proof}

\begin{prop}\label{prop:positive_FS_AP_for_actions2} 
Let $\A$, $\B$  be $\cst$-bundles over $X$, carrying twisted actions $(\alpha,u_\alpha)$ and $(\beta,u_\beta)$ of a discrete groupoid $\G$.
A multiplier $T$ from $\A^{(\alpha,u_{\alpha})}$ to $\B^{(\beta,u_{\beta})}$ has the positive Fourier--Stieltjes approximation property if and only if there is a bounded net $\{\Phi_{i}\}_{i\in I}$ of completely positive maps $\Phi_i:\cst_{\red}(\A^{(\alpha,u_\alpha)})\to C(\B^{(\beta,u_\beta)}|_{V_i})\subseteq \cst_{\red}(\B^{(\beta,u_\beta)})$ with ranges supported in finite sets $V_i\subseteq \G$, $i\in I$, and such that $\Phi_i(a)\to m_T(a)$ for every $a \in \contc(\A^{(\alpha,u_\alpha)})$. 
\end{prop}
\begin{proof}  
If $\{ T_{i} \}_{i\in I} \subseteq FS[(\alpha,u_{\alpha})]$ is a net witnessing the positive Fourier--Stieltjes approximation property for $T$ then the associated maps $\{\m_{T_i}^{\rd}\}_{i\in I}$, given by Theorem~\ref{thm:FSmultiplier_extends_to_reduced_and_full}, have the desired properties by Theorem~\ref{thm:FSmultiplier_positive_definite} and Corollary~\ref{cor:compactly_supported_vs_range}. 
Conversely, assume that $\{\Phi_{i}\}_{i\in I}$ is a bounded net of completely positive maps $\Phi_i : \cst_{\red}(\A^{(\alpha,u_\alpha)}) \to C(\B^{(\beta,u_\beta)}|_{V_i})$ where $V_i$ is finite. 
Then the associated net $\{T^{\Phi_{i}}\}_{i\in I}$ given by Lemma~\ref{lem:Haagerup trick discrete} consists of positive-definite multipliers such that $\sup_{i\in I}\sup_{x\in X}\|(T_i)_{x}\| \leq \sup_{i\in I} \|\Phi_i\|<\infty$. 
Moreover, $\Phi_i(a\delta_g)\to m_T(a\delta_g) = T_{g}(a)\delta_g$ for $a\in A_{r(g)}$, $g\in \G$, implies 
\[
    \m_{T^{\Phi_i}} (a\delta_g) = T^{\Phi_i}(a) \delta_g = \Phi_i(a\delta_g)(g)\delta_g \to \m_T(a\delta_g)(g) \delta_g = T_{g}(a) \delta_g = m_T(a\delta_g) . 
\]
This implies that $\m_{T^{\Phi_i}}(a)\to m_T(a)$ for every $a\in \contc(\A^{(\alpha,u_\alpha)})$.
\end{proof}
 
\begin{cor}\label{cor:characterisation_FS_AP_for_discrete}
A twisted groupoid action $(\alpha,u_{\alpha})$ of a discrete groupoid $\G$ on a bundle $\A$ has the positive Fourier--Stieltjes approximation property if and only if there is a bounded net $\{\Phi_{i}\}_{i\in I}$ of completely positive maps on $\cst_{\red}(\A^{(\alpha,u_{\alpha})})$ that converge pointwise to the identity and that have ranges supported on finite subsets of $\G$. 
\end{cor}
\begin{proof} 
Apply Proposition~\ref{prop:positive_FS_AP_for_actions2} to the identity multiplier.
\end{proof}

The above results make it more evident that the positive Fourier--Stieltjes AP is a `nuclearity condition' for multipliers.
In particular, if in the characterisation in Corollary~\ref{cor:characterisation_FS_AP_for_discrete} one replaces `ranges are supported on finite subsets of $\G$' with `ranges are finite-dimensional' one gets a characterisation of nuclearity of $\cst_{\red}(\A^{(\alpha,u_{\alpha})})$, see \cite[Theorem 3.1]{ChoiEffros}. 
We show that in this context `finite rank' implies `supported on finite sets' using the following lemma, which is probably well known to experts, cf.\ the proof of 
\cite[Theorem 4.3]{mstt}.

\begin{lem}\label{lem:McKee_Skalski_Todorow_Turowska}
Let $A$ and $B$ be $C^*$-algebras. 
Assume $A$ is finite dimensional and let $B_0$ be a dense $*$-subalgebra of $B$. 
For every completely positive map $\Phi : A \to B$ there is a net $\{\Phi_i\}_{i\in I}$ of completely positive maps $\Phi_i : A \to B_0 \subseteq B$ such that $\| \Phi_i \| = \| \Phi \|$ and $\Phi_{i}(a) \to \Phi(a)$ for $a\in A$ (equivalently $\Phi_{i} \to \Phi$ in norm).
\end{lem}
\begin{proof} 
Since $A$ is finite dimensional we have $A = \bigoplus_{k=1}^N \Mat_{n_k}(\C)$. 
Then each map $\Phi_k := \Phi|_{\Mat_{n_k}(\C)}$ (with $k = 1 , \ldots , N$) is contractive and completely positive, and $\Phi(\oplus_{k=1}^N a_k) = \sum_{k=1}^n \Phi_k(a_k)$ for $\oplus_{k=1}^n a_k \in A$. 
Thus it suffices to prove the assertion for each $\Phi_k$. 
Hence we may simply assume that $A=\Mat_{n}(\C)$ for some $n$. 
We may also assume that $\|\Phi\|=1$.

Let $e_{p,q}$, $p , q = 1, \ldots , n$, be the canonical matrix units in $\Mat_{n}(\C)$. 
Since the matrix $(e_{p,q})^n_{p,q=1}\in \Mat_{n}(\Mat_n(\C))$ is positive and $\Phi$ is completely positive, we get that $(\Phi(e_{p,q}))^n_{p,q=1}\in \Mat_{n}(B)^+$. 
As $\Mat_{n}(B_0)^+$ is dense in $\Mat_{n}(B)^+$, given $\varepsilon_i>0$, $i\in I$, with $\varepsilon_i\to 0$, there are matrices $(b^i)_{p,q=1}^n\in \Mat_{n}(B_0)^+$ such that $\| (\Phi(e_{p,q}))^n_{p,q = 1} - (b^i)_{p,q=1}^n \| < \varepsilon_i/n$.
Define $\Phi_i : A \to B$ by 
\[
    \Phi_i \big( (\lambda_{p,q})^n_{p,q=1} \big) = \sum^n_{p,q=1} \lambda_{p,q} b^i_{p,q} .
\] 
By Choi's Theorem \cite[Theorem 3.14]{Paulsen}, $\Phi_i$ is completely positive.  
Since $\| \Phi(e_{p,q}) - \Phi_i(e_{p,q}) \|= \| \Phi(e_{p,q}) - b^i_{p,q}\| < \varepsilon_i/n\stackrel{i}\to 0$ and $\{e_{p,g}\}_{p,q=1}^n$ is a basis in $\Mat_{n}(\C)$, we conclude that $\Psi_i$ converges to $\Psi$ pointwise (and hence also in norm as $\dim(A)<\infty$).
Moreover,
\[
    | \|\Phi_i\|-1| = | \|\Phi_i(1)\| - \|\Phi(1)\| | \leq \| \Phi_i(1) -\Phi(1) \| \leq \sum_{p=1}^n \| \Phi(e_{p,p}) - b^i_{p,p}\|\leq \varepsilon_i/n\stackrel{i}\to 0.
\]
Thus $\|\Phi_i\|\to 1$ and therefore replacing $\Phi_i$ by $\frac{\Phi_i}{\|\Phi_i\|}$ we get the desired net.
\end{proof}




\begin{cor}\label{cor:discrete_Haagerup_approximation} 
Let $\A$, $\B$ be $\cst$-bundles over $X$, carrying twisted actions $(\alpha,u_\alpha)$ and $(\beta,u_\beta)$ of a discrete groupoid $\G$. 
If a map $\Phi : \cst_{\red}(\A^{(\alpha,u_\alpha)}) \to \cst_{\red}(\B^{(\beta,u_\beta)})$ is a pointwise limit of a bounded net $\{\Phi_{i}\}_{i\in I}$ of finite rank completely positive maps $\Phi_i : \cst_{\red}(\A^{(\alpha,u_\alpha)})\to \cst_{\red}(\B^{(\beta,u_\beta)})$, $i\in I$, then the multiplier $T^{\Phi}$ from Lemma~\ref{lem:Haagerup trick discrete} has the positive Fourier--Stieltjes approximation property. 
\end{cor}
\begin{proof}  
Since $\Phi_{i}(\cst_{\red}(\A^{(\alpha,u_\alpha)})$ is finite dimensional, by Lemma~\ref{lem:McKee_Skalski_Todorow_Turowska} there is a net $\{\phi^i_j\}_{j\in J}$ of norm one completely positive maps $\phi_{j}^i:\Phi_{i}(\cst_{\red}(\A^{(\alpha,u_\alpha)}) \to \contc(\B^{(\beta,u_\beta)})\subseteq \cst_{\red}(\B^{(\beta,u_\beta)})$ that converge pointwise to the identity on $\Phi_{i}(\cst_{\red}(\A^{(\alpha,u_\alpha)})$. 
Thus replacing the net $\{\Phi_{i}\}_{i\in I}$ with $\{\phi_{j}^i\circ \Phi_{i}\}_{(i,j)\in I\times J}$ we may assume that each map $\Phi_{i}$ takes values in the dense $*$-subalgebra $\contc(\B^{(\beta,u_\beta)})$ of $\cst_{\red}(\B^{(\beta,u_\beta)})$. 
Then the finite rank for $\Phi_{i}$ implies that the range of $\Phi_{i}$ is supported on a finite subset of $\G$. 
Hence the assertion follows by Proposition~\ref{prop:positive_FS_AP_for_actions2}.
\end{proof}

\begin{rem}\label{rem:about_full_case}
The above results (i.e.\ Lemma~\ref{lem:Haagerup trick discrete}, Proposition~\ref{prop:positive_FS_AP_for_actions2}, and Corollaries~\ref{cor:characterisation_FS_AP_for_discrete} and \ref{cor:discrete_Haagerup_approximation}, remain valid (with essentially the same proofs) when the reduced crossed products are replaced with full crossed products. 
If, in addition, the crossed product for $(\alpha,u_{\alpha})$ is separable, then nets $\{\Phi_{i}\}_{i\in I}$ may be replaced by sequences, and then their boundedness follow from their pointwise convergence.  
\end{rem}

\begin{thm}\label{thm:characterisation of nuclearity}
Let $(\alpha,u_{\alpha})$ be a twisted groupoid action of a discrete groupoid $\G$ on a bundle $\A$.
Then the following conditions are equivalent:
\begin{enumerate}
    \item\label{enu:characterisation of nuclearity1} $(\alpha,u_{\alpha})$ has the positive Fourier--Stieltjes approximation property and $\contz(\A)$ is nuclear;
    \item\label{enu:characterisation of nuclearity1.5} $(\alpha,u_{\alpha})$ has the Fourier--Stieltjes approximation property and $\contz(\A)$ is nuclear;
    \item\label{enu:characterisation of nuclearity2} the reduced crossed product $\cst_{\red}(\A^{(\alpha,u_\alpha)})$ is nuclear;
    \item\label{enu:characterisation of nuclearity3} the full crossed product $\cst(\A^{(\alpha,u_\alpha)})$ is nuclear;
    \item\label{enu:characterisation of nuclearity4} there is a net of finitely supported positive-definite multipliers $\{T_i\}_{i\in I}$ such that for each $i\in I$, $\sup_{x\in X}\|(T_i)_x\|\leq 1$, and for every for every $g \in \G$, $(T_{i})_g$ is of finite rank and $(T_i)_g(a)\to a$ for any $a\in A_{r(g)}$.
\end{enumerate}
If the above equivalent conditions are satisfied, then $\cst(\A^{(\alpha,u_\alpha)}) = \cst_{\red}(\A^{(\alpha,u_\alpha)})$. 
\end{thm}
\begin{proof} 
Clearly, \ref{enu:characterisation of nuclearity1} implies \ref{enu:characterisation of nuclearity1.5}, and \ref{enu:characterisation of nuclearity1.5} implies \ref{enu:characterisation of nuclearity2} and \ref{enu:characterisation of nuclearity3}, as well as $\cst(\A^{(\alpha,u_\alpha)})=\cst_{\red}(\A^{(\alpha,u_\alpha)})$, by Theorem~\ref{thm:approximation_prop_implies_weak_containment}. 
Suppose now that $\cst_{\red}(\A^{(\alpha,u_\alpha)})$ is nuclear. 
Equivalently, by \cite[Theorem 3.1]{ChoiEffros}, the identity on $\cst_{\red}(\A^{(\alpha,u_\alpha)})$ is the pointwise limit of finite-rank completely positive contractions. 
Modifying this net as in the proof of Corollary~\ref{cor:discrete_Haagerup_approximation} one gets a net as in \ref{enu:characterisation of nuclearity4}. 
The same argument works for the full crossed product.  
Hence either \ref{enu:characterisation of nuclearity2} or \ref{enu:characterisation of nuclearity3} implies \ref{enu:characterisation of nuclearity4}. 
Now if $\{T_i\}_{i\in I}$ is a net as in \ref{enu:characterisation of nuclearity4}, it witnesses the positive Fourier--Stieltjes approximation property of $(\alpha,u_{\alpha})$ and the maps $\{m_{T_i}|_{\contz(\A)}$ witness nuclearity of $\contz(\A)$. 
Hence \ref{enu:characterisation of nuclearity4} implies \ref{enu:characterisation of nuclearity1}. 
\end{proof}

\begin{rem}\label{rem:nuclear_vs_AP's} 
The problem of characterising nuclearity of the crossed product for actions on nuclear $C^*$-algebras in terms of an amenability condition has received a lot of attention, usually using von~Neumann algebra techniques, see \cite[Th\'eor\`eme 4.5]{Clare1}, \cite[Corollary 4.11]{BussEchterhoffWillett}, \cite[Theorem A]{Kranz}. 
In particular, assuming that $\contz(\A)$ is nuclear, the aforementioned results combined with Theorem~\ref{thm:characterisation of nuclearity} give that: 
\begin{enumerate}
    \item If $\G=G$ is a discrete group and the twist is trivial, then the Fourier--Stieltjes AP is equivalent to amenability of the action.
    \item If $\G=G$ is a discrete group, then the Fourier--Stieltjes AP is equivalent to Exel's AP for the twisted action.
    \item If $\G$ is countable, $\contz(\A)$ is separable and the twist is trivial, then the Fourier--Stieltjes AP is equivalent to measurewise amenability of the action.
\end{enumerate}
The characterisation of nuclearity of $\cst_{\red}(\A^{(\alpha,u_\alpha)})$ in terms of condition \ref{enu:characterisation of nuclearity4} in Theorem~\ref{thm:characterisation of nuclearity} is a generalisation of \cite[Theorem 4.3]{mstt}, stated for discrete group actions on unital $C^*$-algebras. 
\end{rem}

\subsection{Bimodule maps, Herz--Schur multipliers and strong amenability}

In this section we consider again a general \'etale locally compact Haudorff groupoid $\G$. 
We will consider however only bimodule maps on a crossed product.

\begin{prop}[Haagerup Trick II]\label{prop:Haagerup_trick}
Let $(\alpha,u)$ be a twisted action of an \'etale groupoid $\G$ on a $\cst$-bundle $\A$. 
To every bounded linear $\contz(\A)$-bimodule map $\Phi : \cst_{\red}(\A^{(\alpha,u)}) \to \cst_{\red}(\A^{(\alpha,u)})$ there is an associated bounded strictly continuous central section $\varphi_{\Phi}\in \contb(r^*Z\M(\A))$ which is uniquely determined by the recipe
\begin{equation}\label{eq:varphi_defn}
    \varphi_{\Phi}(g) a(g) = \Phi(a)(g) \qquad \text{ for any $a\in \contc(\A^{(\alpha,u)})$ supported on a bisection}.
\end{equation}
In particular, $\sup_{g\in \G}\|\varphi_{\Phi}(g) \|\leq \|\Phi\|$. 
If $\Phi$ is completely positive then the Herz--Schur multiplier $\varphi_{\Phi} : \G \to r^*Z\M(\A)$ is of positive type.
\end{prop}
\begin{proof}
Fix $g\in \G$ and $U\in \Bis(\G)$ containing $g$. 
Let $a_0 \in A_{r(g)}$ and take any $a \in \contc(\A|_{r(U)})$ and $b\in \contc(\A|_{s(U)})$ such that $a_0 = a(r(g)) = \alpha_g(b(s(g))$. 
Choose an approximate unit $\{e_i^{U}\}_{i \in I} \subseteq \contc(\A|_{r(U)})$ in  $\contz(\A|_{r(U)})$.
We will use the notation from Lemma~\ref{lem:inverse_semigroup_twisted_action}. 
Thus let $a\delta_U\in \contc(\A^{(\alpha,u)})$ be given by $(a\delta_U)(g) =  a(r(g))$, $g\in U$. 
Clearly $a *(e_i^{U}\delta_{U}) \stackrel{i}{\to} a\delta_{U}$ and  $(e_i^{U}\delta_{U})*\alpha_{U}^{-1}(a) \stackrel{i}{\to} a\delta_{U}$ uniformly.   
Using that both $\Phi$ and the embedding $\cst_{\red}(\A^{(\alpha,u)}) \hookrightarrow \contz(\A^{(\alpha,u)})$ (with convolution multiplication) are $\contz(\A^{(\alpha,u)})$-bimodule maps we get
\[
    \lim_{i} a_0\Phi(e_i^{U}\delta_{U})(g) = \lim_{i} \big( a *\Phi(e_i^{U}\delta_{U}) \big)(g) = \lim_{i} \Phi(a *(e_i^{U}\delta_{U})(g) = \Phi(a\delta_{U})(g) , 
\]
and similarly 
\[
    \lim_{i} \Phi(e_i^{U}\delta_{U})(g) a_0 = \lim_{i} \big(\Phi(e_i^{U}\delta_{U})*  \alpha_{U}^{-1}(a) \big) (g) = \lim_{i} \Phi \big( e_i^{U}\delta_{U}*  \alpha_{U}^{-1}(a) \big) (g) = \Phi ( a\delta_{U} ) (g) .
\]
Hence 
the following strict limit in $\M(A_{r(g)})$
\[
    \varphi_{\Phi}(g) := \text{s-}\lim_{i} \Phi( e_i^{U}\delta_{U} ) (g)
\]
exists, is central, and satisfies \eqref{eq:varphi_defn} (which clearly determines $\varphi_{\Phi}(g)$ uniquely). 
We have $\| \varphi_{\Phi}(g) \| = \| \text{s-}\lim_{i} \Phi( e_i^{U}\delta_{U})(g) \| \leq \sup_{i} \| \Phi(e_i^{U}\delta_{U}) \| \leq \|\Phi\|$.

Assume now that $\Phi$ is completely positive. 
Fix $x \in X$, $n \in \N$ and elements $g_i \in \G_x$ and $a_{i} \in A_{x}$ for $1 \leq i \leq n$. 
Find sections $\{ f_i \}_{i = 1}^n \in \contc(\A)$ supported on bisections and such that $f_{i}(g_i) = \alpha_{g_i}(a_{i})$, $1 \leq i \leq n$. 
Then $f_i * f_j^*$ is supported on a bisection and 
\[
    (f_i * f_j^*)(g_i {g_j}^{-1}) = \alpha_{g_i}(a_{i} a_{j}^*)u(g_i g_j^{-1},g_j)^*
\]
for every $1 \leq i,j \leq n$, cf.\ the calculations in the proof of Lemma~\ref{lem:positive_action}. 
Hence by \eqref{eq:varphi_defn} we get
\[
    \Phi( f_{i} * f_j^*) (g_{i}g_{j}^{-1}) u(g_{i}g_{j}^{-1},g_{j}) = \varphi_{\Phi}(g_{i}g_{j}^{-1}) \alpha_{g_{i}}(a_{g_i} a_{g_j}^*) .  
\]
Since the matrix $(f_i * f_j^*)_{i,j = 1}^n\in \Mat_{n}(\cst_{\red}(\A))$ is positive, complete positivity of $\Phi$ implies that 
\[
	\sum_{i,j=1}^n \alpha^{-1}_{g_i} \Big( \Phi( f_{i} * f_j^*)(g_{i}g_{j}^{-1})u(g_{i}g_{j}^{-1},g_{j}) \Big) \geq 0
\]
(to see this replace $m_T$ in the proof of Lemma~\ref{lem:completely_positive_implies_positive_definite} by $\Phi$, see also the proof of Lemma~\ref{lem:positive_action}). 
Combining the last two displayed formulas we obtain
\[
	\sum_{i,j=1}^n \alpha^{-1}_{g_i} \Big( \varphi_{\Phi}(g_{i} g_{j}^{-1} ) \alpha_{g_{i}} ( a_{g_i} a_{g_j}^* ) ) \Big) \geq 0,
\]
which means that the Herz--Schur multiplier $\varphi_{\Phi}$ is of positive type, see Lemma~\ref{lem:positive_action}.
\end{proof}

\begin{rem}
A similar result for actions of twisted groupoids was proved in \cite[Proposition 4.1]{BartoszKangAdam} under the assumption that $\A$ has a continuous unital section. 
Note that the proposition mentioned above should have been stated for bimodule maps (every $*$-preserving, so in particular every completely positive, right $A$-module map is automatically an $A$-bimodule map), and then the proof of Proposition~\ref{prop:Haagerup_trick} can be adapted to get a `non-unital version' of \cite[Proposition 4.1]{BartoszKangAdam}. 
For group actions on unital algebras a version of the above Haagerup trick appears also in the proof of \cite[Theorem 3.6]{dong_ruan}.
\end{rem}

\begin{prop}\label{prop:positive_FS_AP_for_actions3} 
Let $(\alpha,u)$ be  a twisted action of an \'etale groupoid $\G$ on a \cst-bundle $\A$. 
The following conditions are equivalent:
\begin{enumerate}
    \item\label{enu:positive_FS_AP_for_actions21} the action $(\alpha,u)$ is \emph{strongly amenable}, i.e.\ there is a bounded net  of positive-type sections $\{\varphi_i\}_{i\in I}\subseteq \contc(r^*Z\M(\A))$ such that $\sup_{g\in G} \|\varphi_i(g)a(g) - a(g)\|_{\infty}\longrightarrow 0$ for $a \in \contc(\A^{(\alpha,u)})$; 
    \item\label{enu:positive_FS_AP_for_actions22} the identity on $\cst_{\red}(\A^{(\alpha,u)})$ is the pointwise limit of a bounded net $\{\Phi_{i}\}_{i\in I}$ of completely positive $\contz(\A)$-bimodule maps $\Phi_i:\cst_{\red}(\A^{(\alpha,u)})\to C(\A^{(\alpha,u)}|_{V_i})\subseteq \cst_{\red}(\A^{(\alpha,u)})$ with ranges supported on precompact sets $V_i\subseteq \G$, $i\in I$. 
\end{enumerate}
\end{prop}
\begin{proof} 
If $\{\varphi_i\}_{i\in I}$ are as in \ref{enu:positive_FS_AP_for_actions21} then the maps $\{\m_{\varphi_i}^{\rd}\}_{i\in I}$ given by Corollary~\ref{cor:completely_positive_Herz_Schur_multipliers} satisfy the properties described in \ref{enu:positive_FS_AP_for_actions22}. 
Conversely, having a net $\{\Phi_{i}\}_{i\in I}$ as in \ref{enu:positive_FS_AP_for_actions22}, the corresponding maps given by Proposition~\ref{prop:Haagerup_trick} form a bounded net $\{ \varphi_{\Phi_i} \}_{i\in I} \subseteq \contc(r^*Z\M(\A))$ of Herz--Schur multipliers of positive type. 
Moreover, if $a \in \contc(\A^{(\alpha,u)})$ is supported on a bisection, then $\sup_{g\in G} \| \varphi_{\Phi_i}(g)a(g) - a(g) \|_{\infty} = \sup_{g\in G} \lim_{i} \| \Phi(a)(g)  - a(g) \|_{\infty} = 0$.
Hence this holds for any $a \in \contc(\A^{(\alpha,u)})$. 
\end{proof}

\begin{rem}
For the trivial bundle $\A=\C\times X$, that is, the case of just a twisted groupoid $(\G,u)$ (see Example~\ref{ex:continuous-cocycles}), the above result reduces to amenability of the underlying groupoid $\G$, that is, the existence of a net of positive-type functions $\{\varphi_i\}_{i\in I}\subseteq \contc(\G)$ with $\varphi_i(g)\to 1$ uniformly on compacts. 
This is just a re-statement of item \ref{enu:positive_FS_AP_for_actions21} in the above proposition. 
We know that amenability of $\G$ is equivalent to nuclearity of $C^*(\G,u)=C^*_r(\G,u)$ (see Corollary~\ref{cor:amenable-groupoids-vs-FS-AP}), which means that the identity map $C^*(\G,u)\to C^*(\G,u)$ can be approximated (in the point-norm topology) by a composition of (contractive) completely positive linear maps $C^*(\G,u)\to \Mat_{n_i}(\C)\to C^*(\G,u)$. 
This is, in turn, equivalent the approximation property appearing in item \ref{enu:positive_FS_AP_for_actions22} of Proposition~\ref{prop:positive_FS_AP_for_actions3}. 
We are not sure whether this has already appeared in the literature in the way it is presented here.
\end{rem}

\subsection{Decomposable norm vs Fourier--Stieltjes norm}
\label{Section:matrices}

As we noted in the introduction to this section, every map given by a Fourier--Stieltjes multiplier, whether on the reduced or universal crossed product, can be expressed as a linear combination of completely positive maps. 
More interestingly, we can also relate the Fourier--Stieltjes norm of the multiplier to the \emph{decomposable} norm of the associated map, as introduced in \cite{Haagerupdec}. 

We first recall some basic relevant facts. 
Suppose that $A,B$ are $\cst$-algebras and $\Phi : A \to B$ is a linear map.
Haagerup in \cite{Haagerupdec} proved that following conditions are equivalent:
\begin{enumerate}
    \item\label{enu:Haagerupdec1} $\Phi$ is \emph{decomposable}, i.e.\ it can be expressed as a linear combination of completely positive maps from $A$ to $B$;
    \item\label{enu:Haagerupdec2} $\Phi = \Phi_1 - \Phi_2 + i (\Phi_3 - \Phi_4)$ where $\Phi_k : A \to B$ are completely positive for $k=1,\ldots ,4$;
    \item\label{enu:Haagerupdec3} there exists a completely positive map $\Psi : A \to \Mat_2(B)$ of the form
    \[
        \Psi = \begin{pmatrix}\Psi_1 & \Phi \\ 
                    \Phi ^\dagger & \Psi_2 
                \end{pmatrix} 
    \]
Then the maps $\Psi_1,\Psi_2:A \to B$ are necessarily completely positive.
\end{enumerate}
Moreover, the linear space $DC(A,B)$ of decomposable maps from $A$ to $B$ with the norm
\[
    \| \Phi \|_{\mathrm{dec}} := \sup \left\{ \max\{ \|\Phi_1\|, \|\Phi_2\|\}: \Psi = \begin{pmatrix}\Phi_1 & \Phi \\\Phi ^\dagger & \Phi_2 \end{pmatrix} \text{ is completely positive} \right\}
\]
is a Banach space. 
The norm $\| \cdot \|_{\mathrm{dec}}$ dominates the completely bounded norm, coincides with the latter on completely positive maps and makes the set of decomposable maps a Banach space. 
It is also submultiplicative with respect to composition, see \cite{Haagerupdec} for more details.
Thus $C^*$-algebras with decomposable maps as morphisms, equipped with $\|\cdot \|_{\mathrm{dec}}$-norm,  form a Banach category which contains the category of $C^*$-algebras with morphisms being completely positive maps as a Banach subcategory.

The next result can be viewed as an extension or continuation  of Theorem~\ref{thm:FSmultiplier_extends_to_reduced_and_full}.

\begin{prop}\label{prop:FSmultipliers_decomposable}
Let $\A$, $\B$ be $\cst$-bundles over $X$, carrying twisted actions $(\alpha,u_\alpha)$ and $(\beta,u_\beta)$ of $\G$. 
For every Fourier--Stieltjes multiplier $T\in FS[(\alpha,u_{\alpha}),(\beta,u_{\beta})]$ the corresponding completely bounded maps $\m_T^{\rd} : \cst_{\red}(\A^{(\alpha,u_{\alpha})}) \to \cst_{\red}(\B^{(\beta,u_{\beta})})$ and $\m_T^{\f} : \cst(\A^{(\alpha,u_{\alpha})}) \to \cst(\B^{(\beta,u_{\beta})})$ are decomposable, and we have $\|\m_T\|_{cb} \leq \|\m_T^{\rd}\|_{\mathrm{dec}}  \leq \|T\|_{FS}$, $\|\m_T\|_{cb}\leq \|\m_T^{\f}\|_{\mathrm{dec}} \leq \|T\|_{FS}$. 

If $T$ is positive-definite, then $\|\m_T\|_{cb} = \| \m_T^{\rd} \|_{\mathrm{dec}} = \| \m_T^{\f} \|_{\mathrm{dec}} = \| T \|_{FS}$. 
\end{prop}
\begin{proof}
Let $T := T_{\EE,L,\xi,\zeta} = \{ T_g \}_{g\in \G}$ be a Fourier--Stieltjes multiplier. 
By rescaling the respective sections (passing to $\sqrt{\frac{|\zeta\|}{\|\xi\|}}\xi$ and $\sqrt{\frac{|\xi\|}{\|\zeta\|}}\zeta$) we can assume that $\|\xi\|=\|\zeta\|$. 
Using Examples~\ref{ex:matrixbundles}, \ref{ex:matrixtwistedactions} and \ref{ex:matrixreps} consider the $\cst$-correspondence $\Mat_2(\EE)$ from $\A^{(\alpha^{(2)},u_{\alpha}^{(2)})}$ to $\B^{(\beta^{(2)},u_{\beta}^{(2)})}$. 
Define the section $\eta:=\xi \oplus \zeta \in \Mat_2(\EE)$ in the obvious `diagonal' way and consider the associated Fourier--Stieltjes multiplier $\tilde{T} := T_{\Mat_2(\EE), L^{(2)},\eta, \eta}$. 
The operator $m_{\tilde{T}}^{\rd} : \cst_{\red}(\A^{(\alpha^{(2)} , u_{\alpha}^{(2)})}) \to \cst_{\red}(\B^{(\beta^{(2)},u_{\beta}^{(2)})}) $ is completely positive by Theorem~\ref{thm:FSmultiplier_positive_definite}. 
Recall that we have a natural isomorphism
\[ 
    \cst_{\red}( \A^{(\alpha^{(2)},u_{\alpha}^{(2)})} ) \cong \Mat_2( \cst_{\red}(\A^{(\alpha,u_{\alpha})}) )
\]
(and its analogue for $\B^{(\beta,u_{\beta})}$). 
Define the map $\Psi: \cst_{\red}(\A^{(\alpha,u_{\alpha})}) \to \Mat_2(\cst_{\red}(\B^{(\beta,u_{\beta})}))$ via the formula
\[
    \Psi(x) = m_{\tilde{T}}^{\rd} \left( \begin{bmatrix}x & x \\x & x \end{bmatrix} \right), \qquad x \in \cst_{\red}(\A^{(\alpha,u_{\alpha})}). 
\]
As the map $x \mapsto \begin{bmatrix}x & x \\x & x \end{bmatrix} $ is completely positive so is $\Psi$. 
It is now elementary to check that in fact 
\[
    \Psi = \begin{bmatrix} 
                m_{T_{\EE, L, \xi, \xi}}^{\rd} & m_T^{\rd} \\ 
                ({m_T^{\rd}})^\dagger  & m_{T_{\EE, L, \zeta, \zeta}}^{\rd}
            \end{bmatrix} . 
\]
Thus, since  $m_{T_{\EE, L, \zeta, \zeta}}^{\rd}$ and $m_{T_{\EE, L, \xi, \xi}}^{\rd}$  are completely positive, $m_T^{\rd}$ is decomposable.
Moreover, as we rescaled to $\| \xi \| = \| \zeta \|$, by Theorem~\ref{thm:FSmultiplier_positive_definite} we have $\| m_{T_{\EE, L, \zeta, \zeta}}^{\rd} \| = \| m_{T_{\EE, L, \xi, \xi}}^{\rd} \| = \| \xi \| \|\zeta\|$. 
Thus $\| m_T^{\rd} \|_{\mathrm{dec}} \leq \| \xi \| \cdot \|\zeta\|$, which shows $\|\m_T^{\rd}\|_{\mathrm{dec}} \leq \|T\|_{FS}$. 
The same argument gives $\|\m_T^{\f}\|_{\mathrm{dec}} \leq \|T\|_{FS}$. 
The other statements follow easily from Theorem~\ref{thm:FSmultiplier_extends_to_reduced_and_full} and the discussion before the proposition.
\end{proof}

\begin{rem}
One could similarly show that Fourier multipliers yield decomposable `reduced-to-full' maps, as studied in Theorem~\ref{thm:Fmultipliers_are_reduced_to_full}.
\end{rem}

The above result naturally raises the question whether we can identify the decomposable norm of $\m_T$ (say on the reduced crossed product) with the norm $\|T\|_{FS}$. 
In the context of the usual Fourier--Stieltjes algebras of discrete groups such an identification was noted only very recently in \cite{ArhancetKriegler}. 
Below we will show that it remains true in the context of twisted crossed products of \emph{discrete} groupoids. 
Before we do that we need a simple (and likely well-known) lemma.

\begin{lem}\label{lem:matrixcorrespondences}
Let $A$, $B$ be $\cst$-algebras, $n \in \N$. 
If $F$ is a $\cst$-correspondence from $A$ to $B$, then the direct sum $E := \bigoplus_{i=1}^n F$ is a $\cst$-correspondence from $A$ to $\Mat_n(B)$, when equipped with the diagonal left action of $A$ and the scalar product and the action of $\Mat_n(B)$ given by 
\[
\begin{gathered}
    \big\langle (\xi_i)_{i=1}^n , (\zeta_i)_{i=1}^n \big\rangle_{\Mat_n(B)} := \left( \langle \xi_i, \zeta_j\rangle \right)_{i,j=1}^n , \\
    (\xi_i)_{i=1}^n \cdot \left( b_{ij}\right)_{i,j=1}^n := \left( \sum_{j=1}^n \xi_j b_{ji} \right)_{i=1}^n, 
\end{gathered}
\]
where $(\xi_i)_{i=1}^n, (\zeta_i)_{i=1}^n \in \bigoplus_{i=1}^n F$, $\left( b_{ij}\right)_{i,j=1}^n \in \Mat_n(B)$.
	
Conversely, every $\cst$-correspondence from $A$ to $\Mat_n(B)$ (up to unitary equivalence) arises in the way described above from a  $\cst$-correspondence from $A$ to $B$.  
\end{lem}
\begin{proof}
Assume first that $A$ and $B$ are unital.	
The proof amounts to a careful check of algebraic formulas, so we will only indicate the main steps needed to establish the second statement. 
	
Suppose that $E$ is a $\cst$-correspondence from $A$ to $\Mat_n(B)$. 
Denote by $e_{ij}$ ($i,j=1, \ldots, n$) matrix units in $\Mat_n(B)$. 
Define $F = Ee_{11}$, and let $\phi:B \to e_{11}\Mat_n(B) e_{11}$ be the obvious isomorphism. 
It is easy to check that $F$ becomes a $\cst$-correspondence from $A$ to $B$, when equipped with the scalar product 
\[
    \langle f_1, f_2 \rangle_B = \phi^{-1} (\langle f_1, f_2\rangle) ,
\]
and the action
\[
    f_1 \cdot b:= f_1 \phi(b) , 
\] 
where $f_1,f_2 \in F, b \in B$; note that the left action of $A$ is unchanged. 
When convenient, we will write the corresponding identification map from $Ee_{11}$ to $F$ as $\psi_1$.
Now given $j\in \{2, \ldots,n\}$ we define a map $\psi_j:Ee_{jj} \to F$ simply by putting
\[ 
    \psi_j (e) = e e_{j1}, \qquad e \in Ee_{jj} .
\]
Finally we define the map $\Psi : E \to \bigoplus_{i=1}^n F$, with the right hand side understood now as a $\cst$-correspondence from $A$ to $\Mat_n(B)$ defined in accordance with the first statement in the lemma, setting
\[ 
    \Psi(e) = \big( \psi_i (ee_{ii}) \big)_{i=1}^n .
\]
We leave it to the diligent reader to verify that $\Psi$ is indeed a unitary equivalence.	

Consider now the case where $A$ and $B$ need not be unital. 
Let $E$ be again a $\cst$-correspondence from $A$ to $\Mat_n(B)$. 
Then $\M(E)$ is a $\cst$-correspondence from $\M(A)$ to $\M(\Mat_n(B)) \cong \Mat_n(\M(B))$. 
By the unital case (and its proof) we have $\M(E) \cong \bigoplus_{i=1}^n \tilde{F}$, where $\tilde{F}$ is a $\cst$-correspondence from $\M(A)$ to $\M(B)$. 
Set $F = \tilde{F}B \subset \tilde{F}$. 
It is then a matter of carefully checking that the explicit isomorphism from $\M(E)$ to $\bigoplus_{i=1}^n \tilde{F}$ restricts to an isomorphism from $E$ to $\bigoplus_{i=1}^n F$. 
\end{proof}

We will now formulate a version of the above result for $\cst$-correspondence bundles.

\begin{cor}\label{cor:matrixcorrespondencebundles}
Let $\A$, $\B$ be $\cst$-bundles over $X$, carrying twisted actions $(\alpha,u_\alpha)$ and $(\beta,u_\beta)$ of $\G$, let $\EE$ be a $\cst$-$\A$-$\Mat_n(\B)$-correspondence bundle and let $n \in \N$. 
Then $\EE$ is unitarily equivalent to the direct sum $\bigoplus_{i=1}^n \FF$, where $\FF$ is a $\cst$-$\A$-$\B$-correspondence bundle, with pointwise operations described in Lemma~\ref{lem:matrixcorrespondences}. 
If moreover $\tilde{L}$ is an equivariant representation of $\G$ on $\EE$, then we may assume that we also have an equivariant representation $L$ of $\G$ on $\FF$ and in the equivalence above $\tilde{L}_g$ corresponds to a direct sum of $L_g$. 
\end{cor}
\begin{proof}
For the first part it suffices to note that the isomorphisms introduced in Lemma~\ref{lem:matrixcorrespondences} are clearly continuous with respect to the relevant `bundle' operations.

For the second we can argue as in the proof of Lemma~\ref{lem:matrixcorrespondences}, using the second condition in Definition~\ref{def:equivariant_representation} and the form of matrix liftings described in Example~\ref{ex:matrixtwistedactions} (it essentially suffices to observe that in the unital case for every $g \in \G$ the operator $\tilde{L}_g$ maps $E_{s(g)} e_{11}$ into $E_{r(g)} e_{11}$). 
Once again we leave the check to the reader. 
\end{proof}

We are now ready to formulate the main decomposability result.

\begin{thm}\label{thm:FSmultipliers_decomposable_discrete}
Let $(\alpha,u_\alpha)$ and $(\beta, u_\beta)$ be twisted actions of a discrete groupoid $\G$ respectively on $\A$ and $\B$. 
Then for every Fourier--Stieltjes multiplier $T\in FS[(\alpha,u_{\alpha}),(\beta,u_{\beta})]$ we have 
\[
    \| \m_T^{\rd} \|_{\mathrm{dec}} = \| \m_T^{\f} \|_{\mathrm{dec}} = \|T\|_{FS} .
\] 
\end{thm}
\begin{proof}
In view of Proposition~\ref{prop:FSmultipliers_decomposable} it suffices to prove the inequalities $\|\m_T^{\rd}\|_{\mathrm{dec}} , \|\m_T^{\f}\|_{\mathrm{dec}} \geq \|T\|_{FS}$, $\|\m_T^{\rd}\|_{\mathrm{dec}} \geq \|T\|_{FS}$. 
We will prove that $\|\m_T^{\rd}\|_{\mathrm{dec}} \geq \|T\|_{FS}$, as the second inequality follows in the same way.
Write then simply $m_T$ for $\m_T^{\rd}$ and assume that 
\[
    \| \m_T \|_{\mathrm{dec}} \leq 1.
\] 
We will show that $\|T\|_{FS} \leq 1$.

By the definition of the decomposable norm the displayed inequality means that there exists a completely positive map $\Psi : \cst_{\red}(\A^{(\alpha,u_{\alpha})}) \to \Mat_2 \left(\cst_{\red}(\B^{(\beta,u_{\beta})})\right)$ such that 
\[ 
    \Psi = \begin{pmatrix} 
                \Phi_1 & \m_T \\ 
                \m_T^\dagger & \Phi_2\ 
            \end{pmatrix} ,
\]
with $\Phi_1, \Phi_2 : \cst_{\red}(\A^{(\alpha,u_{\alpha})}) \to \cst_{\red}(\B^{(\beta,u_{\beta})})$ completely positive and such that $\| \Phi_1 \| \leq 1$, $\|\Phi_2 \|\leq 1$.
Let us now apply Lemma~\ref{lem:Haagerup trick discrete} to the map $\Psi$, viewed now as a map from $\cst_{\red}(\A^{(\alpha,u_{\alpha})})$ to $\cst_{\red}(\Mat_2(\B)^{(\beta^{(2)},u_{\beta}^{(2)})})$. 
This produces a new completely positive map $\tilde{\Psi}: \cst_{\red}(\A^{(\alpha,u_{\alpha})})\to \cst_{\red}( \Mat_2(\B)^{(\beta^{(2)} , u_{\beta}^{(2)})} )$, which comes from a multiplier. 
In fact, as the Haagerup trick acts on matrices `entrywise', and does not change maps which are already multipliers, we have
\[ 
    \tilde{\Psi} = \begin{pmatrix} 
                        \tilde{\Phi}_1 & \m_T \\ 
                        \m_T^\dagger & \tilde{\Phi}_2
                    \end{pmatrix} ,
\]
with $\tilde{\Phi}_1 , \tilde{\Phi}_2 : \cst_{\red}(\A^{(\alpha,u_{\alpha})}) \to \cst_{\red}(\B^{(\beta,u_{\beta})})$ contractive completely positive maps. 
Applying Lemma~\ref{lem:multiplier_approx_with strict} we may approximate  $\tilde{\Psi}$ by strict multipliers. 
Namely, for an approximate unit $\{ e_i \}_{i\in I}$ in $\contz(\A)$, Lemmas~\ref{lem:multiplier_approx_with strict} and ~\ref{lem:multiplier_approx_with strict2} give  a net $\{\tilde{T}_i\}_{i\in I}\in FS[(\alpha,u_{\alpha}),(\beta^{(2)},u_{\beta}^{(2)})]$ such that for every $i \in I$ we have $\tilde{\Psi}(e_ia  e_i) = \m_{\tilde{T}_i}^{\rd}(a)$, $a\in \cst_{\red}(\A^{(\alpha,u_{\alpha})})$. 
Then
\[ 
    \m_{\tilde{T}_i}^{\rd} =    \begin{pmatrix} 
                                    \tilde{\Phi}_{1,i} & \m_{T_i} \\ 
                                    \m_{T_i}^\dagger & \tilde{\Phi}_{2,i}
                                \end{pmatrix},
\]
where $\tilde{\Phi}_{1,i}, \tilde{\Phi}_{2,i} : \cst_{\red}(\A^{(\alpha,u_{\alpha})}) \to \cst_{\red}(\B^{(\beta,u_{\beta})})$ are contractive completely positive maps (coming from Fourier--Stieltjes multipliers) and $\{T_i\}_{i\in I}\in FS[(\alpha,u_{\alpha}),(\beta,u_{\beta})]$ are such that $\| T_i \|_{FS} \to \|T\|_{FS}$ (again by Lemma~\ref{lem:multiplier_approx_with strict2}). 

Fix for the moment $i \in I$. 
By Theorem~\ref{thm:FSmultiplier_positive_definite} the map $\tilde{\Phi}_i$ is given by a Fourier--Stieltjes multiplier $\tilde{T}_i\in FS[(\alpha,u_{\alpha}),(\beta^{(2)},u_{\beta}^{(2)})]$, $\tilde{T}=T_{\EE,L,\xi,\xi}$ for some $\cst$-$\A$-$\Mat_2(\B)$-correspondence $\EE$ equipped with an equivariant representation $\tilde{L}$ of $\G$ and section $\xi\in \contb(\M(\Mat_2(\EE)))$. 
Applying Theorem~\ref{lem:matrixcorrespondences} to $\EE$, $\tilde{L}$ and $\xi$ we may identify $\EE$ with $\FF \oplus \FF$, $\tilde{L}$ with a direct sum of two copies of an equivariant representation $L$ of $\G$ on $\FF$, and $\xi: = (\zeta, \eta)$ for two sections $\zeta, \eta \in \contb(\M(\FF)$. 
Using this identification we obtain
\[ 
    T_i= T_{L, F, \zeta, \eta}, \;\; \tilde{\Phi}_{1,i} = m_{T_{L,F, \zeta, \zeta}}, \qquad \tilde{\Phi}_{2,i} = m_{T_{L,F, \eta, \eta}} .
\]
Thus finally applying the last part of Theorem~\ref{thm:FSmultiplier_positive_definite} respectively to $\Phi_{1,i}$ and $\Phi_{2,i}$ we deduce that $\|\zeta\|, \|\eta\|\leq 1$. 
But this means that $\|T_i\|_{FS} \leq 1$, which ends the proof, as $\|T_i\|_{FS}\to \|T\|_{FS}$.
\end{proof}

\begin{rem}\label{rem:decbim}
In the context of the last theorem we could define a version of Haagerup's decomposable norm for $\contz(\A)$-bimodule maps, say $\|\cdot\|_{\mathrm{dec},\mathrm{bim}}$, requesting that all the entries of the requested 2 by 2 matrix decomposition are bimodule maps. 
Then the arguments of the above theorem together with the Haagerup trick from Proposition~\ref{prop:Haagerup_trick} would show that even for not necessarily discrete $\G$ and arbitrary Herz--Schur type multiplier $T$ we have $\|T\|_{FS} = \| m_T^{\rd} \|_{\mathrm{dec}, \mathrm{bim}} = \| m_T^{\f} \|_{\mathrm{dec}, \mathrm{bim}}$. 
\end{rem}

\begin{rem}\label{rem:category_of_decomposables}
In contrast to Remark~\ref{rem:functors_from_FS}, where we only had contractive functors, assuming that $\G$ is discrete we get isometric involutive $\contb(X)$-bimodule functors $\m^{\f} : FS_{TA}(\G)\to DC_{TA}(\G)$ and $\m^{\rd} : FS_{TA}(\G)\to DC_{TA}^{\rd}(\G)$ by letting  $DC_{TA}(\G)$  (resp.\ $DC_{TA}^{\rd}(\G)$) to be the category whose objects are  twisted actions of $\G$ on \cst-bundles and morphisms are decomposable maps between full (resp.\ reduced) crossed products that preserve fibers and are equipped with the norm $\|\cdot\|_{\mathrm{dec}}$. 
\end{rem}

As an application (though formally we will use only that the norms $\| \cdot \|_{\mathrm{dec}}$ and $\|\cdot\|_{FS}$ are `equivalent') we characterise the Fourier--Stieltjes AP for discrete groupoid actions in a similar fashion as we did for the positive Fourier--Stieltjes AP in Proposition~\ref{prop:positive_FS_AP_for_actions2} 
and Corollary~\ref{cor:characterisation_FS_AP_for_discrete}.

\begin{prop}\label{prop:FS_AP_for_actions2} 
Let $\A$, $\B$ be $\cst$-bundles over $X$, carrying twisted actions $(\alpha,u_\alpha)$ and $(\beta,u_\beta)$ of a discrete groupoid $\G$. 
A multiplier $T$ from $\A^{(\alpha,u_{\alpha})}$ to $\B^{(\beta,u_{\beta})}$ has the Fourier--Stieltjes approximation property if and only if there is a $\| \cdot \|_{\mathrm{dec}}$-bounded net $\{\Phi_{i}\}_{i\in I}$ of completely bounded maps $\Phi_i : \cst_{\red}(\A^{(\alpha,u_\alpha)}) \to C(\B^{(\beta,u_\beta)} |_{V_i}) \subseteq \cst_{\red}(\B^{(\beta,u_\beta)})$ with ranges supported in finite sets $V_i\subseteq \G$, $i\in I$, and such that $\Phi_i(a)\to m_T(a)$ for every $a\in \contc(\A^{(\alpha,u_\alpha)})$.
\end{prop}
\begin{proof} 
If $\{T_{i}\}_{i\in I} \subseteq FS[(\alpha,u_{\alpha})]$ is a net witnessing the positive Fourier--Stieltjes approximation property for $T$, then the associated maps $\{ \m_{T_i}^{\rd} \}_{i\in I}$, given by Theorem~\ref{thm:FSmultiplier_extends_to_reduced_and_full}, have the desired properties by Proposition~\ref{prop:FSmultipliers_decomposable} and Corollary~\ref{cor:compactly_supported_vs_range}. 
Conversely, assume that $\{ \Phi_{i} \}_{i\in I}$ is a net of decomposable maps such that $M := \sup_{i\in I} \| \Phi_i \|_{\mathrm{dec}} < \infty$ and $\Phi_i : \cst_{\red}(\A^{(\alpha,u_\alpha)}) \to C(\B^{(\beta,u_\beta)}|_{V_i})$ where $V_i$ is finite, $i\in I$. 
We may write each $\Phi_i$ as a linear combination $\sum_{k=0}^{3} i^k \Phi_{i,k}$ (here $i^{k}$ is the imaginary unit to the $k$th power, not an index of the set $I$) where each $\Phi_{i,k}$ is completely positive and $\|\Phi_{i,k}\| \leq \|\Phi_i\|_{\mathrm{dec}}\leq M$, see the formulas in \cite[Page 175]{Haagerupdec}. 
Applying Lemma~\ref{lem:Haagerup trick discrete} to $\Phi_{i}$ and $\Phi_{i,k}$ we obtain completely positive maps $\widetilde{\Phi}_{i}$ and $\widetilde{\Phi}_{i,k}$ coming from multipliers and such that $\|\widetilde{\Phi}_{i,k}\|\leq \|\Phi_{i,k}\|\leq M$ and $\widetilde{\Phi}_{i}=\sum_{k=0}^{3} i^k \widetilde{\Phi}_{i,k}^{(j)}$. 
Choosing an approximate unit $\{e_j\}_{j\in J}$ in $\contz(\A)$ and `strictifying' the considered maps by putting $\widetilde{\Phi}_{i}^{(j)}(a) := \widetilde{\Phi}_{i}(e_j a e_j)$ and $\widetilde{\Phi}_{i,k}^{(j)}(a) := \Phi_{i,k}(e_j a e_j)$ for $a\in  \cst_{\red}(\A^{(\alpha,u_\alpha)})$, by Lemma~\ref{lem:Haagerup trick discrete} we get that $\widetilde{\Phi}_{i,k}^{(j)}=\m_{T_{i,k}^{(j)}}$ for $T_{i,k}^{(j)}\in FS((\alpha,u_\alpha), (\beta,u_\beta) )^+$ and $\|T_{i,k}^{(j)}\|_{FS} \leq \|\widetilde{\Phi}_{i,k}\|\leq M$. 
Thus putting $T_{i}^{(j)}= \sum_{k=0}^{3} i^kT_{i,k}^{(j)}$ we obtain a bounded net $\{T_{i}^{(j)}\}_{(i,j)\in I\times J}\subseteq FS((\alpha,u_\alpha), (\beta,u_\beta) )$ with $\sup_{(i,j)\in I\times J}\| T_{i}^{(j)}\|_{FS}\leq 4M$. 
Moreover, for $a \in \contc(\A^{(\alpha,u_\alpha)})$ we have $\m_{T_{i}^{(j)}}(a)\to \m_T(a)$ because $\Phi_i(a)\to m_T(a)$ and $\widetilde{\Phi}_{i}^{(j)}(a) \to \Phi_i(a)$. 
\end{proof}

\begin{cor}\label{cor:characterisation_general_FS_AP_for_discrete}
A twisted groupoid action $(\alpha,u_{\alpha})$ of a discrete groupoid $\G$ on a \cst-bundle $\A$ has the Fourier--Stieltjes approximation property if and only if there is a $\| \cdot \|_{\mathrm{dec}}$-bounded net $\{\Phi_{i}\}_{i\in I}$ of decomposable maps on $\cst_{\red}(\A^{(\alpha,u_{\alpha})})$ that converge pointwise to the identity and that have ranges supported on finite subsets of $\G$. 
\end{cor}

\section{Open questions and further perspectives}
\label{Sec:QuestionsPerspectives}

In this final section of the paper, we would like to discuss the key questions left open in our work and present some of the perspectives for further research regarding Fourier--Stieltjes algebras and their applications in the context of crossed products by twisted groupoid actions. 

\subsection*{Open questions}
For the first three questions we work with fixed $(\alpha,u_\alpha)$ and $(\beta, u_\beta)$, twisted actions of an \'etale groupoid $\G$ on $\cst$-bundles $\A$ and $\B$ over $X$. 

\begin{quest} \label{quest:norms}
If $T\in F[(\alpha,u_{\alpha}),(\beta,u_{\beta})]$ is a Fourier multiplier, is it true that $\|T\|_{FS} = \|T\|_{F}$? 
\end{quest}

Note that we always have $\|T\|_{FS} \leq \|T\|_{F}$. 
The equality $\|T\|_{FS}=\|T\|_{F}$ is known to hold in the classical case of discrete groups, see \cite{Eymard}; one of possible ways of approaching this question relates to the `functional' picture of Fourier--Stieltjes and Fourier algebras mentioned in the introduction and missing in our context.
In general we do not also know if $F(\alpha,u_{\alpha})$ is closed in $FS(\alpha,u_{\alpha})$. 
Note that the issue appears already in the context of groupoid $\cst$-algebras. 
To avoid that, in \cite[Definition 1.4]{RenaultFourier}, Renault defines the Fourier algebra to be a closure in the Fourier--Stieltjes algebra of the coefficients given by the regular representation, and this convention is also adopted by Oty~\cite{Oty}.
We decided not to do this here.

\begin{quest} \label{quest:positive}
Is it true that $F[(\alpha,u_{\alpha}),(\beta, u_\beta)] \cap FS[(\alpha,u_{\alpha})(\beta, u_\beta)]^+ = F[(\alpha,u_{\alpha}),(\beta, u_\beta)]^+$? 
\end{quest}

If $T\in F[(\alpha,u_{\alpha}),(\beta,u_{\beta})]^+$, then $T$ is positive-definite, by Theorem~\ref{thm:FSmultiplier_positive_definite}. 
However this theorem does not imply that every positive-definite Fourier multiplier is in $F[(\alpha,u_{\alpha}),(\beta,u_{\beta})]^+$, as a priori the associated cyclic equivariant representation of $\G$ might not be regular. 
On the other hand, if we denote by $FS[(\alpha,u_{\alpha}),(\beta,u_{\beta})]^+$ the set of Fourier--Stieltjes multipliers that can represented in the form $T_{\EE,L,\xi,\xi}$, then Theorem~\ref{thm:FSmultiplier_positive_definite} states that $FS[(\alpha,u_{\alpha}),(\beta,u_{\beta})]^+$ consists exactly of positive-definite, bounded, strict multipliers from $\A^{(\alpha,u_{\alpha})}$ to $\B^{(\beta,u_{\beta})}$. 
Thus 
\[
    F[(\alpha,u_{\alpha}),(\alpha, u_\alpha)]^+\subseteq  F[(\alpha,u_{\alpha}),(\alpha, u_\alpha)] \cap FS[(\alpha,u_{\alpha}),(\alpha, u_\alpha)]^+ 
\]
and we do not know whether the reverse inclusion holds.

\begin{quest} \label{quest:compact}
Is it true that  $ F[(\alpha,u_{\alpha}),(\beta, u_\beta)]$ can be described as the closure of compactly supported Fourier--Stieltjes multipliers? 
Either in $\|\cdot\|_F$ or $\|\cdot\|_{FS}$? 
\end{quest}

The statement above is true in the context of discrete groups, as shown already in \cite{Eymard}. 
The analogous statement in the context of groupoid $\cst$-algebras appears in \cite{Renault}, but without proof (only an indication of an argument based on the partitions of unity; however we do not see how to use it to obtain the desired convergence, although in principle one could consider adapting the arguments in the proof of Proposition~\ref{pr:CompactSupportFourier}).

\begin{quest} \label{quest:dec}
Can one find an example of a Fourier--Stieltjes multiplier $T$ such that $\|\m_T^{\rd}\|_{\mathrm{dec}} < \|T\|_{FS}$?
\end{quest}

Theorem~\ref{thm:FSmultipliers_decomposable_discrete} shows that this cannot happen for \emph{discrete} $\G$; the proof of that theorem shows that any counterexample would imply a failure of an even weaker form of the Haagerup trick. 
However we do not know the answer to the question above even for groupoid $\cst$-algebras. 

Note that the problems above are clearly interrelated. 
For example a positive answer to Question~\ref{quest:positive} would yield, via Remark~\ref{rem:decbim}, a positive answer to Question~\ref{quest:norms} for Herz--Schur type multipliers.

\vspace*{0.2 cm}

The following questions are related to the approximation properties of actions/crossed products, and matters discussed in Section~\ref{Sec:Applicationapproximation}.

\begin{quest} \label{quest:diagram}
In what generality can one invert the arrows in the diagram in Remark~\ref{rem:diagram}?
\end{quest}

Some of the reverse implications hold in certain special cases, for example for discrete groupoids (see Theorem \ref{thm:characterisation of nuclearity}) but are not clear in general. 
For instance, strong amenability implies Exel's AP (which is in turn equivalent to ``amenability'') for group actions on $\cst$-algebras, see \cite[Theorem~4.9]{BussEchterhoffWillett}. 
On the other hand, for group actions the implication from weak containment to any of the approximation properties is generally a big mystery although it holds in certain special cases, and there are examples of groupoids that satisfy the weak containment property but are not amenable, see \cite{Willett}.

\begin{quest}
Does nuclearity of the reduced crossed product $\cst_{\red}(\A^{(\alpha,u_{\alpha})})$ imply one of the approximation properties in the diagram in Remark~\ref{rem:diagram}?
\end{quest}

Strong amenability in general is not necessary for nuclearity of $\cst_{\red}(\A^{(\alpha,u_{\alpha})})$, which follows for instance from Suzuki's examples of group actions on unital simple $C^*$-algebras in \cite{Suzuki}.
But it might be relevant when the algebra $\contz(\A)$ is commutative, as for group actions, strong amenability and amenability are equivalent, see \cite[Th\'eor\`eme 4.9]{Clare1}.

By Theorem \ref{thm:characterisation of nuclearity} we know that nuclearity of $\cst_{\red}(\A^{(\alpha,u_{\alpha})})$ implies the positive Fourier--Stieltjes AP when $\G$ is discrete.
When $\G=G$ is a group nuclearity of the crossed product implies the formally stronger Exel's AP, see \cite[Corollary 4.11]{BussEchterhoffWillett} and Remark~\ref{rem:nuclear_vs_AP's}.
This makes us tend to think that nuclearity may imply Exel's AP for \'etale gropouid actions.


\vspace*{0.2 cm}

\subsection*{Perspectives}

One of the deepest modern applications of the Fourier and Fourier--Stieltjes algebras in the classical context of discrete groups is related to using these algebras and their properties as means of distinguishing the corresponding group von Neumann algebras. 
A key role is played here by the space of \emph{completely bounded} multipliers of the Fourier algebra of $\Gamma$, denoted $M_{cb}(A(\Gamma))$, which contains the Fourier--Stieltjes algebra $B(\Gamma)$ and can be used to define finer, often quantitative approximation properties of $\Gamma$ (see for example \cite{ch}, \cite{hk}). 
The starting point to study such questions in our context (or even in the context of actions of discrete groups on $\cst$-algebras) one would need to 
\begin{itemize}
    \item investigate the operator space structure of our Fourier and Fourier--Stieltjes algebras;
    \item analyse the corresponding space of (completely) bounded multipliers of the Banach algebra $F[(\alpha,u),(\alpha,u)]$. 
\end{itemize}
The first step is related to Question~\ref{quest:norms}; we would expect the relevant cb-multiplier norm to coincide both with the Fourier norm and Fourier--Stieltjes norm of a given element of the Fourier algebra. 
Note also \cite[Theorem 3.1]{Renault}, which in the groupoid case (trivial action) expresses completely bounded multipliers on the Fourier algebra of $\G$ via \emph{invariant} elements of the Fourier--Stieltjes algebra of the product groupoid $\G \star \G$. 
We expect that in general it might be useful in this context to use the language of \emph{(dynamical) quantum groupoids}, as introduced for example in \cite{EnockTim}.

\vspace*{0.2 cm}

Another natural direction of study concerns attempting to reconstruct a given groupoid dynamical system from the associated Fourier--Stieltjes or Fourier algebra. 
It is well known that from the Fourier algebra $A(\Gamma)$ of a discrete group $\Gamma$ one can reconstruct $\Gamma$ --- as a set --- via the spectrum (space of characters) of $A(\Gamma)$. 
In the articles \cite{RenaultFourier} or \cite{RamsayWalter}, dealing with the groupoid operator algebras, analogous questions are often phrased as the study of \emph{duality for \'etale groupoids}. 
This is closely related to the issue of identifying the Fourier or Fourier--Stieltjes algebras with the dual spaces (of some sort) associated with the relevant operator algebras. 
As we mentioned in the introduction, in general this is far from clear, and even the case of groupoid $\cst$-algebras is not straightforward, as exemplified by \cite[Theorem 2.3]{Renault}.

Concrete questions one can study are the following:
\begin{itemize}
    \item what equivalence relation between (pairs of) twisted groupoid actions leads to the isometric isomorphism of the corresponding Fourier--Stieltjes spaces/algebras? 
    \item Under what conditions one can for example indeed reconstruct a given twisted groupoid action, up to the equivalence mentioned above, knowing its associated Fourier--Stieltjes algebra?
\end{itemize}
In the first part one should mention the Morita equivalence of $\cst$-dynamical systems, discussed in \cite[Section 4]{BedosConti3} and shown to yield --- under certain restrictive assumptions --- isometric isomorphism of the associated Fourier-Stieltjes algebras. 
In our context this can be of course further enriched by looking at different, but equivalent, groupoids, in the spirit of \cite[Section 3.4]{Sims}. 
The second question is naturally far more difficult, and even partial results would be of great interest.

The two broad topics mentioned above clearly do not exhaust the scope of natural questions one can ask regarding our Fourier--Stieltjes algebras. 
Just to quote significant body of work in the group case, as presented in \cite{KaniuthLau}, one could also ask about the general form of (completely) bounded homomorphisms between Fourier--Stieltjes algebras or about the behaviour of these algebras with respect to the restriction to subgroupoids, related to classical spectral synthesis questions. 
This brings us to the possibility of dropping the \'etale assumption, and considering for example the crossed products by actions of arbitrary locally compact groups. 
The last step naturally would take us far beyond the scope of current work.

\subsection*{Acknowledgements}
AB was partially supported by CNPq/CAPES-Print - Brazil. 
AS was partially supported by the National Science Center (NCN) grant no.~2020/39/I/ST1/01566. 
BKK was partially supported by 
the National Science Centre (NCN), through the WEAVE-UNISONO grant no. 2023/05/Y/ST1/00046 as well as the National Science Center (NCN) grant no.~2019/35/B/ST1/02684.
We thank the anonymous referee for a very careful reading of our paper and many thoughtful remarks improving the presentation.

\bibliography{references}{}

@article {Clare1,
    AUTHOR = {Anantharaman-Delaroche, Claire},
     TITLE = {Syst\`emes dynamiques non commutatifs et moyennabilit\'{e}},
   JOURNAL = {Math. Ann.},
    VOLUME = {279},
      YEAR = {1987},
    NUMBER = {2},
     PAGES = {297--315},
      ISSN = {0025-5831,1432-1807},
   MRCLASS = {46L55 (22D25 22D40 43A07 43A35)},
  MRNUMBER = {919508},
MRREVIEWER = {Elliot\ C.\ Gootman},
}

@article {Clare2,
    AUTHOR = {Anantharaman-Delaroche, Claire},
     TITLE = {Amenability and exactness for dynamical systems and their
              {$C^\ast$}-algebras},
   JOURNAL = {Trans. Amer. Math. Soc.},
    VOLUME = {354},
      YEAR = {2002},
    NUMBER = {10},
     PAGES = {4153--4178},
      ISSN = {0002-9947,1088-6850},
   MRCLASS = {46L55 (22D25 43A07)},
  MRNUMBER = {1926869},
MRREVIEWER = {Berndt\ Brenken},
}

@book {AD-Renault-amenable,
    AUTHOR = {Anantharaman-Delaroche, Claire and Renault, Jean},
     TITLE = {Amenable groupoids},
    SERIES = {Monographies de L'Enseignement Math\'{e}matique [Monographs of
              L'Enseignement Math\'{e}matique]},
    VOLUME = {36},
      NOTE = {With a foreword by Georges Skandalis and Appendix B by E.
              Germain},
 PUBLISHER = {L'Enseignement Math\'{e}matique, Geneva},
      YEAR = {2000},
     PAGES = {196},
      ISBN = {2-940264-01-5},
   MRCLASS = {22A22 (22D25 43A07 46L05 46L10 46L80)},
  MRNUMBER = {1799683},
MRREVIEWER = {Robert S. Doran},
}

@article {Multipliers,
    AUTHOR = {Akemann, Charles A. and Pedersen, Gert K. and Tomiyama, Jun},
     TITLE = {Multipliers of {$C\sp*$}-algebras},
   JOURNAL = {J. Functional Analysis},
    VOLUME = {13},
      YEAR = {1973},
     PAGES = {277--301},
      ISSN = {0022-1236},
   MRCLASS = {46L05},
  MRNUMBER = {470685},
}

@unpublished {ArhancetKriegler,
    AUTHOR = {C\'edric Arhancet and Christoph Kriegler},
     TITLE = {{F}ourier-{S}tieltjes algebras, decomposable {F}ourier multipliers and amenability},
   status = {preprint},
      YEAR = {2022},
    note={arXiv:2205.13823}
}

@article {BedosContiregular,
    AUTHOR = {B\'{e}dos, Erik and Conti, Roberto},
     TITLE = {On discrete twisted {$\rm C^*$}-dynamical systems, {H}ilbert
              {$\rm C^*$}-modules and regularity},
   JOURNAL = {M\"{u}nster J. Math.},
     VOLUME = {5},
      YEAR = {2012},
     PAGES = {183--208},
      ISSN = {1867-5778,1867-5786},
   MRCLASS = {46L55},
  MRNUMBER = {3047632},
MRREVIEWER = {Yoshikazu\ Katayama},
}

@article {BedosConti,
    AUTHOR = {B\'{e}dos, Erik and Conti, Roberto},
     TITLE = {Fourier series and twisted {${\rm C}^{\ast}$}-crossed
              products},
   JOURNAL = {J. Fourier Anal. Appl.},
    VOLUME = {21},
      YEAR = {2015},
    NUMBER = {1},
     PAGES = {32--75},
      ISSN = {1069-5869,1531-5851},
   MRCLASS = {46L55 (43A50 43A55)},
  MRNUMBER = {3302101},
MRREVIEWER = {Franz\ Luef},
}

@article {BedosConti2,
    AUTHOR = {B\'{e}dos, Erik and Conti, Roberto},
     TITLE = {The {F}ourier-{S}tieltjes algebra of a {$C^*$}-dynamical
              system},
   JOURNAL = {Internat. J. Math.},
    VOLUME = {27},
      YEAR = {2016},
    NUMBER = {6},
     PAGES = {1650050, 50},
      ISSN = {0129-167X,1793-6519},
   MRCLASS = {46L55 (37A55 43A50 43A55)},
  MRNUMBER = {3516977},
}

@article {BedosConti3,
    AUTHOR = {B\'{e}dos, Erik and Conti, Roberto},
     TITLE = {The {F}ourier-{S}tieltjes algebra of a {$C^*$}-dynamical
              system {II}},
   JOURNAL = {Studia Math.},
  FJOURNAL = {Studia Mathematica},
    VOLUME = {256},
      YEAR = {2021},
    NUMBER = {2},
     PAGES = {217--239},
}

@book {B,
    AUTHOR = {Blackadar, Bruce},
     TITLE = {{$K$}-theory for operator algebras},
    SERIES = {Mathematical Sciences Research Institute Publications},
    VOLUME = {5},
   EDITION = {Second},
 PUBLISHER = {Cambridge University Press, Cambridge},
      YEAR = {1998},
     PAGES = {xx+300},
      ISBN = {0-521-63532-2},
   MRCLASS = {46L80 (19Kxx 58G12)},
  MRNUMBER = {1656031},
}

@article {Ble,
    AUTHOR = {Blecher, David P.},
     TITLE = {A new approach to {H}ilbert {$C^*$}-modules},
   JOURNAL = {Math. Ann.},
    VOLUME = {307},
      YEAR = {1997},
    NUMBER = {2},
     PAGES = {253--290},
      ISSN = {0025-5831,1432-1807},
   MRCLASS = {46L05 (46C50 46H25)},
  MRNUMBER = {1428873},
MRREVIEWER = {Christian\ Le Merdy},
}

@book {BO,
    AUTHOR = {Brown, Nathanial P. and Ozawa, Narutaka},
     TITLE = {{$C^*$}-algebras and finite-dimensional approximations},
    SERIES = {Graduate Studies in Mathematics},
    VOLUME = {88},
 PUBLISHER = {American Mathematical Society, Providence, RI},
      YEAR = {2008},
     PAGES = {xvi+509},
      ISBN = {978-0-8218-4381-9; 0-8218-4381-8},
   MRCLASS = {46L05 (43A07 46-02 46L10)},
  MRNUMBER = {2391387},
MRREVIEWER = {Mikael\ R\o rdam},
}

@article {BusbySmith,
    AUTHOR = {Busby, Robert C. and Smith, Harvey A.},
     TITLE = {Representations of twisted group algebras},
   JOURNAL = {Trans. Amer. Math. Soc.},
  FJOURNAL = {Transactions of the American Mathematical Society},
    VOLUME = {149},
      YEAR = {1970},
     PAGES = {503--537},
      ISSN = {0002-9947,1088-6850},
   MRCLASS = {46.80 (42.56)},
  MRNUMBER = {264418},
MRREVIEWER = {H.\ Leptin},
}

@article {BussEchterhoffWillett,
    AUTHOR = {Buss, Alcides and Echterhoff, Siegfried and Willett, Rufus},
     TITLE = {Amenability and weak containment for actions of locally compact groups on {$C^*$}-algebras},
   JOURNAL = {Mem. Amer. Math. Soc.},
  FJOURNAL = {Memoirs of the American Mathematical Society},
    VOLUME = {301},
      YEAR = {2024},
    NUMBER = {1513},
     PAGES = {v+88},
      ISSN = {0065-9266,1947-6221},
      ISBN = {978-1-4704-7152-1; 978-1-4704-7957-2},
   MRCLASS = {46L55 (43A35)},
  MRNUMBER = {4808715},
MRREVIEWER = {Paul\ Jolissaint},
}

@article {BussExel0,
    AUTHOR = {Buss, Alcides and Exel, Ruy},
     TITLE = {Twisted actions and regular {F}ell bundles over inverse
              semigroups},
   JOURNAL = {Proc. Lond. Math. Soc. (3)},
    VOLUME = {103},
      YEAR = {2011},
    NUMBER = {2},
     PAGES = {235--270},
      ISSN = {0024-6115,1460-244X},
   MRCLASS = {46L55 (20M18 46L05)},
  MRNUMBER = {2821242},
MRREVIEWER = {Fernando\ Abadie},
}

@article {BussExel,
    AUTHOR = {Buss, Alcides and Exel, Ruy},
     TITLE = {Fell bundles over inverse semigroups and twisted \'{e}tale
              groupoids},
   JOURNAL = {J. Operator Theory},
    VOLUME = {67},
      YEAR = {2012},
    NUMBER = {1},
     PAGES = {153--205},
      ISSN = {0379-4024,1841-7744},
   MRCLASS = {46L55 (20M18 55R65)},
  MRNUMBER = {2881538},
MRREVIEWER = {Valentin\ Deaconu},
}

@article {BussMartinez,
    AUTHOR = {Buss, Alcides and Mart\'{\i}nez, Diego},
     TITLE = {Approximation properties of {F}ell bundles over inverse
              semigroups and non-{H}ausdorff groupoids},
   JOURNAL = {Adv. Math.},
    VOLUME = {431},
      YEAR = {2023},
     PAGES = {Paper No. 109251, 54},
      ISSN = {0001-8708,1090-2082},
   MRCLASS = {46L55 (20M18 46L06)},
  MRNUMBER = {4630196},
}

@article {BussMeyerZhu,
    AUTHOR = {Buss, Alcides and Meyer, Ralf and Zhu, Chenchang},
     TITLE = {A higher category approach to twisted actions on
              {$C^*$}-algebras},
   JOURNAL = {Proc. Edinb. Math. Soc. (2)},
    VOLUME = {56},
      YEAR = {2013},
    NUMBER = {2},
     PAGES = {387--426},
      ISSN = {0013-0915,1464-3839},
   MRCLASS = {46L55 (18D05)},
  MRNUMBER = {3056650},
MRREVIEWER = {Masaharu\ Kusuda},
}

@article {BussMeyer,
    AUTHOR = {Buss, Alcides and Meyer, Ralf},
     TITLE = {Inverse semigroup actions on groupoids},
   JOURNAL = {Rocky Mountain J. Math.},
    VOLUME = {47},
      YEAR = {2017},
    NUMBER = {1},
     PAGES = {53--159},
      ISSN = {0035-7596,1945-3795},
   MRCLASS = {22A22 (16D90 20M18 20M30 46L55)},
  MRNUMBER = {3619758},
MRREVIEWER = {Piotr\ Stachura},
}

@article {ChoiEffros,
    AUTHOR = {Choi, Man Duen and Effros, Edward G.},
     TITLE = {Nuclear {$C\sp*$}-algebras and the approximation property},
   JOURNAL = {Amer. J. Math.},
  FJOURNAL = {American Journal of Mathematics},
    VOLUME = {100},
      YEAR = {1978},
    NUMBER = {1},
     PAGES = {61--79},
      ISSN = {0002-9327},
   MRCLASS = {46L05},
  MRNUMBER = {482238},
MRREVIEWER = {Christopher Lance},
       DOI = {10.2307/2373876},
       URL = {https://doi.org/10.2307/2373876},
}

@article {ch,
    AUTHOR = {Cowling, Michael and Haagerup, Uffe},
     TITLE = {Completely bounded multipliers of the {F}ourier algebra of a
              simple {L}ie group of real rank one},
   JOURNAL = {Invent. Math.},
  FJOURNAL = {Inventiones Mathematicae},
    VOLUME = {96},
      YEAR = {1989},
    NUMBER = {3},
     PAGES = {507--549},
}

@article {Daws,
    AUTHOR = {Daws, Matthew},
     TITLE = {Multipliers, self-induced and dual {B}anach algebras},
   JOURNAL = {Dissertationes Math.},
    VOLUME = {470},
      YEAR = {2010},
     PAGES = {62},
      ISSN = {0012-3862,1730-6310},
   MRCLASS = {46H25 (43A22 46H20 46L07)},
  MRNUMBER = {2681109},
MRREVIEWER = {H.\ G.\ Dales},
}

@article {Deaconu,
    AUTHOR = {Deaconu, Valentin},
     TITLE = {Groupoid actions on {$C^*$}-correspondences},
   JOURNAL = {New York J. Math.},
  FJOURNAL = {New York Journal of Mathematics},
    VOLUME = {24},
      YEAR = {2018},
     PAGES = {1020--1038},
      ISSN = {1076-9803},
   MRCLASS = {46L05},
  MRNUMBER = {3874961},
MRREVIEWER = {Michael\ S.\ Anoussis},
}

@article {Daws2,
    AUTHOR = {Daws, Matthew},
     TITLE = {Multipliers of locally compact quantum groups via {H}ilbert
              {$C^\ast$}-modules},
   JOURNAL = {J. Lond. Math. Soc. (2)},
  FJOURNAL = {Journal of the London Mathematical Society. Second Series},
    VOLUME = {84},
      YEAR = {2011},
    NUMBER = {2},
     PAGES = {385--407},
}

@article {dong_ruan,
    AUTHOR = {Dong, Zhe and Ruan, Zhong-Jin},
     TITLE = {A {H}ilbert module approach to the {H}aagerup property},
   JOURNAL = {Integral Equations Operator Theory},
  FJOURNAL = {Integral Equations and Operator Theory},
    VOLUME = {73},
      YEAR = {2012},
    NUMBER = {3},
     PAGES = {431--454},
      ISSN = {0378-620X,1420-8989},
   MRCLASS = {46L08 (20F65 22D10)},
  MRNUMBER = {2945214},
MRREVIEWER = {Cristian\ Ivanescu},
}

@article {EKQR,
    AUTHOR = {Echterhoff, Siegfried and Kaliszewski, S. and Quigg, John and
              Raeburn, Iain},
     TITLE = {A categorical approach to imprimitivity theorems for
              {$C^*$}-dynamical systems},
   JOURNAL = {Mem. Amer. Math. Soc.},
    VOLUME = {180},
      YEAR = {2006},
    NUMBER = {850},
     PAGES = {viii+169},
      ISSN = {0065-9266,1947-6221},
   MRCLASS = {46L55 (18A99 46L05 46L08 46M15)},
  MRNUMBER = {2203930},
MRREVIEWER = {Judith\ A.\ Packer},
 }

@article {ER,
    AUTHOR = {Echterhoff, Siegfried and Raeburn, Iain},
     TITLE = {Multipliers of imprimitivity bimodules and {M}orita
              equivalence of crossed products},
   JOURNAL = {Math. Scand.},
    VOLUME = {76},
      YEAR = {1995},
    NUMBER = {2},
     PAGES = {289--309},
      ISSN = {0025-5521,1903-1807},
   MRCLASS = {46L05 (46L55)},
  MRNUMBER = {1354585},
}

@article {EnockTim,
    AUTHOR = {Enock, Michel and Timmermann, Thomas},
     TITLE = {Measured quantum transformation groupoids},
   JOURNAL = {J. Noncommut. Geom.},
  FJOURNAL = {Journal of Noncommutative Geometry},
    VOLUME = {10},
      YEAR = {2016},
    NUMBER = {3},
     PAGES = {1143--1214},
}

@article {Exel:amenability,
    AUTHOR = {Exel, Ruy},
     TITLE = {Amenability for {F}ell bundles},
   JOURNAL = {J. Reine Angew. Math.},
    VOLUME = {492},
      YEAR = {1997},
     PAGES = {41--73},
      ISSN = {0075-4102,1435-5345},
   MRCLASS = {46M20 (46L55)},
  MRNUMBER = {1488064},
MRREVIEWER = {Robert\ S.\ Doran},
}

@article {Exel:combinatorial,
    AUTHOR = {Exel, Ruy},
     TITLE = {Inverse semigroups and combinatorial {$C^\ast$}-algebras},
   JOURNAL = {Bull. Braz. Math. Soc. (N.S.)},
  FJOURNAL = {Bulletin of the Brazilian Mathematical Society. New Series.
              Boletim da Sociedade Brasileira de Matem\'{a}tica},
    VOLUME = {39},
      YEAR = {2008},
    NUMBER = {2},
     PAGES = {191--313},
      ISSN = {1678-7544,1678-7714},
   MRCLASS = {46L05 (18B40 46L55)},
  MRNUMBER = {2419901},
MRREVIEWER = {Mark\ Tomforde},
}

@article {Eymard,
    AUTHOR = {Eymard, Pierre},
     TITLE = {L'alg\`ebre de {F}ourier d'un groupe localement compact},
   JOURNAL = {Bull. Soc. Math. France},
  FJOURNAL = {Bulletin de la Soci\'{e}t\'{e} Math\'{e}matique de France},
    VOLUME = {92},
      YEAR = {1964},
     PAGES = {181--236},
      ISSN = {0037-9484},
   MRCLASS = {22.65},
  MRNUMBER = {228628},
MRREVIEWER = {C.\ Dunkl},
}

@book {Fell_Doran,
    AUTHOR = {Fell, James M. G. and Doran, Robert S.},
     TITLE = {Representations of {$^*$}-algebras, locally compact groups,
              and {B}anach {$^*$}-algebraic bundles. {V}ol. 1},
    SERIES = {Pure and Applied Mathematics},
    VOLUME = {125},
      NOTE = {Basic representation theory of groups and algebras},
 PUBLISHER = {Academic Press, Inc., Boston, MA},
      YEAR = {1988},
     PAGES = {xviii+746},
      ISBN = {0-12-252721-6},
   MRCLASS = {46-02 (22Dxx 22Exx 46Hxx 46Lxx 46M20)},
  MRNUMBER = {936628},
MRREVIEWER = {Maurice\ J.\ Dupr\'{e}},
}

@article {Fujita,
    AUTHOR = {Fujita, Masayuki},
     TITLE = {Banach algebra structure in {F}ourier spaces and
              generalization of harmonic analysis on locally compact groups},
   JOURNAL = {J. Math. Soc. Japan},
  FJOURNAL = {Journal of the Mathematical Society of Japan},
    VOLUME = {31},
      YEAR = {1979},
    NUMBER = {1},
     PAGES = {53--67},
}

@article {FS,
    AUTHOR = {Finn-Sell, Martin},
     TITLE = {Fibred coarse embeddings, a-{T}-menability and the coarse
              analogue of the {N}ovikov conjecture},
   JOURNAL = {J. Funct. Anal.},
    VOLUME = {267},
      YEAR = {2014},
    NUMBER = {10},
     PAGES = {3758--3782},
      ISSN = {0022-1236,1096-0783},
   MRCLASS = {22A22 (19K99 20F65 30L05 46L80 58H05)},
  MRNUMBER = {3266245},
MRREVIEWER = {J\'{a}n\ \v{S}pakula},
}

@article {FMR,
    AUTHOR = {Fowler, Neal J. and Muhly, Paul S. and Raeburn, Iain},
     TITLE = {Representations of {C}untz-{P}imsner algebras},
   JOURNAL = {Indiana Univ. Math. J.},
  FJOURNAL = {Indiana University Mathematics Journal},
    VOLUME = {52},
      YEAR = {2003},
    NUMBER = {3},
     PAGES = {569--605},
      ISSN = {0022-2518,1943-5258},
   MRCLASS = {46L05 (46L08 47L30)},
  MRNUMBER = {1986889},
MRREVIEWER = {Baruch\ Solel},
}

@incollection {Haagerupdec,
    AUTHOR = {Haagerup, Uffe},
     TITLE = {Injectivity and decomposition of completely bounded maps},
 BOOKTITLE = {Operator algebras and their connections with topology and
              ergodic theory ({B}u\c{s}teni, 1983)},
    SERIES = {Lecture Notes in Math.},
    VOLUME = {1132},
     PAGES = {170--222},
 PUBLISHER = {Springer, Berlin},
      YEAR = {1985},
      ISBN = {3-540-15643-7},
   MRCLASS = {46L30 (46L05)},
  MRNUMBER = {799569},
MRREVIEWER = {Roger\ R.\ Smith},
}

@article {Haagerup,
    AUTHOR = {Haagerup, Uffe},
     TITLE = {Group {$C^*$}-algebras without the completely bounded
              approximation property},
   JOURNAL = {J. Lie Theory},
    VOLUME = {26},
      YEAR = {2016},
    NUMBER = {3},
     PAGES = {861--887},
      ISSN = {0949-5932},
   MRCLASS = {22D25 (22D15 22E40 22E46 43A22)},
  MRNUMBER = {3476201},
MRREVIEWER = {Tim\ de Laat},
}

@article {hk,
    AUTHOR = {Haagerup, Uffe and Kraus, Jon},
     TITLE = {Approximation properties for group {$C^*$}-algebras and group
              von {N}eumann algebras},
   JOURNAL = {Trans. Amer. Math. Soc.},
  FJOURNAL = {Transactions of the American Mathematical Society},
    VOLUME = {344},
      YEAR = {1994},
    NUMBER = {2},
     PAGES = {667--699},
}

@book {KaniuthLau,
    AUTHOR = {Kaniuth, Eberhard and Lau, Anthony To-Ming},
     TITLE = {Fourier and {F}ourier-{S}tieltjes algebras on locally compact
              groups},
    SERIES = {Mathematical Surveys and Monographs},
    VOLUME = {231},
 PUBLISHER = {American Mathematical Society, Providence, RI},
      YEAR = {2018},
     PAGES = {xi+306},
}

@article {katsura2,
    AUTHOR = {Katsura, Takeshi},
     TITLE = {Ideal structure of {$C^*$}-algebras associated with
              {$C^*$}-correspondences},
   JOURNAL = {Pacific J. Math.},
    VOLUME = {230},
      YEAR = {2007},
    NUMBER = {1},
     PAGES = {107--145},
      ISSN = {0030-8730,1945-5844},
   MRCLASS = {46L05 (46L55)},
  MRNUMBER = {2413377},
MRREVIEWER = {Mark\ Tomforde},
}

@article {Kranz,
    AUTHOR = {Kranz, Julian},
     TITLE = {Amenability for actions of \'{e}tale groupoids on
              {$C^*$}-algebras and {F}ell bundles},
   JOURNAL = {Trans. Amer. Math. Soc.},
    VOLUME = {376},
      YEAR = {2023},
    NUMBER = {7},
     PAGES = {5089--5121},
      ISSN = {0002-9947,1088-6850},
   MRCLASS = {46L55 (22A22)},
  MRNUMBER = {4608439},
MRREVIEWER = {Eusebio\ Gardella},
}

@article {Kumjian,
    AUTHOR = {Kumjian, Alex},
     TITLE = {Fell bundles over groupoids},
   JOURNAL = {Proc. Amer. Math. Soc.},
    VOLUME = {126},
      YEAR = {1998},
    NUMBER = {4},
     PAGES = {1115--1125},
      ISSN = {0002-9939,1088-6826},
   MRCLASS = {46L05 (22A22 22D25 46M20)},
  MRNUMBER = {1443836},
MRREVIEWER = {Robert\ S.\ Doran},
}

@article {KuV,
    AUTHOR = {Kustermans, Johan and Vaes, Stefaan},
     TITLE = {Locally compact quantum groups},
   JOURNAL = {Ann. Sci. \'{E}cole Norm. Sup. (4)},
  FJOURNAL = {Annales Scientifiques de l'\'{E}cole Normale Sup\'{e}rieure.
              Quatri\`eme S\'{e}rie},
    VOLUME = {33},
      YEAR = {2000},
    NUMBER = {6},
     PAGES = {837--934},
}

@article {Bartosz,
    AUTHOR = {Kwa\'{s}niewski, Bartosz K.},
     TITLE = {Crossed products by endomorphisms of {$C_0(X)$}-algebras},
   JOURNAL = {J. Funct. Anal.},
    VOLUME = {270},
      YEAR = {2016},
    NUMBER = {6},
     PAGES = {2268--2335},
      ISSN = {0022-1236,1096-0783},
   MRCLASS = {46L55 (19K35 46L05 46L35)},
  MRNUMBER = {3460241},
MRREVIEWER = {Jonathan\ M.\ Rosenberg},
}

@article {BartoszKangAdam,
    AUTHOR = {Kwa\'{s}niewski, Bartosz K. and Li, Kang and Skalski, Adam},
     TITLE = {The {H}aagerup property for twisted groupoid dynamical
              systems},
   JOURNAL = {J. Funct. Anal.},
    VOLUME = {283},
      YEAR = {2022},
    NUMBER = {1},
     PAGES = {Paper No. 109484, 43},
      ISSN = {0022-1236,1096-0783},
   MRCLASS = {46L05 (22A22 43A22 46L55)},
  MRNUMBER = {4404070},
MRREVIEWER = {Chi-Keung\ Ng},
}

@article {BartoszRalf2,
    AUTHOR = {Kwa\'{s}niewski, Bartosz K. and Meyer, Ralf},
     TITLE = {Essential crossed products for inverse semigroup actions:
              simplicity and pure infiniteness},
   JOURNAL = {Doc. Math.},
    VOLUME = {26},
      YEAR = {2021},
     PAGES = {271--335},
      ISSN = {1431-0635,1431-0643},
   MRCLASS = {46L55 (20M18 22A22 46L05)},
  MRNUMBER = {4246403},
MRREVIEWER = {Takahiro\ Sudo},
}

@book {Lance,
    AUTHOR = {Lance, E. Christopher},
     TITLE = {Hilbert {$C^*$}-modules},
    SERIES = {London Mathematical Society Lecture Note Series},
    VOLUME = {210},
      NOTE = {A toolkit for operator algebraists},
 PUBLISHER = {Cambridge University Press, Cambridge},
      YEAR = {1995},
     PAGES = {x+130},
      ISBN = {0-521-47910-X},
   MRCLASS = {46L05 (46H25 46K05 46L80 46M05 47D25)},
  MRNUMBER = {1325694},
MRREVIEWER = {Robert\ S.\ Doran},
}

@article {LeGall,
    AUTHOR = {Le Gall, Pierre-Yves},
     TITLE = {Th\'{e}orie de {K}asparov \'{e}quivariante et groupo\"{\i}des.
              {I}},
   JOURNAL = {$K$-Theory},
    VOLUME = {16},
      YEAR = {1999},
    NUMBER = {4},
     PAGES = {361--390},
      ISSN = {0920-3036,1573-0514},
   MRCLASS = {19K35 (20L05 22A22 46L80)},
  MRNUMBER = {1686846},
MRREVIEWER = {Evgeniy\ V.\ Troitski\u{\i}},
}

@article {mstt,
    AUTHOR = {McKee, Andrew and Skalski, Adam and Todorov, Ivan G. and
              Turowska, Lyudmila},
     TITLE = {Positive {H}erz-{S}chur multipliers and approximation
              properties of crossed products},
   JOURNAL = {Math. Proc. Cambridge Philos. Soc.},
    VOLUME = {165},
      YEAR = {2018},
    NUMBER = {3},
     PAGES = {511--532},
      ISSN = {0305-0041,1469-8064},
   MRCLASS = {46E10},
  MRNUMBER = {3860401},
MRREVIEWER = {Vesko\ Valov},
}

@article {mtt,
    AUTHOR = {McKee, Andrew and Todorov, Ivan G. and Turowska, Lyudmila},
     TITLE = {Herz-{S}chur multipliers of dynamical systems},
   JOURNAL = {Adv. Math.},
    VOLUME = {331},
      YEAR = {2018},
     PAGES = {387--438},
      ISSN = {0001-8708,1090-2082},
   MRCLASS = {46L07 (47L05 47L10 47L65)},
  MRNUMBER = {3804681},
MRREVIEWER = {Matthew\ P.\ Neal},
}

@article {Mur,
    AUTHOR = {Murphy, Gerard J.},
     TITLE = {Positive definite kernels and {H}ilbert {$C^\ast$}-modules},
   JOURNAL = {Proc. Edinburgh Math. Soc. (2)},
    VOLUME = {40},
      YEAR = {1997},
    NUMBER = {2},
     PAGES = {367--374},
      ISSN = {0013-0915,1464-3839},
   MRCLASS = {46L05 (46E22 46H25)},
  MRNUMBER = {1454031},
MRREVIEWER = {William\ Paschke},
}

@article {Oty,
    AUTHOR = {Oty, Karla J.},
     TITLE = {Fourier-{S}tieltjes algebras of {$r$}-discrete groupoids},
   JOURNAL = {J. Operator Theory},
  FJOURNAL = {Journal of Operator Theory},
    VOLUME = {41},
      YEAR = {1999},
    NUMBER = {1},
     PAGES = {175--197},
}

@article {PackerRaeburn,
    AUTHOR = {Packer, Judith A. and Raeburn, Iain},
     TITLE = {Twisted crossed products of {$C^*$}-algebras},
   JOURNAL = {Math. Proc. Cambridge Philos. Soc.},
  FJOURNAL = {Mathematical Proceedings of the Cambridge Philosophical
              Society},
    VOLUME = {106},
      YEAR = {1989},
    NUMBER = {2},
     PAGES = {293--311},
      ISSN = {0305-0041,1469-8064},
   MRCLASS = {46L55 (22D25 46L40)},
  MRNUMBER = {1002543},
MRREVIEWER = {John\ Quigg},
}

@book {Paulsen,
    AUTHOR = {Paulsen, Vern},
     TITLE = {Completely bounded maps and operator algebras},
    SERIES = {Cambridge Studies in Advanced Mathematics},
    VOLUME = {78},
 PUBLISHER = {Cambridge University Press, Cambridge},
      YEAR = {2002},
     PAGES = {xii+300},
      ISBN = {0-521-81669-6},
   MRCLASS = {46L07 (47A20 47L30)},
  MRNUMBER = {1976867},
MRREVIEWER = {Christian Le Merdy},
}

@article {Paterson,
    AUTHOR = {Paterson, Alan L. T.},
     TITLE = {The {F}ourier algebra for locally compact groupoids},
   JOURNAL = {Canad. J. Math.},
    VOLUME = {56},
      YEAR = {2004},
    NUMBER = {6},
     PAGES = {1259--1289},
      ISSN = {0008-414X,1496-4279},
   MRCLASS = {43A32 (46J99)},
  MRNUMBER = {2102633},
MRREVIEWER = {Humberto\ E.\ Prado},
}

@book {Pedersen,
    AUTHOR = {Pedersen, Gert K.},
     TITLE = {{$C^*$}-algebras and their automorphism groups},
    SERIES = {Pure and Applied Mathematics (Amsterdam)},
   EDITION = {Second},
      NOTE = {Edited and with a preface by S\o ren Eilers and Dorte Olesen},
 PUBLISHER = {Academic Press, London},
      YEAR = {2018},
     PAGES = {xviii+520},
      ISBN = {978-0-12-814122-9},
   MRCLASS = {46Lxx (46-02)},
  MRNUMBER = {3839621},
}

@article {RamsayWalter,
    AUTHOR = {Ramsay, Arlan and Walter, Martin E.},
     TITLE = {Fourier-{S}tieltjes algebras of locally compact groupoids},
   JOURNAL = {J. Funct. Anal.},
    VOLUME = {148},
      YEAR = {1997},
    NUMBER = {2},
     PAGES = {314--367},
      ISSN = {0022-1236,1096-0783},
   MRCLASS = {22D25 (43A40 46L05)},
  MRNUMBER = {1469344},
MRREVIEWER = {Jean\ N.\ Renault},
}

@book {Renault0,
    AUTHOR = {Renault, Jean},
     TITLE = {A groupoid approach to {$C\sp{\ast} $}-algebras},
    SERIES = {Lecture Notes in Mathematics},
    VOLUME = {793},
 PUBLISHER = {Springer, Berlin},
      YEAR = {1980},
     PAGES = {ii+160},
      ISBN = {3-540-09977-8},
   MRCLASS = {46Lxx (22D25 22D40)},
  MRNUMBER = {584266},
MRREVIEWER = {A.\ K.\ Seda},
}

@article {Renault,
    AUTHOR = {Renault, Jean},
     TITLE = {Repr\'{e}sentation des produits crois\'{e}s d'alg\`ebres de
              groupo\"{\i}des},
   JOURNAL = {J. Operator Theory},
  FJOURNAL = {Journal of Operator Theory},
    VOLUME = {18},
      YEAR = {1987},
    NUMBER = {1},
     PAGES = {67--97},
      ISSN = {0379-4024},
   MRCLASS = {46L55 (22A30 46L05)},
  MRNUMBER = {912813},
MRREVIEWER = {Elliot\ C.\ Gootman},
}

@article {Renault2,
    AUTHOR = {Renault, Jean},
     TITLE = {The ideal structure of groupoid crossed product
              {$C^\ast$}-algebras},
      NOTE = {With an appendix by Georges Skandalis},
   JOURNAL = {J. Operator Theory},
  FJOURNAL = {Journal of Operator Theory},
    VOLUME = {25},
      YEAR = {1991},
    NUMBER = {1},
     PAGES = {3--36},
      ISSN = {0379-4024},
   MRCLASS = {46L55},
  MRNUMBER = {1191252},
MRREVIEWER = {John\ Quigg},
}

@article {RenaultFourier,
    AUTHOR = {Renault, Jean},
     TITLE = {The {F}ourier algebra of a measured groupoid and its
              multipliers},
   JOURNAL = {J. Funct. Anal.},
    VOLUME = {145},
      YEAR = {1997},
    NUMBER = {2},
     PAGES = {455--490},
      ISSN = {0022-1236,1096-0783},
   MRCLASS = {43A15 (22A22 43A22 43A35 46L99)},
  MRNUMBER = {1444088},
MRREVIEWER = {Fran\c{c}oise\ Lust-Piquard},
}

@article {Re,
    AUTHOR = {Renault, Jean},
     TITLE = {Cartan subalgebras in {$C^*$}-algebras},
   JOURNAL = {Irish Math. Soc. Bull.},
  FJOURNAL = {Irish Mathematical Society Bulletin},
    		VOLUME = {61},
      YEAR = {2008},
     PAGES = {29--63},
   MRCLASS = {46L85 (37B99)},
  MRNUMBER = {2460017},
MRREVIEWER = {Benton L. Duncan},
}

@article {RussoDye,
    AUTHOR = {Russo, Bernard and Dye, Henry A.},
     TITLE = {A note on unitary operators in {$C\sp{\ast} $}-algebras},
   JOURNAL = {Duke Math. J.},
  FJOURNAL = {Duke Mathematical Journal},
    VOLUME = {33},
      YEAR = {1966},
     PAGES = {413--416},
      ISSN = {0012-7094,1547-7398},
   MRCLASS = {46.65},
  MRNUMBER = {193530},
MRREVIEWER = {J.\ Glimm},
}

@article {Schweizer,
    AUTHOR = {Schweizer, J\"{u}rgen},
     TITLE = {Hilbert {$C^\ast$}-modules with a predual},
   JOURNAL = {J. Operator Theory},
    VOLUME = {48},
      YEAR = {2002},
    NUMBER = {3},
     PAGES = {621--632},
      ISSN = {0379-4024,1841-7744},
   MRCLASS = {46L08},
  MRNUMBER = {1962475},
MRREVIEWER = {William\ Paschke},
}

@article {Sieben98,
    AUTHOR = {Sieben, N\'{a}ndor},
     TITLE = {{$C^*$}-crossed products by twisted inverse semigroup actions},
   JOURNAL = {J. Operator Theory},
  FJOURNAL = {Journal of Operator Theory},
    VOLUME = {39},
      YEAR = {1998},
    NUMBER = {2},
     PAGES = {361--393},
      ISSN = {0379-4024,1841-7744},
   MRCLASS = {46L55},
  MRNUMBER = {1620499},
MRREVIEWER = {Steven\ P.\ Kaliszewski},
}

@incollection {Sims,
    AUTHOR = {Sims, Aidan},
    TITLE={Hausdorff \'etale groupoids and their $\cst$-algebras}, 
    BOOKTITLE = {Operator algebras and dynamics: groupoids, crossed products,
              and {R}okhlin dimension},
    SERIES = {Advanced Courses in Mathematics. CRM Barcelona},
    EDITOR = {Perera, Francesc},
      NOTE = {Lecture notes from the Advanced Course held at Centre de
              Recerca Matem\`atica (CRM) Barcelona, March 13--17, 2017},
 PUBLISHER = {Birkh\"{a}user/Springer, Cham},
      YEAR = {[2020] \copyright 2020},
     PAGES = {x+163},
      ISBN = {978-3-030-39712-8; 978-3-030-39713-5},
   MRCLASS = {46-02 (22A22 37A55 46L35 46L55 47L65)},
  MRNUMBER = {4321941},
}

@article {Sims_Williams,
    AUTHOR = {Sims, Aidan and Williams, Dana P.},
     TITLE = {Amenability for {F}ell bundles over groupoids},
   JOURNAL = {Illinois J. Math.},
    VOLUME = {57},
      YEAR = {2013},
    NUMBER = {2},
     PAGES = {429--444},
      ISSN = {0019-2082,1945-6581},
   MRCLASS = {46L55 (46L05)},
  MRNUMBER = {3263040},
MRREVIEWER = {Clifford\ Neves},
}

@article {Suzuki,
    AUTHOR = {Suzuki, Yuhei},
     TITLE = {Simple equivariant {$\rm C^*$}-algebras whose full and reduced
              crossed products coincide},
   JOURNAL = {J. Noncommut. Geom.},
  FJOURNAL = {Journal of Noncommutative Geometry},
    VOLUME = {13},
      YEAR = {2019},
    NUMBER = {4},
     PAGES = {1577--1585},
}

@article {Takeishi,
    AUTHOR = {Takeishi, Takuya},
     TITLE = {On nuclearity of {$C^*$}-algebras of {F}ell bundles over
              \'{e}tale groupoids},
   JOURNAL = {Publ. Res. Inst. Math. Sci.},
    VOLUME = {50},
      YEAR = {2014},
    NUMBER = {2},
     PAGES = {251--268},
      ISSN = {0034-5318,1663-4926},
   MRCLASS = {46L55},
  MRNUMBER = {3223473},
MRREVIEWER = {Valentin\ Deaconu},
}

@article {Walter,
    AUTHOR = {Walter, Martin E.},
     TITLE = {Dual algebras},
   JOURNAL = {Math. Scand.},
    VOLUME = {58},
      YEAR = {1986},
    NUMBER = {1},
     PAGES = {77--104},
      ISSN = {0025-5521,1903-1807},
   MRCLASS = {46K05 (22D15 22D25 22D35 43A40 46H05)},
  MRNUMBER = {845488},
MRREVIEWER = {Jean-Marie\ Schwartz},
}

@article {Willett,
    AUTHOR = {Willett, Rufus},
     TITLE = {A non-amenable groupoid whose maximal and reduced
              {$C^*$}-algebras are the same},
   JOURNAL = {M\"{u}nster J. Math.},
    VOLUME = {8},
      YEAR = {2015},
    NUMBER = {1},
     PAGES = {241--252},
      ISSN = {1867-5778,1867-5786},
   MRCLASS = {46L55 (22A22 22D25 46L05)},
  MRNUMBER = {3549528},
MRREVIEWER = {Jean\ N.\ Renault},
}

@book {WLbook,
    AUTHOR = {Williams, Dana P.},
     TITLE = {Crossed products of {$C{^\ast}$}-algebras},
    SERIES = {Mathematical Surveys and Monographs},
    VOLUME = {134},
 PUBLISHER = {American Mathematical Society, Providence, RI},
      YEAR = {2007},
     PAGES = {xvi+528},
      ISBN = {978-0-8218-4242-3; 0-8218-4242-0},
   MRCLASS = {46-02 (22D25 46L05 46L35 46L55 46L85)},
  MRNUMBER = {2288954},
MRREVIEWER = {Jonathan\ M.\ Rosenberg},
}
\bibliographystyle{amsalpha}
 
\end{document}